\documentclass{amsart}
\usepackage{amssymb}
\usepackage{amsfonts}
\usepackage{geometry}
\usepackage{color}

\newtheorem{theorem}{Theorem}
\theoremstyle{plain}

\newtheorem{conclusion}[theorem]{Conclusion}

\newtheorem{corollary}[theorem]{Corollary}

\newtheorem{definition}[theorem]{Definition}

\newtheorem{lemma}[theorem]{Lemma}
\newtheorem{notation}[theorem]{Notation}
\newtheorem{problem}[theorem]{Problem}
\newtheorem{proposition}[theorem]{Proposition}
\newtheorem{remark}[theorem]{Remark}

\numberwithin{equation}{section}
\input{tcilatex}
\begin{document}
\title[The Hyt\"{o}nen-Vuorinen $L^{p}$ conjecture with an extended energy
side condition]{The Hyt\"{o}nen-Vuorinen $L^{p}$ conjecture for the Hilbert
transform, with an extended energy side condition, when $\frac{4}{3}<p<4$
and the measures share no point\ masses}
\author[E. T. Sawyer]{Eric T. Sawyer$^\dagger$}
\address{Eric T. Sawyer, Department of Mathematics and Statistics\\
McMaster University\\
1280 Main Street West\\
Hamilton, Ontario L8S 4K1 Canada}
\thanks{$\dagger $ Research supported in part by a grant from the National
Science and Engineering Research Council of Canada.}
\email{sawyer@mcmaster.ca}
\author[B. D. Wick]{Brett D. Wick$^\ddagger$}
\address{Brett D. Wick, Department of Mathematics \& Statistics, Washington
University -- St. Louis, One Brookings Drive, St. Louis, MO USA 63130-4899.}
\email{wick@math.wustl.edu}
\thanks{$\ddagger $ B. D. Wick's research is supported in part by National
Science Foundation Grants DMS \# 1800057, \# 2054863, and \# 2000510 and
Australian Research Council -- DP 220100285.}
\date{\today }

\begin{abstract}
In the case $\frac{4}{3}<p<4$, and assuming $\sigma $ and $\omega $ are
locally finite positive Borel measures on $\mathbb{R}$ without common point
masses, we prove variants of two conjectures of T. Hyt\"{o}nen and E.
Vuorinen from 2018 on \emph{two weight} testing theorems for the Hilbert
transform on weighted\emph{\ L}$^{p}$ spaces, but only in the presence of
extended energy conditions. Namely, if we assume the extended energy
condition 
\begin{equation*}
\int_{\mathbb{R}}\left( \sum_{W\in \mathbb{W}}\left( \frac{\mathrm{P}\left(
W,\Phi _{F}f\sigma \right) }{\ell \left( W\right) }\right) ^{2}\left\vert 
\mathsf{P}_{\left[ W\right] }^{\omega }\right\vert ^{2}Z\right) ^{\frac{p}{2}%
}d\omega \lesssim \mathcal{E}_{p}^{\ell ^{2},\limfunc{ext}}\left( \sigma
,\omega \right) ^{p}\left\Vert f\right\Vert _{L^{p}\left( \sigma \right)
}^{p},
\end{equation*}%
for $f\geq 0$, along with the dual extended energy condition, then the two
weight norm inequality%
\begin{equation*}
\left\Vert H_{\sigma }f\right\Vert _{L^{p}\left( \omega \right) }\lesssim
\left\Vert f\right\Vert _{L^{p}\left( \sigma \right) },\ \ \ \ \ \frac{4}{3}%
<p<4,
\end{equation*}%
holds\newline
(1) \emph{if and only if} the global quadratic interval testing conditions
hold, \newline
(2) \emph{if and only if} the local quadratic interval testing, the
quadratic Muckenhoupt, the quadratic weak boundedness conditions all hold. 
\newline

We also give a slight improvement of the second characterization by
replacing the quadratic Muckenhoupt conditions with two smaller conditions.
\end{abstract}

\maketitle
\tableofcontents

\section{Introduction}

Almost a century ago in 1927, M. Riesz \cite{Rie} proved that the conjugate
function is bounded on $L^{p}\left( \mathbb{T}\right) $ of the circle $%
\mathbb{T}$ for all $1<p<\infty $. Forty-six years later in 1973, R. Hunt,
B. Muckenhoupt and R.L. Wheeden \cite{HuMuWh} showed that the conjugate
function on $\mathbb{T}$, equivalently the Hilbert transform on the real
line $\mathbb{R}$, extends to a bounded operator on a weighted
space\thinspace $L^{p}\left( w\right) $, i.e.%
\begin{equation}
\int_{\mathbb{R}}\left\vert Hf\left( x\right) \right\vert ^{p}w\left(
x\right) dx\lesssim \int_{\mathbb{R}}\left\vert f\left( x\right) \right\vert
^{p}w\left( x\right) dx,\ \ \ \ \ \text{for all }f\in L^{p}\left( w\right) ,
\label{WNI}
\end{equation}%
if and only if the weight $w$ satisfies the remarkable $A_{p}$ condition of
Muckenhoupt\footnote{%
A function theoretic characterization is available in the case $p=2$. For
example, in 1960, H. Helson and G. Szeg\"{o} \cite{HeSz} showed that (\ref%
{WNI}) holds for $p=2$ if and only if $w=e^{u+Hv}$ for some bounded
functions $u,v$ with $\left\Vert u\right\Vert _{\infty }<\frac{\pi }{2}$.},%
\begin{equation*}
\left( \frac{1}{\left\vert I\right\vert }\int_{I}w\left( x\right) dx\right)
\left( \frac{1}{\left\vert I\right\vert }\int_{I}\left( \frac{1}{w\left(
x\right) }\right) ^{\frac{1}{p-1}}dx\right) ^{p-1}\leq A_{p}\left( w\right)
,\ \ \ \ \ \text{for all intervals }I\subset \mathbb{R}.
\end{equation*}

However, an extension of the $A_{p}\left( w\right) $ characterization, even
in the case $p=2$, to the setting of two weights has proven to be much more
difficult, especially in view of the exceptional beauty and utility of the $%
A_{p}\left( w\right) $ condition, which has greatly spoiled expectations%
\footnote{%
The Helson-Szeg\"{o} result was extended in 1983 to the two weight setting
for all $1<p<\infty $ by M. Cotlar and C. Sadosky \cite{CoSa}.}. Many
sufficient conditions involving `bumped-up' variants of $A_{p}\left(
w\right) $ have been obtained by numerous authors over the ensuing years,
but while such conditions are in a sense `checkable', they cannot
characterize the two weight inequality for the Hilbert transform due to its
instability, at least in the case $p=2$, see \cite{AlLuSaUr 2}.

Instead, using the `testing condition' approach originating with E. Sawyer 
\cite{Saw} and G. David and J.-L. Journ\'{e} \cite{DaJo}, groundbreaking
strides were made in 2004 toward a characterization in the case $p=2$ by F.
Nazarov, S. Treil and A. Volberg \cite{NTV4}. Finally, in 2014, the
Nazarov-Treil-Volberg $T1$ conjecture \cite{Vol} on the boundedness of the
Hilbert transform from one weighted space $L^{2}\left( \sigma \right) $ to
another $L^{2}\left( \omega \right) $, was settled affirmatively, when the
measures $\sigma ,\omega $ share no common point masses, in the two part
paper \cite{LaSaShUr3};\cite{Lac} of M. Lacey, C.-Y. Shen, E. Sawyer, I.
Uriarte-Tuero; M. Lacey. Subsequently, T. Hyt\"{o}nen \cite{Hyt} removed the
assumption of no common point masses. A number of modifications of the NTV
approach were required in this solution, including the use in \cite%
{LaSaShUr3} of functional energy and the two weight inequalities for Poisson
integrals from \cite{Saw3}, as well as M. Lacey's use of a size condition,
an upside-down corona construction, and a recursion to bound the obstinate
stopping form in \cite{Lac}, and the use in \cite{Hyt} of two weight
inequalities for Poisson integrals with holes.

The testing approach in the case $p\neq 2$ has proven even more challenging,
going back at least to the 2015 primer of M. Lacey \cite[page 18]{Lac2} - a
major source of difficulty being the resistance of known arguments to using
energy conditions when $p\neq 2$. Moreover, it was recently shown in M.
Alexis, J.-L. Luna-Garcia, E. Sawyer and I. Uriarte-Tuero \cite{AlLuSaUr 2},
that the usual scalar testing and Muckenhoupt conditions do not suffice for
boundedness of the Hilbert transform for any $p\neq 2$.

On the other hand, T. Hyt\"{o}nen and E. Vuorinen \cite{HyVu} have made two
challenging conjectures regarding extensions to weighted $L^{p}$ norms, $%
1<p<\infty $, that involve \emph{quadratic} testing, Muckenhoupt, and weak
boundedness conditions, described in a moment.

The purpose of this paper is to prove two weaker variants of the conjectures
of Hyt\"{o}nen and Vuorinen in the special case that $\frac{4}{3}<p<4$ and
the measures $\sigma $ and $\omega $ share no point masses, and to give a
slight improvement of the second variant in this setting, thus completing a
first step toward a two-weight $L^{p}$ theory for the Hilbert transform. We
remark that there is only one place in the proof where the restriction $p<4$
arises, namely in Lemma \ref{final level} that plays a supporting role in
bounding the \emph{stopping form} at the end of the paper. When possible we
will state and prove our supporting results for $1<p<\infty $ and general
measures, but for simplicity we assume in our theorems that the $\sigma $
and $\omega $ share no common point masses. Here are two open problems
weaker than the full conjectures of Hyt\"{o}nen and Vuorinen. Of course we
can ask for inclusion of the necessary extended energy conditions in these
two problems as well.

\begin{problem}
Do the conjectures of Hyt\"{o}nen and Vuorinen hold for measures \emph{with}
common point masses when $\frac{4}{3}<p<4$?
\end{problem}

\begin{problem}
Do the conjectures of Hyt\"{o}nen and Vuorinen hold for measures \emph{%
without} common point masses when $1<p<\infty $?
\end{problem}

In order to state our theorems, we need a number of definitions. Let $\mu $
be a positive locally finite Borel measure on $\mathbb{R}$, let $\mathcal{D}$
be a dyadic grid on $\mathbb{R}$, and let $\left\{ \bigtriangleup _{Q}^{\mu
}\right\} _{Q\in \mathcal{D}}$ be the associated set of weighted Haar
projections on $L^{2}\left( \mu \right) $, see e.g. \cite{NTV4}. In
particular $\bigtriangleup _{Q}^{\mu }f\left( x\right) =\left\langle
f,h_{Q}^{\mu }\right\rangle _{\mu }h_{Q}^{\mu }\left( x\right) $ where $%
\left\{ h_{Q}^{\mu }\right\} _{Q\in \mathcal{D}}$ is the associated
orthonormal Haar basis (that may include averages over infinite intervals of
finite $\mu $-measure). Denote the Hilbert transform $H$ of a signed measure 
$\nu $ defined by%
\begin{equation*}
H\nu \left( x\right) \equiv \func{pv}\int_{\mathbb{R}}\frac{1}{y-x}d\nu
\left( y\right) ,
\end{equation*}%
and for $1<p<\infty $, consider the two weight norm inequality,%
\begin{equation}
\left\Vert H_{\sigma }f\right\Vert _{L^{p}\left( \omega \right) }=\left(
\int_{\mathbb{R}}\left\vert H_{\sigma }f\right\vert ^{p}d\omega \right) ^{%
\frac{1}{p}}\leq \mathfrak{N}_{H,p}\left( \sigma ,\omega \right) \left(
\int_{\mathbb{R}}\left\vert f\right\vert ^{p}d\sigma \right) ^{\frac{1}{p}}=%
\mathfrak{N}_{H,p}\left( \sigma ,\omega \right) \left\Vert f\right\Vert
_{L^{p}\left( \sigma \right) },  \label{Hilbert'}
\end{equation}%
where $\mathfrak{N}_{H,p}\left( \sigma ,\omega \right) $ is the best
constant taken over all admissible truncations of the Hilbert transform, and
where $H_{\sigma }f\equiv H\left( f\sigma \right) $, see e.g. \cite[%
subsubsection 1.2.1 on page 130]{SaShUr10} for more detail on this
interpretation of the norm inequality that avoids consideration of principal
values (as pioneered by X. Tolsa). The following definitions are for the
most part from Hyt\"{o}nen and Vuorinen \cite{HyVu}, except for the extended
energy conditions.

\subsection{Quadratic testing conditions}

The \emph{local scalar} (forward) interval\emph{\ }testing characteristic $%
\mathfrak{T}_{H,p}^{\func{loc}}\left( \sigma ,\omega \right) $ is defined as
the best constant in%
\begin{equation}
\left\Vert \mathbf{1}_{I}H_{\sigma }\mathbf{1}_{I}\right\Vert _{L^{p}\left(
\omega \right) }\leq \mathfrak{T}_{H,p}^{\func{loc}}\left( \sigma ,\omega
\right) \left\vert I\right\vert _{\sigma }^{\frac{1}{p}},
\label{scalar cube testing}
\end{equation}%
and the \emph{local quadratic} (forward) interval testing characteristic $%
\mathfrak{T}_{H,p}^{\ell ^{2},\func{loc}}\left( \sigma ,\omega \right) $, is
defined as the best constant in%
\begin{equation}
\left\Vert \left( \sum_{i=1}^{\infty }\left( a_{i}\mathbf{1}%
_{I_{i}}H_{\sigma }\mathbf{1}_{I_{i}}\right) ^{2}\right) ^{\frac{1}{2}%
}\right\Vert _{L^{p}\left( \omega \right) }\leq \mathfrak{T}_{H,p}^{\ell
^{2},\func{loc}}\left( \sigma ,\omega \right) \left\Vert \left(
\sum_{i=1}^{\infty }\left( a_{i}\mathbf{1}_{I_{i}}\right) ^{2}\right) ^{%
\frac{1}{2}}\right\Vert _{L^{p}\left( \sigma \right) },
\label{quad cube testing}
\end{equation}%
taken over all sequences of intervals $\left\{ I_{i}\right\} _{i=1}^{\infty
} $, and all sequences of positive numbers $\left\{ a_{i}\right\}
_{i=1}^{\infty }$. The dual scalar and quadratic interval testing
characteristics are obtained by interchanging $\sigma $ and $\omega $, and
replacing $p$ with $p^{\prime }$.

The \emph{global scalar} interval\emph{\ }testing characteristic $\mathfrak{T%
}_{H,p}^{\limfunc{glob}}\left( \sigma ,\omega \right) $ is defined as the
best constant in%
\begin{equation}
\left\Vert H_{\sigma }\mathbf{1}_{I}\right\Vert _{L^{p}\left( \omega \right)
}\leq \mathfrak{T}_{H,p}^{\limfunc{glob}}\left( \sigma ,\omega \right)
\left\vert I\right\vert _{\sigma }^{\frac{1}{p}},  \label{global scalar}
\end{equation}%
and the \emph{global quadratic} testing characteristic $\mathfrak{T}%
_{H,p}^{\ell ^{2},\limfunc{glob}}\left( \sigma ,\omega \right) $ is defined
as the best constant in%
\begin{equation}
\left\Vert \left( \sum_{i=1}^{\infty }\left( a_{i}H_{\sigma }\mathbf{1}%
_{I_{i}}\right) ^{2}\right) ^{\frac{1}{2}}\right\Vert _{L^{p}\left( \omega
\right) }\leq \mathfrak{T}_{H,p}^{\ell ^{2},\limfunc{glob}}\left( \sigma
,\omega \right) \left\Vert \left( \sum_{i=1}^{\infty }\left( a_{i}\mathbf{1}%
_{I_{i}}\right) ^{2}\right) ^{\frac{1}{2}}\right\Vert _{L^{p}\left( \sigma
\right) },  \label{glob}
\end{equation}%
taken over all sequences of intervals $\left\{ I_{i}\right\} _{i=1}^{\infty
} $, and all sequences of positive numbers $\left\{ a_{i}\right\}
_{i=1}^{\infty }$, and similarly for the dual scalar and global
characteristics $\mathfrak{T}_{H,p^{\prime }}^{\limfunc{glob}}\left( \sigma
,\omega \right) $ and $\mathfrak{T}_{H,p^{\prime }}^{\ell ^{2},\limfunc{glob}%
}\left( \omega ,\sigma \right) $. Note that by Khintchine's inequality, the
global quadratic condition (\ref{glob}) can be interpreted as `average
testing' over finite linear spans of indicators of intervals,%
\begin{equation*}
\mathbb{E}_{\pm }\left\Vert H_{\sigma }\left( \sum_{i=1}^{M}\pm a_{i}\mathbf{%
1}_{I_{i}}\right) \right\Vert _{L^{p}\left( \omega \right) }\lesssim \mathbb{%
E}_{\pm }\left\Vert \sum_{i=1}^{M}\pm a_{i}\mathbf{1}_{I_{i}}\right\Vert
_{L^{p}\left( \sigma \right) },
\end{equation*}%
which in the case $p=2$ reduces to the global \emph{scalar} testing
condition (\ref{global scalar}).

\subsection{Extended energy conditions}

Given any interval $F\in \mathcal{D}$, we define the $\left( r,\varepsilon
\right) $\emph{-Whitney} collection $\mathcal{M}_{\left( r,\varepsilon
\right) -\limfunc{deep}}\left( F\right) $ of $F$ to be the set of dyadic
subintervals $W\subset F$ that are maximal with respect to the property that 
$W\subset _{r,\varepsilon }F$. Clearly the intervals in $\mathcal{M}_{\left(
r,\varepsilon \right) -\limfunc{deep}}\left( F\right) $ form a pairwise
disjoint decomposition of $F$.where the extended energy characteristic $%
\mathcal{E}_{p}^{\ell ^{2},\limfunc{ext}}\left( \sigma ,\omega \right) $ is
the best constant in the inequality,

The extended energy characteristic $\mathcal{E}_{p}^{\ell ^{2},\limfunc{ext}%
}\left( \sigma ,\omega \right) $ is defined to be the smallest constant in
the inequality 
\begin{equation}
\int_{\mathbb{R}}\left( \sum_{F\in \mathcal{F}}\sum_{W\in \mathcal{M}%
_{\left( r,\varepsilon \right) -\limfunc{deep}}\left( F\right) \cap \mathcal{%
C}_{\mathcal{F}}\left( F\right) }\left( \frac{\mathrm{P}\left( W,\Phi
_{F}f\sigma \right) }{\ell \left( W\right) }\right) ^{2}\left\vert \mathsf{P}%
_{\mathcal{C}_{\mathcal{F}}\left( F\right) \cap \mathcal{D}\left[ W\right]
}^{\omega }\right\vert ^{2}Z\right) ^{\frac{p}{2}}d\omega \lesssim \mathcal{E%
}_{p}^{\ell ^{2},\limfunc{ext}}\left( \sigma ,\omega \right) ^{p}\left\Vert
f\right\Vert _{L^{p}\left( \sigma \right) }^{p},\ \ \ \ \ \text{for }f\geq 0,
\label{ext ener}
\end{equation}%
where $\mathcal{F}$ is the collection of stopping times associated with $f$,
and%
\begin{equation*}
\Phi _{F}f\equiv \mathbb{E}_{F}^{\sigma }f+f_{F}=\mathbb{E}_{F}^{\sigma
}f+\sum_{I:\ I\supsetneqq F}\mathbf{1}_{I_{F}}\bigtriangleup _{I}^{\sigma }f.
\end{equation*}

\begin{remark}
Note that (\ref{ext ener}) is a \emph{positive} inequality, and like the
quadratic Muckenhoupt characteristics defined in the next subsection,
depends only on the measures $\sigma $ and $\omega $, and not on any
singular operator such as the Hilbert transform. In the special case $p=2$,
the characteristic $\mathcal{E}_{p}^{\ell ^{2},\limfunc{ext}}\left( \sigma
,\omega \right) $ is controlled by testing and Muckenhoupt characteristics,
even the logically larger functional energy characteristic, see e.g. \cite%
{LaSaShUr3} and \cite{SaShUr7}, but we are unable to extend all of the
Poisson inequality arguments used there when $p\neq 2$.
\end{remark}

There is a similar dual quadratic extended energy characteristic $\mathcal{E}%
_{p^{\prime }}^{\ell ^{2},\limfunc{ext},\ast }\left( \omega ,\sigma \right) $
in which $\sigma $ and $\omega $ are interchanged, the index $p$ is replaced
by its conjugate index $p^{\prime }$,%
\begin{equation*}
\int_{\mathbb{R}}\left( \sum_{G\in \mathcal{G}}\sum_{W\in \mathcal{M}%
_{\left( r,\varepsilon \right) -\limfunc{deep}}\left( G\right) \cap \mathcal{%
C}_{\mathcal{G}}\left( G\right) }\left( \frac{\mathrm{P}\left( W,\Phi
_{G}g\omega \right) }{\ell \left( W\right) }\right) ^{2}\left\vert \mathsf{P}%
_{\mathcal{C}_{\mathcal{G}}\left( G\right) \cap \mathcal{D}\left[ W\right]
}^{\sigma }\right\vert ^{2}Z\right) ^{\frac{p^{\prime }}{2}}d\sigma \lesssim 
\mathcal{E}_{p^{\prime }}^{\ell ^{2},\limfunc{ext},\ast }\left( \omega
,\sigma \right) ^{p^{\prime }}\left\Vert g\right\Vert _{L^{p^{\prime
}}\left( \omega \right) }^{p^{\prime }},\ \ \ \ \ \text{for }g\geq 0,
\end{equation*}%
where $\mathcal{G}$ is the collection of stopping times associated with $g$.

\begin{remark}
\label{ref}Note that the Whitney intervals $W$, in the above definition of
the quadratic extended energy characteristic $\mathcal{E}_{p}^{\ell ^{2},%
\limfunc{ext}}\left( \sigma ,\omega \right) $, are restricted to lie in the
corona $\mathcal{C}_{\mathcal{F}}\left( F\right) $. This departure from the
analogous definitions in \cite{LaSaShUr3}, \cite{SaShUr7} and elsewhere in
the literature, constitutes an important simplification of the argument in
the case $p=2$, and is largely responsible for our successful control of the
far form when $p\neq 2$. The key point here is that any $W\in \mathcal{D}$
lies in $\mathcal{M}_{\left( r,\varepsilon \right) -\limfunc{deep}}\left(
F\right) \cap \mathcal{C}_{\mathcal{F}}\left( F\right) $ for at most one $%
F\in \mathcal{F}$.
\end{remark}

\subsection{Main theorems}

Now we can state our first main result, which in the case of no common point
masses, extends the results in \cite{LaSaShUr3},\cite{Lac} and \cite{Hyt} to 
$\frac{4}{3}<p<4$ by replacing $2$ with $p$ at the expense of introducing
quadratic testing characteristics, and answers in the affirmative the first
conjecture in \cite{HyVu} in this range, provided we include the necessary
extended energy conditions in the characterizing conditions. The case of
doubling measures was obtained in \cite{SaWi}, even for vector Riesz
transforms in higher dimensions and $1<p<\infty $, and the case $p=2$ was
done earlier in \cite{AlSaUr}. The case of \emph{dyadic shifts} was resolved
much earlier for $1<p<\infty $ and general measures in \cite{Vuo}.

\begin{theorem}[first conjecture of Hyt\"{o}nen and Vuorinen]
\label{main glob}Suppose $\frac{4}{3}<p<4$, and that $\sigma $ and $\omega $
are locally finite positive Borel measures on $\mathbb{R}$ without common
point masses. Suppose also that the extended energy characteristics $%
\mathcal{E}_{p}^{\ell ^{2},\limfunc{ext}}\left( \sigma ,\omega \right) $ and 
$\mathcal{E}_{p^{\prime }}^{\ell ^{2},\limfunc{ext},\ast }\left( \omega
,\sigma \right) $ are finite. Then under these side conditions, the two
weight norm inequality (\ref{Hilbert'}) holds \emph{if and only if} the
global quadratic interval testing conditions (\ref{glob}) hold. Moreover, we
have%
\begin{equation*}
\mathfrak{N}_{H,p}\left( \sigma ,\omega \right) \lesssim \mathfrak{T}%
_{H,p}^{\ell ^{2},\limfunc{glob}}\left( \sigma ,\omega \right) +\mathfrak{T}%
_{H,p^{\prime }}^{\ell ^{2},\limfunc{glob}}\left( \omega ,\sigma \right) +%
\mathcal{E}_{p}^{\ell ^{2},\limfunc{ext}}\left( \sigma ,\omega \right) +%
\mathcal{E}_{p^{\prime }}^{\ell ^{2},\limfunc{ext},\ast }\left( \omega
,\sigma \right) .
\end{equation*}
\end{theorem}

In order to replace global quadratic testing with the smaller local
quadratic testing, we must include additional quadratic Muckenhoupt
characteristics and a quadratic weak boundedness characteristic, which we
now define.

\subsubsection{Quadratic and scalar tailed Muckenhoupt conditions}

The \emph{global} quadratic Muckenhoupt characteristic $\mathcal{A}%
_{p}^{\ell ^{2},\limfunc{glob}}\left( \sigma ,\omega \right) $ of Hyt\"{o}%
nen and Vuorinen is defined as the best constant in%
\begin{equation}
\left\Vert \left( \sum_{i=1}^{\infty }\left( \mathbf{1}_{I_{i}}\int_{\mathbb{%
R}\setminus I_{i}}\frac{f_{i}\left( x\right) }{\left\vert x-c_{i}\right\vert 
}d\sigma \left( x\right) \right) ^{2}\right) ^{\frac{1}{2}}\right\Vert
_{L^{p}\left( \omega \right) }\leq \mathcal{A}_{p}^{\ell ^{2},\limfunc{glob}%
}\left( \sigma ,\omega \right) \left\Vert \left( \sum_{i=1}^{\infty
}f_{i}^{2}\right) ^{\frac{1}{2}}\right\Vert _{L^{p}\left( \sigma \right) },
\label{quad A2 cond}
\end{equation}%
taken over all sequences of intervals $\left\{ I_{i}\right\} _{i=1}^{\infty
} $ with centers $\left\{ c_{i}\right\} _{i=1}^{\infty }$, and all sequences
of functions $\left\{ f_{i}\right\} _{i=1}^{\infty }$. There is also the
usual dual characteristic defined by interchanging $\sigma $ and $\omega $,
and replacing $p$ by $p^{\prime }$.

We now introduce three smaller Muckenhoupt characteristics whose use we will
track throughout the proof - only the disjoint form requires the triple
Muckenhoupt characteristic, and only the outer form bound requires the
kernel Muckenhoupt characteristic, while the offset Muckenhoupt
characteristic suffices elsewhere.

The smaller \emph{offset }quadratic Muckenhoupt characteristic $A_{p}^{\ell
^{2},\limfunc{offset}}\left( \sigma ,\omega \right) $ is defined as the best
constant in%
\begin{equation}
\left\Vert \left( \sum_{i=1}^{\infty }\left\vert a_{i}\frac{\left\vert
I_{i}^{\ast }\right\vert _{\sigma }}{\left\vert I_{i}^{\ast }\right\vert }%
\right\vert ^{2}\mathbf{1}_{I_{i}}\right) ^{\frac{1}{2}}\right\Vert
_{L^{p}\left( \omega \right) }\leq A_{p}^{\ell ^{2},\limfunc{offset}}\left(
\sigma ,\omega \right) \left\Vert \left( \sum_{i=1}^{\infty }\left\vert
a_{i}\right\vert ^{2}\mathbf{1}_{I_{i}^{\ast }}\right) ^{\frac{1}{2}%
}\right\Vert _{L^{p}\left( \sigma \right) },  \label{quad A2 tailless}
\end{equation}%
where $I_{i}^{\ast }$ is taken over the finitely many dyadic intervals $%
I_{i}^{\ast }$ disjoint from $I_{i}$ and such that $\ell \left( I_{i}^{\ast
}\right) =\ell \left( I_{i}\right) $ and $\limfunc{dist}\left( I_{i}^{\ast
},I_{i}\right) \leq r\ell \left( I_{i}\right) $, and all sequences numbers $%
a_{i}$, where $r$ is the goodness constant from \cite{NTV4} and \cite%
{LaSaShUr3} - see also the section on preliminaries. There is again the
usual dual characteristic defined by interchanging $\sigma $ and $\omega $,
and replacing $p$ by $p^{\prime }$.

There is also an intermediate \emph{triple} quadratic Muckenhoupt
characteristic $A_{p}^{\ell ^{2},\limfunc{trip}}\left( \sigma ,\omega
\right) $ defined as the best constant in%
\begin{equation}
\left\Vert \left( \sum_{i=1}^{\infty }\left( \mathbf{1}_{I_{i}}\int_{3I_{i}%
\setminus I_{i}}\frac{f_{i}\left( x\right) }{\left\vert x-c_{i}\right\vert }%
d\sigma \left( x\right) \right) ^{2}\right) ^{\frac{1}{2}}\right\Vert
_{L^{p}\left( \omega \right) }\leq A_{p}^{\ell ^{2},\limfunc{trip}}\left(
\sigma ,\omega \right) \left\Vert \left( \sum_{i=1}^{\infty
}f_{i}^{2}\right) ^{\frac{1}{2}}\right\Vert _{L^{p}\left( \sigma \right) },
\label{triple Muck}
\end{equation}%
taken over all sequences of intervals $\left\{ I_{i}\right\} _{i=1}^{\infty
} $ with centers $\left\{ c_{i}\right\} _{i=1}^{\infty }$, and all sequences
of functions $\left\{ f_{i}\right\} _{i=1}^{\infty }$ with $\limfunc{supp}%
f_{i}\subset 3I_{i}\setminus I_{i}$, as well as the dual such characteristic 
$A_{p^{\prime }}^{\ell ^{2},\limfunc{trip}}\left( \omega ,\sigma \right) $.

Next, there is the \emph{scalar tailed} Muckenhoupt characteristic defined
by,%
\begin{equation}
\mathcal{A}_{p}\left( \sigma ,\omega \right) \approx \sup_{I\text{ an
interval}}\left( \frac{1}{\left\vert I\right\vert }\int \left( \frac{\ell
\left( I\right) }{\ell \left( I\right) +\limfunc{dist}\left( x,I\right) }%
\right) ^{p}d\omega \left( x\right) \right) ^{\frac{1}{p}}\ \left( \frac{%
\left\vert I\right\vert _{\sigma }}{\left\vert I\right\vert }\right) ^{\frac{%
1}{p^{\prime }}}.  \label{dual scalar tail}
\end{equation}%
We have 
\begin{eqnarray*}
A_{p}^{\ell ^{2},\limfunc{offset}}\left( \sigma ,\omega \right) &\lesssim
&A_{p}^{\ell ^{2},\limfunc{trip}}\left( \sigma ,\omega \right) \lesssim 
\mathcal{A}_{p}^{\ell ^{2},\limfunc{glob}}\left( \sigma ,\omega \right) , \\
\mathcal{A}_{p}\left( \sigma ,\omega \right) &\lesssim &\mathcal{A}%
_{p}^{\ell ^{2},\limfunc{glob}}\left( \sigma ,\omega \right) ,
\end{eqnarray*}%
but there is no obvious relationship we can see between $\mathcal{A}%
_{p}\left( \sigma ,\omega \right) $ and $A_{p}^{\ell ^{2},\limfunc{offset}%
}\left( \sigma ,\omega \right) $.

\subsubsection{Quadratic weak boundedness property}

The \emph{quadratic} weak boundedness characteristic $\mathcal{WBP}%
_{H,p}^{\ell ^{2}}\left( \sigma ,\omega \right) $ is defined as the best
constant in%
\begin{eqnarray}
&&\sum_{i=1}^{\infty }\left\vert \int_{\mathbb{R}}a_{i}H_{\sigma }\mathbf{1}%
_{I_{i}}\left( x\right) b_{i}\mathbf{1}_{J_{i}}\left( x\right) d\omega
\left( x\right) \right\vert  \label{WBP HV} \\
&\leq &\mathcal{WBP}_{H,p}^{\ell ^{2}}\left( \sigma ,\omega \right)
\left\Vert \left( \sum_{i=1}^{\infty }\left\vert a_{i}\mathbf{1}%
_{I_{i}}\right\vert ^{2}\right) ^{\frac{1}{2}}\right\Vert _{L^{p}\left(
\sigma \right) }\left\Vert \left( \sum_{i=1}^{\infty }\left\vert b_{i}%
\mathbf{1}_{J_{i}}\right\vert ^{2}\right) ^{\frac{1}{2}}\right\Vert
_{L^{p^{\prime }}\left( \omega \right) },  \notag
\end{eqnarray}%
taken over all sequences $\left\{ I_{i}\right\} _{i=1}^{\infty }$, $\left\{
J_{i}\right\} _{i=1}^{\infty }$, $\left\{ a_{i}\right\} _{i=1}^{\infty }$
and $\left\{ b_{i}\right\} _{i=1}^{\infty }$of intervals and numbers
respectively where $J_{i}$ denotes any interval \emph{adjacent} to $I_{i}$
with comparable side length up to a factor of $2^{r}$, where here \emph{%
adjacent} means the closures of $I_{i}$ and $J_{i}$ have nonempty
intersection while the interiors of $I_{i}$ and $J_{i}$ have empty
intersection.\ Clearly, this characteristic is symmetric, $\mathcal{WBP}%
_{H,p}^{\ell ^{2}}\left( \sigma ,\omega \right) =\mathcal{WBP}_{H,p}^{\ell
^{2}}\left( \omega ,\sigma \right) $, and is used only in bounding the
comparable form.

\begin{remark}
We observe that in the proofs of our main theorems, we use only a slightly
weaker form of these characteristics, namely where the sequences $\left\{
I_{i}\right\} _{i=1}^{\infty }$, etc., used above are restricted to lie in a
fixed dyadic grid $\mathcal{D}$, provided that we require uniform control
over all grids $\mathcal{D}$.
\end{remark}

Now we can state our second main theorem, which again, in the case of
measures without common point masses, extends the results in \cite{LaSaShUr3}%
,\cite{Lac} and \cite{Hyt} to $\frac{4}{3}<p<4$ by replacing $2$ with $p$,
but using only a local quadratic testing characteristic, at the expense of
introducing additional quadratic Muckenhoupt and weak boundedness
characteristics - of course the extended energy characteristics are included
in the characterizing conditions. The case of doubling measures was again
obtained in \cite{SaWi}, even for\ general Calder\'{o}n-Zygmund operators
and $1<p<\infty $, and with the \textbf{scalar} local testing condition (\ref%
{scalar cube testing}) in place of quadratic local testing, and the case $%
p=2 $ is in \cite{AlSaUr}. As already mentioned, it was shown in \cite%
{AlLuSaUr 2} that one cannot replace all quadratic conditions with their
scalar analogues for any $p\neq 2$. The case of \emph{dyadic shifts} was
again resolved for general measures and all $1<p<\infty $ in \cite{Vuo}.

The following theorem provides a slight improvement over our variant
(including extended energy conditions) of the second conjecture of Hyt\"{o}%
nen and Vuorinen in the range $\frac{4}{3}<p<4$ when the measures share no
point masses, in that the global Muckenhoupt characteristics are replaced by
the smaller triple Muckenhoupt characteristics and scalar tailed Muckenhoupt
characteristics.

\begin{theorem}[second conjecture of Hyt\"{o}nen and Vuorinen]
\label{main}Suppose $\frac{4}{3}<p<4$, and that $\sigma $ and $\omega $ are
locally finite positive Borel measures on $\mathbb{R}$ without common point
masses. Suppose also that the extended energy characteristics $\mathcal{E}%
_{p}^{\ell ^{2},\limfunc{ext}}\left( \sigma ,\omega \right) $ and $\mathcal{E%
}_{p^{\prime }}^{\ell ^{2},\limfunc{ext},\ast }\left( \omega ,\sigma \right) 
$ are finite. Then under these side conditions, the two weight norm
inequality (\ref{Hilbert'}) holds \emph{if and only if} the local quadratic
interval testing conditions (\ref{quad cube testing}) hold, the triple
quadratic Muckenhoupt conditions (\ref{triple Muck}) hold, the scalar tailed
Muckenhouopt conditions (\ref{dual scalar tail}) hold, and the quadratic
weak boundedness property (\ref{WBP HV}) holds. Moreover, we have%
\begin{eqnarray*}
\mathfrak{N}_{H,p}\left( \sigma ,\omega \right) &\lesssim &\mathfrak{T}%
_{H,p}^{\ell ^{2},\func{loc}}\left( \sigma ,\omega \right) +\mathfrak{T}%
_{H,p^{\prime }}^{\ell ^{2},\func{loc}}\left( \omega ,\sigma \right)
+A_{p}^{\ell ^{2},\limfunc{trip}}\left( \sigma ,\omega \right) +A_{p^{\prime
}}^{\ell ^{2},\limfunc{trip}}\left( \omega ,\sigma \right) \\
&&+\mathcal{A}_{p}\left( \sigma ,\omega \right) +\mathcal{A}_{p^{\prime
}}\left( \omega ,\sigma \right) +\mathcal{WBP}_{H,p}^{\ell ^{2}}\left(
\sigma ,\omega \right) \\
&&+\mathcal{E}_{p}^{\ell ^{2},\limfunc{ext}}\left( \sigma ,\omega \right) +%
\mathcal{E}_{p^{\prime }}^{\ell ^{2},\limfunc{ext},\ast }\left( \omega
,\sigma \right) .
\end{eqnarray*}
\end{theorem}

\subsubsection{Dual half-forms}

In order to better understand the extended energy characteristics, write the
bilinear form $\left\langle H_{\sigma }f,g\right\rangle _{\omega }$ as a sum
of two `dual' half-forms,%
\begin{equation*}
\left\langle H_{\sigma }f,g\right\rangle _{\omega }=\left\{ \sum_{\substack{ %
I,J\in \mathcal{D}  \\ \ell \left( J\right) \leq \ell \left( I\right) }}%
+\sum _{\substack{ I,J\in \mathcal{D}  \\ \ell \left( J\right) >\ell \left(
I\right) }}\right\} \left\langle H_{\sigma }\bigtriangleup _{I}^{\sigma
}f,\bigtriangleup _{J}^{\omega }g\right\rangle _{\omega }\equiv \mathsf{B}%
^{\leq }\left( f,g\right) +\mathsf{B}^{>}\left( f,g\right) ,
\end{equation*}%
determined by the relative size of the side lengths $\ell \left( I\right)
,\ell \left( J\right) $ of $I$ and $J$. Denote the norms of these bilinear
half-forms by $\mathfrak{N}_{H,p}^{\leq }\left( \sigma ,\omega \right) $ and 
$\mathfrak{N}_{H,p}^{>}\left( \sigma ,\omega \right) $ respectively, i.e. $%
\mathfrak{N}_{H,p}^{\leq }\left( \sigma ,\omega \right) $ is the best
constant in the inequality,%
\begin{equation*}
\left\vert \mathsf{B}^{\leq }\left( f,g\right) \right\vert \leq \mathfrak{N}%
_{H,p}^{\leq }\left( \sigma ,\omega \right) \left\Vert f\right\Vert
_{L^{p}\left( \sigma \right) }\left\Vert g\right\Vert _{L^{p^{\prime
}}\left( \omega \right) }\ ,
\end{equation*}%
and similarly for $\mathfrak{N}_{H,p}^{>}\left( \sigma ,\omega \right) $.
Note that $\mathfrak{N}_{H,p}\left( \sigma ,\omega \right) \leq \mathfrak{N}%
_{H,p}^{\leq }\left( \sigma ,\omega \right) +\mathfrak{N}_{H,p}^{>}\left(
\sigma ,\omega \right) $. We can actually \emph{characterize} the sum of the
characteristics $\mathfrak{N}_{H,p}^{\leq }\left( \sigma ,\omega \right) +%
\mathfrak{N}_{H,p}^{>}\left( \sigma ,\omega \right) $.

\begin{theorem}[characterization of sum of dual half-form norms]
\label{half}Suppose $\frac{4}{3}<p<4$, and that $\sigma $ and $\omega $ are
locally finite positive Borel measures on $\mathbb{R}$ without common point
masses. There are the following characterizations of the sum of the norms of
the dual half forms in terms of testing, Muckenhoupt, weak boundedness and
extended energy characteristics, 
\begin{eqnarray*}
&&\mathfrak{N}_{H,p}^{\leq }\left( \sigma ,\omega \right) +\mathfrak{N}%
_{H,p}^{>}\left( \sigma ,\omega \right) \\
&&\ \ \ \ \ \approx \mathfrak{T}_{H,p}^{\ell ^{2},\limfunc{glob}}\left(
\sigma ,\omega \right) +\mathfrak{T}_{H,p^{\prime }}^{\ell ^{2},\limfunc{glob%
}}\left( \omega ,\sigma \right) +\mathcal{E}_{p}^{\ell ^{2},\limfunc{ext}%
}\left( \sigma ,\omega \right) +\mathcal{E}_{p^{\prime }}^{\ell ^{2},%
\limfunc{ext},\ast }\left( \omega ,\sigma \right) \\
&&\ \ \ \ \ \approx \mathfrak{T}_{H,p}^{\ell ^{2},\func{loc}}\left( \sigma
,\omega \right) +\mathfrak{T}_{H,p^{\prime }}^{\ell ^{2},\func{loc}}\left(
\omega ,\sigma \right) +A_{p}^{\ell ^{2},\limfunc{trip}}\left( \sigma
,\omega \right) +A_{p^{\prime }}^{\ell ^{2},\limfunc{trip}}\left( \omega
,\sigma \right) \\
&&\ \ \ \ \ \ \ \ \ \ \ \ \ \ \ +\mathcal{A}_{p}\left( \sigma ,\omega
\right) +\mathcal{A}_{p^{\prime }}\left( \omega ,\sigma \right) +\mathcal{WBP%
}_{H,p}^{\ell ^{2}}\left( \sigma ,\omega \right) \\
&&\ \ \ \ \ \ \ \ \ \ \ \ \ \ \ +\mathcal{E}_{p}^{\ell ^{2},\limfunc{ext}%
}\left( \sigma ,\omega \right) +\mathcal{E}_{p^{\prime }}^{\ell ^{2},%
\limfunc{ext},\ast }\left( \omega ,\sigma \right) .
\end{eqnarray*}
\end{theorem}

\subsection{Guide for the reader}

Here we emphasize the most basic ideas used to handle the case $p\neq 2$ of
the main theorems, given that the case $p=2$ was solved back in 2014 using
orthonormal weighted Haar bases in a Hilbert space, something not available
when $p\neq 2$. Neverthess, a large portion of the $p=2$ proof from 2014
finds its way into the arguments here, and the reader is encouraged to have
at least some of the papers \cite{LaSaShUr3}, \cite{Lac}, \cite{Hyt} and 
\cite{Saw7} at hand while reading this one. To get started, we use the idea
of Hyt\"{o}nen and Vuorinen to test, not the \emph{scalar} inequality for $H$
over indicators of intervals, but rather to test the $\ell ^{2}$\emph{-valued%
} extension, which has the same norm, over sequences of indicators of
intervals (times constants). In order to use these and other quadratic
testing conditions effectively, we follow \cite[Subsection 2.1]{Saw7} and
decompose our bilinear form 
\begin{equation*}
\left\langle H_{\sigma }f,g\right\rangle _{\omega }=\sum_{I,J\in \mathcal{D}%
}\left\langle H_{\sigma }\bigtriangleup _{I}^{\sigma }f,\bigtriangleup
_{J}^{\omega }g\right\rangle _{\omega }
\end{equation*}%
into subforms, of which a typical example can be written as%
\begin{equation*}
\mathsf{B}_{\mathcal{P}}\left( f,g\right) =\sum_{\left( I,J\right) \in 
\mathcal{P}}\left\langle H_{\sigma }\bigtriangleup _{I}^{\sigma
}f,\bigtriangleup _{J}^{\omega }g\right\rangle _{\omega }=\sum_{\left(
I,J\right) \in \mathcal{P}}\left\langle \bigtriangleup _{J}^{\omega
}H_{\sigma }\bigtriangleup _{I}^{\sigma }f,\bigtriangleup _{J}^{\omega
}g\right\rangle _{\omega }
\end{equation*}%
for some subset of pairs $\mathcal{P}$ of $\mathcal{D}\times \mathcal{D}$.
We then proceed with the inequalities of Cauchy-Schwarz in $\ell ^{2}$, and H%
\"{o}lder in $L^{p}\left( \omega \right) $, to obtain%
\begin{eqnarray*}
\left\vert \mathsf{B}_{\mathcal{P}}\left( f,g\right) \right\vert
&=&\left\vert \int_{\mathbb{R}}\left\{ \sum_{\left( I,J\right) \in \mathcal{P%
}}\bigtriangleup _{J}^{\omega }H_{\sigma }\bigtriangleup _{I}^{\sigma
}f\left( x\right) \ \bigtriangleup _{J}^{\omega }g\left( x\right) \right\} \
d\omega \left( x\right) \right\vert \\
&\leq &\int_{\mathbb{R}}\sqrt{\sum_{\left( I,J\right) \in \mathcal{P}%
}\left\vert \bigtriangleup _{J}^{\omega }H_{\sigma }\bigtriangleup
_{I}^{\sigma }f\left( x\right) \right\vert ^{2}}\sqrt{\sum_{\left(
I,J\right) \in \mathcal{P}}\left\vert \bigtriangleup _{J}^{\omega }g\left(
x\right) \right\vert ^{2}}\ d\omega \left( x\right) \\
&\leq &\left\Vert \sqrt{\sum_{\left( I,J\right) \in \mathcal{P}}\left\vert
\bigtriangleup _{J}^{\omega }H_{\sigma }\bigtriangleup _{I}^{\sigma }f\left(
x\right) \right\vert ^{2}}\right\Vert _{L^{p}\left( \omega \right)
}\left\Vert \sqrt{\sum_{\left( I,J\right) \in \mathcal{P}}\left\vert
\bigtriangleup _{J}^{\omega }g\left( x\right) \right\vert ^{2}}\right\Vert
_{L^{p^{\prime }}\left( \omega \right) } \\
&=&\left\Vert \left\vert \left\{ \bigtriangleup _{J}^{\omega }H_{\sigma
}\bigtriangleup _{I}^{\sigma }f\right\} _{\left( I,J\right) \in \mathcal{P}%
}\right\vert _{\ell ^{2}}\right\Vert _{L^{p}\left( \omega \right)
}\left\Vert \left\vert \left\{ \bigtriangleup _{J}^{\omega }g\right\}
_{\left( I,J\right) \in \mathcal{P}}\right\vert _{\ell ^{2}}\right\Vert
_{L^{p^{\prime }}\left( \omega \right) }.
\end{eqnarray*}%
At this point, Burkholder's theorem on martingale differences yields a
square function estimate that can be used to show that the second factor is
controlled by $\left\Vert g\right\Vert _{L^{p^{\prime }}\left( \omega
\right) }$ provided the pairs $\left( I,J\right) \in \mathcal{P}$ have only
a bounded number of $I^{\prime }s$ paired with a given $J$. In order to
handle the first factor we need to manipulate the sequence $\left\{
\bigtriangleup _{J}^{\omega }H_{\sigma }\bigtriangleup _{I}^{\sigma
}f\right\} _{\left( I,J\right) \in \mathcal{P}}$ so as to apply one of the
quadratic hypotheses. The entire difficulty with this approach lies in
appropriately decomposing the original bilinear form, and in finding
vector-valued manipulations so that the two goals can be simultaneously
achieved.

In order to proceed further, we need the fundamental insight of Nazarov,
Treil and Volberg that we may restrict our attention to functions with Haar
support consisting of $\func{good}$ intervals, which enjoy crucial geometric
decay properties. We also follow the blueprints of work in the case $p=2$ by
Hyt\"{o}nen, Lacey, Sawyer, Shen, Uriarte-Tuero and Wick, in particular that
of the work in \cite{Saw7} and \cite{Lac}, with a couple of exceptions.
There is no appeal to weighted Poisson inequalities. There is no explicit
use of the size condition in this proof. Instead, the `magical' property $%
\frac{d}{dx}\frac{1}{x}=-\frac{1}{x^{2}}$ for $x\neq 0$, of the convolution
kernel $\frac{1}{x}$ of the Hilbert transform is directly used in the proof
of the $L^{p}$-Stopping Child Lemma, and an elaborate stopping energy
characteristic finishes control of the stopping form. Thus our approach
provides a new proof in the case $p=2$ as well, in view of the fact that the
extended energy can be controlled by testing and Muckenhoupt in this case.

Then most of the decompositions into subsubforms go by the name of \emph{%
corona decompositions}, in which a collection of $\func{good}$ stopping
times (by stopping times we simply mean a collection of dyadic intervals,
whether or not they were chosen by some stopping criterion) is chosen so
that various features of the inner products are controlled in the coronas
lying `between' the stopping times. These features include the averages of $%
f $ over $\func{good}$ intervals in a corona and the total amount of `scalar 
$p $-energy' within a corona. Of course there is a price to pay for
arranging control of these special features, and we are able to pay it only
if there are Carleson type conditions that can be derived from the stopping
time criteria.

All of this information must then be encoded in the sequences $\left\{
\bigtriangleup _{J}^{\omega }H_{\sigma }\bigtriangleup _{I}^{\sigma
}f\right\} _{\left( I,J\right) \in \mathcal{P}}$ in such a way that it can
be exploited, and this requires different approaches in each separate
instance. Examples of this can be found in the ensuing sections where the
major forms are analyzed using a variety of tools. In particular we need the
following preliminary tools from Section 3 below.

\begin{enumerate}
\item \emph{Traditional two weight tools}: the $\func{good}$/$\func{bad}$
interval technology and Poisson inequalities in Lemma \ref{Poisson
inequality} of Nazarov, Treil and Volberg (from \cite[Subsetion 4.1]{NTV4}
and \cite{Vol}); the monotonicity equivalence in Lemma \ref{Energy Lemma}
(from \cite[Section 4]{LaSaShUr3}), including the simple new Lemma \ref%
{energy pointwise}; and standard properties of Carleson measures.

\item $L^{p}$\emph{\ specific two weight tools}: A new vector-valued Theorem %
\ref{using Carleson} for Carleson measures with $1<p<\infty $; the square
function Theorem \ref{square thm} for corona martingale differences that
uses Burkholder's theorem (from \cite{Bur1} and \cite{Bur2}); extension (\ref%
{FS vv}) of the vector-valued maximal inequalities of Fefferman and Stein 
\cite{FeSt} to the dyadic setting of a general measure (observed by J.-L.
Luna-Garcia), as well as the useful Lemma \ref{disjoint supp}; and finally a
Corona Martingale Comparison Principle in Proposition \ref{CMCP} that can be
viewed as a variant of the comparison principles for martingale differences
in J. Zinn \cite{Zin}.
\end{enumerate}

In particular we point to the challenges of the \emph{far} and \emph{stopping%
} forms in Sections 6, 8 and 9, whose analysis takes up much of this paper,
and moreover requires the restriction to $p<4$ for handling the stopping
form.

The main tool used to bound the \emph{far} form is the Intertwining
Proposition, which controls the \emph{far} form by a new \emph{extended}
energy, which is shown to be necessary for the norm inequality\footnote{%
It is interesting to note that even in the case $p=2$ there is no short
direct proof of the necessity of extended energy for the norm inequality.
The traditional approach is to go through the two weight Poisson inequality
and its equivalent Poisson testing conditions, which is a long and difficult
road. Here is this paper we derive necessity using most of the estimates for
the entire proof - again a long and difficult road.}. See Remark \ref{ref}
for more detail on extended energy.

There are three main tools used to bound the stopping form, namely a dual
tree decomposion generalizing Lacey's `upside down' corona construction, a
martingale difference comparison principle that delivers a form of
`orthogonality', and an $L^{p}$-Stopping Child Lemma. At least three major
obstacles appear, with the first arising from the fact that we no longer
have additivity of Hilbert space projections that played a prominent role in
the `upside down' corona construction of Lacey in \cite[page 8]{Lac}. The
second arises from the lack of a counterpart to the Quasi-Orthogonality
Argument in \cite[page 6]{Lac} when $p\neq 2$, which leads to an elaborate
extension of the stopping form, and as Lacey writes in his primer \cite[page
4]{Lac2}, "This argument\ (\textit{referring to the case }$p=2$\textit{\ in 
\cite[page 6]{Lac}}) relies heavily on the Hilbertian structure of the
question." Possibly the most significant obstacle is the failure of our
methods to obtain a suitable extension of the decay in Lemma \ref{final
level} in Subsection 9.6 to $p\geq 4$, thus limiting our control of the
stopping form to $p<4$.

In Section 2 we treat necessity of the quadratic conditions and defer the
necessity of extended energy to Section 9, then the preliminaries are
treated in Section 3, followed by the sufficiency proof of the two main
theorems in Sections 4 (comparable and disjoint forms), 5 (neighbour form),
6 (far form), 7 (paraproduct form), 8 (stopping form, which requires the
restriction $p<4$) and then necessity of extended energy in Section 9.
Concluding remarks are made in Section 10 on the restriction to $\frac{4}{3}%
<p<4$ in controlling the stopping form.

\section{Necessity of testing conditions}

To derive the necessity of the local quadratic testing, quadratic
Muckenhoupt and quadratic weak boundedness conditions, we will need a
special case of the classical Hilbert space valued extension of a bounded
operator from one $L^{p}$ space to another, see e.g. \cite[Theorem 4.5.1]%
{Gra}. Suppose $T$ is bounded from $L^{p}\left( \mathbb{R};\sigma \right) $
to $L^{p}\left( \mathbb{R};\omega \right) $, $0<p<\infty $, and for $\mathbf{%
f}=\left\{ f_{j}\right\} _{j=1}^{\infty }$, define%
\begin{equation*}
T\mathbf{f}\equiv \left\{ Tf_{j}\right\} _{j=1}^{\infty }.
\end{equation*}%
Then $T$ extends to an operator bounded from $L^{p}\left( \ell ^{2};\sigma
\right) $ to $L^{p}\left( \ell ^{2};\omega \right) $ with the same norm,%
\begin{equation*}
\int_{\mathbb{R}}\left\vert T\mathbf{f}\left( x\right) \right\vert _{\ell
^{2}}^{p}d\omega \left( x\right) \leq \left\Vert T\right\Vert _{L^{p}\left(
\sigma \right) \rightarrow L^{p}\left( \omega \right) }^{p}\int_{\mathbb{R}%
}\left\vert \mathbf{f}\left( x\right) \right\vert _{\ell ^{2}}^{p}d\sigma
\left( x\right) ,
\end{equation*}%
which written out in full becomes%
\begin{equation}
\left( \int_{\mathbb{R}}\left( \sqrt{\sum_{j=1}^{\infty }\left\vert
Tf_{j}\left( x\right) \right\vert ^{2}}\right) ^{p}d\omega \left( x\right)
\right) ^{\frac{1}{p}}\leq \left\Vert T\right\Vert _{L^{p}\left( \sigma
\right) \rightarrow L^{p}\left( \omega \right) }\left( \int_{\mathbb{R}%
}\left( \sqrt{\sum_{j=1}^{\infty }\left\vert f_{j}\left( x\right)
\right\vert ^{2}}\right) ^{p}d\sigma \left( x\right) \right) ^{\frac{1}{p}}.
\label{in full}
\end{equation}

\subsection{Necessity of quadratic testing and offset $A_{p}$, and WBP}

We can use the vector-valued inequality (\ref{in full}) with $T=H_{\sigma }$
to obtain the necessity of the global quadratic testing inequality (\ref%
{glob}) for the boundedness of $H$ from $L^{p}\left( \sigma \right) $ to $%
L^{p}\left( \omega \right) $. Indeed, we simply set $f_{j}\equiv
a_{j}H_{\sigma }\mathbf{1}_{I_{i}}$ in (\ref{in full}) to obtain the global
quadratic testing inequality (\ref{glob}). Then we simply note the pointwise
inequality%
\begin{equation*}
\sum_{i=1}^{\infty }\left( a_{i}\mathbf{1}_{I_{i}}H_{\sigma }\mathbf{1}%
_{I_{i}}\right) \left( x\right) ^{2}\leq \sum_{i=1}^{\infty }\left\vert
a_{i}\right\vert ^{2}\left\vert H_{\sigma }\mathbf{1}_{I_{i}}\left( x\right)
\right\vert ^{2}\mathbf{,}
\end{equation*}%
to obtain the local version (\ref{quad cube testing}). Altogether we have, 
\begin{equation*}
\mathfrak{T}_{H,p}^{\ell ^{2},\func{loc}}\left( \sigma ,\omega \right) \leq 
\mathfrak{T}_{H,p}^{\ell ^{2},\limfunc{glob}}\left( \sigma ,\omega \right)
\lesssim \mathfrak{N}_{H,p}\left( \sigma ,\omega \right) .
\end{equation*}

The quadratic offset $A_{p}^{\ell ^{2},\limfunc{offset}}\left( \sigma
,\omega \right) $ characteristic is controlled by the global quadratic
testing characteristic $\mathfrak{T}_{H,p}^{\ell ^{2},\func{global}}\left(
\sigma ,\omega \right) $ using the pointwise estimate $\left\vert H_{\sigma }%
\mathbf{1}_{I_{i}^{\ast }}\left( x\right) \right\vert \gtrsim \frac{%
\left\vert I_{i}^{\ast }\right\vert _{\sigma }}{\left\vert I_{i}\right\vert }
$ for $x\in I_{i}$, and the quadratic weak boundedness condition also
follows from global quadratic testing,%
\begin{eqnarray*}
&&\sum_{i=1}^{\infty }\left\vert \int_{\mathbb{R}}a_{i}H_{\sigma }\mathbf{1}%
_{I_{i}}\left( x\right) b_{i}\mathbf{1}_{J_{i}}\left( x\right) d\omega
\left( x\right) \right\vert \leq \left\Vert \left( \sum_{i=1}^{\infty
}\left( a_{i}H_{\sigma }\mathbf{1}_{I_{i}}\right) ^{2}\right) ^{\frac{1}{2}%
}\right\Vert _{L^{p}\left( \omega \right) }\left\Vert \left(
\sum_{i=1}^{\infty }\left( b_{i}\mathbf{1}_{J_{i}}\right) ^{2}\right) ^{%
\frac{1}{2}}\right\Vert _{L^{p^{\prime }}\left( \omega \right) } \\
&&\ \ \ \ \ \ \ \ \ \ \ \ \ \ \ \lesssim \mathfrak{T}_{H,p}^{\ell ^{2},%
\limfunc{glob}}\left( \sigma ,\omega \right) \left\Vert \left(
\sum_{i=1}^{\infty }\left( a_{i}\mathbf{1}_{I_{i}}\right) ^{2}\right) ^{%
\frac{1}{2}}\right\Vert _{L^{p}\left( \sigma \right) }\left\Vert \left(
\sum_{i=1}^{\infty }\left( b_{i}\mathbf{1}_{J_{i}}\right) ^{2}\right) ^{%
\frac{1}{2}}\right\Vert _{L^{p^{\prime }}\left( \omega \right) }.
\end{eqnarray*}%
Finally, it is claimed without proof in \cite{HyVu} that finiteness of the
global quadratic Muckenhoupt characteristic $\mathcal{A}_{p}^{\ell ^{2},%
\limfunc{glob}}\left( \sigma ,\omega \right) $ is necessary for the norm
inequality. However, the reader can now easily provide a proof modeled on
that for the offset condition above, after writing $f_{i}\mathbf{1}_{\mathbb{%
R}\setminus I_{i}}=f_{i}\mathbf{1}_{L_{i}}+f_{i}\mathbf{1}_{R_{i}}$, where $%
L_{i}$ and $R_{i}$ are the left and right hand components of $\mathbb{R}%
\setminus I_{i}$ respectively\footnote{%
We thank Ignacio Uriarte-Tuero for discussions on this matter.}. Of course,
the triple quadratic Muckenhoupt characteristic $A_{p}^{\ell ^{2},\limfunc{%
trip}}\left( \sigma ,\omega \right) $, and the scalar tailed characteristic $%
\mathcal{A}_{p}\left( \sigma ,\omega \right) $, are controlled by the global
quadratic Muckenhoupt characteristic $\mathcal{A}_{p}^{\ell ^{2},\limfunc{%
glob}}\left( \sigma ,\omega \right) $. Finally, it is well known in the case 
$p=2$ that the scalar tailed characteristic $\mathcal{A}_{p}\left( \sigma
,\omega \right) $ is controlled by the scalar global testing characteristic $%
\mathfrak{T}_{H,p}^{\limfunc{glob}}\left( \sigma ,\omega \right) $, and the
same proof works for $p\neq 2$.

\subsection{Necessity of extended energy for boundedness of the dual half
forms}

Recall the extended energy characteristic $\mathcal{E}_{p}^{\ell ^{2},%
\limfunc{ext}}\left( \sigma ,\omega \right) $ and its dual $\mathcal{E}%
_{p^{\prime }}^{\ell ^{2},\limfunc{ext}}\left( \omega ,\sigma \right) $
given in inequality (\ref{ext ener}).

\begin{theorem}
\label{necc func ener}Suppose $\sigma $ and $\omega $ are locally finite
positive Borel measures on the real line $\mathbb{R}$. For $1<p<\infty $ we
have%
\begin{equation*}
\mathcal{E}_{p}^{\ell ^{2},\limfunc{ext}}\left( \sigma ,\omega \right)
\lesssim \mathfrak{N}_{H,p}^{\leq }\left( \sigma ,\omega \right) \text{ and }%
\mathcal{E}_{p^{\prime }}^{\ell ^{2},\limfunc{ext}}\left( \sigma ,\omega
\right) \lesssim \mathfrak{N}_{H,p}^{>}\left( \sigma ,\omega \right) .
\end{equation*}
\end{theorem}

We defer the proof to Section \ref{ex en section} below. In the case $p=2$,
each of the extended energy characteristics are controlled by the global
quadratic testing characteristics \cite{LaSaShUr3}, but when $p\neq 2$, we
are unable to show even that $\mathcal{E}_{p}^{\ell ^{2},\limfunc{ext}%
}\left( \sigma ,\omega \right) +\mathcal{E}_{p^{\prime }}^{\ell ^{2},%
\limfunc{ext}}\left( \omega ,\sigma \right) \lesssim \mathfrak{N}%
_{H,p}\left( \sigma ,\omega \right) $. Thus when $p=2$, we have $\mathfrak{N}%
_{H,p}\left( \sigma ,\omega \right) \approx \mathfrak{N}_{H,p}^{\leq }\left(
\sigma ,\omega \right) +\mathfrak{N}_{H,p}^{>}\left( \sigma ,\omega \right) $%
, but it remains an open question whether or not this persists for any $%
p\neq 2$.

\section{Preliminaries}

We will need the $\func{good}/\func{bad}$ technology of Nazarov, Treil and
Volberg, a Monotonicity Lemma, a Poisson Decay Lemma, an estimate on sums of
Poisson kernels, a $p$-energy reversal inequality, several properties of
Carleson measures, bounds for square functions using Burkholder's martingale
transform theorem and Kintchine's expectation theorem, an extension of
Fefferman-Stein vector-valued inequalitites to the dyadic maximal function
with arbitrary measures, and finally a Corona Martingale Comparison
Principle in order to control the stopping form in the final section of the
paper - an important new feature of the proof.

Recall the formula%
\begin{equation}
f=\sum_{Q\in \mathcal{D}}\bigtriangleup _{I}^{\sigma }f,\ \ \ \ \ \text{%
where }\bigtriangleup _{Q}^{\sigma }f=\left\langle f,h_{Q}^{\sigma
}\right\rangle _{\sigma }h_{Q}^{\sigma }\text{ and }h_{Q}^{\sigma }=\frac{1}{%
\sqrt{\left\vert Q\right\vert _{\sigma }}}\left( \sqrt{\frac{\left\vert
Q_{-}\right\vert _{\sigma }}{\left\vert Q_{+}\right\vert _{\sigma }}}\mathbf{%
1}_{Q_{+}}-\sqrt{\frac{\left\vert Q_{+}\right\vert _{\sigma }}{\left\vert
Q_{-}\right\vert _{\sigma }}}\mathbf{1}_{Q_{-}}\right) ,  \label{def Haar}
\end{equation}%
and where $Q_{\pm }$ denote the right and left hand children of the interval 
$Q$.

\subsection{Good/bad intervals and functions}

For the purposes of this paper, an \emph{interval} $I=\left[ a,b\right) $
will be taken to be closed on the left and open on the right, unless
otherwise stated. We recall the definition of a $\func{good}$ dyadic
interval from \cite[Subsection 4.1]{NTV4}, see also \cite{LaSaUr2}. We say
that a dyadic interval $J$ is $\left( r,\varepsilon \right) $-\emph{deeply
embedded} in a dyadic interval $K$, or simply $r$\emph{-deeply embedded} in $%
K$, which we write as $J\subset _{r}K$, when $J\subset K$ and both 
\begin{eqnarray}
\ell \left( J\right) &\leq &2^{-r}\ell \left( K\right) ,
\label{def deep embed} \\
\limfunc{dist}\left( J,\partial K\right) &\geq &\frac{1}{2}\ell \left(
J\right) ^{\varepsilon }\ell \left( K\right) ^{1-\varepsilon }.  \notag
\end{eqnarray}

\begin{definition}
Let $r\in \mathbb{N}$ and $0<\varepsilon <1$.

\begin{enumerate}
\item A dyadic interval $J$ is $\left( r,\varepsilon \right) $\emph{-}$\func{%
good}$, or simply $\func{good}$, if for \emph{every} dyadic superinterval $I$%
, it is the case that \textbf{either} $J$ has side length at least $2^{-r}$
times that of $I$, \textbf{or} $J\subset _{r}I$ is $\left( r,\varepsilon
\right) $-deeply embedded in $I$.

\item A dyadic interval $J$ is $\limfunc{child}$\emph{-}$\func{good}$ if $J$
and its two\ dyadic children $J_{\pm }$ are $\func{good}$.

\item Denote by $\mathcal{D}_{\func{good}}$ and $\mathcal{D}_{\func{good}}^{%
\limfunc{child}}$ the set of $\func{good}$ and $\limfunc{child}$\emph{-}$%
\func{good}$ intervals respectively.
\end{enumerate}
\end{definition}

It is shown in \cite[Theorem 4.1 on page 15]{NTV4} that for parameters $%
r,\varepsilon $ sufficiently large and small respectively, the boundedness
of the Hilbert transform $H_{\sigma }:L^{p}\left( \sigma \right) \rightarrow
L^{p}\left( \omega \right) $ can be reduced to testing the bilinear
inequality%
\begin{equation*}
\left\vert \int_{\mathbb{R}}H_{\sigma }f\left( x\right) g\left( x\right)
d\omega \left( x\right) \right\vert \leq C\left\Vert f\right\Vert
_{L^{p}\left( \sigma \right) }\left\Vert g\right\Vert _{L^{p^{\prime
}}\left( \omega \right) }\ ,
\end{equation*}%
uniformly over all dyadic grids $\mathcal{D},$ and all functions $%
f=\sum_{I\in \mathcal{D}}\bigtriangleup _{I}^{\sigma }f\in L^{p}\left(
\sigma \right) \cap L^{2}\left( \sigma \right) $ and $g=\sum_{J\in \mathcal{D%
}}\bigtriangleup _{J}^{\omega }g\in L^{p}\left( \omega \right) \cap
L^{2}\left( \omega \right) $ whose Haar supports $\left\{ I\in \mathcal{D}%
:\bigtriangleup _{I}^{\sigma }f\neq 0\right\} $ are $\left\{ J\in \mathcal{D}%
:\bigtriangleup _{J}^{\omega }g\neq 0\right\} $ are contained in $\mathcal{D}%
_{\limfunc{good}}^{\limfunc{child}}$\footnote{%
Only the case $\mathcal{D}_{\func{good}}$ and $p=2$ is mentioned in \cite[%
Theorem 4.1]{NTV4}, but the proof extends readily to $\mathcal{D}_{\func{good%
}}^{\limfunc{child}}$ and $1<p<\infty $.}. The parameters $r,\varepsilon $
will be fixed sufficiently large and small respectively later in the proof.

\subsection{Poisson and Monotonicity Lemmas}

For any interval $J$ with center $c_{J}$, and any finite measure $\nu $,
define the Poisson integral,%
\begin{equation*}
\mathrm{P}\left( J,\nu \right) \equiv \int_{\mathbb{R}}\frac{\ell \left(
J\right) }{\left( \ell \left( J\right) +\left\vert y-c_{J}\right\vert
\right) ^{2}}d\nu \left( y\right) .
\end{equation*}

\begin{lemma}[{Monotonicity Lemma \protect\cite[Section 4]{LaSaShUr3}}]
\label{Energy Lemma}Fix a locally finite positive Borel measure $\omega $.
Let $J\ $be a cube in $\mathcal{D}$. Let $\nu $ be a positive measure
supported in $\mathbb{R}\setminus 2J$. Let $H$ be the Hilbert transform.
Then for any $\beta \in \mathbb{R}$, we have the monotonicity principle,%
\begin{equation*}
\left\vert \left\langle H\nu ,h_{J}^{\omega }\right\rangle _{\omega
}\right\vert \approx \frac{\mathrm{P}\left( J,\nu \right) }{\ell \left(
J\right) }\left\vert \int_{J}\left( x-\beta \right) h_{J}^{\omega }\left(
x\right) d\omega \left( x\right) \right\vert =\frac{\mathrm{P}\left( J,\nu
\right) }{\ell \left( J\right) }\left\vert \left\langle Z-\beta
,h_{J}^{\omega }\right\rangle _{\omega }\right\vert ,
\end{equation*}%
where $Z\left( x\right) =x$ is the identity function on the real line.
\end{lemma}

Due to the importance of this result for the Hilbert transform, we repeat
the short proof here.

\begin{proof}
With $c_{J}$ equal to the center of$\ J$, and $\beta \in \mathbb{R}$, we
have that $\left( x-c_{J}\right) h_{J}^{\omega }\left( x\right) $ doesn't
change sign on $J$ by (\ref{def Haar}) and so, 
\begin{eqnarray*}
\left\vert \left\langle H\nu ,h_{J}^{\omega }\right\rangle _{\omega
}\right\vert &=&\left\vert \int_{J}\left( \int_{\mathbb{R}\setminus 2J}\frac{%
1}{y-x}d\nu \left( y\right) \right) h_{J}^{\omega }\left( x\right) d\omega
\left( x\right) \right\vert \\
&=&\left\vert \int_{J}\left( \int_{\mathbb{R}\setminus 2J}\frac{1}{y-x}-%
\frac{1}{y-c_{J}}d\nu \left( y\right) \right) h_{J}^{\omega }\left( x\right)
d\omega \left( x\right) \right\vert \\
&=&\left\vert \int_{J}\left( \int_{\mathbb{R}\setminus 2J}\frac{x-c_{J}}{%
\left( y-x\right) \left( y-c_{J}\right) }d\nu \left( y\right) \right)
h_{J}^{\omega }\left( x\right) d\omega \left( x\right) \right\vert \\
&=&\int_{J}\left( \int_{\mathbb{R}\setminus 2J}\frac{\ell \left( J\right) }{%
\left\vert \left( y-x\right) \left( y-c_{J}\right) \right\vert }d\nu \left(
y\right) \right) \left\vert \frac{x-c_{J}}{\ell \left( J\right) }%
h_{J}^{\omega }\left( x\right) \right\vert d\omega \left( x\right) \\
&\approx &\mathrm{P}\left( J,\nu \right) \int_{J}\left\vert \frac{x-c_{J}}{%
\ell \left( J\right) }h_{J}^{\omega }\left( x\right) \right\vert d\omega
\left( x\right) ,
\end{eqnarray*}%
and using%
\begin{eqnarray*}
\int_{J}\left\vert \frac{x-c_{J}}{\ell \left( J\right) }h_{J}^{\omega
}\left( x\right) \right\vert d\omega \left( x\right) &=&\frac{1}{\ell \left(
J\right) }\int_{J}\left( x-c_{J}\right) h_{J}^{\omega }\left( x\right)
d\omega \left( x\right) \\
&=&\frac{1}{\ell \left( J\right) }\int_{J}\left( x-\beta \right)
h_{J}^{\omega }\left( x\right) d\omega \left( x\right) ,
\end{eqnarray*}%
we obtain%
\begin{equation*}
\left\vert \left\langle H\nu ,h_{J}^{\omega }\right\rangle _{\omega
}\right\vert \approx \frac{\mathrm{P}\left( J,\nu \right) }{\ell \left(
J\right) }\left\vert \int_{J}\left( x-\beta \right) h_{J}^{\omega }\left(
x\right) d\omega \left( x\right) \right\vert =\frac{\mathrm{P}\left( J,\nu
\right) }{\ell \left( J\right) }\left\vert \left\langle Z,h_{J}^{\omega
}\right\rangle _{\omega }\right\vert .
\end{equation*}
\end{proof}

Here is a \emph{pointwise} corollary of the Monotonicity Lemma \ref{Energy
Lemma}, that estimates a Haar projection of $H\left( \mathbf{1}_{K}\sigma
\right) $.

\begin{corollary}
\label{point mon}For $J,K\in \mathcal{D}$ with $2J$ disjoint from $K$, we
have the pointwise estimate,%
\begin{equation*}
\left\vert \Delta _{J}^{\omega }H\left( \mathbf{1}_{K}\sigma \right) \left(
x\right) \right\vert \approx \frac{\mathrm{P}\left( J,\mathbf{1}_{K}\sigma
\right) }{\ell \left( J\right) }\left\vert \Delta _{J}^{\omega }Z\left(
x\right) \right\vert \leq \mathrm{P}\left( J,\mathbf{1}_{K}\sigma \right) 
\mathbf{1}_{J}\left( x\right) .
\end{equation*}
\end{corollary}

\begin{proof}
The Monotonicity Lemma \ref{Energy Lemma} yields%
\begin{equation*}
\left\vert \Delta _{J}^{\omega }H\left( \mathbf{1}_{K}\sigma \right) \left(
x\right) \right\vert =\left\vert \left\langle H\left( \mathbf{1}_{K}\sigma
\right) ,h_{J}^{\omega }\right\rangle _{\omega }\right\vert \left\vert
h_{J}^{\omega }\left( x\right) \right\vert \approx \frac{\mathrm{P}\left( J,%
\mathbf{1}_{K}\sigma \right) }{\ell \left( J\right) }\left\vert \left\langle
Z,h_{J}^{\omega }\right\rangle _{\omega }\right\vert \left\vert
h_{J}^{\omega }\left( x\right) \right\vert ,
\end{equation*}%
and then the following calculation completes the proof of the corollary,%
\begin{eqnarray*}
&&\left\vert \Delta _{J}^{\omega }Z\left( x\right) \right\vert =\left\vert
\Delta _{J}^{\omega }\left( Z-c_{J}\right) \left( x\right) \right\vert \\
&=&\left\vert E_{J_{-}}^{\omega }\left( Z-c_{J}\right) -E_{J}^{\omega
}\left( Z-c_{J}\right) \right\vert \mathbf{1}_{J_{-}}\left( x\right)
+\left\vert E_{J_{+}}^{\omega }\left( Z-c_{J}\right) -E_{J}^{\omega }\left(
Z-c_{J}\right) \right\vert \mathbf{1}_{J_{+}}\left( x\right) \\
&\leq &2\left[ \ell \left( J\right) \mathbf{1}_{J_{-}}\left( x\right) +\ell
\left( J\right) \mathbf{1}_{J_{+}}\left( x\right) \right] =2\ell \left(
J\right) \mathbf{1}_{J}\left( x\right) .
\end{eqnarray*}
\end{proof}

The corollary applies to separated intervals $J,K$ and we now show that the
upper bound in the corollary holds more generally for \emph{disjoint}
intervals $J$ and $K$, including in particular adjacent intervals.

\begin{lemma}
\label{energy pointwise}For $J\in \mathcal{D}$ that is disjoint from $K\in 
\mathcal{D}$, we have the pointwise estimate,%
\begin{equation*}
\left\vert \Delta _{J}^{\omega }H\left( \mathbf{1}_{K}\sigma \right) \left(
x\right) \right\vert \lesssim \mathrm{P}\left( J,\mathbf{1}_{K}\sigma
\right) \mathbf{1}_{J}\left( x\right) .
\end{equation*}
\end{lemma}

\begin{proof}
We have%
\begin{eqnarray*}
&&\Delta _{J}^{\omega }H\left( \mathbf{1}_{K}\sigma \right) \left( x\right)
=\left\langle H\left( \mathbf{1}_{K}\sigma \right) ,h_{J}^{\omega
}\right\rangle _{\omega }h_{J}^{\omega }\left( x\right) =\int_{J}\left[
H\left( \mathbf{1}_{K}\sigma \right) \left( y\right) -\gamma \right]
h_{J}^{\omega }\left( y\right) h_{J}^{\omega }\left( x\right) d\omega \left(
y\right) \\
&=&\int_{J}\left[ H\left( \mathbf{1}_{K}\sigma \right) \left( y\right)
-\gamma \right] h_{J}^{\omega }\left( y\right) h_{J}^{\omega }\left(
x\right) \left\{ \mathbf{1}_{J_{-}}\left( y\right) \mathbf{1}_{J_{-}}\left(
x\right) +\mathbf{1}_{J_{+}}\left( y\right) \mathbf{1}_{J_{+}}\left(
x\right) \right\} d\omega \left( y\right) \\
&&+\int_{J}\left[ H\left( \mathbf{1}_{K}\sigma \right) \left( y\right)
-\gamma \right] h_{J}^{\omega }\left( y\right) h_{J}^{\omega }\left(
x\right) \left\{ \mathbf{1}_{J_{-}}\left( y\right) \mathbf{1}_{J_{+}}\left(
x\right) +\mathbf{1}_{J_{+}}\left( y\right) \mathbf{1}_{J_{-}}\left(
x\right) \right\} d\omega \left( y\right) \\
&\equiv &A+B+C+D.
\end{eqnarray*}%
Hence using the formula (\ref{def Haar}) for the Haar function above,%
\begin{eqnarray*}
h_{J}^{\omega }\left( y\right) h_{J}^{\omega }\left( x\right) \mathbf{1}%
_{J_{-}}\left( y\right) \mathbf{1}_{J_{-}}\left( x\right) &=&\frac{1}{%
\left\vert J\right\vert _{\omega }}\frac{\left\vert J_{+}\right\vert
_{\omega }}{\left\vert J_{-}\right\vert _{\omega }}\mathbf{1}_{J_{-}}\left(
y\right) \mathbf{1}_{J_{-}}\left( x\right) , \\
h_{J}^{\omega }\left( y\right) h_{J}^{\omega }\left( x\right) \mathbf{1}%
_{J_{+}}\left( y\right) \mathbf{1}_{J_{+}}\left( x\right) &=&\frac{1}{%
\left\vert J\right\vert _{\omega }}\frac{\left\vert J_{-}\right\vert
_{\omega }}{\left\vert J_{+}\right\vert _{\omega }}\mathbf{1}_{J_{+}}\left(
y\right) \mathbf{1}_{J_{+}}\left( x\right) , \\
h_{J}^{\omega }\left( y\right) h_{J}^{\omega }\left( x\right) \mathbf{1}%
_{J_{-}}\left( y\right) \mathbf{1}_{J_{+}}\left( x\right) &=&-\frac{1}{%
\left\vert J\right\vert _{\omega }}\mathbf{1}_{J_{-}}\left( y\right) \mathbf{%
1}_{J_{+}}\left( x\right) , \\
h_{J}^{\omega }\left( y\right) h_{J}^{\omega }\left( x\right) \mathbf{1}%
_{J_{+}}\left( y\right) \mathbf{1}_{J_{-}}\left( x\right) &=&-\frac{1}{%
\left\vert J\right\vert _{\omega }}\mathbf{1}_{J_{+}}\left( y\right) \mathbf{%
1}_{J_{-}}\left( x\right) .
\end{eqnarray*}%
Then we have with $\gamma =H\left( \mathbf{1}_{K}\sigma \right) \left(
c_{J}\right) $,%
\begin{eqnarray*}
\left\vert A\right\vert &=&\left\vert \int_{J}\left[ H\left( \mathbf{1}%
_{K}\sigma \right) \left( y\right) -\gamma \right] \frac{1}{\left\vert
J\right\vert _{\omega }}\frac{\left\vert J_{+}\right\vert _{\omega }}{%
\left\vert J_{-}\right\vert _{\omega }}\mathbf{1}_{J_{-}}\left( y\right)
d\omega \left( y\right) \right\vert \mathbf{1}_{J_{-}}\left( x\right)
\lesssim \mathrm{P}\left( J,\mathbf{1}_{K}\sigma \right) \mathbf{1}%
_{J_{-}}\left( x\right) , \\
\left\vert B\right\vert &=&\left\vert \int_{J}\left[ H\left( \mathbf{1}%
_{K}\sigma \right) \left( y\right) -\gamma \right] \frac{1}{\left\vert
J\right\vert _{\omega }}\frac{\left\vert J_{-}\right\vert _{\omega }}{%
\left\vert J_{+}\right\vert _{\omega }}\mathbf{1}_{J_{+}}\left( y\right)
d\omega \left( y\right) \right\vert \mathbf{1}_{J_{+}}\left( x\right)
\lesssim \mathrm{P}\left( J,\mathbf{1}_{K}\sigma \right) \mathbf{1}%
_{J_{+}}\left( x\right) , \\
\left\vert C\right\vert &=&\left\vert \int_{J}\left[ H\left( \mathbf{1}%
_{K}\sigma \right) \left( y\right) -\gamma \right] \frac{1}{\left\vert
J\right\vert _{\omega }}\mathbf{1}_{J_{-}}\left( y\right) d\omega \left(
y\right) \right\vert \mathbf{1}_{J_{+}}\left( x\right) \lesssim \mathrm{P}%
\left( J,\mathbf{1}_{K}\sigma \right) \mathbf{1}_{J_{+}}\left( x\right) , \\
\left\vert D\right\vert &=&\left\vert \int_{J}\left[ H\left( \mathbf{1}%
_{K}\sigma \right) \left( y\right) -\gamma \right] \frac{1}{\left\vert
J\right\vert _{\omega }}\mathbf{1}_{J_{+}}\left( y\right) d\omega \left(
y\right) \right\vert \mathbf{1}_{J_{-}}\left( x\right) \lesssim \mathrm{P}%
\left( J,\mathbf{1}_{K}\sigma \right) \mathbf{1}_{J_{-}}\left( x\right) .
\end{eqnarray*}
\end{proof}

We will need the following critical Poisson Decay Lemma of Nazarov, Treil
and Volberg from \cite{Vol}.

\begin{lemma}[Poisson Decay Lemma]
\label{Poisson inequality}Suppose $J\subset I\subset K$ are dyadic intervals
and that $d\left( J,\partial I\right) >2\ell \left( J\right) ^{\varepsilon
}\ell \left( I\right) ^{1-\varepsilon }$ for some $0<\varepsilon <\frac{1}{2}
$. Then for any locally finite positive Borel measure $\mu $ we have%
\begin{equation}
\mathrm{P}(J,\mu \mathbf{1}_{K\setminus I})\lesssim \left( \frac{\ell \left(
J\right) }{\ell \left( I\right) }\right) ^{1-2\varepsilon }\mathrm{P}(I,\mu 
\mathbf{1}_{K\setminus I}).  \label{e.Jsimeq}
\end{equation}
\end{lemma}

There is an extension of the pointwise inequality (\ref{PE bound}) to a
variant involving `absolute' projections 
\begin{equation*}
\left\vert \mathsf{P}_{\Lambda }^{\omega }\right\vert g\left( x\right)
\equiv \sqrt{\sum_{J\in \Lambda }\left\vert \bigtriangleup _{J}^{\omega
}g\left( x\right) \right\vert ^{2}},
\end{equation*}%
where $\Lambda \subset \mathcal{D}_{\func{good}}\left[ I\right] $.

\begin{lemma}
\label{ener rev}Suppose $I,F\in \mathcal{D}_{\func{good}}$ with $I\subset F$%
, and let $\Lambda \subset \mathcal{D}_{\func{good}}\left[ I\right] $. Then
with $\mathcal{W}_{\func{good}}^{\limfunc{trip}}\left( I\right) $ denoting
the maximal good subintervals of $I$ whose triples are contained in $I$, and
with $\Lambda \left[ K\right] \equiv \Lambda \cap \mathcal{D}\left[ K\right] 
$, 
\begin{equation}
\sqrt{\sum_{K\in \mathcal{W}_{\func{good}}^{\limfunc{trip}}\left( I\right)
}\left( \frac{\mathrm{P}\left( K,\mathbf{1}_{F\setminus I}\sigma \right) }{%
\ell \left( K\right) }\right) ^{2}\left( \left\vert \mathsf{P}_{\Lambda %
\left[ K\right] }^{\omega }\right\vert Z\left( x\right) \right) ^{2}}\approx
\left\vert \mathsf{P}_{\Lambda }^{\omega }\right\vert H_{\sigma }\mathbf{1}%
_{F\setminus I}\left( x\right) .  \label{PE bound''}
\end{equation}
\end{lemma}

\begin{proof}
Recall from the pointwise montonicity principle in Corollary \ref{point mon}
above that,%
\begin{equation*}
\left\vert \bigtriangleup _{J}^{\omega }H_{\sigma }\mathbf{1}_{F\setminus
I}\left( x\right) \right\vert \approx \frac{\mathrm{P}\left( J,\mathbf{1}%
_{F\setminus I}\sigma \right) }{\ell \left( J\right) }\left\vert
\bigtriangleup _{J}^{\omega }Z\left( x\right) \right\vert ,\ \ \ \ \ \text{%
for }2J\cap I=\emptyset \ .
\end{equation*}%
Moreover, we also have $\frac{\mathrm{P}\left( J,\mathbf{1}_{F\setminus
I}\sigma \right) }{\ell \left( J\right) }\approx \frac{\mathrm{P}\left( K,%
\mathbf{1}_{F\setminus I}\sigma \right) }{\ell \left( K\right) }$ for $K\in 
\mathcal{W}_{\func{good}}^{\limfunc{trip}}\left( I\right) $, and so noting $%
J\subset K\subset I$,%
\begin{eqnarray*}
&&\left\vert \mathsf{P}_{\Lambda }^{\omega }\right\vert H_{\sigma }\mathbf{1}%
_{F\setminus I}\left( x\right) ^{2}=\sum_{K\in \mathcal{W}_{\func{good}}^{%
\limfunc{trip}}\left( I\right) }\sum_{J\in \Lambda \left[ K\right]
}\left\vert \bigtriangleup _{J}^{\omega }H_{\sigma }\mathbf{1}_{F\setminus
I}\left( x\right) \right\vert ^{2}\approx \sum_{K\in \mathcal{W}_{\func{good}%
}^{\limfunc{trip}}\left( I\right) }\sum_{J\in \Lambda \left[ K\right]
}\left( \frac{\mathrm{P}\left( J,\mathbf{1}_{F\setminus I}\sigma \right) }{%
\ell \left( J\right) }\right) ^{2}\left\vert \bigtriangleup _{J}^{\omega
}Z\left( x\right) \right\vert ^{2} \\
&\approx &\sum_{K\in \mathcal{W}_{\func{good}}^{\limfunc{trip}}\left(
I\right) }\left( \frac{\mathrm{P}\left( K,\mathbf{1}_{F\setminus I}\sigma
\right) }{\ell \left( K\right) }\right) ^{2}\sum_{J\in \Lambda \left[ K%
\right] }\left\vert \bigtriangleup _{J}^{\omega }Z\left( x\right)
\right\vert ^{2}=\sum_{K\in \mathcal{W}_{\func{good}}^{\limfunc{trip}}\left(
I\right) }\left( \frac{\mathrm{P}\left( K,\mathbf{1}_{F\setminus I}\sigma
\right) }{\ell \left( K\right) }\right) ^{2}\left\vert \mathsf{P}_{\Lambda %
\left[ K\right] }^{\omega }\right\vert Z\left( x\right) ^{2}.
\end{eqnarray*}
\end{proof}

\subsection{Reversal of the Mononotonicity Lemma\label{Sub Rev}}

Our proof will use a stopping energy inequality in order to bound the
stopping form. But first we note an easy pointwise precursor to this
inequality. Let $\left\{ I_{r}\right\} _{r=1}^{\infty }$ be a pairwise
disjoint decomposition of an interval $I$ into subintervals $I_{r}$. Then
for $x,y\in I_{r}\subset I$ with $y<x$:%
\begin{equation}
\frac{\mathrm{P}\left( I_{r},\mathbf{1}_{I\setminus I_{r}}\sigma \right) }{%
\ell \left( I_{r}\right) }\left[ x-y\right] \leq 2\left[ H_{\sigma }\mathbf{1%
}_{I\setminus I_{r}}\left( x\right) -H_{\sigma }\mathbf{1}_{I\setminus
I_{r}}\left( y\right) \right] .  \label{simple reversal}
\end{equation}%
Fix $r$ for the moment and set $I_{r}=\left[ a,b\right] $. We now fix $c\in
\left( a,b\right) $ such that\footnote{%
See \cite{LaSaUr2} for the easy modifications in the case when no such $c$
exists.} $\left\vert \left[ a,c\right] \right\vert _{\omega }=\left\vert %
\left[ c,b\right] \right\vert _{\omega }=\frac{1}{2}\left\vert \left[ a,b%
\right] \right\vert _{\omega }$, and set $I_{r,-}\equiv \left[ a,c\right] $
and $I_{r,+}\equiv \left[ c,b\right] $. Then for $x\in I_{r,+}$, average
both sides of the displayed inequality in the variable $y$ over $I_{r,-}$
with respect to $\omega $ to obtain%
\begin{equation}
\mathbf{1}_{I_{r,+}}\left( x\right) \frac{\mathrm{P}\left( I_{r},\mathbf{1}%
_{I\setminus I_{r}}\sigma \right) }{\ell \left( I_{r}\right) }\left[
x-E_{I_{r,-}}^{\omega }Z\right] \leq 2\mathbf{1}_{I_{r,+}}\left( x\right) %
\left[ H_{\sigma }\mathbf{1}_{I\setminus I_{r}}\left( x\right)
-E_{I_{r,-}}^{\omega }\left( H_{\sigma }\mathbf{1}_{I\setminus I_{r}}\right) %
\right] .  \label{PE bound}
\end{equation}

\begin{remark}
Using (\ref{PE bound}), one can control the $p$-Poisson-energy
characteristic $\mathcal{E}_{p}\left( \sigma ,\omega \right) $ by the scalar
testing characteristic $\mathfrak{T}_{H,p}^{\func{loc}}\left( \sigma ,\omega
\right) $ and Muckehhoupt characteristic $\mathcal{A}_{p}\left( \sigma
,\omega \right) $. See \cite[Proposition 2.11 ]{LaSaUr2} for the case $p=2$,
and below for the case $p\neq 2$. However, the Muckehhoupt characteristic $%
\mathcal{A}_{p}\left( \sigma ,\omega \right) $ can be dropped, see Lemma \ref%
{energy cond} below.
\end{remark}

\subsection{Carleson measures}

Here we recall some simple properties of Carleson measures and conditions
from \cite{LaSaShUr3} and \cite{SaShUr7}, where the case $p=2$ was
considered - the general case $1<p<\infty $ is similar. Let $\mathcal{F}%
\subset \mathcal{D}$ and let the corona $\mathcal{C}_{\mathcal{F}}\left(
F\right) $ consist of all intervals contained in $F$ that are not contained
in any smaller interval from $\mathcal{F}$. We say that the triple $\left(
C_{0},\mathcal{F},\alpha _{\mathcal{F}}\right) $ constitutes \emph{stopping
data} for a function $f\in L_{loc}^{1}\left( \mu \right) $ if,

(\textbf{1}): $\ \ E_{I}^{\mu }\left\vert f\right\vert \leq \alpha _{%
\mathcal{F}}\left( F\right) $ for all $I\in \mathcal{C}_{\mathcal{F}}\left(
F\right) $ and $F\in \mathcal{F}$,

(\textbf{2}): $\ \ \sum_{F^{\prime }\in \mathcal{F}:\ F^{\prime }\subset
F}\left\vert F^{\prime }\right\vert _{\mu }\leq C_{0}\left\vert F\right\vert
_{\mu }$ for all $F\in \mathcal{F}$,

(\textbf{3}): $\ \ \sum_{F\in \mathcal{F}}\alpha _{\mathcal{F}}\left(
F\right) ^{p}\left\vert F\right\vert _{\mu }\mathbf{\leq }%
C_{0}^{p}\left\Vert f\right\Vert _{L^{p}\left( \mu \right) }^{p}$,

(\textbf{4}): $\ \ \alpha _{\mathcal{F}}\left( F\right) \leq \alpha _{%
\mathcal{F}}\left( F^{\prime }\right) $ whenever $F^{\prime },F\in \mathcal{F%
}$ with $F^{\prime }\subset F$.

Moreover there is the following useful consequence of (2) and (3) that says
the sequence $\left\{ \alpha _{\mathcal{F}}\left( F\right) \mathbf{1}%
_{F}\right\} _{F\in \mathcal{F}}$ has a \emph{quasiorthogonal} property
relative to $f$ with a constant $C_{0}^{\prime }$ depending only on $C_{0}$:%
\begin{equation}
\left\Vert \sum_{F\in \mathcal{F}}\alpha _{\mathcal{F}}\left( F\right) 
\mathbf{1}_{F}\right\Vert _{L^{p}\left( \mu \right) }^{p}\leq C_{0}^{\prime
}\left\Vert f\right\Vert _{L^{p}\left( \mu \right) }^{p}.  \label{q orth}
\end{equation}%
Indeed, this follows easily from the fact that the Carleson condition (2)
implies a geometric decay in levels of the tree $\mathcal{F}$, namely that
there are positive constants $\delta $ and $C_{\delta }$, depending on $%
C_{0} $, such that if $\mathfrak{C}_{\mathcal{F}}^{\left( n\right) }\left(
F\right) $ denotes the set of $n^{th}$ generation children of $F$ in $%
\mathcal{F}$,%
\begin{equation}
\sum_{F^{\prime }\in \mathfrak{C}_{\mathcal{F}}^{\left( n\right) }\left(
F\right) }\left\vert F^{\prime }\right\vert _{\mu }\leq C_{\delta
}2^{-\delta n}\left\vert F\right\vert _{\mu },\ \ \ \ \ \text{for all }n\geq
0\text{ and }F\in \mathcal{F}.  \label{geom decay}
\end{equation}%
To see this well known result of Carleson, let $\beta _{k}\left( F\right)
\equiv \sum_{F^{\prime }\in \mathfrak{C}_{\mathcal{F}}^{\left( k\right)
}\left( F\right) }\left\vert F^{\prime }\right\vert _{\mu }$ and note that $%
\beta _{k+1}\left( F\right) \leq \beta _{k}\left( F\right) $ implies that
for any integer $N\geq C_{0}$, we have%
\begin{equation*}
\left( N+1\right) \beta _{N}\left( F\right) \leq \sum_{k=0}^{N}\beta
_{k}\left( F\right) \leq C_{0}\left\vert F\right\vert _{\mu }\ ,
\end{equation*}%
and hence 
\begin{equation*}
\beta _{N}\left( F\right) \leq \frac{C_{0}}{N+1}\left\vert F\right\vert
_{\mu }<\frac{1}{2}\left\vert F\right\vert _{\mu }\ ,\ \ \ \ \ \text{for }%
F\in \mathcal{F}\text{ and }N=\left[ 2C_{0}\right] _{0}.
\end{equation*}%
It follows by iteration that 
\begin{equation*}
\beta _{\ell N}\left( F\right) \leq \frac{1}{2}\beta _{\left( \ell -1\right)
N}\left( F\right) \leq ...\leq \frac{1}{2^{\ell }}\beta _{0}\left( F\right) =%
\frac{1}{2^{\ell }}\left\vert F\right\vert _{\mu },\ \ \ \ \ \ell =0,1,2,...
\end{equation*}%
and so given $n\in \mathbb{N}$, choose $\ell $ such that $\ell N\leq
n<\left( \ell +1\right) N$, and note that 
\begin{equation*}
\sum_{F^{\prime }\in \mathfrak{C}_{\mathcal{F}}^{\left( n\right) }\left(
F\right) }\left\vert F^{\prime }\right\vert _{\mu }=\beta _{n}\left(
F\right) \leq \beta _{\ell N}\left( F\right) \leq \frac{1}{2^{\frac{n}{N}}}%
\left\vert F\right\vert _{\mu }=2^{-\frac{n}{\left[ 2C_{0}\right] }%
}\left\vert F\right\vert _{\mu }=2^{-n\delta }\left\vert F\right\vert _{\mu
}\ ,
\end{equation*}%
which proves the geometric decay (\ref{geom decay}). Inequality (\ref{q orth}%
) will actually be proved in the course of proving the next theorem. With a
slight abuse of notation we will also refer to inequality (3) above as a
quasiorthogonality property.

The following inequality will be used in controlling both the far form and
the paraproduct form later on. Its proof uses a technique introduced in \cite%
[see the proof of the bound for the paraproduct form]{SaWi}, that goes back
decades in other situations, see e.g. \cite{BoBo}. Actually, the proof here
predates that in \cite{SaWi}, but this result was not needed in the setting
of doubling measures, and so was left out of \cite{SaWi}.

\begin{theorem}
\label{using Carleson}Suppose that the triple $\left( C_{0},\mathcal{F}%
,\alpha _{\mathcal{F}}\right) $ constitutes \emph{stopping data} for a
function $f\in L_{loc}^{1}\left( \mu \right) $, and for $\kappa \in \mathbb{Z%
}_{+}$, set 
\begin{equation*}
\alpha _{\mathcal{F}}^{\kappa }\left( x\right) \equiv \left\{ \alpha _{%
\mathcal{F}}\left( F\right) \mathbf{1}_{F^{\kappa }}\left( x\right) \right\}
_{F\in \mathcal{F}}\ \text{where }F^{\kappa }\equiv \bigcup_{G\in \mathfrak{C%
}_{\mathcal{F}}^{\left( \kappa \right) }\left( F\right) }G\text{ }.
\end{equation*}%
Then for $1<p<\infty $, 
\begin{equation}
\int_{\mathbb{R}}\left\vert \alpha _{\mathcal{F}}^{\kappa }\left( x\right)
\right\vert _{\ell ^{2}}^{p}d\mu \left( x\right) =\int_{\mathbb{R}}\left(
\sum_{F\in \mathcal{F}}\left\vert \alpha _{\mathcal{F}}\left( F\right)
\right\vert ^{2}\mathbf{1}_{F^{\kappa }}\left( x\right) \right) ^{\frac{p}{2}%
}d\mu \left( x\right) \leq C_{\delta }2^{-\delta \kappa }\sum_{F\in \mathcal{%
F}}\alpha _{\mathcal{F}}\left( F\right) ^{p}\left\vert F\right\vert _{\mu }\
,  \label{The claim}
\end{equation}%
where $\delta >0$ is the constant in (\ref{geom decay}). The inequality can
be reversed for $\kappa =0$ and $2\leq p<\infty $.
\end{theorem}

\begin{proof}
We begin with the observation that $F_{1}^{\kappa }\subset F_{2}^{\kappa }$
whenever $F_{1}\subset F_{2}$, which will be used repeatedly below without
further mention. We now claim that for $1<p<\infty $, 
\begin{equation}
\int_{\mathbb{R}}\left( \sum_{F\in \mathcal{F}}\left\vert \alpha _{\mathcal{F%
}}\left( F\right) \right\vert ^{2}\mathbf{1}_{F^{\kappa }}\left( x\right)
\right) ^{\frac{p}{2}}d\mu \left( x\right) \leq C_{\delta }2^{-\delta \kappa
}\sum_{F\in \mathcal{F}}\alpha _{\mathcal{F}}\left( F\right) ^{p}\left\vert
F\right\vert _{\mu }.  \label{first claim}
\end{equation}%
Indeed, for $1<p\leq 2$ (and even for $0<p\leq 2$), the inequality follows
from the trivial inequality $\left\Vert \cdot \right\Vert _{\ell ^{q}}\leq
\left\Vert \cdot \right\Vert _{\ell ^{1}}$ for $0<q\leq 1$,%
\begin{eqnarray*}
&&\int_{\mathbb{R}}\left( \sum_{F\in \mathcal{F}}\left\vert \alpha _{%
\mathcal{F}}\left( F\right) \right\vert ^{2}\mathbf{1}_{F^{\kappa }}\left(
x\right) \right) ^{\frac{p}{2}}d\mu \left( x\right) \leq \int_{\mathbb{R}%
}\sum_{F\in \mathcal{F}}\left\vert \alpha _{\mathcal{F}}\left( F\right)
\right\vert ^{p}\mathbf{1}_{F^{\kappa }}\left( x\right) d\mu \left( x\right)
\\
&&\ \ \ \ \ \ \ \ \ \ \ \ \ \ \ =\sum_{F\in \mathcal{F}}\alpha _{\mathcal{F}%
}\left( F\right) ^{p}\left\vert F^{\kappa }\right\vert _{\mu }\leq C_{\delta
}2^{-\delta \kappa }\sum_{F\in \mathcal{F}}\alpha _{\mathcal{F}}\left(
F\right) ^{p}\left\vert F\right\vert _{\mu },
\end{eqnarray*}%
where $\delta >0$ is the geometric decay in generations exponent in (\ref%
{geom decay}).

Now we turn to the case $p\geq 2$. When $p=2m$ is an even positive integer,
we will set%
\begin{equation*}
\mathcal{F}_{\ast }^{2m}\equiv \left\{ \left( F_{1},...,F_{2m}\right) \in 
\mathcal{F}\times ...\times \mathcal{F}:F_{i}\subset F_{j}\text{ for }1\leq
i\leq j\leq 2m\text{, and }F_{i}=F_{i+1}\text{ for all odd }i\right\} ,
\end{equation*}%
and then by symmetry we can arrange the intervals below in nondecreasing
order to obtain%
\begin{eqnarray*}
&&\int_{\mathbb{R}}\left( \sum_{F\in \mathcal{F}}\left\vert \alpha _{%
\mathcal{F}}\left( F\right) \mathbf{1}_{F^{\kappa }}\left( x\right)
\right\vert ^{2}\right) ^{\frac{p}{2}}d\mu \left( x\right) =\int_{\mathbb{R}%
}\left( \sum_{F\in \mathcal{F}}\left\vert \alpha _{\mathcal{F}}\left(
F\right) \mathbf{1}_{F^{\kappa }}\left( x\right) \right\vert ^{2}\right)
^{m}d\mu \left( x\right) \\
&=&\int_{\mathbb{R}}\sum_{\left( F_{1},...,F_{2m}\right) \in \mathcal{F}%
^{2m}}\alpha _{\mathcal{F}}\left( F_{1}\right) ...\alpha _{\mathcal{F}%
}\left( F_{2m}\right) \mathbf{1}_{F_{1}^{\kappa }\cap ...\cap F_{2m}^{\kappa
}}d\mu \left( x\right) \\
&=&C_{m}\int_{\mathbb{R}}\sum_{\left( F_{1},...,F_{2m}\right) \in \mathcal{F}%
_{\ast }^{2m}}\alpha _{\mathcal{F}}\left( F_{1}\right) ...\alpha _{\mathcal{F%
}}\left( F_{2m}\right) \mathbf{1}_{F_{1}^{\kappa }\cap ...\cap
F_{2m}^{\kappa }}d\mu \left( x\right) \\
&=&C_{m}\sum_{\left( F_{1},...,F_{2m}\right) \in \mathcal{F}_{\ast
}^{2m}}\alpha _{\mathcal{F}}\left( F_{1}\right) ...\alpha _{\mathcal{F}%
}\left( F_{2m}\right) \left\vert F_{1}^{\kappa }\right\vert _{\mu }=C_{m}%
\func{Int}^{\kappa }\left( m\right) ,
\end{eqnarray*}%
where from the geometric decay in (\ref{geom decay}), we obtain%
\begin{eqnarray}
\func{Int}^{\kappa }\left( m\right) &\equiv &\sum_{\left(
F_{1},...,F_{2m}\right) \in \mathcal{F}_{\ast }^{2m}}\alpha _{\mathcal{F}%
}\left( F_{1}\right) ...\alpha _{\mathcal{F}}\left( F_{2m}\right) \left\vert
F_{1}^{\kappa }\right\vert _{\mu }\lesssim 2^{-\delta \kappa }\func{Int}%
\left( m\right) ,  \label{Int} \\
\text{where }\func{Int}\left( m\right) &\equiv &\sum_{\left(
F_{1},...,F_{2m}\right) \in \mathcal{F}_{\ast }^{2m}}\alpha _{\mathcal{F}%
}\left( F_{1}\right) ...\alpha _{\mathcal{F}}\left( F_{2m}\right) \left\vert
F_{1}\right\vert _{\mu }.  \notag
\end{eqnarray}

We now set about showing that%
\begin{equation*}
\func{Int}\left( m\right) \lesssim \sum_{F\in \mathcal{F}}\left\vert \alpha
_{\mathcal{F}}\left( F\right) \right\vert ^{2m}\left\vert F\right\vert _{\mu
}\ .
\end{equation*}%
For this, we first prove (\ref{q orth}) in order to outline the main idea.
Using the geometric decay in (\ref{geom decay}) once more we obtain 
\begin{eqnarray*}
\sum_{n=0}^{\infty }\sum_{F^{\prime }\in \mathfrak{C}_{\mathcal{F}}^{\left(
n\right) }\left( F\right) :\ }\alpha _{\mathcal{F}}\left( F^{\prime }\right)
\left\vert F^{\prime }\right\vert _{\mu } &\leq &\sum_{n=0}^{\infty }\sqrt{%
\sum_{F^{\prime }\in \mathfrak{C}_{\mathcal{F}}^{\left( n\right) }\left(
F\right) }\alpha _{\mathcal{F}}\left( F^{\prime }\right) ^{2}\left\vert
F^{\prime }\right\vert _{\mu }}C_{\delta }2^{-\delta n}\sqrt{\left\vert
F\right\vert _{\mu }} \\
&\leq &C_{\delta }\sqrt{\left\vert F\right\vert _{\mu }}\sqrt{%
\sum_{n=0}^{\infty }2^{-\delta n}\sum_{F^{\prime }\in \mathfrak{C}_{\mathcal{%
F}}^{\left( n\right) }\left( F\right) }\alpha _{\mathcal{F}}\left( F^{\prime
}\right) ^{2}\left\vert F^{\prime }\right\vert _{\mu }},
\end{eqnarray*}%
and hence that$\ $%
\begin{eqnarray*}
&&\sum_{F\in \mathcal{F}}\alpha _{\mathcal{F}}\left( F\right) \left\{
\sum_{n=0}^{\infty }\sum_{F^{\prime }\in \mathfrak{C}_{\mathcal{F}}^{\left(
n\right) }\left( F\right) }\alpha _{\mathcal{F}}\left( F^{\prime }\right)
\left\vert F^{\prime }\right\vert _{\mu }\right\} \\
&\lesssim &\sum_{F\in \mathcal{F}}\alpha _{\mathcal{F}}\left( F\right) \sqrt{%
\left\vert F\right\vert _{\mu }}\sqrt{\sum_{n=0}^{\infty }2^{-\delta
n}\sum_{F^{\prime }\in \mathfrak{C}_{\mathcal{F}}^{\left( n\right) }\left(
F\right) }\alpha _{\mathcal{F}}\left( F^{\prime }\right) ^{2}\left\vert
F^{\prime }\right\vert _{\mu }} \\
&\lesssim &\left( \sum_{F\in \mathcal{F}}\alpha _{\mathcal{F}}\left(
F\right) ^{2}\left\vert F\right\vert _{\mu }\right) ^{\frac{1}{2}}\left(
\sum_{n=0}^{\infty }2^{-\delta n}\sum_{F\in \mathcal{F}}\sum_{F^{\prime }\in 
\mathfrak{C}_{\mathcal{F}}^{\left( n\right) }\left( F\right) }\alpha _{%
\mathcal{F}}\left( F^{\prime }\right) ^{2}\left\vert F^{\prime }\right\vert
_{\mu }\right) ^{\frac{1}{2}} \\
&\lesssim &\left\Vert f\right\Vert _{L^{2}\left( \mu \right) }\left(
\sum_{F^{\prime }\in \mathcal{F}}\alpha _{\mathcal{F}}\left( F^{\prime
}\right) ^{2}\left\vert F^{\prime }\right\vert _{\mu }\right) ^{\frac{1}{2}%
}\lesssim \left\Vert f\right\Vert _{L^{2}\left( \mu \right) }^{2}.
\end{eqnarray*}%
This proves (\ref{q orth}) since $\left\Vert \sum_{F\in \mathcal{F}}\alpha _{%
\mathcal{F}}\left( F\right) \mathbf{1}_{F}\right\Vert _{L^{2}\left( \mu
\right) }^{2}$ is dominated by twice the left hand side above.

We now adapt this last argument to apply to (\ref{Int}). For example, in the
case $m=2$, we have that%
\begin{eqnarray*}
&&\func{Int}\left( 2\right) =\sum_{F_{4}\in \mathcal{F}}\alpha _{\mathcal{F}%
}\left( F_{4}\right) \sum_{F_{3}\subset F_{4}}\alpha _{\mathcal{F}}\left(
F_{3}\right) \sum_{F_{2}\subset F_{3}}\alpha _{\mathcal{F}}\left(
F_{2}\right) \sum_{F_{1}\subset F_{2}}\alpha _{\mathcal{F}}\left(
F_{1}\right) \left\vert F_{1}\right\vert _{\mu } \\
&=&\sum_{F_{4}\in \mathcal{F}}\alpha _{\mathcal{F}}\left( F_{4}\right)
\left( \sum_{n_{3}=0}^{\infty }\sum_{F_{3}\in \mathfrak{C}_{\mathcal{F}%
}^{\left( n_{3}\right) }\left( F_{4}\right) }\alpha _{\mathcal{F}}\left(
F_{3}\right) \left( \sum_{n_{2}=0}^{\infty }\sum_{F_{2}\in \mathfrak{C}_{%
\mathcal{F}}^{\left( n_{2}\right) }\left( F_{3}\right) }\alpha _{\mathcal{F}%
}\left( F_{2}\right) \left( \sum_{n_{1}=0}^{\infty }\sum_{F_{1}\in \mathfrak{%
C}_{\mathcal{F}}^{\left( n_{1}\right) }\left( F_{2}\right) }\alpha _{%
\mathcal{F}}\left( F_{1}\right) \left\vert F_{1}\right\vert _{\mu }\right)
\right) \right)
\end{eqnarray*}%
which is at most (we continue to write $m$ in place of $2$ until the very
end of the argument)%
\begin{eqnarray*}
&&C_{\delta }\sum_{n_{3}=0}^{\infty }\sum_{n_{2}=0}^{\infty
}\sum_{n_{1}=0}^{\infty }\sum_{F_{4}\in \mathcal{F}}\alpha _{\mathcal{F}%
}\left( F_{4}\right) \sum_{F_{3}\in \mathfrak{C}_{\mathcal{F}}^{\left(
n_{3}\right) }\left( F_{4}\right) }\alpha _{\mathcal{F}}\left( F_{3}\right)
\\
&&\ \ \ \ \ \ \ \ \ \ \ \ \ \ \ \ \ \ \ \ \ \ \ \ \ \ \ \ \ \ \times
\sum_{F_{2}\in \mathfrak{C}_{\mathcal{F}}^{\left( n_{2}\right) }\left(
F_{3}\right) }\alpha _{\mathcal{F}}\left( F_{2}\right) \left( 2^{-\delta
n_{1}}\left\vert F_{2}\right\vert _{\mu }\right) ^{\frac{2m-1}{2m}}\left(
\sum_{F_{1}\in \mathfrak{C}_{\mathcal{F}}^{\left( n_{1}\right) }\left(
F_{2}\right) }\alpha _{\mathcal{F}}\left( F_{1}\right) ^{2m}\left\vert
F_{1}\right\vert _{\mu }\right) ^{\frac{1}{2m}} \\
&=&C_{\delta }\sum_{n_{3}=0}^{\infty }\sum_{n_{2}=0}^{\infty
}\sum_{n_{1}=0}^{\infty }2^{-\delta \frac{2m-1}{2m}n_{1}}\sum_{F_{4}\in 
\mathcal{F}}\alpha _{\mathcal{F}}\left( F_{4}\right) \sum_{F_{3}\in 
\mathfrak{C}_{\mathcal{F}}^{\left( n_{3}\right) }\left( F_{4}\right) }\alpha
_{\mathcal{F}}\left( F_{3}\right) \\
&&\ \ \ \ \ \ \ \ \ \ \ \ \ \ \ \ \ \ \ \ \ \ \ \ \ \ \ \ \ \ \times
\sum_{F_{2}\in \mathfrak{C}_{\mathcal{F}}^{\left( n_{2}\right) }\left(
F_{3}\right) }\alpha _{\mathcal{F}}\left( F_{2}\right) \left\vert
F_{2}\right\vert _{\mu }^{\frac{1}{2m}}\left( \sum_{F_{1}\in \mathfrak{C}_{%
\mathcal{F}}^{\left( n_{1}\right) }\left( F_{2}\right) }\alpha _{\mathcal{F}%
}\left( F_{1}\right) ^{2m}\left\vert F_{1}\right\vert _{\mu }\right) ^{\frac{%
1}{2m}}\left\vert F_{2}\right\vert _{\mu }^{1-\frac{2}{2m}},
\end{eqnarray*}%
which is in turn dominated by%
\begin{eqnarray*}
&&C_{\delta }\sum_{n_{3}=0}^{\infty }\sum_{n_{2}=0}^{\infty
}\sum_{n_{1}=0}^{\infty }2^{-\delta \frac{2m-1}{2m}n_{1}}\sum_{F_{4}\in 
\mathcal{F}}\alpha _{\mathcal{F}}\left( F_{4}\right) \sum_{F_{3}\in 
\mathfrak{C}_{\mathcal{F}}^{\left( n_{3}\right) }\left( F_{4}\right) }\alpha
_{\mathcal{F}}\left( F_{3}\right) \\
&&\times \left( \sum_{F_{2}\in \mathfrak{C}_{\mathcal{F}}^{\left(
n_{2}\right) }\left( F_{3}\right) }\alpha _{\mathcal{F}}\left( F_{2}\right)
^{2m}\left\vert F_{2}\right\vert _{\mu }\right) ^{\frac{1}{2m}}\left(
\sum_{F_{2}\in \mathfrak{C}_{\mathcal{F}}^{\left( n_{2}\right) }\left(
F_{3}\right) }\sum_{F_{1}\in \mathfrak{C}_{\mathcal{F}}^{\left( n_{1}\right)
}\left( F_{2}\right) }\alpha _{\mathcal{F}}\left( F_{1}\right)
^{2m}\left\vert F_{1}\right\vert _{\mu }\right) ^{\frac{1}{2m}}\left(
2^{-\delta n_{2}}\left\vert F_{3}\right\vert _{\mu }\right) ^{\frac{2m-2}{2m}%
},
\end{eqnarray*}%
where in the last line we have applied H\"{o}lder's inequality with
exponents $\left( 2m,2m,\frac{2m}{2m-2}\right) $, and then used that $%
\sum_{F_{2}\in \mathfrak{C}_{\mathcal{F}}^{\left( n_{2}\right) }\left(
F_{3}\right) }\left\vert F_{2}\right\vert _{\mu }\leq C_{\delta }2^{-\delta
n_{2}}\left\vert F_{3}\right\vert _{\mu }$.

Continuing in this way, we dominate the sum above by%
\begin{eqnarray*}
&\lesssim &\sum_{n_{3}=0}^{\infty }\sum_{n_{2}=0}^{\infty
}\sum_{n_{1}=0}^{\infty }2^{-\delta \frac{2m-1}{2m}n_{1}}\sum_{F_{4}\in 
\mathcal{F}}\alpha _{\mathcal{F}}\left( F_{4}\right) \sum_{F_{3}\in 
\mathfrak{C}_{\mathcal{F}}^{\left( n_{3}\right) }\left( F_{4}\right) }\alpha
_{\mathcal{F}}\left( F_{3}\right) \\
&&\times \left( \sum_{F_{2}\in \mathfrak{C}_{\mathcal{F}}^{\left(
n_{2}\right) }\left( F_{3}\right) }\alpha _{\mathcal{F}}\left( F_{2}\right)
^{2m}\left\vert F_{2}\right\vert _{\mu }\right) ^{\frac{1}{2m}}\left(
\sum_{F_{2}\in \mathfrak{C}_{\mathcal{F}}^{\left( n_{2}\right) }\left(
F_{3}\right) }\sum_{F_{1}\in \mathfrak{C}_{\mathcal{F}}^{\left( n_{1}\right)
}\left( F_{2}\right) }\alpha _{\mathcal{F}}\left( F_{1}\right)
^{2m}\left\vert F_{1}\right\vert _{\mu }\right) ^{\frac{1}{2m}}\left(
2^{-\delta n_{2}}\left\vert F_{3}\right\vert _{\mu }\right) ^{1-\frac{2}{2m}}
\\
&=&\sum_{n_{3}=0}^{\infty }\sum_{n_{2}=0}^{\infty }\sum_{n_{1}=0}^{\infty
}2^{-\delta \left( 1-\frac{1}{2m}\right) n_{1}-\delta \left( 1-\frac{2}{2m}%
\right) n_{2}}\sum_{F_{4}\in \mathcal{F}}\alpha _{\mathcal{F}}\left(
F_{4}\right) \\
&&\times \sum_{F_{3}\in \mathfrak{C}_{\mathcal{F}}^{\left( n_{3}\right)
}\left( F_{4}\right) }\alpha _{\mathcal{F}}\left( F_{3}\right) \left\vert
F_{3}\right\vert _{\mu }^{\frac{1}{2m}}\left( \sum_{F_{2}\in \mathfrak{C}_{%
\mathcal{F}}^{\left( n_{2}\right) }\left( F_{3}\right) }\alpha _{\mathcal{F}%
}\left( F_{2}\right) ^{2m}\left\vert F_{2}\right\vert _{\mu }\right) ^{\frac{%
1}{2m}} \\
&&\times \left( \sum_{F_{2}\in \mathfrak{C}_{\mathcal{F}}^{\left(
n_{2}\right) }\left( F_{3}\right) }\sum_{F_{1}\in \mathfrak{C}_{\mathcal{F}%
}^{\left( n_{1}\right) }\left( F_{2}\right) }\alpha _{\mathcal{F}}\left(
F_{1}\right) ^{2m}\left\vert F_{1}\right\vert _{\mu }\right) ^{\frac{1}{2m}%
}\left\vert F_{3}\right\vert _{\mu }^{1-\frac{3}{2m}}
\end{eqnarray*}%
and continuing with $\frac{2m-4}{2m}=0$ for $m=2$, we have the upper bound,%
\begin{eqnarray*}
&&\sum_{n_{3}=0}^{\infty }\sum_{n_{2}=0}^{\infty }\sum_{n_{1}=0}^{\infty
}2^{-\delta \left[ \left( 1-\frac{1}{2m}\right) n_{1}+\left( 1-\frac{2}{2m}%
\right) n_{2}+\left( 1-\frac{3}{2m}\right) n_{3}\right] }\sum_{F_{4}\in 
\mathcal{F}}\alpha _{\mathcal{F}}\left( F_{4}\right) \left\vert
F_{4}\right\vert _{\mu }^{\frac{1}{2m}}\left( \sum_{F_{3}\in \mathfrak{C}_{%
\mathcal{F}}^{\left( n_{3}\right) }\left( F_{4}\right) }\alpha _{\mathcal{F}%
}\left( F_{3}\right) ^{2m}\left\vert F_{3}\right\vert _{\mu }\right) ^{\frac{%
1}{2m}} \\
&&\times \left( \sum_{F_{3}\in \mathfrak{C}_{\mathcal{F}}^{\left(
n_{3}\right) }\left( F_{4}\right) }\sum_{F_{2}\in \mathfrak{C}_{\mathcal{F}%
}^{\left( n_{2}\right) }\left( F_{3}\right) }\alpha _{\mathcal{F}}\left(
F_{2}\right) ^{2m}\left\vert F_{2}\right\vert _{\mu }\right) ^{\frac{1}{2m}}
\\
&&\times \left( \sum_{F_{3}\in \mathfrak{C}_{\mathcal{F}}^{\left(
n_{3}\right) }\left( F_{4}\right) }\sum_{F_{2}\in \mathfrak{C}_{\mathcal{F}%
}^{\left( n_{2}\right) }\left( F_{3}\right) }\sum_{F_{1}\in \mathfrak{C}_{%
\mathcal{F}}^{\left( n_{1}\right) }\left( F_{2}\right) }\alpha _{\mathcal{F}%
}\left( F_{1}\right) ^{2m}\left\vert F_{1}\right\vert _{\mu }\right) ^{\frac{%
1}{2m}}\left\vert F_{4}\right\vert _{\mu }^{\frac{2m-4}{2m}},
\end{eqnarray*}%
which is at most%
\begin{eqnarray*}
&&\sum_{n_{3}=0}^{\infty }\sum_{n_{2}=0}^{\infty }\sum_{n_{1}=0}^{\infty
}2^{-\delta \left[ \left( 1-\frac{1}{2m}\right) n_{1}+\left( 1-\frac{2}{2m}%
\right) n_{2}+\left( 1-\frac{3}{2m}\right) n_{3}\right] }\left(
\sum_{F_{4}\in \mathcal{F}}\alpha _{\mathcal{F}}\left( F_{4}\right)
^{2m}\left\vert F_{4}\right\vert _{\mu }\right) ^{\frac{1}{2m}} \\
&&\times \left( \sum_{F_{4}\in \mathcal{F}}\sum_{F_{3}\in \mathfrak{C}_{%
\mathcal{F}}^{\left( n_{3}\right) }\left( F_{4}\right) }\alpha _{\mathcal{F}%
}\left( F_{3}\right) ^{2m}\left\vert F_{3}\right\vert _{\mu }\right) ^{\frac{%
1}{2m}}\left( \sum_{F_{4}\in \mathcal{F}}\sum_{F_{3}\in \mathfrak{C}_{%
\mathcal{F}}^{\left( n_{3}\right) }\left( F_{4}\right) }\sum_{F_{2}\in 
\mathfrak{C}_{\mathcal{F}}^{\left( n_{2}\right) }\left( F_{3}\right) }\alpha
_{\mathcal{F}}\left( F_{2}\right) ^{2m}\left\vert F_{2}\right\vert _{\mu
}\right) ^{\frac{1}{2m}} \\
&&\times \left( \sum_{F_{4}\in \mathcal{F}}\sum_{F_{3}\in \mathfrak{C}_{%
\mathcal{F}}^{\left( n_{3}\right) }\left( F_{4}\right) }\sum_{F_{2}\in 
\mathfrak{C}_{\mathcal{F}}^{\left( n_{2}\right) }\left( F_{3}\right)
}\sum_{F_{1}\in \mathfrak{C}_{\mathcal{F}}^{\left( n_{1}\right) }\left(
F_{2}\right) }\alpha _{\mathcal{F}}\left( F_{1}\right) ^{2m}\left\vert
F_{1}\right\vert _{\mu }\right) ^{\frac{1}{2m}}.
\end{eqnarray*}%
Finally, since 
\begin{eqnarray*}
\sum_{F_{4}\in \mathcal{F}}\sum_{F_{3}\in \mathfrak{C}_{\mathcal{F}}^{\left(
n_{3}\right) }\left( F_{4}\right) }\sum_{F_{2}\in \mathfrak{C}_{\mathcal{F}%
}^{\left( n_{2}\right) }\left( F_{3}\right) }\sum_{F_{1}\in \mathfrak{C}_{%
\mathcal{F}}^{\left( n_{1}\right) }\left( F_{2}\right) }\alpha _{\mathcal{F}%
}\left( F_{1}\right) ^{2m}\left\vert F_{1}\right\vert _{\mu } &\leq
&\sum_{F\in \mathcal{F}}\alpha _{\mathcal{F}}\left( F\right) ^{2m}\left\vert
F\right\vert _{\mu }, \\
\sum_{F_{4}\in \mathcal{F}}\sum_{F_{3}\in \mathfrak{C}_{\mathcal{F}}^{\left(
n_{3}\right) }\left( F_{4}\right) }\sum_{F_{2}\in \mathfrak{C}_{\mathcal{F}%
}^{\left( n_{2}\right) }\left( F_{3}\right) }\alpha _{\mathcal{F}}\left(
F_{2}\right) ^{2m}\left\vert F_{2}\right\vert _{\mu } &\leq &\sum_{F\in 
\mathcal{F}}\alpha _{\mathcal{F}}\left( F\right) ^{2m}\left\vert
F\right\vert _{\mu }, \\
\sum_{F_{4}\in \mathcal{F}}\sum_{F_{3}\in \mathfrak{C}_{\mathcal{F}}^{\left(
n_{3}\right) }\left( F_{4}\right) }\alpha _{\mathcal{F}}\left( F_{3}\right)
^{2m}\left\vert F_{3}\right\vert _{\mu } &\leq &\sum_{F\in \mathcal{F}%
}\alpha _{\mathcal{F}}\left( F\right) ^{2m}\left\vert F\right\vert _{\mu },
\end{eqnarray*}%
we obtain that $\func{Int}\left( 2\right) $ is dominated by%
\begin{equation*}
\sum_{n_{3}=0}^{\infty }\sum_{n_{2}=0}^{\infty }\sum_{n_{1}=0}^{\infty
}2^{-\delta \left[ \left( 1-\frac{1}{2m}\right) n_{1}+\left( 1-\frac{2}{2m}%
\right) n_{2}+\left( 1-\frac{3}{2m}\right) n_{3}\right] }\sum_{F\in \mathcal{%
F}}\alpha _{\mathcal{F}}\left( F\right) ^{2m}\left\vert F\right\vert _{\mu
}=C_{\delta ,p}\sum_{F\in \mathcal{F}}\alpha _{\mathcal{F}}\left( F\right)
^{2m}\left\vert F\right\vert _{\mu }\ .
\end{equation*}%
This together with (\ref{Int}), proves%
\begin{equation*}
\int_{\mathbb{R}}\left\vert \alpha _{\mathcal{F}}^{\kappa }\left( x\right)
\right\vert _{\ell ^{2}}^{4}d\mu \left( x\right) \lesssim 2^{-\delta \kappa
}\sum_{F\in \mathcal{F}}\alpha _{\mathcal{F}}\left( F\right) ^{4}\left\vert
F\right\vert _{\mu }\ .
\end{equation*}

Similarly, we can show for $m\geq 3$ that%
\begin{equation*}
\int_{\mathbb{R}}\left\vert \alpha _{\mathcal{F}}^{\kappa }\left( x\right)
\right\vert _{\ell ^{2}}^{2m}d\mu \left( x\right) \lesssim 2^{-\delta \kappa
}\sum_{F\in \mathcal{F}}\alpha _{\mathcal{F}}\left( F\right) ^{2m}\left\vert
F\right\vert _{\mu }.
\end{equation*}%
Altogether then we have%
\begin{equation*}
\int_{\mathbb{R}}\left\vert \alpha _{\mathcal{F}}^{\kappa }\left( x\right)
\right\vert _{\ell ^{2}}^{p}d\mu \left( x\right) \lesssim 2^{-\delta \kappa
}\sum_{F\in \mathcal{F}}\alpha _{\mathcal{F}}\left( F\right) ^{p}\left\vert
F\right\vert _{\mu },\ \ \ \ \ \text{for }p\in \left( 0,2\right] \cup
\left\{ 2m\right\} _{m\in \mathbb{N}}\ ,
\end{equation*}%
where $\alpha _{\mathcal{F}}^{\kappa }\left( x\right) \equiv \left\{ \alpha
_{\mathcal{F}}\left( F\right) \mathbf{1}_{F^{\kappa }}\left( x\right)
\right\} _{F\in \mathcal{F}}$. Marcinkiewicz interpolation \cite[Theorem
1.18 on page 480]{GaRu} applied with the linear operator taking sequences of
numbers $\left\{ \alpha _{\mathcal{F}}\left( F\right) \right\} _{F\in 
\mathcal{F}}\in \ell ^{p}\left( \mathcal{F},\left\vert F\right\vert _{\mu
}\right) $ to sequences of functions $\left\{ \alpha _{\mathcal{F}}\left(
F\right) \mathbf{1}_{F^{\kappa }}\left( x\right) \right\} _{F\in \mathcal{F}%
}\in L^{p}\left( \ell ^{2};\omega \right) $, now gives this inequality for
all $1<p<\infty $, and this completes the proof of (\ref{first claim}),
which is the inequality in (\ref{The claim}).

For the reverse inequality when $\kappa =0$ and $2\leq p<\infty $, we have
with $\alpha _{\mathcal{F}}\left( x\right) =\alpha _{\mathcal{F}}^{0}\left(
x\right) $ that%
\begin{eqnarray*}
&&\int_{\mathbb{R}}\left\vert \alpha _{\mathcal{F}}\left( x\right)
\right\vert _{\ell ^{2}}^{p}d\mu \left( x\right) =\int_{\mathbb{R}}\left(
\sum_{F\in \mathcal{F}}\left\vert \alpha _{\mathcal{F}}\left( F\right) 
\mathbf{1}_{F}\left( x\right) \right\vert ^{2}\right) ^{\frac{p}{2}}d\mu
\left( x\right) \\
&\gtrsim &\int_{\mathbb{R}}\sum_{F\in \mathcal{F}}\left\vert \alpha _{%
\mathcal{F}}\left( F\right) \mathbf{1}_{F}\left( x\right) \right\vert
^{p}d\mu \left( x\right) =\sum_{F\in \mathcal{F}}\alpha _{\mathcal{F}}\left(
F\right) ^{p}\left\vert F\right\vert _{\mu }.
\end{eqnarray*}
\end{proof}

\subsection{$\mathcal{F}$-square functions}

Recall that the Haar square function%
\begin{equation*}
\mathcal{S}f\left( x\right) =\mathcal{S}_{\limfunc{Haar}}^{\mu }f\left(
x\right) \equiv \left( \sum_{I\in \mathcal{D}}\left\vert \bigtriangleup
_{I}^{\mu }f\left( x\right) \right\vert ^{2}\right) ^{\frac{1}{2}}
\end{equation*}%
is bounded on $L^{p}\left( \mu \right) $ for any $1<p<\infty $ and any
locally finite positive Borel measure $\mu $ by Burkholder's theorem \cite%
{Bur1}, \cite{Bur2} and Khintchine's inequality - see also the excellent
lecture notes \cite[Exercise 4 on page 19]{Hyt2} - simply because $\mathcal{S%
}_{\limfunc{Haar}}^{\mu }$ is a martingale difference square function.

We now recall extensions of this result to more complicated corona square
functions with locally finite positive Borel measures on $\mathbb{R}$ that
were derived in \cite{SaWi}\ (and treated there in $\mathbb{R}^{n}$). Fix a $%
\mathcal{D}$-dyadic interval $F_{0}$, let $\mu $\ be a locally finite
positive Borel measure on $F_{0}$, and suppose that $\mathcal{F}$ is a
subset of $\mathcal{D}\left[ F_{0}\right] \equiv \left\{ I\in \mathcal{D}%
:I\subset F_{0}\right\} $. The collection $\left\{ \mathcal{C}_{\mathcal{F}%
}\left( F\right) \right\} _{F\in \mathcal{F}}$ of subsets $\mathcal{C}_{%
\mathcal{F}}\left( F\right) \subset \mathcal{D}\left[ F_{0}\right] $ is
defined by%
\begin{equation*}
\mathcal{C}_{\mathcal{F}}\left( F\right) \equiv \left\{ I\in \mathcal{D}%
:I\subset F\text{ and }I\not\subset F^{\prime }\text{ for any }\mathcal{F}%
\text{-child }F^{\prime }\text{ of }F\right\} ,\ \ \ \ \ F\in \mathcal{F},
\end{equation*}%
and%
\begin{eqnarray*}
&&\mathcal{C}_{\mathcal{F}}\left( F\right) \text{ is connected for each }%
F\in \mathcal{F}, \\
&&F\in \mathcal{C}_{\mathcal{F}}\left( F\right) \text{ and }I\in \mathcal{C}%
_{\mathcal{F}}\left( F\right) \Longrightarrow I\subset F\text{ for each }%
F\in \mathcal{F}, \\
&&\mathcal{C}_{\mathcal{F}}\left( F\right) \cap \mathcal{C}_{\mathcal{F}%
}\left( F^{\prime }\right) =\emptyset \text{ for all distinct }F,F^{\prime
}\in \mathcal{F}, \\
&&\mathcal{D}\left[ F_{0}\right] =\bigcup_{F\in \mathcal{F}}\mathcal{C}_{%
\mathcal{F}}\left( F\right) .
\end{eqnarray*}%
The subset $\mathcal{C}_{\mathcal{F}}\left( F\right) $ of $\mathcal{D}$ is
referred to as the $\mathcal{F}$-corona with top $F$, and the decomposition\ 
$\mathcal{D}\left[ F_{0}\right] =\bigcup_{F\in \mathcal{F}}\mathcal{C}_{%
\mathcal{F}}\left( F\right) $ is referred to as the corresponding corona
decomposition. We emphasize that there is no assumption of $\func{good}$\
intervals here.

Define the corona projections $\mathsf{P}_{\mathcal{C}_{\mathcal{F}}\left(
F\right) }^{\mu }\equiv \sum_{I\in \mathcal{C}_{\mathcal{F}}\left( F\right)
}\bigtriangleup _{I}^{\mu }$ and group them together according to their
depth in the tree $\mathcal{F}$ into the projections%
\begin{equation*}
\mathsf{P}_{k}^{\mu }\equiv \sum_{F\in \mathfrak{C}_{\mathcal{F}}^{k}\left(
F_{0}\right) }\mathsf{P}_{\mathcal{C}_{\mathcal{F}}\left( F\right) }^{\mu }\
.
\end{equation*}%
Note that the $k^{th}$ grandchildren $F\in \mathfrak{C}_{\mathcal{F}%
}^{k}\left( F_{0}\right) $ are pairwise disjoint and hence so are the
supports of the functions $\mathsf{P}_{\mathcal{C}_{\mathcal{F}}\left(
F\right) }^{\mu }f$ for $F\in \mathfrak{C}_{\mathcal{F}}^{k}\left(
F_{0}\right) $. Define the $\mathcal{F}$-square function $\mathcal{S}_{%
\mathcal{F}}f$ by%
\begin{equation*}
\mathcal{S}_{\mathcal{F}}f\left( x\right) =\left( \sum_{k=0}^{\infty
}\left\vert \mathsf{P}_{k}^{\mu }f\left( x\right) \right\vert ^{2}\right) ^{%
\frac{1}{2}}=\left( \sum_{F\in \mathcal{F}}\left\vert \mathsf{P}_{\mathcal{C}%
_{\mathcal{F}}\left( F\right) }^{\mu }f\left( x\right) \right\vert
^{2}\right) ^{\frac{1}{2}}=\left( \sum_{F\in \mathcal{F}}\left\vert
\sum_{I\in \mathcal{C}_{\mathcal{F}}\left( F\right) }\bigtriangleup
_{I}^{\mu }f\left( x\right) \right\vert ^{2}\right) ^{\frac{1}{2}}.
\end{equation*}

\begin{theorem}[\protect\cite{SaWi}]
\label{square thm}Suppose $\mu $ is a locally finite positive Borel measure
on $\mathbb{R}$, and let $\mathcal{F}\subset \mathcal{D}$\footnote{%
It was assumed in \cite{SaWi} that\ $\mathcal{F}$ is $\mu $-Carleson, but
this was a misprint.}. Then for $1<p<\infty $,%
\begin{equation*}
\left\Vert \mathcal{S}_{\mathcal{F}}f\right\Vert _{L^{p}\left( \mu \right)
}\approx \left\Vert f\right\Vert _{L^{p}\left( \mu \right) }.
\end{equation*}
\end{theorem}

Another square function that will arise in related forms is 
\begin{eqnarray*}
\mathcal{S}_{\rho ,\delta }f\left( x\right) &\equiv &\left( \sum_{I\in 
\mathcal{D}\ :x\in I}\left\vert \mathsf{P}_{I}^{\rho ,\delta }f\left(
x\right) \right\vert ^{2}\right) ^{\frac{1}{2}}, \\
\text{where }\mathsf{P}_{I}^{\rho ,\delta }f\left( x\right) &\equiv
&\sum_{J\in \mathcal{D}:\ 2^{-\rho }\ell \left( I\right) \leq \ell \left(
J\right) \leq 2^{\rho }\ell \left( I\right) }2^{-\delta \limfunc{dist}\left(
J,I\right) }\bigtriangleup _{J}^{\mu }f\left( x\right) .
\end{eqnarray*}

\begin{theorem}
\label{square thm nearby}Suppose $\mu $ is a locally finite positive Borel
measure on $\mathbb{R}$, and let $0<\rho ,\delta <1$. Then for $1<p<\infty $,%
\begin{equation*}
\left\Vert \mathcal{S}_{\rho ,\delta }f\right\Vert _{L^{p}\left( \mu \right)
}\leq C_{p,\rho ,\delta }\left\Vert f\right\Vert _{L^{p}\left( \mu \right) }.
\end{equation*}
\end{theorem}

\begin{proof}
It is easy to see that $\mathcal{S}_{\rho ,\delta }f\left( x\right) \leq
C_{\rho ,\delta }\mathcal{S}_{\limfunc{Haar}}f\left( x\right) $, and the
boundedness of $\mathcal{S}_{\rho ,\delta }$ now follows from the
boundedness of the Haar square function $\mathcal{S}_{\limfunc{Haar}}$.
\end{proof}

More generally, for $\Lambda \subset \mathcal{D}\left[ I\right] $ we define
projections%
\begin{equation*}
\mathsf{P}_{\Lambda }^{\omega }g\left( x\right) \equiv \sum_{J\in \Lambda
}\bigtriangleup _{J}^{\omega }g\left( x\right) ,
\end{equation*}%
and their associated `absolute' projections 
\begin{equation}
\left\vert \mathsf{P}_{\Lambda }^{\omega }\right\vert g\left( x\right)
\equiv \sqrt{\sum_{J\in \Lambda }\left\vert \bigtriangleup _{J}^{\omega
}g\left( x\right) \right\vert ^{2}}.  \label{abs proj}
\end{equation}

\subsubsection{Corona martingales}

The special type of martingale $\left\{ f_{k}\right\} _{k=1}^{\infty }$ that
we will be working with in this paper is that for which there is

\begin{enumerate}
\item an interval $F_{0}$ (thought of as the universe) and a subset $%
\mathcal{F}\subset \mathcal{D}\left[ F_{0}\right] $,

\item an increasing sequence $\left\{ \mathcal{E}_{k}\right\} _{k=0}^{\infty
}$ of $\sigma $-algebras of the form%
\begin{equation*}
\mathcal{E}_{k}\equiv \left\{ E\text{ Borel }\subset F_{0}:E\cap F\in
\left\{ \emptyset ,F\right\} \text{ for all }F\in \mathfrak{C}_{\mathcal{F}%
}^{\left( k\right) }\left( F_{0}\right) \right\} ,
\end{equation*}

\item and a function $f\in L^{p}\left( \mu \right) \cap L^{2}\left( \mu
\right) $ such that%
\begin{equation*}
f_{k}\left( x\right) =\mathsf{E}_{k}^{\mu }f\left( x\right) ,\ \ \ \ \ x\in 
\mathbb{R},
\end{equation*}%
where%
\begin{eqnarray*}
\mathsf{E}_{k}^{\mu }f\left( x\right) &\equiv &\left\{ 
\begin{array}{ccc}
E_{F}^{\mu }f & \text{ if } & x\in F\text{ for some }F\in \mathfrak{C}_{%
\mathcal{F}}^{\left( k\right) }\left( F_{0}\right) \\ 
f\left( x\right) & \text{ if } & x\in F_{0}\setminus \bigcup \mathfrak{C}_{%
\mathcal{F}}^{\left( k\right) }\left( F_{0}\right)%
\end{array}%
\right. ; \\
&\text{ }&\text{and where }\bigcup \mathfrak{C}_{\mathcal{F}}^{\left(
k\right) }\left( F_{0}\right) \equiv \bigcup_{F\in \mathfrak{C}_{\mathcal{F}%
}^{\left( k\right) }\left( F_{0}\right) }F.
\end{eqnarray*}
\end{enumerate}

Note that the sequence $\left\{ \mathsf{P}_{k}^{\mu }f\left( x\right)
\right\} _{F\in \mathcal{F}}$ of corona projections of the function $f$ is
the \emph{martingale difference sequence} of the $L^{p}$ bounded martingale $%
\left\{ \mathsf{E}_{k}^{\mu }f\left( x\right) \right\} _{F\in \mathcal{F}}$
with respect to the increasing sequence $\left\{ \mathcal{E}_{k}\right\}
_{k=0}^{\infty }$ of $\sigma $-algebras generated by the `atoms' $F\in 
\mathfrak{C}_{\mathcal{F}}^{\left( k\right) }\left( F_{0}\right) $.

\begin{definition}
\label{def mart}We refer to the above construction of a martingale
difference sequence associated with the function $f$, as the $\mathcal{F}$%
-corona martingale difference sequence of $f$.
\end{definition}

\begin{conclusion}
\label{assoc mart}To any function $f\in L^{p}\left( \mu \right) \cap
L^{2}\left( \mu \right) $ and any subset $\mathcal{F}\subset \mathcal{D}%
\left[ F_{0}\right] $, we can associate an $L^{p}$ bounded corona martingale 
$\left\{ \mathsf{E}_{k}^{\mu }f\left( x\right) \right\} _{F\in \mathcal{F}}$%
, whose martingale properties can then be exploited - e.g. Burkholder's
theorem which leads to boundedness of the associated square function, and
the Corona Martingale Comparison Principle in Propositon \ref{CMCP} below.
\end{conclusion}

\subsubsection{Iterated corona martingales\label{iter subsub sec}}

Given stopping times $\mathcal{Q}\subset \mathcal{A}$ in a finite set, we
can view the corona decomposition associated with $\mathcal{A}$ as an \emph{%
iterated} corona decomposition associated with $\mathcal{Q}\circ \mathcal{A}$%
, where the iterated stopping time $\mathcal{Q}\circ \mathcal{A}$ is thought
of as the union of the restricted stopping times $\mathcal{A}\left[ Q\right]
\equiv \mathcal{A}\cap \mathcal{C}_{\mathcal{Q}}\left( Q\right) $ for each $%
Q\in \mathcal{Q}$. The reason for taking this point of view is that the
corona decomposition of each corona $\mathcal{C}_{\mathcal{Q}}\left(
Q\right) $ into coronas $\left\{ \mathcal{C}_{\mathcal{A}\left[ Q\right]
}\left( A\right) \right\} _{A\in \mathcal{A}\left[ Q\right] }$ may carry
special information that is less visible when we view the corona
decomposition $\left\{ \mathcal{C}_{\mathcal{A}}\left( A\right) \right\}
_{A\in \mathcal{A}}$ abstractly. In fact the usual martingale difference
sequence $\left\{ h_{k}\right\} _{k=1}^{\infty }$ associated with a function 
$h$ and the stopping times $\mathcal{A}$, blurs any such information, since
the coronas at a given level in $\mathcal{A}$ may be associated with $Q$'s
at many different levels in $\mathcal{Q}$. We will now define the \emph{%
iterated} martingale difference sequence associated with $\mathcal{Q}\circ 
\mathcal{A}$ which doesn't suffer from this defect, and clearly displays any
information peculiar to the coronas $\mathcal{C}_{\mathcal{Q}}\left(
Q\right) $.

Define the \emph{depth} of a stopping time $\mathcal{S}$ to be the length $%
\limfunc{depth}\left( \mathcal{S}\right) $ of the longest tower in $\mathcal{%
S}$. Then in an iterated corona $\mathcal{Q}\circ \mathcal{A}$ we define the 
\emph{iterated} difference sequence starting with the ground level of $%
\mathcal{Q}$, which we assume is the single interval $T$. Define $h_{1}$ to
be the difference sequence associated with the sequence of coronas 
\begin{equation*}
\mathcal{C}_{\mathcal{A}\left[ T\right] }^{\left( 1\right) }\left( T\right)
=\left\{ \mathcal{C}_{\mathcal{A}}\left( A\right) \right\} _{A\in \mathfrak{C%
}_{\mathcal{A}}\left( T\right) }
\end{equation*}%
which are the $\mathcal{A}\left[ T\right] $ children of $T$. Then set $h_{k}$
to be the difference sequence associated with the coronas $\mathcal{C}_{%
\mathcal{A}\left[ T\right] }^{\left( k\right) }\left( T\right) $ for $1\leq
k\leq D_{1}$ where%
\begin{equation*}
D_{1}\equiv \limfunc{depth}\left( \mathcal{A}\left[ T\right] \right) .
\end{equation*}
Continuing in this way beyond this point would only add vanishing difference
sequences, and corresponding repeated $\sigma $-algebras for each vanishing
difference.

Next define $h_{D_{1}+1}$ to be the difference sequence associated with the
sequence of coronas $\left\{ \mathcal{C}_{\mathcal{A}\left[ Q\right]
}^{\left( 1\right) }\left( Q\right) \right\} _{Q\in \mathfrak{C}_{\mathcal{A}%
}\left( T\right) }$ at level one in $\mathcal{Q}$, and set $h_{D_{1}+k}$ to
be the difference sequence associated with the sequence of coronas $\left\{ 
\mathcal{C}_{\mathcal{A}\left[ Q\right] }^{\left( k\right) }\left( Q\right)
\right\} _{Q\in \mathfrak{C}_{\mathcal{A}}\left( T\right) }$ for $1\leq
k\leq D_{2}$ where%
\begin{equation*}
D_{2}\equiv \max_{Q\in \mathfrak{C}_{\mathcal{A}}\left( T\right) }\limfunc{%
depth}\left( \mathcal{A}\left[ Q\right] \right) .
\end{equation*}

At this point we have defined the iterated difference sequence $\left\{
h_{\ell }\right\} _{\ell =1}^{D_{1}+D_{2}}$ up to $D_{1}+D_{2}$, and we now
also define the \emph{iterated} distance $\limfunc{dist}{}_{\mathcal{Q}\circ 
\mathcal{A}}\left( A,T\right) $ from the root $T$ to an interval $A\in 
\mathcal{A}\left[ Q\right] $ for some $Q\in \mathfrak{C}_{\mathcal{Q}}^{%
\left[ 1\right] }\left( T\right) =\left\{ T\right\} \cup \mathfrak{C}_{%
\mathcal{A}}\left( T\right) $, by%
\begin{equation*}
\limfunc{dist}{}_{\mathcal{Q}\circ \mathcal{A}}\left( A,T\right) \equiv
\left\{ 
\begin{array}{ccc}
\limfunc{dist}{}_{\mathcal{A}\left[ T\right] }\left( A,T\right) & \text{ if }
& A\in \mathcal{A}\left[ T\right] \\ 
D_{1}+\limfunc{dist}{}_{\mathcal{A}\left[ Q\right] }\left( A,Q\right) & 
\text{ if } & A\in \mathcal{A}\left[ Q\right]%
\end{array}%
\right. ,\ \ \ \ \ A\in \mathcal{A}\left[ Q\right] ,Q\in \mathfrak{C}_{%
\mathcal{Q}}^{\left[ 1\right] }\left( T\right) .
\end{equation*}%
We also denote $\limfunc{dist}{}_{\mathcal{Q}\circ \mathcal{A}}\left(
A,T\right) $ by $\limfunc{xdist}_{\mathcal{A}}\left( A,T\right) $ when $%
\mathcal{Q}$ is understood. Note that this iterated distance is in general
larger than the corona distance $\limfunc{dist}{}_{\mathcal{A}}\left(
A,T\right) $.

Then we continue by defining $h_{D_{1}+D_{2}+k}$ to be the difference
sequence associated with the sequence of coronas $\left\{ \mathcal{C}_{%
\mathcal{A}\left[ Q\right] }^{\left( k\right) }\left( Q\right) \right\}
_{Q\in \mathfrak{C}_{\mathcal{Q}}^{\left( 2\right) }\left( T\right) }$ for $%
1\leq k\leq D_{3}$ where%
\begin{equation*}
D_{3}\equiv \max_{Q\in \mathfrak{C}_{\mathcal{Q}}^{\left( 2\right) }\left(
T\right) }\limfunc{depth}\left( \mathcal{A}\left[ Q\right] \right) .
\end{equation*}%
We also define $\limfunc{dist}{}_{\mathcal{Q}\circ \mathcal{A}}\left(
A,T\right) $ for $A\in \mathcal{A}\left[ Q\right] ,Q\in \mathfrak{C}_{%
\mathcal{Q}}^{\left[ 2\right] }\left( T\right) $ in the analogous way.

We then continue this process of defining 
\begin{equation}
h_{k},\ \ \ D_{k},\ \ \ \text{and }\limfunc{dist}{}_{\mathcal{Q}\circ 
\mathcal{A}}\left( A,T\right) =\limfunc{xdist}{}_{\mathcal{A}}\left(
A,T\right) \text{ for }A\in \mathcal{A}\left[ Q\right] \text{ with }Q\in 
\mathfrak{C}_{\mathcal{Q}}^{\left[ k-1\right] }\left( T\right) ,
\label{iter defs}
\end{equation}%
until $k$ has reached $\limfunc{depth}\left( \mathcal{Q}\circ \mathcal{A}%
\right) =D_{1}+D_{2}+...+D_{N}$, where $N=\limfunc{depth}\left( \mathcal{Q}%
\right) $ is the depth of the stopping times $\mathcal{Q}$, i.e. $\mathfrak{C%
}_{\mathcal{Q}}^{\left( N\right) }\left( T\right) \neq \emptyset $ and $%
\mathfrak{C}_{\mathcal{Q}}^{\left( N+1\right) }\left( T\right) =\emptyset $.

We refer to this construction of the \emph{iterated} martingale difference
sequence $\left\{ h_{k}\right\} _{k=1}^{\limfunc{depth}\left( \mathcal{Q}%
\circ \mathcal{A}\right) }$ associated with the function $h$, as the \emph{%
regularization} of the $\mathcal{A}$-corona martingale difference sequence $%
\left\{ f_{k}\right\} _{k=1}^{\limfunc{depth}\left( \mathcal{A}\right) }$,
defined in Definition \ref{def mart}, by the iterated stopping times $%
\mathcal{Q}\circ \mathcal{A}$. Note that $\limfunc{depth}\left( \mathcal{A}%
\right) $ is typically much smaller than $\limfunc{depth}\left( \mathcal{Q}%
\circ \mathcal{A}\right) $.

Finally, we associate to each $A\in \mathcal{A}$, an ordered pair $\left(
d_{1},d_{2}\right) $ where $d_{1}=\limfunc{dist}_{\mathcal{A}}\left(
A,Q\right) $ and $d_{2}=\limfunc{dist}_{\mathcal{Q}}\left( Q,T\right) $
where $Q$ is the unique interval in $\mathcal{Q}$ such that $A\in \mathcal{C}%
_{\mathcal{Q}}\left( Q\right) $. Note that the ordered pairs associated to
intervals $A$ at a fixed level $\limfunc{xdist}_{\mathcal{A}}\left(
A,T\right) $ all coincide. If we let $t$ denote the level in the iterated
martingale difference sequence, then we can unambiguously define%
\begin{equation*}
\left( d_{1}\left( t\right) ,d_{2}\left( t\right) \right) \text{ to be
associated to }A\text{ where }t=\limfunc{xdist}{}_{\mathcal{A}}\left(
A,T\right) .
\end{equation*}

\begin{conclusion}
\label{concl it}Suppose we are presented with a martingale difference
sequence $\left\{ f_{k}\right\} _{k=1}^{\infty }$ for $f=\sum_{k=1}^{\infty
}f_{k}$ relative to a collection of stopping times $\mathcal{A}$ as in
Definition \ref{def mart}. In the special case when $\mathcal{A}$ has an
iterated structure arising from stopping times $\mathcal{Q}\subset \mathcal{A%
}$, we can also write $f=\sum_{\ell =1}^{\infty }h_{\ell }$, where the
iterated martingale difference sequence $\left\{ h_{\ell }\right\} _{\ell
=1}^{\infty }$ is \emph{finer} than $\left\{ h_{k}\right\} _{k=1}^{\infty }$
and has a \emph{regularizing} property, i.e. there is a function $\ell
\rightarrow k=k\left( \ell \right) \leq \ell $ mapping $\mathbb{N}$ to $%
\mathbb{N}$, such that each interval $A\in \mathcal{A}$ associated with a
projection $\mathsf{P}_{\mathcal{C}_{\mathcal{A}}\left( A\right) }$
occurring in the function $f_{\ell }$, is contained an interval $Q\in 
\mathcal{Q}\subset \mathcal{A}$ associated with a projection $\mathsf{P}_{%
\mathcal{C}_{\mathcal{A}}\left( Q\right) }$ occurring in the function $%
h_{k\left( \ell \right) }$. Moreover, there is an iterated distance $%
\limfunc{dist}{}_{\mathcal{Q}\circ \mathcal{A}}\left( A,T\right) $, often
denoted $\limfunc{xdist}_{\mathcal{A}}\left( A,T\right) $ when the iteration
is understood, in the tree $\mathcal{Q}\circ \mathcal{A}$ satisfying%
\begin{eqnarray}
\limfunc{dist}{}_{\mathcal{Q}\circ \mathcal{A}}\left( A,T\right) &=&\limfunc{%
xdist}{}_{\mathcal{A}}\left( A,T\right) =D_{1}+D_{2}+...+D_{m-1}+\limfunc{%
dist}{}_{\mathcal{A}\left[ Q\right] }\left( A,Q\right) ,
\label{def iter dist} \\
\text{for }A &\in &\mathcal{A}\left[ Q\right] \text{ with }Q\in \mathfrak{C}%
_{\mathcal{Q}}^{\left( m\right) }\left( T\right) .  \notag
\end{eqnarray}%
There is also a pair $\left( d_{1}\left( t\right) ,d_{2}\left( t\right)
\right) $ such that%
\begin{eqnarray*}
d_{1}\left( t\right) &=&\limfunc{dist}{}_{\mathcal{A}}\left( A,Q\right) 
\text{ and }d_{2}\left( t\right) =\limfunc{dist}{}_{\mathcal{Q}}\left(
Q,T\right) , \\
\text{for all }A &\in &\mathcal{A}\text{ with }t=\limfunc{xdist}\ _{\mathcal{%
A}}\left( A,T\right) \text{.}
\end{eqnarray*}
\end{conclusion}

\subsection{Vector-valued inequalities}

For any locally finite positive Borel measure $\mu $ on $\mathbb{R}$, let $%
M_{\mu }^{\func{dy}}$ denote the dyadic maximal function,%
\begin{equation*}
M_{\mu }^{\func{dy}}f\left( x\right) \equiv \sup_{x\in I\in \mathcal{D}%
}\left( \frac{1}{\left\vert I\right\vert _{\mu }}\int_{I}\left\vert
f\right\vert d\mu \right) \mathbf{1}_{I}\left( x\right) ,
\end{equation*}%
which is well-known to be bounded on $L^{p}\left( \mu \right) $ for $1<p\leq
\infty $ (since the weak type $\left( 1,1\right) $ and strong type $\left(
\infty ,\infty \right) $ constants are both $1$). We need the $\ell ^{2}$
vector-valued inequality of Fefferman and Stein for the dyadic maximal
operator $M_{\mu }$ for $1<p<\infty $, namely 
\begin{equation}
\left\Vert \left\vert M_{\mu }^{\func{dy}}\mathbf{f}\right\vert _{\ell
^{2}}\right\Vert _{L^{p}\left( \mu \right) }\lesssim \left\Vert \left\vert 
\mathbf{f}\right\vert _{\ell ^{2}}\right\Vert _{L^{p}\left( \mu \right) },\
\ \ \ \ 1<p<\infty ,  \label{FS vv}
\end{equation}%
where $\mathbf{f}\left( x\right) =\left\{ f_{i}\left( x\right) \right\}
_{i=1}^{\infty }$, $M_{\mu }^{\func{dy}}\mathbf{f}=\left\{ M_{\mu }^{\func{dy%
}}f_{i}\left( x\right) \right\} _{i=1}^{\infty }$ and $\left\vert \mathbf{h}%
\left( x\right) \right\vert _{\ell ^{2}}=\sqrt{\sum_{i=1}^{\infty
}\left\vert h_{i}\left( x\right) \right\vert ^{2}}$. We are unable to find
this statement explicitly in the literature for general measures, and we
thank Jos\'{e}-Luis Luna-Garcia for pointing out to us that the case $p\geq
2 $ follows from the duality argument in \cite{FeSt}, and that the case $%
1<p\leq 2$ then follows from the weak type $\left( 1,1\right) $ inequality
in \cite[Theorem A.15 on page 247]{CrMaPe}, together with Marcinkiewicz
interpolation for Banach space valued functions, see e.g. \cite[Theorem 1.18
on page 480]{GaRu}. For the convenience of the reader, we repeat the short\
arguments suggested by Jos\'{e} here.

\begin{proof}[Proof of (\protect\ref{FS vv}) (J.-L. Luna-Garcia)]
First, for any weight $w$, we claim that 
\begin{equation}
\int_{\mathbb{R}}\left\vert M_{\mu }^{\func{dy}}f\left( x\right) \right\vert
^{q}w\left( x\right) d\mu \left( x\right) \leq C_{q}\int_{\mathbb{R}%
}\left\vert f\left( x\right) \right\vert ^{q}M_{\mu }^{\func{dy}}w\left(
x\right) d\mu \left( x\right) ,\ \ \ \ \ 1<q<\infty .  \label{FS 1979}
\end{equation}%
Indeed, let $\lambda >0$ and suppose $\Omega _{\lambda }\equiv \left\{
M_{\mu }^{\func{dy}}f>\lambda \right\} =\overset{\cdot }{\bigcup }%
_{j=1}^{\infty }I_{j}$ where $I_{j}$ are the maximal dyadic intervals
satisfying $\frac{1}{\left\vert I_{j}\right\vert _{\mu }}\int_{I_{j}}\left%
\vert f\right\vert d\mu >\lambda $. Then we have the weak type $\left(
1,1\right) $ inequality,%
\begin{eqnarray*}
\left\vert \left\{ M_{\mu }^{\func{dy}}f>\lambda \right\} \right\vert _{w\mu
} &=&\sum_{j}\left\vert I_{j}\right\vert _{w\mu }=\sum_{j}\left( \frac{1}{%
\left\vert I_{j}\right\vert _{\mu }}\left\vert I_{j}\right\vert _{w\mu
}\right) \left\vert I_{j}\right\vert _{\mu } \\
&\leq &\sum_{j}\left( \frac{1}{\left\vert I_{j}\right\vert _{\mu }}%
\left\vert I_{j}\right\vert _{w\mu }\right) \frac{1}{\lambda }%
\int_{I_{j}}\left\vert f\right\vert d\mu \leq \frac{1}{\lambda }\int_{%
\mathbb{R}}\left\vert f\right\vert \left( M_{\mu }^{\func{dy}}w\right) d\mu ,
\end{eqnarray*}%
as well as the strong type $\left( \infty ,\infty \right) $ inequality.
Marcinkiewicz interpolation now gives (\ref{FS 1979}).

From (\ref{FS 1979}) with $q=\frac{p}{2}\geq 1$ we have,%
\begin{eqnarray*}
&&\left\Vert \left\vert M_{\mu }^{\func{dy}}\mathbf{f}\right\vert _{\ell
^{2}}\right\Vert _{L^{p}\left( \mu \right) }^{2}=\left( \int_{\mathbb{R}%
}\left( \sum_{i=1}^{\infty }\left\vert M_{\mu }^{\func{dy}}f_{i}\right\vert
^{2}\right) ^{\frac{p}{2}}d\mu \right) ^{\frac{2}{p}}=\sup_{\left\Vert
g\right\Vert _{L^{q^{\prime }}\left( \sigma \right) }=1}\int_{\mathbb{R}%
}\left( \sum_{i=1}^{\infty }\left\vert M_{\mu }^{\func{dy}}f_{i}\right\vert
^{2}\right) gd\mu \\
&\leq &\sup_{\left\Vert g\right\Vert _{L^{q^{\prime }}\left( \mu \right)
}=1}\int_{\mathbb{R}}\sum_{i=1}^{\infty }\left\vert f_{i}\right\vert
^{2}M_{\mu }^{\func{dy}}gd\mu \leq \left( \int_{\mathbb{R}}\left(
\sum_{i=1}^{\infty }\left\vert f_{i}\right\vert ^{2}\right) ^{q}d\mu \right)
^{\frac{1}{q}}\left( \int_{\mathbb{R}}\left( M_{\mu }^{\func{dy}}g\right)
^{q^{\prime }}d\mu \right) ^{\frac{1}{q^{\prime }}} \\
&\leq &C_{q}\left( \int_{\mathbb{R}}\left( \sum_{i=1}^{\infty }\left\vert
f_{i}\right\vert ^{2}\right) ^{\frac{p}{2}}d\mu \right) ^{\frac{2}{p}}\left(
\int_{\mathbb{R}}\left\vert g\right\vert ^{q^{\prime }}d\mu \right) ^{\frac{1%
}{q^{\prime }}}=C_{q}\left\Vert \left\vert \mathbf{f}\right\vert _{\ell
^{2}}\right\Vert _{L^{p}\left( \mu \right) }^{2}\ .
\end{eqnarray*}%
This completes the proof that (\ref{FS vv}) holds for $p\geq 2$.

The weak type $\left( 1,1\right) $ inequality in \cite[Theorem A.15 on page
247]{CrMaPe} says that%
\begin{equation*}
\left\vert \left\{ \left\vert M_{\mu }^{\func{dy}}\mathbf{f}\right\vert
_{\ell ^{2}}>\lambda \right\} \right\vert _{\mu }\leq \frac{C}{\lambda }\int
\left\vert \mathbf{f}\right\vert _{\ell ^{2}}d\mu ,
\end{equation*}%
and now the Marcinkiewicz interpolation theorem in \cite[Theorem 1.18 on
page 480]{GaRu} completes the proof of (\ref{FS vv}).
\end{proof}

\begin{description}
\item[Projections and maximal operators] We will sometimes apply (\ref{FS vv}%
) in conjunction with the fact that, by the telescoping identity for Haar
projections $\left\{ \bigtriangleup _{I}^{\mu }\right\} _{I\in \mathcal{D}}$%
, a projection $\mathsf{P}_{\Lambda }^{\mu }$ with $\Lambda $ a connected
subset of $\mathcal{D}\left[ S\right] \setminus \left\{ S\right\} $ for some 
$S\in \mathcal{D}$, is dominated pointwise by the dyadic maximal operator, 
\begin{equation}
\left\vert \mathsf{P}_{\Lambda }^{\mu }f\left( x\right) \right\vert \leq
2M_{\mu }^{\func{dy}}\left( \mathbf{1}_{S}f\right) \left( x\right) .
\label{PM}
\end{equation}%
Indeed, for any $x$, if $I$ is the smallest interval in $\Lambda $
containing $x$, and if $F\in \Lambda $ is the largest, then%
\begin{eqnarray*}
&&\mathsf{P}_{\Lambda }^{\mu }f\left( x\right) =\sum_{K\in \left[
I,S_{I}\right) }\left( E_{K}^{\mu }f-E_{\pi K}^{\mu }f\right) =E_{I}^{\mu
}f\left( x\right) -E_{\pi F}^{\mu }f\left( x\right) , \\
&&\text{and }\left\vert \mathsf{P}_{\Lambda }^{\mu }f\left( x\right)
\right\vert \leq 2M_{\mu }^{\func{dy}}\left( \mathbf{1}_{S}f\right) \left(
x\right) \text{ \ since }\pi F\subset S,
\end{eqnarray*}%
where $\pi K$ is the parent of $K$ in the dyadic grid $\mathcal{D}$. Since $%
\mathsf{P}_{\Lambda }^{\mu }f=\mathsf{P}_{\Lambda }^{\mu }\left( \mathsf{P}%
_{\Lambda }^{\mu }f\right) $ we also have%
\begin{equation}
\left\vert \mathsf{P}_{\Lambda }^{\mu }f\left( x\right) \right\vert \leq
2M_{\mu }^{\func{dy}}\left( \mathsf{P}_{\Lambda }^{\mu }f\right) \left(
x\right) .  \label{PM'}
\end{equation}
\end{description}

We will also need the following `disjoint support' lemma for vector-valued
functions. Define the mixed norm space%
\begin{equation*}
L^{p}\left( \ell ^{2};\mu \right) \equiv \left\{ \mathbf{f}=\left(
f_{i}\right) _{i=1}^{\infty }:\left\Vert \mathbf{f}\right\Vert _{L^{p}\left(
\ell ^{2};\mu \right) }\equiv \left( \int_{\mathbb{R}}\left\vert \mathbf{f}%
\left( x\right) \right\vert _{\ell ^{2}}^{p}d\mu \left( x\right) \right) ^{%
\frac{1}{p}}<\infty \right\} .
\end{equation*}

\begin{lemma}
\label{disjoint supp}Let $1<p<\infty $. Suppose $\left\{ f^{n}\right\}
_{n=1}^{\infty }$ is a sequence in $L^{p}\left( \ell ^{2},\mu \right) $
where $f^{n}\left( x\right) =\left\{ f_{k}^{n}\left( x\right) \right\}
_{k=1}^{\infty }\in \ell ^{2}$, and that for each $x\in \mathbb{R}$, the $%
\mathbb{N}$-supports 
\begin{equation*}
\mathbb{N}\text{-}\limfunc{supp}f^{n}\left( x\right) \equiv \left\{ k\in 
\mathbb{N}:f_{k}^{n}\left( x\right) \neq 0\right\}
\end{equation*}%
of $f^{n}\left( x\right) $ are pairwise disjoint in $n$, i.e.%
\begin{equation}
\mathbb{N}\text{-}\limfunc{supp}f^{n}\left( x\right) \cap \mathbb{N}\text{-}%
\limfunc{supp}f^{m}\left( x\right) =\emptyset ,\ \ \ \ \ \text{for }n\neq m.
\label{N supp}
\end{equation}%
Then%
\begin{equation}
\left\Vert \sum_{n=1}^{\infty }f^{n}\right\Vert _{L^{p}\left( \ell ^{2};\mu
\right) }^{p}=\int_{\mathbb{R}}\left( \sum_{n=1}^{\infty }\left\vert
f^{n}\left( x\right) \right\vert _{\ell ^{2}}^{2}\right) ^{\frac{p}{2}}d\mu
\left( x\right) .  \label{first disj}
\end{equation}%
If the functions $f^{n}\left( x\right) $ are pairwise disjoint in $x$, i.e.
the $\mathbb{R}$-supports 
\begin{equation*}
\mathbb{R}\text{-}\limfunc{supp}f^{n}\left( x\right) \equiv \left\{ x\in 
\mathbb{R}:f^{n}\left( x\right) \neq 0\right\}
\end{equation*}%
of $f^{n}$ satisfy%
\begin{equation}
\mathbb{R}\text{-}\limfunc{supp}f^{n}\cap \mathbb{R}\text{-}\limfunc{supp}%
f^{m}=\emptyset ,\ \ \ \ \ \text{for }n\neq m,  \label{R supp}
\end{equation}%
then 
\begin{equation}
\left\Vert \sum_{n=1}^{\infty }f^{n}\right\Vert _{L^{p}\left( \ell ^{2};\mu
\right) }^{p}=\sum_{n=1}^{\infty }\left\Vert f^{n}\right\Vert _{L^{p}\left(
\ell ^{2};\mu \right) }^{p}\ .  \label{second disj}
\end{equation}
\end{lemma}

\begin{proof}
For each $x\in \mathbb{R}$, the disjoint $\mathbb{N}$-support hypothesis on
the sequence $\left\{ f^{n}\left( x\right) \right\} _{n=1}^{\infty }$ yields 
$\left\vert \sum_{n=1}^{\infty }f^{n}\left( x\right) \right\vert _{\ell
^{2}}^{2}=\sum_{n=1}^{\infty }\left\vert f^{n}\left( x\right) \right\vert
_{\ell ^{2}}^{2}$, which gives (\ref{first disj}). If the sequence $\left\{
f^{n}\left( x\right) \right\} _{n=1}^{\infty }$ is pairwise disjoint in $x$,
then for all $1<p<\infty \,.$ we have%
\begin{equation*}
\left\Vert \sum_{n=1}^{\infty }f^{n}\right\Vert _{L^{p}\left( \ell ^{2};\mu
\right) }^{p}=\int_{\mathbb{R}}\left( \sum_{n=1}^{\infty }\left\vert
f^{n}\left( x\right) \right\vert _{\ell ^{2}}^{2}\right) ^{\frac{p}{2}}d\mu
\left( x\right) =\int_{\mathbb{R}}\sum_{n=1}^{\infty }\left\vert f^{n}\left(
x\right) \right\vert _{\ell ^{2}}^{p}d\mu \left( x\right)
=\sum_{n=1}^{\infty }\left\Vert f^{n}\right\Vert _{L^{p}\left( \ell ^{2};\mu
\right) }^{p}.
\end{equation*}
\end{proof}

Inequality (\ref{second disj}) will be used throughout the paper, and
especially in the proof of the Corona Martingale Comparison Principle.

\begin{corollary}
\label{disjoint supp'}Let $1<p<\infty $. Suppose $\left\{ f^{n}\right\}
_{n=1}^{\infty }$ and $\left\{ g^{n}\right\} _{n=1}^{\infty }$ are sequences
in $L^{p}\left( \ell ^{2};\mu \right) $, each satisfying (\ref{R supp}), and
that there is $\eta >0$ such that $\left\Vert f^{n}\right\Vert _{L^{p}\left(
\ell ^{2};\mu \right) }\leq \eta \left\Vert g^{n}\right\Vert _{L^{p}\left(
\ell ^{2};\mu \right) }$ for all $n\in \mathbb{N}$. Then%
\begin{equation*}
\left\Vert \sum_{n=1}^{\infty }f^{n}\right\Vert _{L^{p}\left( \ell ^{2};\mu
\right) }\leq \eta \left\Vert \sum_{n=1}^{\infty }g^{n}\right\Vert
_{L^{p}\left( \ell ^{2};\mu \right) }.
\end{equation*}
\end{corollary}

\begin{proof}
From the lemma above we have%
\begin{equation*}
\left\Vert \sum_{n=1}^{\infty }f^{n}\right\Vert _{L^{p}\left( \ell ^{2};\mu
\right) }^{p}=\sum_{n=1}^{\infty }\left\Vert f^{n}\right\Vert _{L^{p}\left(
\ell ^{2};\mu \right) }^{p}\leq \eta ^{p}\sum_{n=1}^{\infty }\left\Vert
g^{n}\right\Vert _{L^{p}\left( \ell ^{2};\mu \right) }^{p}=\eta
^{p}\left\Vert \sum_{n=1}^{\infty }g^{n}\right\Vert _{L^{p}\left( \ell
^{2};\mu \right) }^{p}.
\end{equation*}
\end{proof}

We close this section with a technical lemma that is needed in connection
with the dual tree decomposition below that generalizes the upside down
corona construction of M.\ Lacey in \cite{Lac}.

\begin{lemma}
\label{q<1}Let $1<p<2$. Suppose $G\left( x\right) =\left\{ G_{k}\left(
x\right) \right\} _{k=1}^{\infty }$ and $B\left( x\right) =\left\{
B_{k}\left( x\right) \right\} _{k=1}^{\infty }$ are two sequences of
functions on the real line with pairwise disjoint $\mathbb{N}$-supports, 
\begin{equation*}
G_{k}\left( x\right) B_{k}\left( x\right) =0,\ \ \ \ \ \text{for all }k\text{
and }x.
\end{equation*}%
Then there is a positive constant $c>0$ (independent of the sequences) such
that%
\begin{equation}
\max \left\{ \frac{\Lambda ^{p}}{\left\Vert G\right\Vert _{L^{p}\left( \ell
^{2};\omega \right) }^{p}},\left( \frac{\Lambda ^{p}}{\left\Vert
G\right\Vert _{L^{p}\left( \ell ^{2};\omega \right) }^{p}}\right) ^{\frac{p}{%
2}}\right\} \geq c\frac{\left\Vert B\right\Vert _{L^{p}\left( \ell
^{2};\omega \right) }^{p}}{\left\Vert G\right\Vert _{L^{p}\left( \ell
^{2};\omega \right) }^{p}},  \label{suff}
\end{equation}%
where%
\begin{equation*}
\Lambda ^{p}\equiv \left\Vert G+B\right\Vert _{L^{p}\left( \ell ^{2};\omega
\right) }^{p}-\left\Vert G\right\Vert _{L^{p}\left( \ell ^{2};\omega \right)
}^{p}.
\end{equation*}
\end{lemma}

\begin{proof}
Lemma \ref{disjoint supp}, together with the pairwise disjoint $\mathbb{N}$%
-support hypotheisis, shows that 
\begin{eqnarray*}
&&\Lambda ^{p}=\int_{\mathbb{R}}\left( \sum_{k=1}^{\infty }\left\vert
G_{k}\left( x\right) +B_{k}\left( x\right) \right\vert ^{2}\right) ^{\frac{p%
}{2}}d\omega \left( x\right) -\int_{\mathbb{R}}\left( \sum_{k=1}^{\infty
}\left\vert G_{k}\left( x\right) \right\vert ^{2}\right) ^{\frac{p}{2}%
}d\omega \left( x\right) \\
&=&\int_{\mathbb{R}}\left( \sum_{k=1}^{\infty }\left\vert G_{k}\left(
x\right) \right\vert ^{2}+\sum_{k=1}^{\infty }\left\vert B_{k}\left(
x\right) \right\vert ^{2}\right) ^{\frac{p}{2}}d\omega \left( x\right)
-\int_{\mathbb{R}}\left( \sum_{k=1}^{\infty }\left\vert G_{k}\left( x\right)
\right\vert ^{2}\right) ^{\frac{p}{2}}d\omega \left( x\right) \\
&=&\int_{\mathbb{R}}\left( \left\vert G\left( x\right) \right\vert _{\ell
^{2}}^{2}+\left\vert B\left( x\right) \right\vert _{\ell ^{2}}^{2}\right) ^{%
\frac{p}{2}}d\omega \left( x\right) -\int_{\mathbb{R}}\left\vert G\left(
x\right) \right\vert _{\ell ^{2}}^{p}d\omega \left( x\right) =\int_{\mathbb{R%
}}\left( g\left( x\right) ^{2}+b\left( x\right) ^{2}\right) ^{\frac{p}{2}%
}d\omega \left( x\right) -\int_{\mathbb{R}}g\left( x\right) ^{p}d\omega
\left( x\right) ,
\end{eqnarray*}%
where $g\left( x\right) \equiv \left\vert G\left( x\right) \right\vert
_{\ell ^{2}}$ and $b\left( x\right) \equiv \left\vert B\left( x\right)
\right\vert _{\ell ^{2}}$ are the functions we work with from now on.

For $0<\eta \ll 1$, we write%
\begin{eqnarray*}
\mathbb{R} &=&L_{\eta }+C_{\eta }+R_{\eta }\ , \\
\text{where }L_{\eta } &\equiv &\left\{ b\left( x\right) \leq \eta g\left(
x\right) \right\} , \\
C_{\eta } &\equiv &\left\{ \eta g\left( x\right) <b\left( x\right) <\frac{1}{%
\eta }g\left( x\right) \right\} , \\
\text{and }R_{\eta } &\equiv &\left\{ g\left( x\right) \leq \eta b\left(
x\right) \right\} ,
\end{eqnarray*}%
and then decompose%
\begin{equation*}
\int_{\mathbb{R}}\left( g\left( x\right) ^{2}+b\left( x\right) ^{2}\right) ^{%
\frac{p}{2}}d\omega \left( x\right) =\left\{ \int_{L_{\eta }}+\int_{C_{\eta
}}+\int_{R_{\eta }}\right\} \left( g\left( x\right) ^{2}+b\left( x\right)
^{2}\right) ^{\frac{p}{2}}d\omega \left( x\right) \equiv T_{L_{\eta
}}+T_{C_{\eta }}+T_{R_{\eta }}.
\end{equation*}

We have%
\begin{eqnarray*}
\left( 1+\frac{1}{\eta ^{2}}\right) ^{\frac{p}{2}}\int_{L_{\eta }}b\left(
x\right) ^{p}d\omega \left( x\right) &\leq &T_{L_{\eta }}\leq \left( 1+\eta
^{2}\right) ^{\frac{p}{2}}\int_{L_{\eta }}g\left( x\right) ^{p}d\omega
\left( x\right) , \\
\left( 1+\eta ^{2}\right) ^{\frac{p}{2}}\int_{C_{\eta }}g\left( x\right)
^{p}d\omega \left( x\right) &\leq &T_{C_{\eta }}\leq \left( 1+\frac{1}{\eta
^{2}}\right) ^{\frac{p}{2}}\int_{C_{\eta }}g\left( x\right) ^{p}d\omega
\left( x\right) , \\
\left( 1+\frac{1}{\eta ^{2}}\right) ^{\frac{p}{2}}\int_{R_{\eta }}g\left(
x\right) ^{p}d\omega \left( x\right) &\leq &T_{R_{\eta }}\leq \left( 1+\eta
^{2}\right) ^{\frac{p}{2}}\int_{R_{\eta }}b\left( x\right) ^{p}d\omega
\left( x\right) ,
\end{eqnarray*}%
and so%
\begin{eqnarray*}
\Lambda ^{p} &=&\int_{\mathbb{R}}\left( g\left( x\right) ^{2}+b\left(
x\right) ^{2}\right) ^{\frac{p}{2}}d\omega \left( x\right) -\int_{\mathbb{R}%
}g\left( x\right) ^{p}d\omega \left( x\right) \\
&=&\int_{L_{\eta }\cup C_{\eta }\cup R_{\eta }}\left[ \left( g\left(
x\right) ^{2}+b\left( x\right) ^{2}\right) ^{\frac{p}{2}}-g\left( x\right)
^{p}\right] d\omega \left( x\right) \\
&\geq &\int_{L_{\eta }}\left[ \left( g\left( x\right) ^{2}+b\left( x\right)
^{2}\right) ^{\frac{p}{2}}-g\left( x\right) ^{p}\right] d\omega \left(
x\right) \\
&&+\left[ \left( 1+\eta ^{2}\right) ^{\frac{p}{2}}-1\right] \int_{C_{\eta
}}g\left( x\right) ^{p}d\omega \left( x\right) \\
&&+\int_{R_{\eta }}\left[ \left( g\left( x\right) ^{2}+b\left( x\right)
^{2}\right) ^{\frac{p}{2}}-g\left( x\right) ^{p}\right] d\omega \left(
x\right) \\
&\geq &\left[ \left( 1+\eta ^{2}\right) ^{\frac{p}{2}}-1\right]
\int_{C_{\eta }}\left( \eta b\left( x\right) \right) ^{p}d\omega \left(
x\right) +\left\{ \int_{L_{\eta }}+\int_{R_{\eta }}\right\} \left[ \left(
g\left( x\right) ^{2}+b\left( x\right) ^{2}\right) ^{\frac{p}{2}}-g\left(
x\right) ^{p}\right] d\omega \left( x\right) .
\end{eqnarray*}

In particular,%
\begin{equation*}
\int_{C_{\eta }}b\left( x\right) ^{p}d\omega \left( x\right) \leq \frac{%
\Lambda ^{p}}{\left[ \left( 1+\eta ^{2}\right) ^{\frac{p}{2}}-1\right] \eta
^{p}}\lesssim \eta ^{-p-2}\Lambda ^{p}.
\end{equation*}%
Since $g\left( x\right) \leq \eta b\left( x\right) $ on $R_{\eta }$, we also
have%
\begin{eqnarray*}
\Lambda ^{p} &\geq &\int_{R_{\eta }}\left[ \left( g\left( x\right)
^{2}+b\left( x\right) ^{2}\right) ^{\frac{p}{2}}-g\left( x\right) ^{p}\right]
d\omega \left( x\right) =\int_{R_{\eta }}\left[ \left( \left( \frac{g\left(
x\right) }{b\left( x\right) }\right) ^{2}+1\right) ^{\frac{p}{2}}-\left( 
\frac{g\left( x\right) }{b\left( x\right) }\right) ^{p}\right] b\left(
x\right) ^{p}d\omega \left( x\right) \\
&\geq &\int_{R_{\eta }}\left( 1-\eta ^{p}\right) b\left( x\right)
^{p}d\omega \left( x\right) =\left( 1-\eta ^{p}\right) \int_{R_{\eta
}}b\left( x\right) ^{p}d\omega \left( x\right) ,
\end{eqnarray*}%
and so altogether we have%
\begin{equation}
\int_{C_{\eta }\cup R_{\eta }}b\left( x\right) ^{p}d\omega \left( x\right)
\lesssim \eta ^{-p-2}\Lambda ^{p}.  \label{altog}
\end{equation}%
Now we continue differently in two exhaustive cases.

\bigskip

\textbf{Case I: }$\int_{\mathbb{R}}b\left( x\right) ^{p}d\omega \left(
x\right) \leq \eta ^{-1}\int_{C_{\eta }\cup R_{\eta }}b\left( x\right)
^{p}d\omega \left( x\right) $.

\bigskip

In this case (\ref{altog}) yields%
\begin{equation*}
\int_{\mathbb{R}}b\left( x\right) ^{p}d\omega \left( x\right) \lesssim \eta
^{-p-2-1}\Lambda ^{p}.
\end{equation*}

\bigskip

\textbf{Case II: }$\int_{C_{\eta }\cup R_{\eta }}b\left( x\right)
^{p}d\omega \left( x\right) <\eta \int_{\mathbb{R}}b\left( x\right)
^{p}d\omega \left( x\right) $.

\bigskip

Since $b\left( x\right) \leq \eta g\left( x\right) $ on $L_{\eta }$, we have
for $\eta \leq \frac{1}{2}$, 
\begin{eqnarray*}
&&\Lambda ^{p}\geq \int_{L_{\eta }}\left[ \left( g\left( x\right)
^{2}+b\left( x\right) ^{2}\right) ^{\frac{p}{2}}-g\left( x\right) ^{p}\right]
d\omega \left( x\right) =\int_{L_{\eta }}\left[ \left( 1+\left( \frac{%
b\left( x\right) }{g\left( x\right) }\right) ^{2}\right) ^{\frac{p}{2}}-1%
\right] g\left( x\right) ^{p}d\omega \left( x\right) \\
&\approx &\int_{L_{\eta }}\left( \frac{b\left( x\right) }{g\left( x\right) }%
\right) ^{2}g\left( x\right) ^{p}d\omega \left( x\right) =\int_{L_{\eta }}%
\frac{b\left( x\right) ^{2}}{g\left( x\right) ^{2-p}}d\omega \left( x\right)
.
\end{eqnarray*}%
Recall that $1<p<2$, so that we can apply H\"{o}lder's inequality with
exponents $\frac{2}{p}$ and $\frac{2}{2-p}$ to obtain 
\begin{eqnarray*}
&&\int_{L_{\eta }}b\left( x\right) ^{p}d\omega \left( x\right)
=\int_{L_{\eta }}\frac{b\left( x\right) ^{p}}{g\left( x\right) ^{p\left( 1-%
\frac{p}{2}\right) }}g\left( x\right) ^{p\left( 1-\frac{p}{2}\right)
}d\omega \left( x\right) \\
&\leq &\left( \int_{L_{\eta }}\frac{b\left( x\right) ^{2}}{g\left( x\right)
^{2-p}}d\omega \left( x\right) \right) ^{\frac{p}{2}}\left( \int_{L_{\eta
}}g\left( x\right) ^{p\left( 1-\frac{p}{2}\right) \frac{2}{2-p}}d\omega
\left( x\right) \right) ^{1-\frac{p}{2}} \\
&=&\left( \int_{L_{\eta }}\frac{b\left( x\right) ^{2}}{g\left( x\right)
^{2-p}}d\omega \left( x\right) \right) ^{\frac{p}{2}}\left( \int_{L_{\eta
}}g\left( x\right) ^{p}d\omega \left( x\right) \right) ^{1-\frac{p}{2}%
}\lesssim \left( \Lambda ^{p}\right) ^{\frac{p}{2}}\left( \int_{\mathbb{R}%
}g\left( x\right) ^{p}d\omega \left( x\right) \right) ^{1-\frac{p}{2}}.
\end{eqnarray*}

Because we are in Case II, we also have%
\begin{eqnarray*}
\int_{\mathbb{R}}b\left( x\right) ^{p}d\omega \left( x\right)
&=&\int_{L_{\eta }}b\left( x\right) ^{p}d\omega \left( x\right)
+\int_{C_{\eta }\cup R_{\eta }}b\left( x\right) ^{p}d\omega \left( x\right)
\\
&\leq &\left( \Lambda ^{p}\right) ^{\frac{p}{2}}\left( \int_{\mathbb{R}%
}g\left( x\right) ^{p}d\omega \left( x\right) \right) ^{1-\frac{p}{2}}+\eta
\int_{\mathbb{R}}b\left( x\right) ^{p}d\omega \left( x\right) \\
&\Longrightarrow &\int_{\mathbb{R}}b\left( x\right) ^{p}d\omega \left(
x\right) \leq \frac{\left( \Lambda ^{p}\right) ^{\frac{p}{2}}\left( \int_{%
\mathbb{R}}g\left( x\right) ^{p}d\omega \left( x\right) \right) ^{1-\frac{p}{%
2}}}{1-\eta }.
\end{eqnarray*}

Thus altogether we have shown that for $\eta \leq \frac{1}{2}$, 
\begin{equation}
\frac{\int_{\mathbb{R}}b\left( x\right) ^{p}d\omega \left( x\right) }{\int_{%
\mathbb{R}}g\left( x\right) ^{p}d\omega \left( x\right) }\leq \max \left\{
\eta ^{-p-2-1}\frac{\Lambda ^{p}}{\int_{\mathbb{R}}g\left( x\right)
^{p}d\omega \left( x\right) },\frac{1}{1-\eta }\left( \frac{\Lambda ^{p}}{%
\int_{\mathbb{R}}g\left( x\right) ^{p}d\omega \left( x\right) }\right) ^{%
\frac{p}{2}}\right\} ,  \label{q2}
\end{equation}%
and this completes the proof of (\ref{suff}) upon taking $\eta =\frac{1}{2}$.
\end{proof}

\section{Beginning the proof of the main theorems}

We build our proof on the decomposition used in \cite{Saw7} - with the
exception of the bounds for the far and stopping forms, which require new
arguments. We assume that the Haar supports of the functions $f\in
L^{p}\left( \sigma \right) \cap L^{2}\left( \sigma \right) $ and $g\in
L^{p^{\prime }}\left( \omega \right) \cap L^{2}\left( \omega \right) $ in
the proof are contained in the $\limfunc{child}$-$\func{good}$ grid $%
\mathcal{D}_{\func{good}}^{\limfunc{child}}$. Here is a brief schematic
diagram of the initial twelve decompositions made below,

\begin{equation}
\fbox{$%
\begin{array}{ccccccccccc}
\left\langle H_{\sigma }f,g\right\rangle _{\omega } &  &  &  &  &  &  &  & 
&  &  \\ 
\downarrow &  &  &  &  &  &  &  &  &  &  \\ 
\mathsf{B}_{\func{below}}\left( f,g\right) & + & \mathsf{B}_{\func{above}%
}\left( f,g\right) & + & \mathsf{B}_{\limfunc{disj}}\left( f,g\right) & + & 
\mathsf{B}_{\limfunc{disj}}^{\ast }\left( f,g\right) & + & \mathsf{B}_{%
\limfunc{comp}}\left( f,g\right) & + & \mathsf{B}_{\limfunc{comp}}^{\ast
}\left( f,g\right) \\ 
\downarrow &  & \downarrow &  &  &  &  &  &  &  &  \\ 
\downarrow &  & \limfunc{similar} &  &  &  &  &  &  &  &  \\ 
\downarrow &  &  &  &  &  &  &  &  &  &  \\ 
\mathsf{B}_{\func{neigh}}\left( f,g\right) & + & \mathsf{B}_{\func{far}%
}\left( f,g\right) & + & \mathsf{B}_{\limfunc{para}}\left( f,g\right) & + & 
\mathsf{B}_{\limfunc{stop}}\left( f,g\right) &  &  &  & 
\end{array}%
$}\ ,  \label{brief}
\end{equation}%
where we note that there is a similar decomposition of $\mathsf{B}_{\func{%
above}}\left( f,g\right) $ into dual neighbour, far, paraproduct and
stopping forms.

The long-range portion $\mathsf{B}_{\limfunc{disj}}^{\func{long}}\left(
f,g\right) $ of the disjoint form will be controlled by the triple quadratic
Muckenhoupt characteristics, but will \emph{also} be controlled by the
global quadratic testing characteristics. Similarly for the far form $%
\mathsf{B}_{\func{far}}\left( f,g\right) $. Here are all of the bounds
listed in the order we will prove them (all bounds except for (9) are valid
for $1<p<\infty $, and their duals follow by symmetry)\footnote{%
The stopping form $\mathsf{B}_{\limfunc{stop}}\left( f,g\right) $ depends
only on the scalar testing characteristic $\mathfrak{T}_{H,p}^{\func{loc}%
}\left( \sigma ,\omega \right) $ through the stopping energy $\mathfrak{X}%
_{p}\left( \sigma ,\omega \right) $ and the $\sigma $-Carleson condition.}:%
\begin{eqnarray}
\boldsymbol{(1)} &&\ \left\vert \mathsf{B}_{\limfunc{comp}}\left( f,g\right)
\right\vert \lesssim \left( \mathfrak{T}_{H,p}^{\ell ^{2},\func{loc}}\left(
\sigma ,\omega \right) +A_{p}^{\ell ^{2},\limfunc{offset}}\left( \sigma
,\omega \right) +\mathcal{WBP}_{H,p}^{\ell ^{2}}\left( \sigma ,\omega
\right) \right) \ \left\Vert f\right\Vert _{L^{p}\left( \sigma \right)
}\left\Vert g\right\Vert _{L^{p^{\prime }}\left( \omega \right) }\ ,
\label{bounds} \\
\boldsymbol{(2)} &&\ \left\vert \mathsf{B}_{\limfunc{disj}}^{\func{long}%
}\left( f,g\right) \right\vert \lesssim A_{p}^{\ell ^{2},\limfunc{trip}%
}\left( \sigma ,\omega \right) \ \left\Vert f\right\Vert _{L^{p}\left(
\sigma \right) }\left\Vert g\right\Vert _{L^{p^{\prime }}\left( \omega
\right) }\ ,  \notag \\
\boldsymbol{(3)} &&\ \left\vert \mathsf{B}_{\limfunc{disj}}^{\func{long}%
}\left( f,g\right) \right\vert \lesssim \mathfrak{T}_{H,p}^{\ell ^{2},%
\limfunc{glob}}\left( \sigma ,\omega \right) \ \left\Vert f\right\Vert
_{L^{p}\left( \sigma \right) }\left\Vert g\right\Vert _{L^{p^{\prime
}}\left( \omega \right) }\ ,  \notag \\
\boldsymbol{(4)} &&\ \left\vert \mathsf{B}_{\limfunc{disj}}^{\limfunc{mid}%
}\left( f,g\right) \right\vert \lesssim A_{p}^{\ell ^{2},\limfunc{offset}%
}\left( \sigma ,\omega \right) \ \left\Vert f\right\Vert _{L^{p}\left(
\sigma \right) }\left\Vert g\right\Vert _{L^{p^{\prime }}\left( \omega
\right) }\ ,  \notag \\
\boldsymbol{(5)} &&\ \left\vert \mathsf{B}_{\func{neigh}}\left( f,g\right)
\right\vert \lesssim A_{p}^{\ell ^{2},\limfunc{offset}}\left( \sigma ,\omega
\right) \ \left\Vert f\right\Vert _{L^{p}(\sigma )}\left\Vert g\right\Vert
_{L^{p^{\prime }}(\omega )}\ ,  \notag \\
\boldsymbol{(6)} &&\ \left\vert \mathsf{B}_{\func{far}}\left( f,g\right)
\right\vert \lesssim \left( \mathfrak{T}_{H,p}^{\ell ^{2},\limfunc{loc}%
}\left( \sigma ,\omega \right) +A_{p}^{\ell ^{2},\limfunc{trip}}\left(
\sigma ,\omega \right) +\mathcal{A}_{p}^{\limfunc{punct}}\left( \sigma
,\omega \right) +\mathcal{E}_{H,p}^{\ell ^{2},\limfunc{ext}}\left( \sigma
,\omega \right) \right) \ \left\Vert f\right\Vert _{L^{p}(\sigma
)}\left\Vert g\right\Vert _{L^{p^{\prime }}(\omega )}\ ,  \notag \\
\boldsymbol{(7)} &&\ \left\vert \mathsf{B}_{\func{far}}\left( f,g\right)
\right\vert \lesssim \left( \mathfrak{T}_{H,p}^{\ell ^{2},\limfunc{glob}%
}\left( \sigma ,\omega \right) +\mathcal{E}_{H,p}^{\ell ^{2},\limfunc{ext}%
}\left( \sigma ,\omega \right) \right) \ \left\Vert f\right\Vert
_{L^{p}(\sigma )}\left\Vert g\right\Vert _{L^{p^{\prime }}(\omega )}\ , 
\notag \\
\boldsymbol{(8)} &&\ \left\vert \mathsf{B}_{\limfunc{para}}\left( f,g\right)
\right\vert \lesssim \mathfrak{T}_{H,p}^{\ell ^{2},\func{loc}}\left( \sigma
,\omega \right) \ \left\Vert f\right\Vert _{L^{p}(\sigma )}\left\Vert
g\right\Vert _{L^{p^{\prime }}(\omega )}\ ,  \notag \\
\boldsymbol{(9)} &&\ \left\vert \mathsf{B}_{\limfunc{stop}}\left( f,g\right)
\right\vert \lesssim \mathfrak{T}_{H,p}^{\func{loc}}\left( \sigma ,\omega
\right) \ \left\Vert f\right\Vert _{L^{p}(\sigma )}\left\Vert g\right\Vert
_{L^{p^{\prime }}(\omega )},\ 1<p<4\ .  \notag
\end{eqnarray}

These bounds, together with the necessity results above, complete the proofs
of both Theorems \ref{main glob} and \ref{main} (the dual of $\left(
9\right) $ requires $1<p^{\prime }<4$ as well) because the global quadratic
testing characteristic dominates all the other characteristics, with the
exception of the triple quadratic Muckenhoupt characteristic. Note that the
quadratic weak boundedness characteristic is used only for the \emph{%
comparable} form, that the triple quadratic Muckenhoupt characteristic is
used only for the long-range portion of the \emph{disjoint} form, and that
the scalar tailed Muckenhoupt characteristic is used only for the \emph{far}
form. Each of the bounds (2) - (5) and (7) - (9) involve just one of our
quadratic hypotheses on the right hand side, while in the first bound (1),
the reader can easily check that the comparable form can be naturally
decomposed into three pieces (with overlapping, separted and adjacent
intervals), each of which are bounded by just one of the characteristics,
but there are too many decompositions of the functional energy inequality to
separate out where each characteristic is being used for control of the far
form.

We now describe these decompositions in detail, repeating what is needed
from \cite{Saw7} for the convenience of the reader. Following \cite{NTV4}
and \cite{LaSaShUr3}, we fix a dyadic grid $\mathcal{D}$ and assume without
loss of generality, see e.g. \cite[(4.3) in Section 4]{NTV4}, that both $f$
and $g$ are supported in a fixed dyadic interval $T\in \mathcal{D}$ and have 
$\int_{T}fd\sigma =\int_{T}gd\omega =0$. We first expand the Hilbert
transform bilinear form $\left\langle H_{\sigma }f,g\right\rangle _{\omega }$
in terms of the Haar decompositions of $f$ and $g$,%
\begin{equation*}
\left\langle H_{\sigma }f,g\right\rangle _{\omega }=\sum_{I,J\in \mathcal{D}%
}\left\langle H_{\sigma }\bigtriangleup _{I}^{\sigma }f,\bigtriangleup
_{J}^{\omega }g\right\rangle _{\omega },
\end{equation*}%
and then assuming the Haar supports of $f$ and $g$ lie in $\mathcal{D}_{%
\func{good}}^{\limfunc{child}}$, we decompose the double sum above as
follows,%
\begin{eqnarray*}
\left\langle H_{\sigma }f,g\right\rangle _{\omega } &=&\left\{ \sum 
_{\substack{ I,J\in \mathcal{D}  \\ J\subset _{\tau }I}}+\sum_{\substack{ %
I,J\in \mathcal{D}  \\ I\subset _{\tau }J}}+\sum_{\substack{ I,J\in \mathcal{%
D}  \\ I\cap J=\emptyset \text{ and }\ell \left( J\right) <2^{-\tau }\ell
\left( I\right) }}+\sum_{\substack{ I,J\in \mathcal{D}  \\ I\cap J=\emptyset 
\text{ and }\ell \left( I\right) <2^{-\tau }\ell \left( J\right) }}\right. \\
&&\ \ \ \ \ \ \ \ \ \ \ \ \ \ \ \ \ \ \ \ \ \ \ \ \ \ \ \ \ \ +\left. \sum 
_{\substack{ I,J\in \mathcal{D}  \\ J\subset I\text{ and }\ell \left(
J\right) \geq 2^{-\tau }\ell \left( I\right) }}+\sum_{\substack{ I,J\in 
\mathcal{D}  \\ I\subset J\text{ and }\ell \left( I\right) \geq 2^{-\tau
}\ell \left( J\right) }}\right\} \left\langle H_{\sigma }\bigtriangleup
_{I}^{\sigma }f,\bigtriangleup _{J}^{\omega }g\right\rangle _{\omega } \\
&\equiv &\mathsf{B}_{\func{below}}\left( f,g\right) +\mathsf{B}_{\func{above}%
}\left( f,g\right) +\mathsf{B}_{\limfunc{disj}}\left( f,g\right) +\mathsf{B}%
_{\limfunc{disj}}^{\ast }\left( f,g\right) +\mathsf{B}_{\limfunc{comp}%
}\left( f,g\right) +\mathsf{B}_{\limfunc{comp}}^{\ast }\left( f,g\right) ,
\end{eqnarray*}%
where 
\begin{equation*}
\tau =r+1,
\end{equation*}%
and where $J\subset _{\tau }I$ is defined in (\ref{def deep embed}). The
first two forms are symmetric, and so it suffices to prove the boundedness
of just one of them, say $\mathsf{B}_{\func{below}}\left( f,g\right) $,$\ $%
for all $1<p<\infty $. Indeed, with the more precise definitions%
\begin{equation*}
\mathsf{B}_{\func{below}}^{H,\left( \sigma ,\omega \right) }\left(
f,g\right) \equiv \sum_{\substack{ I,J\in \mathcal{D}  \\ J\subset _{\tau }I 
}}\left\langle H_{\sigma }\bigtriangleup _{I}^{\sigma }f,\bigtriangleup
_{J}^{\omega }g\right\rangle _{\omega }\ \text{and }\mathsf{B}_{\func{above}%
}^{H,\left( \sigma ,\omega \right) }\left( f,g\right) \equiv \sum_{\substack{
I,J\in \mathcal{D}  \\ I\subset _{\tau }J}}\left\langle H_{\sigma
}\bigtriangleup _{I}^{\sigma }f,\bigtriangleup _{J}^{\omega }g\right\rangle
_{\omega }\ ,
\end{equation*}%
we have%
\begin{equation*}
\mathsf{B}_{\func{above}}^{H,\left( \sigma ,\omega \right) }\left(
f,g\right) =\sum_{\substack{ J,I\in \mathcal{D}  \\ J\subset _{\tau }I}}%
\left\langle H_{\sigma }\bigtriangleup _{J}^{\sigma }f,\bigtriangleup
_{I}^{\omega }g\right\rangle _{\omega }=-\sum_{\substack{ I,J\in \mathcal{D} 
\\ J\subset _{\tau }I}}\left\langle H_{\omega }\bigtriangleup _{I}^{\omega
}g,\bigtriangleup _{J}^{\sigma }f\right\rangle _{\sigma }=-\mathsf{B}_{\func{%
below}}^{H,\left( \omega ,\sigma \right) }\left( g,f\right) .
\end{equation*}

Using a Calder\'{o}n-Zygmund corona decomposition with parameter $\Gamma >1$%
, we will later decompose the below form $\mathsf{B}_{\func{below}}\left(
f,g\right) $ into another four forms, 
\begin{equation*}
\mathsf{B}_{\func{below}}\left( f,g\right) =\mathsf{B}_{\func{neigh}}\left(
f,g\right) +\mathsf{B}_{\func{far}}\left( f,g\right) +\mathsf{B}_{\limfunc{%
para}}\left( f,g\right) +\mathsf{B}_{\limfunc{stop}}\left( f,g\right) ,
\end{equation*}%
in which there is control of averages of $f$ in each corona. At this point
we will have twelve forms in our decomposition of the inner product $%
\left\langle H_{\sigma }f,g\right\rangle _{\omega }$.

\subsection{Comparable form}

We will bound the comparable form%
\begin{equation*}
\mathsf{B}_{\limfunc{comp}}\left( f,g\right) =\sum_{\substack{ I,J\in 
\mathcal{D}  \\ J\subset I\text{ and }\ell \left( J\right) \geq 2^{-\tau
}\ell \left( I\right) }}\left\langle \mathbf{1}_{J}H_{\sigma }\bigtriangleup
_{I}^{\sigma }f,\bigtriangleup _{J}^{\omega }g\right\rangle _{\omega }
\end{equation*}%
for $1<p<\infty $, by the local quadratic testing, offset quadratic
Muckenhoupt, and quadratic weak boundedness characteristics, i.e. we prove%
\begin{equation*}
\left\vert \mathsf{B}_{\limfunc{comp}}\left( f,g\right) \right\vert \lesssim
\left( \mathfrak{T}_{H,p}^{\ell ^{2},\func{loc}}\left( \sigma ,\omega
\right) +A_{p}^{\ell ^{2},\limfunc{offset}}\left( \sigma ,\omega \right) +%
\mathcal{WBP}_{H,p}^{\ell ^{2}}\left( \sigma ,\omega \right) \right)
\left\Vert f\right\Vert _{L^{p}\left( \sigma \right) }\left\Vert
g\right\Vert _{L^{p^{\prime }}\left( \omega \right) }\lesssim \mathfrak{T}%
_{H,p}^{\ell ^{2},\limfunc{glob}}\left( \sigma ,\omega \right) \ .
\end{equation*}%
This is the only place in this paper where we use the quadratic weak
boundedness characteristic $\mathcal{WBP}_{H,p}^{\ell ^{2}}\left( \sigma
,\omega \right) $. Note also that the second inequality has already been
proved in the section on necessity.

We write 
\begin{eqnarray*}
\bigtriangleup _{I}^{\sigma }f &=&\left( E_{I_{\limfunc{left}}}^{\sigma
}\bigtriangleup _{I}^{\sigma }f\right) \mathbf{1}_{I_{\limfunc{left}%
}}+\left( E_{I_{\limfunc{right}}}^{\sigma }\bigtriangleup _{I}^{\sigma
}f\right) \mathbf{1}_{I_{\limfunc{right}}}\ , \\
\bigtriangleup _{J}^{\omega }g &=&\left( E_{J_{\limfunc{left}}}^{\sigma
}\bigtriangleup _{J}^{\omega }g\right) \mathbf{1}_{J_{\limfunc{left}%
}}+\left( E_{J_{\limfunc{right}}}^{\sigma }\bigtriangleup _{J}^{\omega
}g\right) \mathbf{1}_{J_{\limfunc{right}}}
\end{eqnarray*}%
and%
\begin{eqnarray*}
\mathsf{B}_{\limfunc{comp}}\left( f,g\right) &=&\sum_{\substack{ I,J\in 
\mathcal{D}  \\ J\subset I\text{ and }\ell \left( J\right) \geq 2^{-\tau
}\ell \left( I\right) }}\left( E_{I_{\limfunc{left}}}^{\sigma
}\bigtriangleup _{I}^{\sigma }f\right) \left( E_{J_{\limfunc{left}}}^{\sigma
}\bigtriangleup _{J}^{\omega }g\right) \left\langle \mathbf{1}_{J_{\limfunc{%
left}}}H_{\sigma }\mathbf{1}_{I_{\limfunc{left}}},\mathbf{1}_{J_{\limfunc{%
left}}}\right\rangle _{\omega } \\
&&+\sum_{\substack{ I,J\in \mathcal{D}  \\ J\subset I\text{ and }\ell \left(
J\right) \geq 2^{-\tau }\ell \left( I\right) }}\left( E_{I_{\limfunc{left}%
}}^{\sigma }\bigtriangleup _{I}^{\sigma }f\right) \left( E_{J_{\limfunc{right%
}}}^{\sigma }\bigtriangleup _{J}^{\omega }g\right) \left\langle \mathbf{1}%
_{J_{\limfunc{right}}}H_{\sigma }\mathbf{1}_{I_{\limfunc{left}}},\mathbf{1}%
_{J_{\limfunc{right}}}\right\rangle _{\omega } \\
&&+\sum_{\substack{ I,J\in \mathcal{D}  \\ J\subset I\text{ and }\ell \left(
J\right) \geq 2^{-\tau }\ell \left( I\right) }}\left( E_{I_{\limfunc{right}%
}}^{\sigma }\bigtriangleup _{I}^{\sigma }f\right) \left( E_{J_{\limfunc{left}%
}}^{\sigma }\bigtriangleup _{J}^{\omega }g\right) \left\langle \mathbf{1}%
_{J_{\limfunc{left}}}H_{\sigma }\mathbf{1}_{I_{\limfunc{right}}},\mathbf{1}%
_{J_{\limfunc{left}}}\right\rangle _{\omega } \\
&&+\sum_{\substack{ I,J\in \mathcal{D}  \\ J\subset I\text{ and }\ell \left(
J\right) \geq 2^{-\tau }\ell \left( I\right) }}\left( E_{I_{\limfunc{right}%
}}^{\sigma }\bigtriangleup _{I}^{\sigma }f\right) \left( E_{J_{\limfunc{right%
}}}^{\sigma }\bigtriangleup _{J}^{\omega }g\right) \left\langle \mathbf{1}%
_{J_{\limfunc{right}}}H_{\sigma }\mathbf{1}_{I_{\limfunc{right}}},\mathbf{1}%
_{J_{\limfunc{right}}}\right\rangle _{\omega } \\
&\equiv &\mathsf{B}_{\limfunc{comp}}^{\limfunc{left},\limfunc{left}}\left(
f,g\right) +\mathsf{B}_{\limfunc{comp}}^{\limfunc{left},\limfunc{right}%
}\left( f,g\right) +\mathsf{B}_{\limfunc{comp}}^{\limfunc{right},\limfunc{%
left}}\left( f,g\right) +\mathsf{B}_{\limfunc{comp}}^{\limfunc{right},%
\limfunc{right}}\left( f,g\right) .
\end{eqnarray*}%
If the pair of intervals $I_{\zeta /\eta }$ and $J_{\zeta /\eta }$ are
disjoint, $\zeta ,\eta \in \left\{ \limfunc{left},\limfunc{right}\right\} $,
then the sums above are immediately controlled by the quadratic weak
boundedness characteristic when they are adjacent, and by the quadratic
offset Muckenhoupt characteristic when they are not.

If the intervals overlap, then using the $\ell ^{2}$ Cauchy-Schwarz and $%
L^{p}\left( \omega \right) $ H\"{o}lder inequalities we obtain 
\begin{eqnarray*}
&&\left\vert \mathsf{B}_{\limfunc{comp}}\left( f,g\right) \right\vert \leq
\sum_{\zeta ,\eta \in \left\{ \limfunc{left},\limfunc{right}\right\}
}\left\vert \mathsf{B}_{\limfunc{comp}}^{\zeta ,\eta }\left( f,g\right)
\right\vert \\
&\leq &\sum_{\zeta ,\eta \in \left\{ \limfunc{left},\limfunc{right}\right\}
}\left\Vert \left( \sum_{\substack{ I,J\in \mathcal{D}  \\ J\subset I\text{, 
}\ell \left( J\right) \geq 2^{-\tau }\ell \left( I\right) \text{ and }%
I_{\zeta }\cap J_{\eta }\neq \emptyset }}\left\vert \left( E_{I_{\zeta
}}^{\sigma }\bigtriangleup _{I}^{\sigma }f\right) \mathbf{1}_{J_{\eta
}}H_{\sigma }\mathbf{1}_{I_{\zeta }}\left( x\right) \right\vert ^{2}\right)
^{\frac{1}{2}}\right\Vert _{L^{p}\left( \omega \right) } \\
&&\ \ \ \ \ \ \ \ \ \ \ \ \ \ \ \ \ \ \ \ \ \ \ \ \ \times \left\Vert \left(
\sum_{\substack{ I,J\in \mathcal{D}  \\ J\subset I\text{ and }\ell \left(
J\right) \geq 2^{-\tau }\ell \left( I\right) }}\left\vert \left( E_{J_{\eta
}}^{\sigma }\bigtriangleup _{J}^{\omega }g\right) \mathbf{1}_{J_{\eta
}}\left( x\right) \right\vert ^{2}\right) ^{\frac{1}{2}}\right\Vert
_{L^{p^{\prime }}\left( \omega \right) }.
\end{eqnarray*}%
The second factor is dominated by%
\begin{equation*}
\left\Vert \left( \sum_{\substack{ I,J\in \mathcal{D}  \\ J\subset I\text{
and }\ell \left( J\right) \geq 2^{-\tau }\ell \left( I\right) }}\left\vert
\left( \bigtriangleup _{J}^{\omega }g\right) \left( x\right) \right\vert
^{2}\right) ^{\frac{1}{2}}\right\Vert _{L^{p^{\prime }}\left( \omega \right)
}\lesssim \left\Vert \mathcal{S}^{\omega }g\right\Vert _{L^{p^{\prime
}}\left( \omega \right) }\approx \left\Vert g\right\Vert _{L^{p^{\prime
}}\left( \omega \right) }
\end{equation*}%
by the square function estimate in Theorem \ref{square thm}, and since there
are only $\tau +1$ intervals $I\in \mathcal{D}$ with $J\subset I$ and $\ell
\left( J\right) \geq 2^{-\tau }\ell \left( I\right) $.

We now turn to the first factor on the right side above where the intervals $%
I_{\zeta }$ and $J_{\eta }$ overlap, and consider separately the cases $J=I$
and $J_{\eta }\subset I_{\zeta }$ for some choice of $\zeta ,\eta \in
\left\{ \limfunc{left},\limfunc{right}\right\} $. For the case $J=I$ we have%
\begin{eqnarray*}
\left\Vert \left( \sum_{I\in \mathcal{D}}\left( E_{I_{\limfunc{left}%
}}^{\sigma }\bigtriangleup _{I}^{\sigma }f\right) ^{2}\left\vert \mathbf{1}%
_{I_{\limfunc{left}}}H_{\sigma }\mathbf{1}_{I_{\limfunc{left}}}\right\vert
^{2}\right) ^{\frac{1}{2}}\right\Vert _{L^{p}\left( \omega \right) } &\leq &%
\mathfrak{T}_{H,p}^{\ell ^{2},\func{loc}}\left( \sigma ,\omega \right)
\left\Vert \left( \sum_{I\in \mathcal{D}}\left( E_{I_{\limfunc{left}%
}}^{\sigma }\bigtriangleup _{I}^{\sigma }f\right) ^{2}\mathbf{1}_{I_{%
\limfunc{left}}}\right) ^{\frac{1}{2}}\right\Vert _{L^{p}\left( \sigma
\right) } \\
&\lesssim &\mathfrak{T}_{H,p}^{\ell ^{2},\func{loc}}\left( \sigma ,\omega
\right) \left\Vert f\right\Vert _{L^{p}\left( \sigma \right) },
\end{eqnarray*}%
and similarly%
\begin{equation*}
\left\Vert \left( \sum_{I\in \mathcal{D}}\left( E_{I_{\limfunc{right}%
}}^{\sigma }\bigtriangleup _{I}^{\sigma }f\right) ^{2}\left\vert \mathbf{1}%
_{I_{\limfunc{right}}}H_{\sigma }\mathbf{1}_{I_{\limfunc{right}}}\right\vert
^{2}\right) ^{\frac{1}{2}}\right\Vert _{L^{p}\left( \omega \right) }\lesssim 
\mathfrak{T}_{H,p}^{\ell ^{2},\func{loc}}\left( \sigma ,\omega \right)
\left\Vert f\right\Vert _{L^{p}\left( \sigma \right) }\ .
\end{equation*}%
For those $J_{\eta }\subset I_{\zeta }$, we apply the local quadratic
testing condition to $H_{\sigma }\mathbf{1}_{I_{\zeta }}$, and again finish
with the square function Theorem \ref{square thm}. Finally, the estimate for
the dual comparable form $\mathsf{B}_{\limfunc{comp}}^{\ast }\left(
f,g\right) $ is handled symmetrically.

\subsection{Disjoint form}

Here we prove the following estimates for the disjoint form $\mathsf{B}_{%
\limfunc{disj}}\left( f,g\right) $ with absolute values inside the sum: 
\begin{eqnarray}
&&\left\vert \mathsf{B}_{\limfunc{disj}}\right\vert \left( f,g\right) \equiv
\dsum\limits_{\substack{ I,J\in \mathcal{D}  \\ J\cap I=\emptyset \text{ and 
}\ell \left( J\right) <2^{-\tau }\ell \left( I\right) }}\left\vert
\left\langle H_{\sigma }\bigtriangleup _{I}^{\sigma }f,\bigtriangleup
_{J}^{\omega }g\right\rangle _{\omega }\right\vert  \label{routine'} \\
&\lesssim &A_{p}^{\ell ^{2},\limfunc{trip}}\left( \sigma ,\omega \right)
\left\Vert f\right\Vert _{L^{p}\left( \sigma \right) }\left\Vert
g\right\Vert _{L^{p^{\prime }}\left( \omega \right) },\ \ \ \ \ 1<p<\infty ,
\notag
\end{eqnarray}%
and%
\begin{equation}
\left\vert \mathsf{B}_{\limfunc{disj}}\right\vert \left( f,g\right) \lesssim 
\mathfrak{T}_{H,p}^{\ell ^{2},\limfunc{glob}}\left( \sigma ,\omega \right)
\left\Vert f\right\Vert _{L^{p}\left( \sigma \right) }\left\Vert
g\right\Vert _{L^{p^{\prime }}\left( \omega \right) },\ \ \ \ \ 1<p<\infty .
\label{routine''}
\end{equation}

This is the only place in this paper where we use the triple quadratic
Muckenhoupt characteristic $A_{p}^{\ell ^{2},\limfunc{trip}}\left( \sigma
,\omega \right) $, or make \emph{direct} use of the global quadratic testing
characteristic $\mathfrak{T}_{H,p}^{\ell ^{2},\limfunc{glob}}\left( \sigma
,\omega \right) $.

\begin{proof}[Proof of (\protect\ref{routine'}) and (\protect\ref{routine''})%
]
We further decompose the form $\left\vert \mathsf{B}_{\limfunc{disj}%
}\right\vert \left( f,g\right) $ as 
\begin{eqnarray*}
&&\left\vert \mathsf{B}_{\limfunc{disj}}\right\vert \left( f,g\right) \equiv
\sum_{\substack{ I,J\in \mathcal{D}  \\ I\cap J=\emptyset \text{ and }\ell
\left( J\right) \leq 2^{-\tau }\ell \left( I\right) }}\left\langle H_{\sigma
}\bigtriangleup _{I}^{\sigma }f,\bigtriangleup _{J}^{\omega }g\right\rangle
_{\omega } \\
&=&\sum_{I\in \mathcal{D}}\sum_{\substack{ J\in \mathcal{D}:\ \ell \left(
J\right) \leq \ell \left( I\right)  \\ d\left( J,I\right) >\ell \left(
I\right) }}\left\langle H_{\sigma }\bigtriangleup _{I}^{\sigma
}f,\bigtriangleup _{J}^{\omega }g\right\rangle _{\omega }+\sum_{I\in 
\mathcal{D}}\sum_{\substack{ J\in \mathcal{D}:\ \ell \left( J\right) \leq
\ell \left( I\right)  \\ d\left( J,I\right) \leq \ell \left( I\right) }}%
\left\langle H_{\sigma }\bigtriangleup _{I}^{\sigma }f,\bigtriangleup
_{J}^{\omega }g\right\rangle _{\omega } \\
&\equiv &\mathcal{A}^{\limfunc{long}}\left( f,g\right) +\mathcal{A}^{%
\limfunc{mid}}\left( f,g\right) .
\end{eqnarray*}

\bigskip

\textbf{The long-range case: }Here we prove that the long-range form $%
\mathcal{A}^{\limfunc{long}}\left( f,g\right) $ can be bounded \textbf{either%
} by the triple quadratic Muckenhoupt characteristic $A_{p}^{\ell ^{2},%
\limfunc{trip}}\left( \sigma ,\omega \right) $, \textbf{or} by the global
quadratic testing characteristic $\mathfrak{T}_{H,p}^{\ell ^{2},\limfunc{glob%
}}\left( \sigma ,\omega \right) $.

\bigskip

\textbf{Claim \#1}%
\begin{equation}
\left\vert \mathcal{A}^{\limfunc{long}}\left( f,g\right) \right\vert \leq
\sum_{I\in \mathcal{D}}\sum_{\substack{ J\in \mathcal{D}:\ \ell \left(
J\right) \leq \ell \left( I\right)  \\ d\left( J,I\right) >\ell \left(
I\right) }}\left\vert \int_{\mathbb{R}}\left( H_{\sigma }\bigtriangleup
_{I}^{\sigma }f\right) \bigtriangleup _{J}^{\omega }gd\omega \right\vert
\lesssim A_{p}^{\ell ^{2},\limfunc{trip}}\left( \sigma ,\omega \right)
\left\Vert f\right\Vert _{L^{p}\left( \sigma \right) }\left\Vert
g\right\Vert _{L^{p^{\prime }}\left( \omega \right) }.  \label{long est}
\end{equation}%
\textbf{Proof}: In the sum in the middle of the display above, we pigeonhole
the intervals $I$ and $J$ relative to intervals $K\in \mathcal{D}$. Let $%
N,s,t\in \mathbb{N}$. For $K\in \mathcal{D}_{N}\equiv \left\{ I\in \mathcal{D%
}:\ell \left( I\right) =2^{N}\right\} $, we restrict $I$ and $J$ to $I\in 
\mathcal{D}_{N-s}$ and $J\in \mathcal{D}_{N-s-t}$ respectively and write%
\begin{eqnarray*}
&&\sum_{I\in \mathcal{D}}\sum_{\substack{ J\in \mathcal{D}:\ \ell \left(
J\right) \leq \ell \left( I\right)  \\ d\left( J,I\right) >\ell \left(
I\right) }}\left\vert \int_{\mathbb{R}}\left( H_{\sigma }\bigtriangleup
_{I}^{\sigma }f\right) \bigtriangleup _{J}^{\omega }gd\omega \right\vert \\
&=&\sum_{s,t\in \mathbb{N}}\left\{ \sum_{N\in \mathbb{Z}}\sum_{K\in \mathcal{%
D}_{N}}\sum_{\substack{ I\in \mathcal{D}_{N-s}  \\ I\subset K}}\sum 
_{\substack{ J\in \mathcal{D}_{N-s-t}  \\ d\left( J,I\right) \approx \ell
\left( K\right) }}\left\vert \int_{\mathbb{R}}\left( H_{\sigma
}\bigtriangleup _{I}^{\sigma }f\right) \bigtriangleup _{J}^{\omega }gd\omega
\right\vert \right\} =\sum_{s,t\in \mathbb{N}}W_{s,t}, \\
&&\text{where }W_{s,t}\equiv \sum_{N\in \mathbb{Z}}\sum_{K\in \mathcal{D}%
_{N}}\sum_{\substack{ I\in \mathcal{D}_{N-s}  \\ I\subset K}}\sum_{\substack{
J\in \mathcal{D}_{N-s-t}  \\ d\left( J,I\right) \approx \ell \left( K\right) 
}}\left\vert \int_{\mathbb{R}}\left( H_{\sigma }\bigtriangleup _{I}^{\sigma
}f\right) \bigtriangleup _{J}^{\omega }gd\omega \right\vert ,
\end{eqnarray*}%
and observe that%
\begin{equation*}
W_{s,t}\leq \int_{\mathbb{R}}\left\{ \sum_{N\in \mathbb{Z}}\sum_{K\in 
\mathcal{D}_{N}}\sum_{\substack{ J\in \mathcal{D}_{N-s-t}  \\ d\left(
J,I\right) \approx \ell \left( K\right) }}\left\vert \bigtriangleup
_{J}^{\omega }H_{\sigma }\sum_{\substack{ I\in \mathcal{D}_{N-s}  \\ %
I\subset K}}\bigtriangleup _{I}^{\sigma }f\left( x\right) \right\vert
\left\vert \bigtriangleup _{J}^{\omega }g\left( x\right) \right\vert
\right\} d\omega \left( x\right) .
\end{equation*}%
By the Monotonicity Lemma and Poisson Decay Lemma, this is bounded by,%
\begin{eqnarray*}
W_{s,t} &\leq &\int_{\mathbb{R}}\left\{ \sum_{N\in \mathbb{Z}}\sum_{K\in 
\mathcal{D}_{N}}\sum_{\substack{ J\in \mathcal{D}_{N-s-t}  \\ d\left(
J,I\right) \approx \ell \left( K\right) }}\mathrm{P}\left( J,\sum_{\substack{
I\in \mathcal{D}_{N-s}  \\ I\subset K}}\left\vert \bigtriangleup
_{I}^{\sigma }f\right\vert \sigma \right) \mathbf{1}_{J}\left( x\right)
\left\vert \bigtriangleup _{J}^{\omega }g\left( x\right) \right\vert
\right\} d\omega \left( x\right) \\
&\lesssim &2^{-\left( s+t\right) \left( 1-2\varepsilon \right) }\int_{%
\mathbb{R}}\sum_{N\in \mathbb{Z}}\sum_{K\in \mathcal{D}_{N}}\mathrm{P}\left(
K,f_{K}^{N}\sigma \right) \mathbf{1}_{J}\left( x\right) \ g_{K}^{N}\left(
x\right) \ d\omega \left( x\right) ,
\end{eqnarray*}%
where%
\begin{eqnarray*}
f_{K}^{N}\left( x\right) &\equiv &\sum_{\substack{ I\in \mathcal{D}_{N-s} 
\\ I\subset K}}\left\vert \bigtriangleup _{I}^{\sigma }f\left( x\right)
\right\vert \text{ and }g_{K}^{N}\left( x\right) \equiv \sum_{\substack{ %
J\in \mathcal{D}_{N-s-t}  \\ d\left( J,I\right) \approx \ell \left( K\right) 
}}\left\vert \bigtriangleup _{J}^{\omega }g\left( x\right) \right\vert , \\
\text{and }g_{K}^{N}\left( x\right) ^{2} &\leq &2^{s+t}\sum_{\substack{ J\in 
\mathcal{D}_{N-s-t}  \\ d\left( J,I\right) \approx \ell \left( K\right) }}%
\left\vert \bigtriangleup _{J}^{\omega }g\left( x\right) \right\vert ^{2}.
\end{eqnarray*}

Thus we have%
\begin{eqnarray*}
&&W_{s,t}\lesssim 2^{-\left( s+t\right) \left( 1-2\varepsilon \right)
}\sum_{N\in \mathbb{Z}}\int_{\mathbb{R}}\sum_{K\in \mathcal{D}_{N}}\mathrm{P}%
\left( K,f_{K}^{N}\sigma \right) \mathbf{1}_{3K\setminus K}\left( x\right) \
g_{K}^{N}\left( x\right) \ d\omega \left( x\right) \\
&\leq &2^{-\left( s+t\right) \left( 1-2\varepsilon \right) }\left\Vert
\left( \sum_{N\in \mathbb{Z}}\sum_{K\in \mathcal{D}_{N}}\mathrm{P}\left(
K,f_{K}^{N}\sigma \right) ^{2}\mathbf{1}_{3K\setminus K}\left( x\right)
\right) ^{\frac{1}{2}}\right\Vert _{L^{p}\left( \omega \right) }\left\Vert
\left( \sum_{N\in \mathbb{Z}}\sum_{K\in \mathcal{D}_{N}}\left\vert
g_{K}^{N}\left( x\right) \right\vert ^{2}\right) ^{\frac{1}{2}}\right\Vert
_{L^{p^{\prime }}\left( \omega \right) } \\
&\leq &2^{-\left( s+t\right) \left( 1-2\varepsilon \right) }\left\Vert
\left( \sum_{N\in \mathbb{Z}}\sum_{K\in \mathcal{D}_{N}}\left( \frac{%
\int_{K}f_{K}^{N}d\sigma }{\left\vert K\right\vert }\right) ^{2}\mathbf{1}%
_{3K\setminus K}\left( x\right) \right) ^{\frac{1}{2}}\right\Vert
_{L^{p}\left( \omega \right) }2^{\frac{s+t}{2}}\left\Vert \left( \sum_{N\in 
\mathbb{Z}}\sum_{K\in \mathcal{D}_{N}}\sum_{\substack{ J\in \mathcal{D}%
_{N-s-t}  \\ d\left( J,I\right) \approx \ell \left( K\right) }}\left\vert
\bigtriangleup _{J}^{\omega }g\left( x\right) \right\vert ^{2}\right) ^{%
\frac{1}{2}}\right\Vert _{L^{p^{\prime }}\left( \omega \right) } \\
&\leq &2^{-\left( s+t\right) \left( 1-2\varepsilon \right) }A_{p}^{\ell ^{2},%
\limfunc{trip}}\left( \sigma ,\omega \right) \left\Vert \left( \sum_{N\in 
\mathbb{Z}}\sum_{K\in \mathcal{D}_{N}}f_{K}^{N}\left( x\right) ^{2}\right) ^{%
\frac{1}{2}}\right\Vert _{L^{p}\left( \sigma \right) }2^{\frac{s+t}{2}%
}\left\Vert g\right\Vert _{L^{p^{\prime }}\left( \omega \right) }\ ,
\end{eqnarray*}%
where we have used the triple quadratic Muckenhoupt condition in the last
line, after breaking up the annulus $3K\setminus K$ into its left and right
hand intervals, and where we have also used $\left\Vert \mathcal{S}%
g\right\Vert _{L^{p^{\prime }}\left( \omega \right) }\approx \left\Vert
g\right\Vert _{L^{p^{\prime }}\left( \omega \right) }$.

Moreover, this last line is dominated by%
\begin{eqnarray*}
W_{s,t} &\lesssim &2^{-\left( s+t\right) \left( \frac{1}{2}-2\varepsilon
\right) }A_{p}^{\ell ^{2},\limfunc{trip}}\left( \sigma ,\omega \right)
\left\Vert \left( \sum_{N\in \mathbb{Z}}\sum_{K\in \mathcal{D}_{N}}\sum 
_{\substack{ I\in \mathcal{D}_{N-s}  \\ I\subset K}}\left\vert
\bigtriangleup _{I}^{\sigma }f\left( x\right) \right\vert ^{2}\right) ^{%
\frac{1}{2}}\right\Vert _{L^{p}\left( \sigma \right) }\left\Vert
g\right\Vert _{L^{p^{\prime }}\left( \omega \right) } \\
&\lesssim &2^{-\left( s+t\right) \left( \frac{1}{2}-2\varepsilon \right)
}A_{p}^{\ell ^{2},\limfunc{trip}}\left( \sigma ,\omega \right) \left\Vert
f\right\Vert _{L^{p}\left( \sigma \right) }\left\Vert g\right\Vert
_{L^{p^{\prime }}\left( \omega \right) }\ ,
\end{eqnarray*}%
this time using $\left\Vert \mathcal{S}f\right\Vert _{L^{p}\left( \sigma
\right) }\approx \left\Vert f\right\Vert _{L^{p}\left( \sigma \right) }$.
Finally we sum in $s$ and $t$ to obtain%
\begin{eqnarray*}
\left\vert \mathcal{A}^{\limfunc{long}}\left( f,g\right) \right\vert
&\lesssim &\sum_{s,t\in \mathbb{N}}W_{s,t}\lesssim \left( \sum_{s,t\in 
\mathbb{N}}2^{-\left( s+t\right) \left( \frac{1}{2}-2\varepsilon \right)
}\right) A_{p}^{\ell ^{2},\limfunc{trip}}\left( \sigma ,\omega \right)
\left\Vert f\right\Vert _{L^{p}\left( \sigma \right) }\left\Vert
g\right\Vert _{L^{p^{\prime }}\left( \omega \right) } \\
&\lesssim &C_{\varepsilon }A_{p}^{\ell ^{2},\limfunc{trip}}\left( \sigma
,\omega \right) \left\Vert f\right\Vert _{L^{p}\left( \sigma \right)
}\left\Vert g\right\Vert _{L^{p^{\prime }}\left( \omega \right) },
\end{eqnarray*}%
provided we take $0<\varepsilon <\frac{1}{4}$.

\bigskip

\textbf{Claim \#2}%
\begin{equation}
\left\vert \mathcal{A}^{\limfunc{long}}\left( f,g\right) \right\vert \leq
\sum_{I\in \mathcal{D}}\sum_{\substack{ J\in \mathcal{G}:\ \ell \left(
J\right) \leq \ell \left( I\right)  \\ d\left( J,I\right) >\ell \left(
I\right) }}\left\vert \int_{\mathbb{R}}\left( H_{\sigma }\bigtriangleup
_{I}^{\sigma }f\right) \bigtriangleup _{J}^{\omega }gd\omega \right\vert
\lesssim \mathfrak{T}_{H,p}^{\ell ^{2},\limfunc{glob}}\left( \sigma ,\omega
\right) \left\Vert f\right\Vert _{L^{p}\left( \sigma \right) }\left\Vert
g\right\Vert _{L^{p^{\prime }}\left( \omega \right) }.  \label{long est'}
\end{equation}

\textbf{Proof}: Let $t\in \mathbb{Z}_{+}$, and restricting $I$ and $J$ to $%
d\left( J,I\right) >\ell \left( I\right) $ and $\ell \left( J\right)
=2^{-t}\ell \left( I\right) $, we write%
\begin{eqnarray*}
&&\sum_{I\in \mathcal{D}}\sum_{\substack{ J\in \mathcal{D}:\ \ell \left(
J\right) \leq \ell \left( I\right)  \\ d\left( J,I\right) >\ell \left(
I\right) }}\left\vert \int_{\mathbb{R}}\left( H_{\sigma }\bigtriangleup
_{I}^{\sigma }f\right) \bigtriangleup _{J}^{\omega }gd\omega \right\vert
=\sum_{t=0}^{\infty }W_{t}\ , \\
&&\ \ \ \ \ \text{where }W_{t}\equiv \sum_{I\in \mathcal{D}}\sum_{\substack{ %
J\in \mathcal{D}:\ \ell \left( J\right) =2^{-t}\ell \left( I\right)  \\ %
d\left( J,I\right) >\ell \left( I\right) }}\left\vert \int_{\mathbb{R}%
}\left( H_{\sigma }\bigtriangleup _{I}^{\sigma }f\right) \bigtriangleup
_{J}^{\omega }gd\omega \right\vert .
\end{eqnarray*}%
Now%
\begin{eqnarray*}
\int_{\mathbb{R}}\left( H_{\sigma }\bigtriangleup _{I}^{\sigma }f\right)
\bigtriangleup _{J}^{\omega }gd\omega &=&\sum_{+,-}\left( E_{\pm }^{\sigma
}\bigtriangleup _{I}^{\sigma }f\right) \int_{\mathbb{R}}\left( H_{\sigma }%
\mathbf{1}_{I_{\pm }}\right) \bigtriangleup _{J}^{\omega }gd\omega , \\
\text{where }\int_{\mathbb{R}}\left( H_{\sigma }\mathbf{1}_{I_{\pm }}\right)
\bigtriangleup _{J}^{\omega }gd\omega &=&\int_{\mathbb{R}}\left(
\bigtriangleup _{J}^{\omega }H_{\sigma }\mathbf{1}_{I_{\pm }}\right)
\bigtriangleup _{J}^{\omega }gd\omega , \\
\text{and }\bigtriangleup _{J}^{\omega }H_{\sigma }\mathbf{1}_{I_{\pm
}}\left( x\right) &=&h_{J}^{\omega }\left( x\right) \left\langle H_{\sigma }%
\mathbf{1}_{I_{\pm }},h_{J}^{\omega }\right\rangle _{\omega }, \\
\text{and }\left\vert \left\langle H_{\sigma }\mathbf{1}_{I_{\pm
}},h_{J}^{\omega }\right\rangle _{\omega }\right\vert &=&\left\vert
\int_{J}\int_{I_{\pm }}\left\{ \frac{1}{y-z}-\frac{1}{y-c_{J}}\right\}
d\sigma \left( y\right) h_{J}^{\omega }\left( z\right) d\omega \left(
z\right) \right\vert \\
&=&\left\vert \int_{J}\int_{I_{\pm }}\frac{\left( z-c_{J}\right)
h_{J}^{\omega }\left( z\right) }{\left( y-z\right) \left( y-c_{J}\right) }%
d\sigma \left( y\right) d\omega \left( z\right) \right\vert \\
&=&\int_{J}\int_{I_{\pm }}\left\vert \frac{\left( z-c_{J}\right)
h_{J}^{\omega }\left( z\right) }{\left( y-z\right) \left( y-c_{J}\right) }%
\right\vert d\sigma \left( y\right) d\omega \left( z\right) ,
\end{eqnarray*}%
since 
\begin{equation*}
h_{J}^{\omega }=\sqrt{\frac{\left\vert J_{-}\right\vert _{\omega }\left\vert
J_{+}\right\vert _{\omega }}{\left\vert J\right\vert _{\omega }}}\left( 
\frac{1}{\left\vert J_{+}\right\vert _{\omega }}\mathbf{1}_{J_{+}}-\frac{1}{%
\left\vert J_{-}\right\vert _{\omega }}\mathbf{1}_{J_{-}}\right) ,
\end{equation*}%
implies that neither $\left( z-c_{J}\right) h_{J}^{\omega }\left( z\right) $
nor $\left( y-z\right) \left( y-c_{J}\right) $ changes sign in the region of
integration. Thus%
\begin{equation*}
\left\vert \left\langle H_{\sigma }\mathbf{1}_{I_{\pm }},h_{J}^{\omega
}\right\rangle _{\omega }\right\vert \approx \left( \int_{I_{\pm }}\frac{1}{%
\left( y-c_{J}\right) ^{2}}d\sigma \left( y\right) \right) \left\vert
\int_{J}\left( z-c_{J}\right) h_{J}^{\omega }\left( z\right) d\omega \left(
z\right) \right\vert
\end{equation*}%
and%
\begin{eqnarray*}
\int_{J}\left( z-c_{J}\right) h_{J}^{\omega }\left( z\right) d\omega \left(
z\right) &=&\int_{J}\left( z-c_{J}\right) \sqrt{\frac{\left\vert
J_{-}\right\vert _{\omega }\left\vert J_{+}\right\vert _{\omega }}{%
\left\vert J\right\vert _{\omega }}}\left( \frac{1}{\left\vert
J_{+}\right\vert _{\omega }}\mathbf{1}_{J_{+}}-\frac{1}{\left\vert
J_{-}\right\vert _{\omega }}\mathbf{1}_{J_{-}}\right) d\omega \left( z\right)
\\
&=&\sqrt{\frac{\left\vert J_{-}\right\vert _{\omega }\left\vert
J_{+}\right\vert _{\omega }}{\left\vert J\right\vert _{\omega }}}\left( 
\frac{1}{\left\vert J_{+}\right\vert _{\omega }}\int_{J_{+}}\left(
z-c_{J}\right) d\omega \left( z\right) -\frac{1}{\left\vert J_{-}\right\vert
_{\omega }}\int_{J_{-}}\left( z-c_{J}\right) d\omega \left( z\right) \right)
\\
&=&\sqrt{\frac{\left\vert J_{-}\right\vert _{\omega }\left\vert
J_{+}\right\vert _{\omega }}{\left\vert J\right\vert _{\omega }}}\left(
m_{J_{+}}-m_{J_{-}}\right) ,
\end{eqnarray*}%
where $m_{J_{+}}=w_{+}-c_{J}$ and $m_{J_{-}}=w_{-}-c_{J}$ with $w_{\pm }\in
J_{\pm }$. Thus%
\begin{eqnarray*}
&&\left\vert \bigtriangleup _{J}^{\omega }H_{\sigma }\mathbf{1}_{I_{\pm
}}\left( x\right) \right\vert =\left\vert h_{J}^{\omega }\left( x\right)
\left\langle H_{\sigma }\mathbf{1}_{I_{\pm }},h_{J}^{\omega }\right\rangle
_{\omega }\right\vert \\
&\approx &\left( \int_{I_{\pm }}\frac{1}{\left( y-c_{J}\right) ^{2}}d\sigma
\left( y\right) \right) \left\vert h_{J}^{\omega }\left( x\right)
\int_{J}\left( z-c_{J}\right) h_{J}^{\omega }\left( z\right) d\omega \left(
z\right) \right\vert \\
&=&\left( \int_{I_{\pm }}\frac{1}{\left( y-c_{J}\right) ^{2}}d\sigma \left(
y\right) \right) \left\vert m_{J_{+}}-m_{J_{-}}\right\vert \sqrt{\frac{%
\left\vert J_{-}\right\vert _{\omega }\left\vert J_{+}\right\vert _{\omega }%
}{\left\vert J\right\vert _{\omega }}}\left\vert h_{J}^{\omega }\left(
x\right) \right\vert \\
&=&\left( \int_{I_{\pm }}\frac{1}{\left( y-c_{J}\right) ^{2}}d\sigma \left(
y\right) \right) \left\vert m_{J_{+}}-m_{J_{-}}\right\vert \frac{\left\vert
J_{-}\right\vert _{\omega }\left\vert J_{+}\right\vert _{\omega }}{%
\left\vert J\right\vert _{\omega }}\left( \frac{1}{\left\vert
J_{+}\right\vert _{\omega }}\mathbf{1}_{J_{+}}\left( x\right) -\frac{1}{%
\left\vert J_{-}\right\vert _{\omega }}\mathbf{1}_{J_{-}}\left( x\right)
\right) ,
\end{eqnarray*}%
where 
\begin{equation*}
\left\vert m_{J_{+}}-m_{J_{-}}\right\vert =\left\vert w_{+}-w_{-}\right\vert
\leq \ell \left( J\right) =2^{-t}\ell \left( I\right) .
\end{equation*}

We also have%
\begin{equation*}
\inf_{z\in J}H_{\sigma }I_{\pm }\left( z\right) =\inf_{z\in J}\int_{I_{\pm }}%
\frac{1}{y-z}d\sigma \left( y\right) \geq \frac{\left\vert I_{\pm
}\right\vert _{\sigma }}{2\limfunc{dist}\left( J,I\right) },
\end{equation*}%
and so altogether we obtain%
\begin{eqnarray*}
&&\left\vert \bigtriangleup _{J}^{\omega }H_{\sigma }\mathbf{1}_{I_{\pm
}}\left( x\right) \right\vert \leq \left( \int_{I_{\pm }}\frac{1}{\left(
y-c_{J}\right) ^{2}}d\sigma \left( y\right) \right) 2^{-t}\ell \left(
I\right) \mathbf{1}_{J}\left( x\right) \\
&\leq &\frac{2}{\limfunc{dist}\left( J,I\right) ^{2}}\left\vert I_{\pm
}\right\vert _{\sigma }2^{-t}\ell \left( I\right) \mathbf{1}_{J}\left(
x\right) \leq 2^{2-t}\frac{\ell \left( I\right) }{\limfunc{dist}\left(
J,I\right) }H_{\sigma }I_{\pm }\left( x\right) \mathbf{1}_{J}\left( x\right)
.
\end{eqnarray*}%
We conclude that for each fixed $I\in \mathcal{D}$,%
\begin{eqnarray*}
&&\left\vert \sum_{\substack{ J\in \mathcal{D}:\ \ell \left( J\right)
=2^{-t}\ell \left( I\right)  \\ d\left( J,I\right) >\ell \left( I\right) }}%
\int_{\mathbb{R}}\left( H_{\sigma }\bigtriangleup _{I}^{\sigma }f\right)
\bigtriangleup _{J}^{\omega }gd\omega \right\vert \\
&=&\left\vert \sum_{+,-}\left( E_{\pm }^{\sigma }\bigtriangleup _{I}^{\sigma
}f\right) \int_{\mathbb{R}}\sum_{\substack{ J\in \mathcal{D}:\ \ell \left(
J\right) =2^{-t}\ell \left( I\right)  \\ d\left( J,I\right) >\ell \left(
I\right) }}\left( \bigtriangleup _{J}^{\omega }H_{\sigma }\mathbf{1}_{I_{\pm
}}\right) \left( x\right) \bigtriangleup _{J}^{\omega }g\left( x\right)
d\omega \left( x\right) \right\vert \\
&\leq &2^{2-t}\sum_{+,-}\int_{\mathbb{R}}\sum_{\substack{ J\in \mathcal{D}:\
\ell \left( J\right) =2^{-t}\ell \left( I\right)  \\ d\left( J,I\right)
>\ell \left( I\right) }}\left( \left\vert E_{\pm }^{\sigma }\bigtriangleup
_{I}^{\sigma }f\right\vert \left\vert H_{\sigma }I_{\pm }\left( x\right)
\right\vert \mathbf{1}_{J}\left( x\right) \right) \left( \frac{\ell \left(
I\right) }{\limfunc{dist}\left( J,I\right) }\left\vert \bigtriangleup
_{J}^{\omega }g\left( x\right) \right\vert \right) d\omega \left( x\right)
\end{eqnarray*}%
and so%
\begin{eqnarray*}
&&\left\vert \sum_{I\in \mathcal{D}}\sum_{\substack{ J\in \mathcal{D}:\ \ell
\left( J\right) =2^{-t}\ell \left( I\right)  \\ d\left( J,I\right) >\ell
\left( I\right) }}\int_{\mathbb{R}}\left( H_{\sigma }\bigtriangleup
_{I}^{\sigma }f\right) \bigtriangleup _{J}^{\omega }gd\omega \right\vert \\
&\leq &2^{2-t}\sum_{+,-}\int_{\mathbb{R}}\sqrt{\sum_{I\in \mathcal{D}}\sum 
_{\substack{ J\in \mathcal{D}:\ \ell \left( J\right) =2^{-t}\ell \left(
I\right)  \\ d\left( J,I\right) >\ell \left( I\right) }}\left\vert E_{\pm
}^{\sigma }\bigtriangleup _{I}^{\sigma }f\right\vert ^{2}\left\vert
H_{\sigma }I_{\pm }\left( x\right) \right\vert ^{2}\mathbf{1}_{J}\left(
x\right) } \\
&&\ \ \ \ \ \ \ \ \ \ \ \ \ \ \ \ \ \ \ \ \ \ \ \ \ \times \sqrt{\sum_{I\in 
\mathcal{D}}\sum_{\substack{ J\in \mathcal{D}:\ \ell \left( J\right)
=2^{-t}\ell \left( I\right)  \\ d\left( J,I\right) >\ell \left( I\right) }}%
\left( \frac{\ell \left( I\right) }{\limfunc{dist}\left( J,I\right) }\right)
^{2}\left\vert \bigtriangleup _{J}^{\omega }g\left( x\right) \right\vert ^{2}%
}d\omega \left( x\right) \\
&\leq &2^{2-t}\sum_{+,-}\left\Vert \sqrt{\sum_{I\in \mathcal{D}}\sum 
_{\substack{ J\in \mathcal{D}:\ \ell \left( J\right) =2^{-t}\ell \left(
I\right)  \\ d\left( J,I\right) >\ell \left( I\right) }}\left\vert E_{\pm
}^{\sigma }\bigtriangleup _{I}^{\sigma }f\right\vert ^{2}\left\vert
H_{\sigma }I_{\pm }\left( x\right) \right\vert ^{2}\mathbf{1}_{J}\left(
x\right) }\right\Vert _{L^{p}\left( \omega \right) } \\
&&\ \ \ \ \ \ \ \ \ \ \ \ \ \ \ \ \ \ \ \ \ \ \ \ \ \times \left\Vert \sqrt{%
\sum_{I\in \mathcal{D}}\sum_{\substack{ J\in \mathcal{D}:\ \ell \left(
J\right) =2^{-t}\ell \left( I\right)  \\ d\left( J,I\right) >\ell \left(
I\right) }}\left( \frac{\ell \left( I\right) }{\limfunc{dist}\left(
J,I\right) }\right) ^{2}\left\vert \bigtriangleup _{J}^{\omega }g\left(
x\right) \right\vert ^{2}}\right\Vert _{L^{p^{\prime }}\left( \omega \right)
},
\end{eqnarray*}%
where the first norm satisifies%
\begin{eqnarray*}
&&\left\Vert \sqrt{\sum_{I\in \mathcal{D}}\sum_{\substack{ J\in \mathcal{D}%
:\ \ell \left( J\right) =2^{-t}\ell \left( I\right)  \\ d\left( J,I\right)
>\ell \left( I\right) }}\left\vert E_{\pm }^{\sigma }\bigtriangleup
_{I}^{\sigma }f\right\vert ^{2}\left\vert H_{\sigma }I_{\pm }\right\vert ^{2}%
\mathbf{1}_{J}}\right\Vert _{L^{p}\left( \omega \right) }\lesssim \left\Vert 
\sqrt{\sum_{I\in \mathcal{D}}\left\vert E_{\pm }^{\sigma }\bigtriangleup
_{I}^{\sigma }f\right\vert ^{2}\left\vert H_{\sigma }I_{\pm }\right\vert ^{2}%
}\right\Vert _{L^{p}\left( \omega \right) } \\
&\lesssim &\mathfrak{T}_{H,p}^{\ell ^{2},\limfunc{glob}}\left( \sigma
,\omega \right) \left\Vert \sqrt{\sum_{I\in \mathcal{D}}\left\vert E_{\pm
}^{\sigma }\bigtriangleup _{I}^{\sigma }f\right\vert ^{2}\mathbf{1}_{I_{\pm
}}}\right\Vert _{L^{p}\left( \sigma \right) }\leq \mathfrak{T}_{H,p}^{\ell
^{2},\limfunc{glob}}\left( \sigma ,\omega \right) \left\Vert \sqrt{%
\sum_{I\in \mathcal{D}}\left\vert \bigtriangleup _{I}^{\sigma }f\right\vert
^{2}}\right\Vert _{L^{p}\left( \sigma \right) }\lesssim \mathfrak{T}%
_{H,p}^{\ell ^{2},\limfunc{glob}}\left( \sigma ,\omega \right) \left\Vert
f\right\Vert _{L^{p}\left( \sigma \right) },
\end{eqnarray*}%
(note that we choose either $+$ throughout or $-$ throughout) and the second
norm satisfies%
\begin{equation*}
\left\Vert \sqrt{\sum_{I\in \mathcal{D}}\sum_{\substack{ J\in \mathcal{D}:\
\ell \left( J\right) =2^{-t}\ell \left( I\right)  \\ d\left( J,I\right)
>\ell \left( I\right) }}\left( \frac{\ell \left( I\right) }{\limfunc{dist}%
\left( J,I\right) }\right) ^{2}\left\vert \bigtriangleup _{J}^{\omega
}g\right\vert ^{2}}\right\Vert _{L^{p^{\prime }}\left( \omega \right)
}\lesssim \left\Vert \sqrt{\sum_{J\in \mathcal{D}}\left\vert \bigtriangleup
_{J}^{\omega }g\right\vert ^{2}}\right\Vert _{L^{p^{\prime }}\left( \omega
\right) }\lesssim \left\Vert g\right\Vert _{L^{p^{\prime }}\left( \omega
\right) }.
\end{equation*}%
The square function inequalities in Theorem \ref{square thm} were used in
both estimates above.

\bigskip

\textbf{The mid range case}: Here we prove that the mid-range form $\mathcal{%
A}^{\limfunc{mid}}\left( f,g\right) $ can be bounded by the quadratic offset
Muckenhoupt characteristic $A_{p}^{\ell ^{2},\limfunc{offset}}$.

\bigskip

Let%
\begin{equation*}
\mathcal{P}\equiv \left\{ \left( I,J\right) \in \mathcal{D}\times \mathcal{D}%
:J\text{ is good},\ \ell \left( J\right) \leq 2^{-\tau }\ell \left( I\right)
,\text{ }J\subset 3I\setminus I\right\} .
\end{equation*}%
Now we pigeonhole the lengths of $I$ and $J$ and the distance between them
by defining%
\begin{equation*}
\mathcal{P}_{N,m}^{t}\equiv \left\{ \left( I,J\right) \in \mathcal{D}\times 
\mathcal{D}:J\text{ is }\func{good},\ \ell \left( I\right) =2^{N},\ \ell
\left( J\right) =2^{N-t},\text{ }J\subset 3I\setminus I,\ 2^{N-m-1}\leq 
\limfunc{dist}\left( I,J\right) \leq 2^{N-m}\right\} .
\end{equation*}%
Note that the closest a good cube $J$ can come to $I$ is determined by the
goodness inequality, which gives this bound for $2^{N-m}\geq \limfunc{dist}%
\left( I,J\right) $: 
\begin{eqnarray*}
&&2^{N-m}\geq \frac{1}{2}\ell \left( I\right) ^{1-\varepsilon }\ell \left(
J\right) ^{\varepsilon }=\frac{1}{2}2^{N\left( 1-\varepsilon \right)
}2^{\left( N-t\right) \varepsilon }=\frac{1}{2}2^{N-\varepsilon t}; \\
&&\text{which implies }0\leq m\leq \varepsilon t,
\end{eqnarray*}%
where the last inequality holds because we are in the case of the mid-range
term.

Now we use $\mathsf{Q}_{I,t,N,m}^{\omega }\equiv \sum_{J\in \mathcal{D}:\
\left( I,J\right) \in \mathcal{P}_{N,m}^{t}\text{\ }}\bigtriangleup
_{J}^{\omega }$, and apply Cauchy-Schwarz in $I$ with $J\subset 3I\setminus
I $\ to get%
\begin{eqnarray*}
&&\left\vert \sum_{N\in \mathbb{Z}}\sum_{t\in \mathbb{N}}\sum_{m=0}^{%
\varepsilon t}\dsum\limits_{\left( I,J\right) \in \mathcal{P}%
_{N,m}^{t}}\left\langle H_{\sigma }\bigtriangleup _{I}^{\sigma
}f,\bigtriangleup _{J}^{\omega }g\right\rangle _{\omega }\right\vert \\
&\leq &\sum_{t\in \mathbb{N}}\sum_{m=0}^{\varepsilon t}\left\vert \int_{%
\mathbb{R}}\sum_{N\in \mathbb{Z}}\dsum\limits_{I\in \mathcal{D}_{N}}\mathsf{Q%
}_{I,t,N,m}^{\omega }H_{\sigma }\bigtriangleup _{I}^{\sigma }f\left(
x\right) \ \mathsf{Q}_{I,t,N,m}^{\omega }g\left( x\right) d\omega \left(
x\right) \right\vert \\
&\leq &\sum_{t\in \mathbb{N}}\sum_{m=0}^{\varepsilon t}\int_{\mathbb{R}%
}\left( \sum_{N\in \mathbb{Z}}\dsum\limits_{I\in \mathcal{D}_{N}}\left\vert 
\mathsf{Q}_{I,t,N,m}^{\omega }H_{\sigma }\bigtriangleup _{I}^{\sigma
}f\left( x\right) \right\vert ^{2}\right) ^{\frac{1}{2}}\left( \sum_{N\in 
\mathbb{Z}}\dsum\limits_{I\in \mathcal{D}_{N}}\left\vert \mathsf{Q}%
_{I,t,N,m}^{\omega }g\left( x\right) \right\vert ^{2}\right) ^{\frac{1}{2}%
}d\omega \left( x\right) \\
&\lesssim &\sum_{t\in \mathbb{N}}\sum_{m=0}^{\varepsilon t}\left\Vert \left(
\sum_{N\in \mathbb{Z}}\dsum\limits_{I\in \mathcal{D}_{N}}\left\vert \mathsf{Q%
}_{I,t,N,m}^{\omega }H_{\sigma }\bigtriangleup _{I}^{\sigma }f\left(
x\right) \right\vert ^{2}\right) ^{\frac{1}{2}}\right\Vert _{L^{p}\left(
\omega \right) }\left\Vert \left( \sum_{N\in \mathbb{Z}}\dsum\limits_{I\in 
\mathcal{D}_{N}}\left\vert \mathsf{Q}_{I,t,N,m}^{\omega }g\left( x\right)
\right\vert ^{2}\right) ^{\frac{1}{2}}\right\Vert _{L^{p^{\prime }}\left(
\omega \right) },
\end{eqnarray*}%
where the second factor is at most $C\left\Vert g\right\Vert _{L^{p^{\prime
}}\left( \omega \right) }$ by the pairwise disjoint Haar supports of the
projections $\mathsf{Q}_{I,t,N,m}^{\omega }$.

Now recall that for fixed $I$, the intervals $J$ such that $\left(
I,J\right) \in \mathcal{P}_{N,m}^{t}$ satisfy%
\begin{equation*}
\ell \left( I\right) =2^{N},\ \ell \left( J\right) =2^{N-t},\text{ }J\subset
3I\setminus I,\ 2^{N-m-1}\leq \limfunc{dist}\left( I,J\right) \leq 2^{N-m},
\end{equation*}%
and so for $y\in I$ we have $\left\vert y-c_{J}\right\vert \geq 2^{N-m}$ and
so%
\begin{equation*}
\mathrm{P}\left( J,\left\vert \bigtriangleup _{I}^{\sigma }f\right\vert
\sigma \right) =\int_{I}\frac{\ell \left( J\right) }{\left( \ell \left(
J\right) +\left\vert y-c_{J}\right\vert \right) ^{2}}\left\vert
\bigtriangleup _{I}^{\sigma }f\left( y\right) \right\vert d\sigma \left(
y\right) \lesssim 2^{N-t-2\left( N-m\right) }\int_{I}\left\vert
\bigtriangleup _{I}^{\sigma }f\left( y\right) \right\vert d\sigma \left(
y\right) ,
\end{equation*}%
and we obtain%
\begin{equation*}
\left\vert \mathsf{Q}_{I,t,N,m}^{\omega }H_{\sigma }\bigtriangleup
_{I}^{\sigma }f\left( x\right) \right\vert \lesssim \sum_{J\in \mathcal{D}:\
\left( I,J\right) \in \mathcal{P}_{N,m}^{t}}\mathrm{P}\left( J,\left\vert
\bigtriangleup _{I}^{\sigma }f\right\vert \sigma \right) \mathbf{1}%
_{J}\left( x\right) \lesssim 2^{-t+2m}\left( \int_{I}\left\vert
\bigtriangleup _{I}^{\sigma }f\left( y\right) \right\vert d\sigma \left(
y\right) \right) \ \mathbf{1}_{3I\setminus I}\left( x\right) .
\end{equation*}%
Thus the first factor satisfies,%
\begin{eqnarray*}
&&\left\Vert \left( \sum_{N\in \mathbb{Z}}\dsum\limits_{I\in \mathcal{D}%
_{N}}\left\vert \mathsf{Q}_{I,t,N,m}^{\omega }H_{\sigma }\left(
\bigtriangleup _{I}^{\sigma }f\right) \left( x\right) \right\vert
^{2}\right) ^{\frac{1}{2}}\right\Vert _{L^{p}\left( \omega \right) } \\
&\lesssim &2^{-t+2m}\left\Vert \left( \sum_{N\in \mathbb{Z}%
}\dsum\limits_{I\in \mathcal{D}_{N}}\left( \int_{I_{\limfunc{left}%
}}\left\vert \bigtriangleup _{I}^{\sigma }f\right\vert d\sigma +\int_{I_{%
\limfunc{right}}}\left\vert \bigtriangleup _{I}^{\sigma }f\right\vert
d\sigma \right) ^{2}\right) ^{\frac{1}{2}}\mathbf{1}_{3I\setminus I}\left(
x\right) \right\Vert _{L^{p}\left( \omega \right) } \\
&\lesssim &2^{-t+2m}\left\Vert \left( \sum_{N\in \mathbb{Z}%
}\dsum\limits_{I\in \mathcal{D}_{N}}\left( \frac{\left\vert I_{\limfunc{left}%
}\right\vert _{\sigma }}{\left\vert I\right\vert }\right) ^{2}\left( \frac{1%
}{\left\vert I_{\limfunc{left}}\right\vert _{\sigma }}\int_{I_{\limfunc{left}%
}}\left\vert \bigtriangleup _{I}^{\sigma }f\right\vert d\sigma \right)
^{2}\right) ^{\frac{1}{2}}\mathbf{1}_{3I\setminus I}\left( x\right)
\right\Vert _{L^{p}\left( \omega \right) } \\
&&+2^{-t+2m}\left\Vert \left( \sum_{N\in \mathbb{Z}}\dsum\limits_{I\in 
\mathcal{D}_{N}}\left( \frac{\left\vert I_{\limfunc{right}}\right\vert
_{\sigma }}{\left\vert I\right\vert }\right) ^{2}\left( \frac{1}{\left\vert
I_{\limfunc{right}}\right\vert _{\sigma }}\int_{I_{\limfunc{right}%
}}\left\vert \bigtriangleup _{I}^{\sigma }f\right\vert d\sigma \right)
^{2}\right) ^{\frac{1}{2}}\mathbf{1}_{3I\setminus I}\left( x\right)
\right\Vert _{L^{p}\left( \omega \right) }.
\end{eqnarray*}%
By the quadratic offset Muckenhoupt condition, the first term on the right
hand side involving $I_{\limfunc{left}}$ is at most%
\begin{eqnarray*}
&&A_{p}^{\ell ^{2},\limfunc{offset}}\left( \sigma ,\omega \right)
2^{-t+2m}\left\Vert \left( \sum_{N\in \mathbb{Z}}\dsum\limits_{I\in \mathcal{%
D}_{N}}\left( \frac{1}{\left\vert I_{\limfunc{left}}\right\vert _{\sigma }}%
\int_{I_{\limfunc{left}}}\left\vert \bigtriangleup _{I}^{\sigma
}f\right\vert d\sigma \right) ^{2}\mathbf{1}_{I_{\limfunc{left}}}\left(
x\right) \right) ^{\frac{1}{2}}\right\Vert _{L^{p}\left( \sigma \right) } \\
&\lesssim &A_{p}^{\ell ^{2},\limfunc{offset}}\left( \sigma ,\omega \right)
2^{-t+2m}\left\Vert \left( \sum_{N\in \mathbb{Z}}\dsum\limits_{I\in \mathcal{%
D}_{N}}\left\vert \bigtriangleup _{I_{\limfunc{left}}}^{\sigma }f\left(
x\right) \right\vert ^{2}\mathbf{1}_{I_{\limfunc{left}}}\left( x\right)
\right) ^{\frac{1}{2}}\right\Vert _{L^{p}\left( \sigma \right) }\lesssim
A_{p}^{\ell ^{2},\limfunc{offset}}\left( \sigma ,\omega \right)
2^{-t+2m}\left\Vert f\right\Vert _{L^{p}(\sigma )},
\end{eqnarray*}%
where we have used that $\left\vert \bigtriangleup _{I}^{\sigma
}f\right\vert $ is constant on $I_{\limfunc{left}}$, followed by the square
function bound. Similarly for the second term on the right hand side
involving $I_{\limfunc{right}}$.

Summing in $t$ and $m$ we then have%
\begin{eqnarray*}
&&\left\vert \mathcal{A}^{\limfunc{mid}}\left( f,g\right) \right\vert
=\left\vert \sum_{N\in \mathbb{Z}}\sum_{t\in \mathbb{N}}\sum_{m=0}^{%
\varepsilon t}\dsum\limits_{\left( I,J\right) \in \mathcal{P}%
_{N,m}^{t}}\left\langle H_{\sigma }\left( \bigtriangleup _{I}^{\sigma
}f\right) ,\bigtriangleup _{J}^{\omega }g\right\rangle _{\omega }\right\vert
\leq \sum_{t\in \mathbb{N}}\sum_{m=0}^{\varepsilon t}\left\vert \sum_{N\in 
\mathbb{Z}}\dsum\limits_{\left( I,J\right) \in \mathcal{P}%
_{N,m}^{t}}\left\langle H_{\sigma }\left( \bigtriangleup _{I}^{\sigma
}f\right) ,\bigtriangleup _{J}^{\omega }g\right\rangle _{\omega }\right\vert
\\
&\lesssim &\sum_{t\in \mathbb{N}}\sum_{m=0}^{\varepsilon t}A_{p}^{\ell ^{2},%
\limfunc{offset}}\left( \sigma ,\omega \right) 2^{-t+2m}\left\Vert
f\right\Vert _{L^{p}(\sigma )}\left\Vert g\right\Vert _{L^{p^{\prime
}}\left( \omega \right) }\lesssim A_{p}^{\ell ^{2},\limfunc{offset}}\left(
\sigma ,\omega \right) \left\Vert f\right\Vert _{L^{p}(\sigma )}\left\Vert
g\right\Vert _{L^{p^{\prime }}\left( \omega \right) },
\end{eqnarray*}%
since $\sum_{t\in \mathbb{N}}\sum_{m=0}^{\varepsilon t}2^{-t+2m}\leq
\sum_{t\in \mathbb{N}}\left( 1+\varepsilon t\right) 2^{-t\left(
1-2\varepsilon \right) }\leq C_{\varepsilon }$ if $0<\varepsilon <\frac{1}{2}
$.

This completes the proof of both (\ref{routine'}) and (\ref{routine''}).
\end{proof}

\section{Decomposition of the below form}

Let%
\begin{equation*}
\mathcal{P}_{\func{below}}\equiv \left\{ \left( I,J\right) \in \mathcal{D}_{%
\func{good}}^{\limfunc{child}}\times \mathcal{D}_{\func{good}}^{\limfunc{%
child}}:J\subset _{\tau }I\right\}
\end{equation*}%
be the set of pairs of $\limfunc{child}$-$\func{good}$ dyadic intervals $%
\left( I,J\right) $ with $J$ at least $\tau $ levels below and inside $I$.
We begin by splitting the $\func{below}$ form into home and neighbour forms,
where $\theta K$ denotes the dyadic sibling of $K\in \mathcal{D}$, and $%
I_{J} $ denotes the child of $I$ that contains $J$,%
\begin{eqnarray*}
\mathsf{B}_{\func{below}}\left( f,g\right) &=&\sum_{\left( I,J\right) \in 
\mathcal{P}_{\func{below}}}\left\langle H_{\sigma }\left( \mathbf{1}%
_{I_{J}}\bigtriangleup _{I}^{\sigma }f\right) ,\bigtriangleup _{J}^{\omega
}g\right\rangle _{\omega }+\sum_{\left( I,J\right) \in \mathcal{P}_{\func{%
below}}}\left\langle H_{\sigma }\left( \mathbf{1}_{\theta
I_{J}}\bigtriangleup _{I}^{\sigma }f\right) ,\bigtriangleup _{J}^{\omega
}g\right\rangle _{\omega } \\
&\equiv &\mathsf{B}_{\func{home}}\left( f,g\right) +\mathsf{B}_{\func{neigh}%
}\left( f,g\right) .
\end{eqnarray*}

\subsection{Neighbour form}

The neighbour form is controlled by the quadratic offset Muckenhoupt
condition using Lemma \ref{Energy Lemma} and the fact that the intervals $J$
are good, namely we claim%
\begin{equation*}
\left\vert \mathsf{B}_{\func{neigh}}\left( f,g\right) \right\vert \leq
C_{\varepsilon }A_{p}^{\ell ^{2},\limfunc{offset}}\left( \sigma ,\omega
\right) \left\Vert f\right\Vert _{L^{p}(\sigma )}\left\Vert g\right\Vert
_{L^{p^{\prime }}(\omega )},\ \ \ \ \ 1<p<\infty .
\end{equation*}%
We have 
\begin{equation*}
\left\langle H_{\sigma }\left( \mathbf{1}_{\theta \left( I_{J}\right)
}\bigtriangleup _{I}^{\sigma }f\right) ,\bigtriangleup _{J}^{\omega
}g\right\rangle _{\omega }=E_{\theta \left( I_{J}\right) }^{\sigma }\Delta
_{I}^{\sigma }f\cdot \left\langle H_{\sigma }\mathbf{1}_{\theta \left(
I_{J}\right) },\bigtriangleup _{J}^{\omega }g\right\rangle _{\omega },
\end{equation*}%
and thus we can write%
\begin{equation}
\mathsf{B}_{\func{neigh}}\left( f,g\right) =\sum_{I,J\in \mathcal{D}_{\func{%
good}}\text{ and }J\subset _{\tau }I}\left( E_{\theta \left( I_{J}\right)
}^{\sigma }\Delta _{I}^{\sigma }f\right) \ \left\langle H_{\sigma }\mathbf{1}%
_{\theta \left( I_{J}\right) },\Delta _{J}^{\omega }g\right\rangle _{\omega
}\ .  \label{neighbour term}
\end{equation}

To see the claim, momentarily fix an integer $s\geq \tau $. Now we
pigeonhole pairs $\left( I,J\right) $ of intervals by requiring $J\in 
\mathfrak{C}_{\mathcal{D}}^{\left( s\right) }\left( I\right) $, i.e. $%
J\subset I$ and $\ell \left( J\right) =2^{-s}\ell \left( I\right) $, and we
further separate the two cases where $I_{J}=I_{\pm }$, the right and left
children of $I$, so that we have 
\begin{equation*}
\mathsf{B}_{\func{neigh}}\left( f,g\right) =\sum_{I}\sum_{+,-}\sum_{s=\tau
}^{\infty }=\sum_{+,-}\sum_{s=\tau }^{\infty }\sum_{J\in \mathfrak{C}_{%
\mathcal{D}}^{\left( s-1\right) }\left( I_{\pm }\right) }\langle H_{\sigma
}\left( \mathbf{1}_{I_{\mp }}\Delta _{I}^{\sigma }f\right) ,\Delta
_{J}^{\omega }g\rangle _{\omega }.
\end{equation*}%
Observe that%
\begin{align*}
\left\vert A_{\pm }\left( I,s\right) \right\vert & \equiv \left\vert
\sum_{J\in \mathfrak{C}_{\mathcal{D}}^{\left( s-1\right) }\left( I_{\pm
}\right) }\langle H_{\sigma }\left( \mathbf{1}_{I_{\mp }}\Delta _{I}^{\sigma
}f\right) ,\Delta _{J}^{\omega }g\rangle _{\omega }\right\vert \\
& =\left\vert \sum_{I,J\in \mathcal{D}_{\func{good}}\text{ and }J\in 
\mathfrak{C}_{\mathcal{D}}^{\left( s-1\right) }\left( I_{\pm }\right)
}\left( E_{I_{\mp }}^{\sigma }\Delta _{I}^{\sigma }f\right) \ \left\langle
\Delta _{J}^{\omega }H_{\sigma }\mathbf{1}_{I_{\mp }},\Delta _{J}^{\omega
}g\right\rangle _{\omega }\right\vert \\
& =\left\vert \int_{\mathbb{R}}\sum_{I,J\in \mathcal{D}_{\func{good}}\text{
and }J\in \mathfrak{C}_{\mathcal{D}}^{\left( s-1\right) }\left( I_{\pm
}\right) }\left( E_{I_{\mp }}^{\sigma }\Delta _{I}^{\sigma }f\right) \
\Delta _{J}^{\omega }H_{\sigma }\mathbf{1}_{I_{\mp }}\left( x\right) \
\Delta _{J}^{\omega }g\left( x\right) d\omega \left( x\right) \right\vert .
\end{align*}%
If we now apply $\ell ^{2}$ Cauchy-Schwarz followed by $L^{p}\left( \omega
\right) $ H\"{o}lder, we obtain%
\begin{equation*}
\left\vert A_{\pm }\left( I,s\right) \right\vert \leq \left\Vert \left(
\sum_{I\in \mathcal{D}\text{, }J\in \mathfrak{C}_{\mathcal{D}}^{\left(
s-1\right) }\left( I_{\pm }\right) }\left\vert \left( E_{I_{\mp }}^{\sigma
}\Delta _{I}^{\sigma }f\right) \ \Delta _{J}^{\omega }H_{\sigma }\mathbf{1}%
_{I_{\mp }}\left( x\right) \right\vert ^{2}\right) ^{\frac{1}{2}}\right\Vert
_{L^{p}\left( \omega \right) }\left\Vert \left( \sum_{I\in \mathcal{D}\text{%
, }J\in \mathfrak{C}_{\mathcal{D}}^{\left( s-1\right) }\left( I_{\pm
}\right) }\left\vert \Delta _{J}^{\omega }g\left( x\right) \right\vert
^{2}\right) ^{\frac{1}{2}}\right\Vert _{L^{p^{\prime }}\left( \omega \right)
}.
\end{equation*}%
The second factor is at most $C\left\Vert g\right\Vert _{L^{p^{\prime
}}\left( \omega \right) }$ by the square function inequality.

For the first factor we use the pointwise Monotonicity Lemma \ref{energy
pointwise},%
\begin{equation*}
\left\vert \Delta _{J}^{\omega }H_{\sigma }\mathbf{1}_{I_{\mp }}\left(
x\right) \right\vert \leq 2\mathrm{P}\left( J,\mathbf{1}_{I_{\mp }}\sigma
\right) \mathbf{1}_{J}\left( x\right) ,
\end{equation*}%
and then the Poisson Decay Lemma \ref{Poisson inequality}, to obtain%
\begin{equation*}
\left\vert \Delta _{J}^{\omega }H_{\sigma }\mathbf{1}_{I_{\mp }}\left(
x\right) \right\vert \leq 2\mathrm{P}\left( J,\mathbf{1}_{I_{\mp }}\sigma
\right) \mathbf{1}_{J}\left( x\right) \leq C_{\varepsilon }2^{-\left(
1-2\varepsilon \right) s}\mathrm{P}\left( I,\mathbf{1}_{I_{\mp }}\sigma
\right) \mathbf{1}_{J}\left( x\right) .
\end{equation*}%
Thus the first factor is bounded by%
\begin{eqnarray*}
&&C_{\varepsilon }2^{-\left( 1-2\varepsilon \right) s}\left\Vert \left(
\sum_{I,J\in \mathcal{D}_{\func{good}}\text{ and }J\in \mathfrak{C}_{%
\mathcal{D}}^{\left( s\right) }\left( I_{\pm }\right) }\left\vert E_{I_{\mp
}}^{\sigma }\Delta _{I}^{\sigma }f\right\vert ^{2}\mathrm{P}\left( I,\mathbf{%
1}_{I_{\mp }}\sigma \right) ^{2}\mathbf{1}_{J}\left( x\right) \right) ^{%
\frac{1}{2}}\right\Vert _{L^{p}\left( \omega \right) } \\
&\leq &C_{\varepsilon }2^{-\left( 1-2\varepsilon \right) s}\left\Vert \left(
\sum_{I\in \mathcal{D}_{\func{good}}}\left\vert E_{I_{\mp }}^{\sigma }\Delta
_{I}^{\sigma }f\right\vert ^{2}\left( \frac{\left\vert I_{\mp }\right\vert
_{\sigma }}{\left\vert I\right\vert }\right) ^{2}\mathbf{1}_{I_{\pm }}\left(
x\right) \right) ^{\frac{1}{2}}\right\Vert _{L^{p}\left( \omega \right) },
\end{eqnarray*}%
upon using $\mathrm{P}\left( I,\mathbf{1}_{I_{\mp }}\sigma \right) \lesssim 
\frac{\left\vert I_{\mp }\right\vert _{\sigma }}{\left\vert I\right\vert }$
and $\sum_{J\in \mathfrak{C}_{\mathcal{D}}^{\left( s-1\right) }\left( I_{\pm
}\right) }\mathbf{1}_{J}\left( x\right) =\mathbf{1}_{I_{\pm }}\left(
x\right) $.

Now we conclude from the quadratic offset Muckenhoupt condition that the
above term is at most 
\begin{eqnarray*}
&&C_{\varepsilon }2^{-\left( 1-2\varepsilon \right) s}\left\Vert \left(
\sum_{I\in \mathcal{D}_{\func{good}}}\left\vert E_{I_{\mp }}^{\sigma }\Delta
_{I}^{\sigma }f\right\vert ^{2}\left( \frac{\left\vert I_{\mp }\right\vert
_{\sigma }}{\left\vert I\right\vert }\right) ^{2}\mathbf{1}_{I_{\pm
}}\right) ^{\frac{1}{2}}\right\Vert _{L^{p}\left( \omega \right) } \\
&\leq &C_{\varepsilon }A_{p}^{\ell ^{2},\limfunc{offset}}\left( \sigma
,\omega \right) 2^{-\left( 1-2\varepsilon \right) s}\left\Vert \left(
\sum_{I\in \mathcal{D}_{\func{good}}}\left\vert E_{I_{\mp }}^{\sigma }\Delta
_{I}^{\sigma }f\right\vert ^{2}\mathbf{1}_{I_{\pm }}\right) ^{\frac{1}{2}%
}\right\Vert _{L^{p}\left( \sigma \right) } \\
&=&C_{\varepsilon }A_{p}^{\ell ^{2},\limfunc{offset}}\left( \sigma ,\omega
\right) 2^{-\left( 1-2\varepsilon \right) s}\left\Vert \left( \sum_{I\in 
\mathcal{D}_{\func{good}}}\left\vert \Delta _{I}^{\sigma }f\right\vert ^{2}%
\mathbf{1}_{I_{\pm }}\right) ^{\frac{1}{2}}\right\Vert _{L^{p}\left( \sigma
\right) } \\
&\leq &C_{\varepsilon }2^{-\left( 1-2\varepsilon \right) s}A_{p}^{\ell ^{2},%
\limfunc{offset}}\left( \sigma ,\omega \right) \left\Vert f\right\Vert
_{L^{p}\left( \sigma \right) },
\end{eqnarray*}%
where we have used that $\left\vert \Delta _{I}^{\sigma }f\right\vert $ is
constant on each child of $I$, followed by the square function estimate.
Then we can sum in $\pm $ and $s$ to obtain for $0<\varepsilon <\frac{1}{2}$,%
\begin{eqnarray*}
\left\vert \mathsf{B}_{\func{neigh}}\left( f,g\right) \right\vert &\leq
&\sum_{+,-}\sum_{s=\tau }^{\infty }\left\vert A_{\pm }\left( I,s\right)
\right\vert \leq C_{\varepsilon }A_{p}^{\ell ^{2},\limfunc{offset}}\left(
\sigma ,\omega \right) \left( \sum_{+,-}\sum_{s=\tau }^{\infty }2^{-\left(
1-2\varepsilon \right) s}\right) \left\Vert f\right\Vert _{L^{p}\left(
\sigma \right) }\left\Vert g\right\Vert _{L^{p^{\prime }}\left( \omega
\right) } \\
&\leq &C_{\varepsilon }A_{p}^{\ell ^{2},\limfunc{offset}}\left( \sigma
,\omega \right) \left\Vert f\right\Vert _{L^{p}\left( \sigma \right)
}\left\Vert g\right\Vert _{L^{p^{\prime }}\left( \omega \right) }\ .
\end{eqnarray*}

\section{Decomposition of the home form by CZ and $p$-energy coronas}

In order to control the home form, we must pigeonhole the pairs of intervals 
$\left( I,J\right) \in \mathcal{P}_{\func{below}}$ into a collection of
pairwise disjoint corona `boxes' in which both $\sigma $-averages of $f$,
and Poisson-Energies of the measure pair $\left( \sigma ,\omega \right) $,
are controlled. Then we split the home form into two forms according to this
decomposition, which we call the diagonal and far forms. But first we need
to construct the Calder\'{o}n-Zygmund / Poisson-Energy corona decomposition.

Fix $\Gamma >1$ and a large dyadic interval $T$. Define a sequence of
stopping times $\left\{ \mathcal{F}_{n}\right\} _{n=0}^{\infty }$ depending
on $T$, $\sigma $ and $\omega $ recursively as follows. Let $\mathcal{F}%
_{0}=\left\{ T\right\} $. Given $\mathcal{F}_{n}$, define $\mathcal{F}_{n+1}$
to consist of the \textbf{maximal} $\func{good}$ intervals $I^{\prime }$ for
which there is $I\in \mathcal{F}_{n}$ with $I^{\prime }\subset I$ and%
\begin{eqnarray}
&&\text{\textbf{either} }\left( \frac{\mathrm{P}\left( I^{\prime },\mathbf{1}%
_{I\setminus I^{\prime }}\sigma \right) }{\ell \left( I^{\prime }\right) }%
\right) ^{p}\mathsf{E}_{p}\left( J,\omega \right) ^{p}\frac{\left\vert
I^{\prime }\right\vert _{\omega }}{\left\vert I^{\prime }\right\vert
_{\sigma }}>\Gamma ,  \label{energy stop crit} \\
&&\text{\textbf{or} }\frac{1}{\left\vert I^{\prime }\right\vert _{\sigma }}%
\int_{I^{\prime }}\left\vert \mathsf{P}_{\mathcal{D}\left[ I\right]
}^{\sigma }f\right\vert d\sigma >4\frac{1}{\left\vert I\right\vert _{\sigma }%
}\int_{I}\left\vert \mathsf{P}_{\mathcal{D}\left[ I\right] }^{\sigma
}f\right\vert d\sigma ,  \notag
\end{eqnarray}%
where the $p$-energy $\mathsf{E}_{p}\left( J,\omega \right) $ is defined by 
\begin{equation}
\mathsf{E}_{p}\left( J,\omega \right) \equiv \left( \frac{1}{\left\vert
J\right\vert _{\omega }}\int_{J}\left\vert x-\frac{1}{\left\vert
J\right\vert _{\omega }}\int_{J}zd\omega \left( z\right) \right\vert
^{p}d\omega \left( x\right) \right) ^{\frac{1}{p}}\approx \left( \frac{1}{%
\left\vert J\right\vert _{\omega }}\int_{J}\left( \sum_{J^{\prime }\subset
J}\left\vert \bigtriangleup _{J^{\prime }}^{\omega }Z\left( x\right)
\right\vert ^{2}\right) ^{\frac{p}{2}}d\omega \left( x\right) \right) ^{%
\frac{1}{p}},  \label{def p energy}
\end{equation}%
where the equivalence follows from the square function Theorem \ref{square
thm} applied to $\mathbf{1}_{J}\left( Z-\frac{1}{\left\vert J\right\vert
_{\omega }}\int_{J}zd\omega \left( z\right) \right) $ - the $p$-energy $%
\mathsf{E}_{p}\left( J,\omega \right) $ is an $L^{p}$ analogue of the energy
functional introduced in \cite{LaSaUr2}. Also, $\mathsf{P}_{\mathcal{D}\left[
I\right] }^{\sigma }=\sum_{K\in \mathcal{D}:\ K\subset I}\bigtriangleup
_{K}^{\sigma }$, and $Z$ denotes the identity function on $\mathbb{R}$, $%
Z\left( z\right) =z$ for all $z\in \mathbb{R}$, so that 
\begin{equation*}
\bigtriangleup _{J}^{\omega }Z\left( x\right) \equiv \left\langle
Z,h_{J}^{\omega }\right\rangle _{\omega }h_{J}^{\omega }\left( x\right) ,\ \
\ \ \ \text{where }\left\langle Z,h_{J}^{\omega }\right\rangle _{\omega
}=\int_{J}zd\omega \left( z\right) =\int_{J}\left( z-c\right) d\omega \left(
z\right) ,\ \ \ c\in \mathbb{R}.
\end{equation*}%
Set $\mathcal{F}\equiv \bigcup_{n=0}^{\infty }\mathcal{F}_{n}$, which we
refer to as the Calder\'{o}n-Zygmund/Poisson-Energy stopping times for the
dyadic interval $T$, with measures $\sigma $ and $\omega $. Note that $%
\mathcal{F}\subset \mathcal{D}_{\func{good}}$.

\begin{notation}
\label{corona notation}Denote the associated corona with top $F\in \mathcal{F%
}$ by 
\begin{equation*}
\mathcal{C}_{\mathcal{F}}\left( F\right) \equiv \left\{ I\in \mathcal{D}%
:I\subset F\text{ and }I\not\subset F^{\prime }\in \mathcal{F}\text{ for any 
}F^{\prime }\subsetneqq F\right\} ,
\end{equation*}%
and the grandchildren at depth $m\in \mathbb{N}$ of $F$ in the tree $%
\mathcal{F}$ by $\mathfrak{C}_{\mathcal{F}}^{\left( m\right) }\left(
F\right) $, with $\mathfrak{C}_{\mathcal{F}}^{\left( 1\right) }\left(
F\right) $ abbreviated to $\mathfrak{C}_{\mathcal{F}}\left( F\right) $. We
will also denote by $\mathcal{C}_{\mathcal{F}}^{\left( k\right) }\left(
F\right) \equiv \bigcup_{G\in \mathfrak{C}_{\mathcal{F}}^{\left( k\right)
}\left( F\right) }\mathcal{C}_{\mathcal{F}}\left( G\right) $ the union of
all $\mathcal{F}$-coronas at depth $k$ below $F$, and we will denote by $%
\mathcal{C}_{\mathcal{F}}^{\left[ m\right] }\left( F\right) $ (note the use
of square brackets in the exponent) the union of $\mathcal{C}_{\mathcal{F}%
}^{\left( k\right) }\left( F\right) $ for $0\leq k\leq m$. We will
consistantly use calligraphic font $\mathcal{C}$ to denote coronas, and
fraktur font $\mathfrak{C}$ to denote children, and will consistently use
superscripts $\left( m\right) $ with parentheses to denote objects `$m$
levels below', and superscripts $\left[ m\right] $ with brackets to denote
objects `at most $m$ levels below'.
\end{notation}

Finally, we define%
\begin{equation*}
\alpha _{\mathcal{F}}\left( F\right) \equiv \sup_{G\in \mathcal{F}:\
G\supset F}E_{G}^{\sigma }\left\vert \mathsf{P}_{\mathcal{D}\left[ \pi _{%
\mathcal{F}}G\right] }^{\sigma }f\right\vert ,\ \ \ \ \ \text{for }F\in 
\mathcal{F}.
\end{equation*}

The point of introducing the corona decomposition $\mathcal{D}\left[ T\right]
=\overset{\cdot }{\bigcup }_{F\in \mathcal{F}}\mathcal{C}_{\mathcal{F}%
}\left( F\right) $ is that, in each $\func{good}$ corona $\mathcal{C}_{%
\mathcal{F}}^{\func{good}}\left( F\right) \equiv \mathcal{C}_{\mathcal{F}%
}\left( F\right) \cap \mathcal{D}_{\func{good}}$, we obtain control of both
the averages of projections of $f$,%
\begin{equation}
E_{I}^{\sigma }\left\vert \mathsf{P}_{\mathcal{D}\left[ F\right]
}f\right\vert \equiv \frac{1}{\left\vert I\right\vert _{\sigma }}%
\int_{I}\left\vert \mathsf{P}_{\mathcal{D}\left[ F\right] }f\right\vert
d\sigma \leq 4E_{F}^{\sigma }\left\vert f\right\vert ,  \label{en con}
\end{equation}%
by negating the second inequality in (\ref{energy stop crit}), as well as
control of the Stopping-Energy functional,%
\begin{equation}
\mathfrak{X}_{F}\left( \sigma ,\omega \right) ^{p}\equiv \sup_{I\in \mathcal{%
C}_{\mathcal{F}}\left( F\right) \cap \mathcal{D}_{\func{good}}^{\limfunc{%
child}}}\left( \frac{\mathrm{P}\left( I,\mathbf{1}_{F\setminus I}\sigma
\right) }{\ell \left( I\right) }\right) ^{p}\mathsf{E}_{p}\left( I,\omega
\right) ^{p}\frac{\left\vert I\right\vert _{\omega }}{\left\vert
I\right\vert _{\sigma }},  \label{PE char}
\end{equation}
by negating the first inequality in (\ref{energy stop crit}), i.e.%
\begin{equation*}
\frac{E_{I}^{\sigma }\left\vert \mathsf{P}_{F}f\right\vert }{E_{F}^{\sigma
}\left\vert \mathsf{P}_{F}f\right\vert }\leq 4\text{ and }\mathfrak{X}%
_{F;p}\left( \sigma ,\omega \right) ^{p}\leq \Gamma ,\ \ \ \ \ \text{for all 
}I\in \mathcal{C}_{\mathcal{F}}^{\func{good}}\left( F\right) \text{ and }%
F\in \mathcal{F}.
\end{equation*}%
In particular, this inequality shows that the \emph{Stopping-Energy
characteristic }%
\begin{equation*}
\mathfrak{X}_{\mathcal{F};p}\left( \sigma ,\omega \right) \equiv \sup_{F\in 
\mathcal{F}}\mathfrak{X}_{F;p}\left( \sigma ,\omega \right)
\end{equation*}%
of $\sigma $ and $\omega $ with respect to the stopping times $\mathcal{F}$,
is dominated by the parameter $\Gamma $ chosen in (\ref{energy stop crit}).

\subsection{Necessity of the $p$-energy condition}

The proof of the stopping form bound will use the $L^{p}$-analogue of the
Poisson-energy characteristic introduced in \cite[(1.9) on page 3]{LaSaUr2}, 
\begin{equation}
\mathcal{E}_{p}\left( \sigma ,\omega \right) ^{p}\equiv \sup_{I\in \mathcal{D%
}}\sum_{\overset{\cdot }{\bigcup }_{r=1}^{\infty }I_{r}\subset I}\left( 
\frac{\mathrm{P}\left( I_{r},\mathbf{1}_{I\setminus I_{r}}\sigma \right) }{%
\ell \left( I_{r}\right) }\right) ^{p}\mathsf{E}_{p}\left( I_{r},\omega
\right) ^{p}\frac{\left\vert I_{r}\right\vert _{\omega }}{\left\vert
I\right\vert _{\sigma }}\text{ and its dual }\mathcal{E}_{p^{\prime }}\left(
\omega ,\sigma \right) ,  \label{energy char}
\end{equation}%
where the supremum is taken over all pairwise disjoint subdecompositions of
an interval $I$ into dyadic subintervals $I_{r}\in \mathcal{D}\left[ I\right]
$, and where the $p$-energy $\mathsf{E}_{p}\left( J,\omega \right) $ is
defined in (\ref{def p energy}). Now we show the $p$-energy characteristic $%
\mathcal{E}_{p}\left( \sigma ,\omega \right) $ is controlled by the scalar
testing characteristic.

\begin{lemma}
\label{energy cond}For $1<p<\infty $ we have%
\begin{equation*}
\mathfrak{X}_{\mathcal{F},p}\left( \sigma ,\omega \right) \lesssim \mathcal{E%
}_{p}\left( \sigma ,\omega \right) \lesssim \mathfrak{T}_{H,p}\left( \sigma
,\omega \right) .
\end{equation*}
\end{lemma}

\begin{proof}
Let $\left\{ I_{r}\right\} _{r=1}^{\infty }$ be a pairwise disjoint
decomposition of an interval $I$ into subintervals $I_{r}$. We begin with
inequality (\ref{simple reversal}) from Subsection \ref{Sub Rev}, namely
that for $x,y\in I_{r}\subset I$:%
\begin{equation*}
\frac{\mathrm{P}\left( I_{r},\mathbf{1}_{I\setminus I_{r}}\sigma \right) }{%
\ell \left( I_{r}\right) }\left\vert x-y\right\vert \leq 2\left\vert H%
\mathbf{1}_{I\setminus I_{r}}\sigma \left( x\right) -H\mathbf{1}_{I\setminus
I_{r}}\sigma \left( y\right) \right\vert .
\end{equation*}%
Then taking the $p^{th}$ power and integrating in $d\omega \left( x\right) $
and $d\omega \left( y\right) $ on both sides gives%
\begin{eqnarray*}
&&\left( \frac{\mathrm{P}\left( I_{r},\mathbf{1}_{I\setminus I_{r}}\sigma
\right) }{\ell \left( I_{r}\right) }\right)
^{p}\int_{I_{r}}\int_{I_{r}}\left\vert x-y\right\vert ^{p}d\omega \left(
x\right) d\omega \left( y\right) \\
&\leq &2^{p}\int_{I_{r}}\int_{I_{r}}\left\vert H\mathbf{1}_{I\setminus
I_{r}}\sigma \left( x\right) -H\mathbf{1}_{I\setminus I_{r}}\sigma \left(
y\right) \right\vert ^{p}d\omega \left( x\right) d\omega \left( y\right) \\
&\lesssim &\int_{I_{r}}\int_{I_{r}}\left\vert H\mathbf{1}_{I\setminus
I_{r}}\sigma \left( x\right) \right\vert ^{p}d\omega \left( x\right) d\omega
\left( y\right) +\int_{I_{r}}\int_{I_{r}}\left\vert H\mathbf{1}_{I\setminus
I_{r}}\sigma \left( y\right) \right\vert ^{p}d\omega \left( x\right) d\omega
\left( y\right) \\
&\lesssim &\left\vert I_{r}\right\vert _{\omega }\int_{I_{r}}\left\vert H%
\mathbf{1}_{I\setminus I_{r}}\sigma \left( x\right) \right\vert ^{p}d\omega
\left( x\right) .
\end{eqnarray*}%
Using 
\begin{eqnarray*}
\int_{I_{r}}\left\vert x-E_{I_{r}}^{\omega }Z\right\vert ^{p}d\omega \left(
x\right) &=&\int_{I_{r}}\left\vert \frac{1}{\left\vert I_{r}\right\vert
_{\omega }}\int_{I_{r}}\left( x-y\right) d\omega \left( y\right) \right\vert
^{p}d\omega \left( x\right) \\
&\leq &\frac{1}{\left\vert I_{r}\right\vert _{\omega }}\int_{I_{r}}%
\int_{I_{r}}\left\vert x-y\right\vert ^{p}d\omega \left( y\right) d\omega
\left( x\right) ,
\end{eqnarray*}%
we obtain%
\begin{eqnarray*}
&&\left( \frac{\mathrm{P}\left( I_{r},\mathbf{1}_{I\setminus I_{r}}\sigma
\right) }{\ell \left( I_{r}\right) }\right) ^{p}\int_{I_{r}}\left\vert
x-E_{I_{r}}^{\omega }Z\right\vert ^{p}d\omega \left( x\right) \lesssim
\int_{I_{r}}\left\vert H\mathbf{1}_{I\setminus I_{r}}\sigma \left( x\right)
\right\vert ^{p}d\omega \left( x\right) \\
&\lesssim &\int_{I_{r}}\left\vert H\mathbf{1}_{I}\sigma \left( x\right)
\right\vert ^{p}d\omega \left( x\right) +\int_{I_{r}}\left\vert H\mathbf{1}%
_{I_{r}}\sigma \left( x\right) \right\vert ^{p}d\omega \left( x\right) .
\end{eqnarray*}%
Now summing in $r$ yields%
\begin{eqnarray*}
&&\sum_{r=1}^{\infty }\left( \frac{\mathrm{P}\left( I_{r},\mathbf{1}%
_{I\setminus I_{r}}\sigma \right) }{\left\vert I_{r}\right\vert }\right)
^{p}\int_{I_{r}}\left\vert x-E_{I_{r}}^{\omega }Z\right\vert ^{p}d\omega
\left( x\right) \\
&\lesssim &\int_{I}\left\vert H\mathbf{1}_{I}\sigma \left( x\right)
\right\vert ^{p}d\omega \left( x\right) +\mathfrak{T}_{H,p}\left( \sigma
,\omega \right) \sum_{r=1}^{\infty }\left\vert I_{r}\right\vert _{\sigma
}\lesssim \mathfrak{T}_{H,p}\left( \sigma ,\omega \right) \left\vert
I\right\vert _{\sigma }\ .
\end{eqnarray*}%
The first inequality $\mathfrak{X}_{\mathcal{F},p}\left( \sigma ,\omega
\right) \lesssim \mathcal{E}_{p}\left( \sigma ,\omega \right) $ in the
statement of the lemma follows directly from the $p$-energy stopping time
construction.
\end{proof}

For $1<s<\infty $ define%
\begin{eqnarray*}
M_{\omega ,s}^{\func{dy}}h\left( x\right) &\equiv &\sup_{Q\in \mathcal{D}:\
x\in Q}\left( \frac{1}{\left\vert Q\right\vert _{\omega }}\int_{Q}\left\vert
h\right\vert ^{s}d\omega \right) ^{\frac{1}{s}}, \\
\mathsf{E}_{p,s}\left( I,\omega \right) &\equiv &\left( \int_{I}M_{\omega
,s}^{\func{dy}}\left[ \left( Z-E_{Q}^{\omega }Z\right) \mathbf{1}_{I}\right]
^{p}d\omega \right) ^{\frac{1}{p}},
\end{eqnarray*}%
and the \emph{stronger} $p$-energy constant,%
\begin{equation*}
\mathcal{E}_{p,s}\left( \sigma ,\omega \right) ^{p}\equiv \sup_{I\in 
\mathcal{D}}\sum_{\overset{\cdot }{\bigcup }_{r=1}^{\infty }I_{r}\subset
I}\left( \frac{\mathrm{P}\left( I_{r},\mathbf{1}_{I\setminus I_{r}}\sigma
\right) }{\ell \left( I_{r}\right) }\right) ^{p}\mathsf{E}_{p,s}\left(
I_{r},\omega \right) ^{p}\frac{\left\vert I_{r}\right\vert _{\omega }}{%
\left\vert I\right\vert _{\sigma }}\text{ and its dual }\mathcal{E}%
_{p^{\prime },s}\left( \omega ,\sigma \right) ,
\end{equation*}%
where the supremum is taken over all pairwise disjoint subdecompositions of
an interval $I$ into dyadic subintervals $I_{r}\in \mathcal{D}\left[ I\right]
$.

\begin{lemma}
For $1<s<p<\infty $ we have%
\begin{equation*}
\mathfrak{X}_{\mathcal{F},p}\left( \sigma ,\omega \right) \lesssim \mathcal{E%
}_{p}\left( \sigma ,\omega \right) \leq \mathcal{E}_{p,s}\left( \sigma
,\omega \right) \lesssim \mathfrak{T}_{H,p}\left( \sigma ,\omega \right) .
\end{equation*}
\end{lemma}

\begin{proof}
Let $\left\{ I_{r}\right\} _{r=1}^{\infty }$ be a pairwise disjoint
decomposition of an interval $I$ into subintervals $I_{r}$. We begin with
inequality (\ref{simple reversal}) from Subsection \ref{Sub Rev}, namely
that for $x,y\in I_{r}\subset I$:%
\begin{equation*}
\frac{\mathrm{P}\left( I_{r},\mathbf{1}_{I\setminus I_{r}}\sigma \right) }{%
\ell \left( I_{r}\right) }\left\vert x-y\right\vert \leq 2\left\vert H%
\mathbf{1}_{I\setminus I_{r}}\sigma \left( x\right) -H\mathbf{1}_{I\setminus
I_{r}}\sigma \left( y\right) \right\vert .
\end{equation*}%
Then taking the average in $y$ over $I_{r}$ yields%
\begin{eqnarray*}
&&\frac{\mathrm{P}\left( I_{r},\mathbf{1}_{I\setminus I_{r}}\sigma \right) }{%
\ell \left( I_{r}\right) }\left\vert x-E_{I_{r}}^{\omega }Z\right\vert =%
\frac{\mathrm{P}\left( I_{r},\mathbf{1}_{I\setminus I_{r}}\sigma \right) }{%
\ell \left( I_{r}\right) }\left\vert \frac{1}{\left\vert I_{r}\right\vert
_{\omega }}\int_{I_{r}}\left( x-y\right) d\omega \left( y\right) \right\vert
\\
&\leq &\frac{\mathrm{P}\left( I_{r},\mathbf{1}_{I\setminus I_{r}}\sigma
\right) }{\ell \left( I_{r}\right) }\frac{1}{\left\vert I_{r}\right\vert
_{\omega }}\int_{I_{r}}\left\vert x-y\right\vert d\omega \left( y\right)
\leq 2\frac{1}{\left\vert I_{r}\right\vert _{\omega }}\int_{I_{r}}\left\vert
H\mathbf{1}_{I\setminus I_{r}}\sigma \left( x\right) -H\mathbf{1}%
_{I\setminus I_{r}}\sigma \left( y\right) \right\vert d\omega \left( y\right)
\\
&\equiv &2E_{I_{r}}^{\omega ,y}\left\vert H\mathbf{1}_{I\setminus
I_{r}}\sigma \left( x\right) -H\mathbf{1}_{I\setminus I_{r}}\sigma \left(
y\right) \right\vert ,
\end{eqnarray*}%
where $E_{I_{r}}^{\omega ,y}$ denotes taking the $\omega $-average in $y$
over the cube $I_{r}$. Then multiplying both sides by the indicator $\mathbf{%
1}_{I_{r}}$ and then applying the maximal operator $M_{\omega ,s}^{\func{dy}%
,x}$ in $x$ gives%
\begin{eqnarray*}
&&\frac{\mathrm{P}\left( I_{r},\mathbf{1}_{I\setminus I_{r}}\sigma \right) }{%
\ell \left( I_{r}\right) }M_{\omega ,s}^{\func{dy}}\left[ \left(
Z-E_{I_{r}}^{\omega }Z\right) \mathbf{1}_{I_{r}}\right] \left( x\right) \leq
2M_{\omega ,s}^{\func{dy},x}E_{I_{r}}^{\omega ,y}\left\vert H\mathbf{1}%
_{I\setminus I_{r}}\sigma \left( x\right) \mathbf{1}_{I_{r}}\left( x\right)
-H\mathbf{1}_{I\setminus I_{r}}\sigma \left( y\right) \mathbf{1}%
_{I_{r}}\left( y\right) \right\vert \\
&\leq &2M_{\omega ,s}^{\func{dy},x}E_{I_{r}}^{\omega ,y}\left\vert \mathbf{1}%
_{I_{r}}H\mathbf{1}_{I\setminus I_{r}}\sigma \left( x\right) \right\vert
+2M_{\omega ,s}^{\func{dy},x}E_{I_{r}}^{\omega ,y}\left\vert \mathbf{1}%
_{I_{r}}H\mathbf{1}_{I\setminus I_{r}}\sigma \left( y\right) \right\vert \\
&\leq &2M_{\omega ,s}^{\func{dy}}\mathbf{1}_{I_{r}}H\mathbf{1}_{I\setminus
I_{r}}\sigma \left( x\right) +2E_{I_{r}}^{\omega }\left\vert \mathbf{1}%
_{I_{r}}H\mathbf{1}_{I\setminus I_{r}}\sigma \right\vert \leq 4M_{\omega
,s}^{\func{dy}}\mathbf{1}_{I_{r}}H\mathbf{1}_{I\setminus I_{r}}\sigma \left(
x\right) .
\end{eqnarray*}%
Now raise this to the power $p$ and integrate with respect to $\omega $ to
obtain,%
\begin{eqnarray*}
&&\left( \frac{\mathrm{P}\left( I_{r},\mathbf{1}_{I\setminus I_{r}}\sigma
\right) }{\ell \left( I_{r}\right) }\right) ^{p}\int M_{\omega ,s}^{\func{dy}%
}\left[ \left( Z-E_{I_{r}}^{\omega }Z\right) \mathbf{1}_{I_{r}}\right]
\left( x\right) ^{p}d\omega \left( x\right) \\
&\leq &4^{p}\int \left( M_{\omega ,s}^{\func{dy}}\left[ \mathbf{1}%
_{I_{r}}\left( H\mathbf{1}_{I\setminus I_{r}}\sigma \right) \right] \left(
x\right) \right) ^{p}d\omega \left( x\right) \leq
4^{p}\int_{I_{r}}\left\vert \left( H\mathbf{1}_{I\setminus I_{r}}\sigma
\right) \left( x\right) \right\vert ^{p}d\omega \left( x\right) \\
&\lesssim &\int_{I_{r}}\left\vert H\mathbf{1}_{I}\sigma \left( x\right)
\right\vert ^{p}d\omega \left( x\right) +\int_{I_{r}}\left\vert H\mathbf{1}%
_{I_{r}}\sigma \left( y\right) \right\vert ^{p}d\omega \left( x\right) .
\end{eqnarray*}%
Finally, summing in $r$ produces,%
\begin{eqnarray*}
&&\sum_{r=1}^{\infty }\left( \frac{\mathrm{P}\left( I_{r},\mathbf{1}%
_{I\setminus I_{r}}\sigma \right) }{\left\vert I_{r}\right\vert }\right)
^{p}\int M_{\omega ,s}^{\func{dy}}\left[ \left( Z-E_{I_{r}}^{\omega
}Z\right) \mathbf{1}_{I_{r}}\right] \left( x\right) ^{p}d\omega \left(
x\right) \\
&\lesssim &\int_{I}\left\vert H\mathbf{1}_{I}\sigma \left( x\right)
\right\vert ^{p}d\omega \left( x\right) +\mathfrak{T}_{H,p}\left( \sigma
,\omega \right) \sum_{r=1}^{\infty }\left\vert I_{r}\right\vert _{\sigma
}\lesssim \mathfrak{T}_{H,p}\left( \sigma ,\omega \right) \left\vert
I\right\vert _{\sigma }\ ,
\end{eqnarray*}%
which proves the inequality $\mathcal{E}_{p,s}\left( \sigma ,\omega \right)
\lesssim \mathfrak{T}_{H,p}\left( \sigma ,\omega \right) \mathcal{\ }$for $%
1<s<p<\infty $. The inequality $\mathcal{E}_{p}\left( \sigma ,\omega \right)
\leq \mathcal{E}_{p,s}\left( \sigma ,\omega \right) $ is standard, and the
remaining inequalities in the statement of the lemma were proved in the
earlier lemma.
\end{proof}

\subsection{Consequences of the CZ and $p$-energy corona decomposition}

If we \emph{assume} the finiteness of the energy characteristic $\mathcal{E}%
_{p}\left( \sigma ,\omega \right) $ in (\ref{energy char}) (which is often
referred to as the \emph{energy condition}), and if we take $\Gamma >\max
\left\{ 8\mathcal{E}_{p}\left( \sigma ,\omega \right) ,8\mathfrak{T}%
_{H,p}\left( \sigma ,\omega \right) \right\} $ in (\ref{energy stop crit}),
we obtain a $\sigma $-Carleson\footnote{%
Such conditions are more commonly referred to as $\sigma $-sparse conditions
nowadays.} condition for the Calder\'{o}n-Zygmund/Poisson-Energy stopping
times $\mathcal{F}$,%
\begin{eqnarray*}
&&\sum_{F^{\prime }\in \mathfrak{C}_{\mathcal{F}}\left( F\right) }\left\vert
F^{\prime }\right\vert _{\sigma } \\
&\leq &\frac{1}{\Gamma }\sum_{F^{\prime }\in \mathfrak{C}_{\mathcal{F}%
}\left( F\right) }\min \left\{ \frac{\int_{F^{\prime }}\left\vert \mathsf{P}%
_{F}f\right\vert d\sigma }{E_{F}^{\sigma }\left\vert \mathsf{P}%
_{F}f\right\vert },\left( \frac{\mathrm{P}\left( F^{\prime },\mathbf{1}%
_{F\setminus F^{\prime }}\sigma \right) }{\ell \left( F^{\prime }\right) }%
\right) ^{p}\mathsf{E}_{p}\left( F^{\prime },\omega \right) ^{p}\left\vert
F^{\prime }\right\vert _{\omega },\frac{1}{\left\vert F^{\prime }\right\vert
_{\sigma }}\int_{F^{\prime }}\left\vert M_{\sigma }\mathbf{1}_{F}H_{\sigma }%
\mathbf{1}_{F}\right\vert ^{p}d\omega \right\} \\
&&\ \ \ \ \ \ \ \ \ \ \ \ \ \ \ \ \ \ \ \ \leq \left( \frac{1}{4}\left\vert
F\right\vert _{\sigma }+\frac{\mathcal{E}_{p}\left( \sigma ,\omega \right)
^{p}}{\Gamma }\left\vert F\right\vert _{\sigma }+\frac{\mathfrak{T}%
_{H,p}\left( \sigma ,\omega \right) \left( \sigma ,\omega \right) ^{p}}{%
\Gamma }\left\vert F\right\vert _{\sigma }\right) <\frac{1}{2}\left\vert
F\right\vert _{\sigma },\ \ \ \ \ \text{for all }F\in \mathcal{F}\ ,
\end{eqnarray*}%
since 
\begin{equation*}
\int_{\mathbb{R}}\left\vert M_{\sigma }\mathbf{1}_{F}H_{\sigma }\mathbf{1}%
_{F}\right\vert ^{p}d\omega \leq C_{p}\int_{F}\left\vert \mathbf{1}%
_{F}H_{\sigma }\mathbf{1}_{F}\right\vert ^{p}d\omega \leq C_{p}\mathfrak{T}%
_{H,p}\left( \sigma ,\omega \right) \left( \sigma ,\omega \right)
^{p}\left\vert F\right\vert _{\sigma },
\end{equation*}%
which can then be iterated to obtain \emph{geometric decay in generations}, 
\begin{equation}
\sum_{G\in \mathfrak{C}_{\mathcal{F}}^{\left( m\right) }\left( F\right)
}\left\vert G\right\vert _{\sigma }\leq C_{\delta }2^{-\delta m}\left\vert
F\right\vert _{\sigma },\ \ \ \ \ \text{for all }m\in \mathbb{N}\text{ and }%
F\in \mathcal{F}\ .  \label{sparse}
\end{equation}%
In addition we obtain the quasiorthogonality inequality, given as (3) in the
definition of stopping data above, 
\begin{equation}
\sum_{F\in \mathcal{F}}\left\vert F\right\vert _{\sigma }\alpha _{\mathcal{F}%
}\left( F\right) ^{p}\leq C\int_{\mathbb{R}}\left\vert f\right\vert
^{p}d\sigma ,\ \ \ \ \ 1<p<\infty ,  \label{qorth}
\end{equation}%
which follows easily from that in \cite{LaSaShUr3}, \cite{SaShUr7} or \cite%
{LaWi}, or equivalently from the Carleson embedding theorem, upon noting
that $E_{F}^{\sigma }\left\vert \mathsf{P}_{F}f\right\vert =E_{F}^{\sigma
}\left\vert f-E_{F}^{\sigma }f\right\vert \leq 2E_{F}^{\sigma }\left\vert
f\right\vert $. In fact the reader can easily verify that the triple $\left(
C_{0},\mathcal{F},\alpha _{\mathcal{F}}\right) $ constitutes \emph{stopping
data} for the function $f\in L^{p}\left( \sigma \right) $ for some constant $%
C_{0}$ depending on $\Gamma $, and hence satisfies the stronger
quasiorthogonality property (\ref{q orth}) as well.

The finiteness of the energy characteristic $\mathcal{E}_{p}\left( \sigma
,\omega \right) $ will be needed both to control the Stopping-Energy
characteristic $\mathfrak{X}_{\mathcal{F}}\left( \sigma ,\omega \right)
\lesssim \mathcal{E}_{p}\left( \sigma ,\omega \right) $, which is needed to
control the stopping form, and to enforce (\ref{sparse}), that is in turn
needed to control the far, paraproduct and stopping forms. Finally, we can
appeal to Lemma \ref{energy cond} for 
\begin{equation}
\mathcal{E}_{p}\left( \sigma ,\omega \right) \lesssim \mathfrak{T}%
_{H,p}\left( \sigma ,\omega \right) ,\ \ \ \ \ 1<p<\infty ,
\label{energy bound}
\end{equation}%
that controls $\mathcal{E}_{p}\left( \sigma ,\omega \right) $ by the testing
characteristic for the \emph{Hilbert transform}. Unfortunately this simple
inequality fails, even with a Muckenhoupt characteristic added to the right
hand side, for most other Calder\'{o}n-Zygmund operators in place of the
Hilbert transform, including Riesz transforms in higher dimensions, see \cite%
{SaShUr11} and \cite{Saw5}, and this failure limits the \ current proof to
essentially just the Hilbert transform and similar operators on the real
line as in \cite{SaShUr11}.

\subsubsection{Pigeonholing in corona boxes}

Now we can pigeonhole the pairs of intervals arising in the sum defining the
below form. Given the corona decomposition of $\mathcal{D}$ according to the
Calder\'{o}n-Zygmund stopping times $\mathcal{F}$ constructed above, we
define the analogous decomposition of,%
\begin{eqnarray*}
&&\mathcal{P}_{\func{below}}\equiv \left\{ \left( I,J\right) \in \mathcal{D}%
\times \mathcal{D}:J\subset _{\tau }I\right\} =\bigcup_{F,G\in \mathcal{F}:\
G\subset F}\left[ \mathcal{C}_{\mathcal{F}}\left( F\right) \times \mathcal{C}%
_{\mathcal{F}}\left( G\right) \right] \cap \mathcal{P}_{\func{below}} \\
&=&\left\{ \bigcup_{F\in \mathcal{F}}\left[ \mathcal{C}_{\mathcal{F}}\left(
F\right) \times \mathcal{C}_{\mathcal{F}}\left( F\right) \right] \cap 
\mathcal{P}_{\func{below}}\right\} \bigcup \left\{ \bigcup_{F,G\in \mathcal{F%
}:\ G\subsetneqq F}\left[ \mathcal{C}_{\mathcal{F}}\left( F\right) \times 
\mathcal{C}_{\mathcal{F}}\left( G\right) \right] \cap \mathcal{P}_{\func{%
below}}\right\} \\
&\equiv &\mathcal{P}_{\func{diag}}\bigcup \mathcal{P}_{\func{far}}\ .
\end{eqnarray*}%
Then we consider the corresponding decomposition of the home form into
diagonal and far forms,%
\begin{eqnarray*}
\mathsf{B}_{\func{home}}\left( f,g\right) &=&\sum_{\left( I,J\right) \in 
\mathcal{P}_{\func{diag}}}\left\langle H_{\sigma }\left( \mathbf{1}%
_{I_{J}}\bigtriangleup _{I}^{\sigma }f\right) ,\bigtriangleup _{J}^{\omega
}g\right\rangle _{\omega }+\sum_{\left( I,J\right) \in \mathcal{P}_{\func{far%
}}}\left\langle H_{\sigma }\left( \mathbf{1}_{I_{J}}\bigtriangleup
_{I}^{\sigma }f\right) ,\bigtriangleup _{J}^{\omega }g\right\rangle _{\omega
} \\
&\equiv &\mathsf{B}_{\func{diag}}\left( f,g\right) +\mathsf{B}_{\func{far}%
}\left( f,g\right) .
\end{eqnarray*}

We next decompose the $\func{far}$ form into corona pieces using $\mathcal{P}%
_{\func{far}}^{F,G}\equiv \left[ \mathcal{C}_{\mathcal{F}}\left( F\right)
\times \mathcal{C}_{\mathcal{F}}\left( G\right) \right] \cap \mathcal{P}_{%
\func{below}}$,

\begin{eqnarray*}
\mathsf{B}_{\func{far}}\left( f,g\right) &=&\sum_{F,G\in \mathcal{F}:\
G\subsetneqq F}\left\langle H_{\sigma }\left( \sum_{I\in \mathcal{C}_{%
\mathcal{F}}\left( F\right) }\mathbf{1}_{I_{J}}\bigtriangleup _{I}^{\sigma
}f\right) ,\sum_{J\in \mathcal{C}_{\mathcal{F}}\left( G\right) :\ J\subset
_{\tau }I}\bigtriangleup _{J}^{\omega }g\right\rangle _{\omega } \\
&=&\sum_{F,G\in \mathcal{F}:\ G\subsetneqq F}\sum_{\left( I,J\right) \in 
\mathcal{P}_{\func{far}}^{F,G}}\left\langle H_{\sigma }\left( \mathbf{1}%
_{I_{J}}\bigtriangleup _{I}^{\sigma }f\right) ,\bigtriangleup _{J}^{\omega
}g\right\rangle _{\omega }=\sum_{F,G\in \mathcal{F}:\ G\subsetneqq F}\mathsf{%
B}_{\func{far}}^{F,G}\left( f,g\right) \\
\text{where }\mathsf{B}_{\func{far}}^{F,G}\left( f,g\right) &\equiv
&\sum_{\left( I,J\right) \in \mathcal{P}_{\func{far}}^{F,G}}\left\langle
H_{\sigma }\left( \mathbf{1}_{I_{J}}\bigtriangleup _{I}^{\sigma }f\right)
,\bigtriangleup _{J}^{\omega }g\right\rangle _{\omega }.
\end{eqnarray*}%
Now for $m>\tau $ and $F\in \mathcal{F}$ we define%
\begin{eqnarray*}
\mathsf{B}_{\func{far}}^{F,m}\left( f,g\right) &\equiv &\sum_{G\in \mathfrak{%
C}_{\mathcal{F}}^{\left( m\right) }\left( F\right) }\mathsf{B}_{\func{far}%
}^{F,G}\left( f,g\right) =\sum_{G\in \mathfrak{C}_{\mathcal{F}}^{\left(
m\right) }\left( F\right) }\sum_{\left( I,J\right) \in \mathcal{P}_{\func{far%
}}^{F,G}}\left\langle H_{\sigma }\left( \mathbf{1}_{I_{J}}\bigtriangleup
_{I}^{\sigma }f\right) ,\bigtriangleup _{J}^{\omega }g\right\rangle _{\omega
} \\
&=&\sum_{F^{\prime }\in \mathfrak{C}_{\mathcal{F}}\left( F\right)
}\sum_{G\in \mathfrak{C}_{\mathcal{F}}^{\left( m-1\right) }\left( F^{\prime
}\right) }\sum_{J\in \mathcal{C}_{\mathcal{F}}\left( G\right) }\left\langle
H_{\sigma }\left( \sum_{I\in \mathcal{C}_{\mathcal{F}}\left( F\right) :\
J\subset _{\tau }I}\mathbf{1}_{I_{J}}\bigtriangleup _{I}^{\sigma }f\right)
,\bigtriangleup _{J}^{\omega }g\right\rangle _{\omega }.
\end{eqnarray*}

We will now control the $\func{far}$ form for $1<p<\infty $ in the remainder
of this section, and finally control the \emph{diagonal} form in the last
two sections.

\subsection{Far form}

Here we will control the $\func{far}$ form $\mathsf{B}_{\func{far}}\left(
f,g\right) $ by quadratic local testing and the extended energy
characteristic $\mathcal{E}_{p}^{\ell ^{2},\limfunc{ext}}\left( \sigma
,\omega \right) $. Recall that the $\func{far}$ form is defined by%
\begin{equation*}
\mathsf{B}_{\func{far}}\left( f,g\right) \equiv \sum_{m=1}^{\infty
}\sum_{F\in \mathcal{F}}\mathsf{B}_{\func{far}}^{F,m}\left( f,g\right) ,
\end{equation*}%
where 
\begin{eqnarray*}
\mathsf{B}_{\func{far}}^{F,m}\left( f,g\right) &\equiv &\sum_{G\in \mathfrak{%
C}_{\mathcal{F}}^{\left( m\right) }\left( F\right) }\mathsf{B}_{\func{far}%
}^{F,G}\left( f,g\right) =\sum_{G\in \mathfrak{C}_{\mathcal{F}}^{\left(
m\right) }\left( F\right) }\sum_{\left( I,J\right) \in \mathcal{P}_{\func{far%
}}^{F,G}}\left\langle H_{\sigma }\left( \mathbf{1}_{I_{J}}\bigtriangleup
_{I}^{\sigma }f\right) ,\bigtriangleup _{J}^{\omega }g\right\rangle _{\omega
} \\
&=&\sum_{G\in \mathfrak{C}_{\mathcal{F}}^{\left( m\right) }\left( F\right)
}\sum_{J\in \mathcal{C}_{\mathcal{F}}\left( G\right) }\left\langle H_{\sigma
}\left( \sum_{I\in \mathcal{C}_{\mathcal{F}}\left( F\right) :\ J\subset
_{\tau }I}\mathbf{1}_{I_{J}}\bigtriangleup _{I}^{\sigma }f\right)
,\bigtriangleup _{J}^{\omega }g\right\rangle _{\omega }.
\end{eqnarray*}%
Thus we can write%
\begin{eqnarray*}
\mathsf{B}_{\func{far}}\left( f,g\right) &=&\sum_{m=1}^{\infty }\sum_{F\in 
\mathcal{F}}\sum_{G\in \mathfrak{C}_{\mathcal{F}}^{\left( m\right) }\left(
F\right) }\sum_{J\in \mathcal{C}_{\mathcal{F}}\left( G\right) }\left\langle
H_{\sigma }\left( \sum_{I\in \mathcal{C}_{\mathcal{F}}\left( F\right) :\
J\subset _{\tau }I}\mathbf{1}_{I_{J}}\bigtriangleup _{I}^{\sigma }f\right)
,\bigtriangleup _{J}^{\omega }g\right\rangle _{\omega } \\
&=&\sum_{G\in \mathcal{F}}\sum_{J\in \mathcal{C}_{\mathcal{F}}\left(
G\right) }\left\langle H_{\sigma }\left( \sum_{I\in \mathcal{D}:\
G\subsetneqq I\text{ and }J\subset _{\tau }I}\mathbf{1}_{I_{J}}%
\bigtriangleup _{I}^{\sigma }f\right) ,\bigtriangleup _{J}^{\omega
}g\right\rangle _{\omega },
\end{eqnarray*}%
which we will usually consider with the dummy variable $G$ replaced by $F$,%
\begin{equation*}
\mathsf{B}_{\func{far}}\left( f,g\right) =\sum_{F\in \mathcal{F}}\sum_{J\in 
\mathcal{C}_{\mathcal{F}}\left( F\right) }\left\langle H_{\sigma }\left(
\sum_{I\in \left( F,T\right] \text{ and }J\subset _{\tau }I}\mathbf{1}%
_{I_{J}}\bigtriangleup _{I}^{\sigma }f\right) ,\bigtriangleup _{J}^{\omega
}g\right\rangle _{\omega }.
\end{equation*}%
Given any collection $\mathcal{H}\subset \mathcal{D}$ of intervals, and a
dyadic interval $J$, we define the corresponding Haar projection $\mathsf{P}%
_{\mathcal{H}}^{\omega }$ and its localization $\mathsf{P}_{\mathcal{H}%
;J}^{\omega }$ to $J$ by%
\begin{equation}
\mathsf{P}_{\mathcal{H}}^{\omega }=\sum_{H\in \mathcal{H}}\bigtriangleup
_{H}^{\omega }\text{ and }\mathsf{P}_{\mathcal{H};J}^{\omega }=\sum_{H\in 
\mathcal{H}:\ H\subset J}\bigtriangleup _{H}^{\omega }\ .
\label{def localization}
\end{equation}

Recall that given any interval $F\in \mathcal{D}$, we defined the $\left(
r,\varepsilon \right) $\emph{-Whitney} collection $\mathcal{M}_{\left(
r,\varepsilon \right) -\limfunc{deep}}\left( F\right) $ of $F$ to be the set
of dyadic subintervals $W\subset F$ that are maximal with respect to the
property that $W\subset _{r,\varepsilon }F$, and hence form a pairwise
disjoint decomposition of $F$. Recall also that $\mathcal{E}_{p}^{\ell ^{2},%
\limfunc{ext}}\left( \sigma ,\omega \right) $ denotes the smallest constant
in the `extended energy' inequality (\ref{ext ener}), i.e. 
\begin{equation*}
\int \left( \sum_{F\in \mathcal{F}}\sum_{W\in \mathcal{M}_{\left(
r,\varepsilon \right) -\limfunc{deep}}\left( F\right) \cap \mathcal{C}_{%
\mathcal{F}}\left( F\right) }\left( \frac{\mathrm{P}\left( W,\Phi
_{F}f\sigma \right) }{\ell \left( W\right) }\right) ^{2}\left\vert \mathsf{P}%
_{\mathcal{C}_{\mathcal{F}}\left( F\right) \cap \mathcal{D}\left[ W\right]
}^{\omega }\right\vert ^{2}Z\right) ^{\frac{p}{2}}d\omega \lesssim \mathcal{E%
}_{p}^{\ell ^{2},\limfunc{ext}}\left( \sigma ,\omega \right) ^{p}\left\Vert
f\right\Vert _{L^{p}\left( \sigma \right) }^{p},\ \ \ \ \ \text{for }f\geq 0,
\end{equation*}%
where $\mathcal{F}$ is the collection of stopping times associated with $f$,
and%
\begin{equation*}
\Phi _{F}f\equiv \mathbb{E}_{F}^{\sigma }f+f_{F}=\mathbb{E}_{F}^{\sigma
}f+\sum_{I:\ I\supsetneqq F}\mathbf{1}_{I_{F}}\bigtriangleup _{I}^{\sigma }f.
\end{equation*}

There is a similar definition of the dual characteristic $\mathcal{E}%
_{p^{\prime }}^{\ell ^{2},\limfunc{ext}}\left( \omega ,\sigma \right) $. The
Intertwining Proposition will control the following Intertwining form,%
\begin{equation*}
\mathsf{B}_{\func{Inter}}\left( f,g\right) \equiv \sum_{F\in \mathcal{F}}\
\sum_{I:\ I\supsetneqq F}\ \left\langle H_{\sigma }\left( \mathbf{1}%
_{I_{F}}\bigtriangleup _{I}^{\sigma }f\right) ,\mathsf{P}_{\mathcal{C}_{%
\mathcal{F}}\left( F\right) }^{\omega }g\right\rangle _{\omega }\ ,
\end{equation*}%
whose difference from $\mathsf{B}_{\func{far}}\left( f,g\right) $ is%
\begin{eqnarray*}
\mathsf{B}_{\func{far}}\left( f,g\right) -\mathsf{B}_{\func{Inter}}\left(
f,g\right) &=&\sum_{F\in \mathcal{F}}\sum_{I\in \left( F,T\right]
}\sum_{J\in \mathcal{C}_{\mathcal{F}}\left( F\right) \text{ and }J\subset
_{\tau }I}\left\langle H_{\sigma }\left( \mathbf{1}_{I_{J}}\bigtriangleup
_{I}^{\sigma }f\right) ,\bigtriangleup _{J}^{\omega }g\right\rangle _{\omega
} \\
&&-\sum_{F\in \mathcal{F}}\ \sum_{I\in \left( F,T\right] }\ \sum_{J\in 
\mathcal{C}_{\mathcal{F}}\left( F\right) }\left\langle H_{\sigma }\left( 
\mathbf{1}_{I_{F}}\bigtriangleup _{I}^{\sigma }f\right) ,\bigtriangleup
_{J}^{\omega }g\right\rangle _{\omega } \\
&=&\sum_{F\in \mathcal{F}}\sum_{I\in \left( F,T\right] }\sum_{\substack{ %
J\in \mathcal{C}_{\mathcal{F}}\left( F\right)  \\ \ell \left( J\right) \geq
\ell \left( F\right) -\tau \text{ and }J\subset _{\tau }I}}\left\langle
H_{\sigma }\left( \mathbf{1}_{I_{F}}\bigtriangleup _{I}^{\sigma }f\right)
,\bigtriangleup _{J}^{\omega }g\right\rangle _{\omega }.
\end{eqnarray*}%
Just as for the comparable form $\mathsf{B}_{\limfunc{comp}}\left(
f,g\right) $, this difference form is controlled by 
\begin{eqnarray*}
&&\left\vert \mathsf{B}_{\func{far}}\left( f,g\right) -\mathsf{B}_{\func{%
Inter}}\left( f,g\right) \right\vert \leq \sum_{F\in \mathcal{F}}\sum_{I\in
\left( F,T\right] }\sum_{\substack{ J\in \mathcal{C}_{\mathcal{F}}\left(
F\right)  \\ \ell \left( J\right) \geq \ell \left( F\right) -\tau \text{ and 
}J\subset _{\tau }I}}\left\vert \left\langle H_{\sigma }\left( \mathbf{1}%
_{I_{F}}\bigtriangleup _{I}^{\sigma }f\right) ,\bigtriangleup _{J}^{\omega
}g\right\rangle _{\omega }\right\vert \\
&\lesssim &\left( \mathfrak{T}_{H,p}^{\ell ^{2},\func{loc}}\left( \sigma
,\omega \right) +\mathcal{A}_{p}^{\ell ^{2}}\left( \sigma ,\omega \right) +%
\mathcal{WBP}_{H,p}^{\ell ^{2}}\left( \sigma ,\omega \right) \right) \
\left\Vert f\right\Vert _{L^{p}\left( \sigma \right) }\left\Vert
g\right\Vert _{L^{p^{\prime }}\left( \omega \right) }\ ,
\end{eqnarray*}%
which is also bounded by $\mathfrak{T}_{H,p}^{\ell ^{2},\limfunc{glob}%
}\left( \sigma ,\omega \right) \left\Vert f\right\Vert _{L^{p}\left( \sigma
\right) }\left\Vert g\right\Vert _{L^{p^{\prime }}\left( \omega \right) }$.

\begin{definition}
\label{sigma carleson n}A collection $\mathcal{F}$ of dyadic intervals is $%
\sigma $\emph{-Carleson} if%
\begin{equation*}
\sum_{F\in \mathcal{F}:\ F\subset S}\left\vert F\right\vert _{\sigma }\leq
C_{\mathcal{F}}\left( \sigma \right) \left\vert S\right\vert _{\sigma },\ \
\ \ \ S\in \mathcal{F}.
\end{equation*}%
The constant $C_{\mathcal{F}}\left( \sigma \right) $ is referred to as the
Carleson norm of $\mathcal{F}$.
\end{definition}

We now show that the extended energy inequality (\ref{ext ener}), together
with quadratic interval testing, suffices to prove the Intertwining
Proposition.

Let $\mathcal{F}$ be any subset of $\mathcal{D}$. For any $J\in \mathcal{D}$%
, we define $\pi _{\mathcal{F}}^{0}J$ to be the smallest $F\in \mathcal{F}$
that contains $J$. Then for $s\geq 1$, we recursively define $\pi _{\mathcal{%
F}}^{s}J$ to be the smallest $F\in \mathcal{F}$ that \emph{strictly}
contains $\pi _{\mathcal{F}}^{s-1}J$. This definition satisfies $\pi _{%
\mathcal{F}}^{s+t}J=\pi _{\mathcal{F}}^{s}\pi _{\mathcal{F}}^{t}J$ for all $%
s,t\geq 0$ and $J\in \mathcal{D}$. In particular $\pi _{\mathcal{F}%
}^{s}J=\pi _{\mathcal{F}}^{s}F$ where $F=\pi _{\mathcal{F}}^{0}J$. In the
special case $\mathcal{F}=\mathcal{D}$ we often suppress the subscript $%
\mathcal{F}$ and simply write $\pi ^{s}$ for $\pi _{\mathcal{D}}^{s}$.
Finally, for $F\in \mathcal{F}$, we write $\mathfrak{C}_{\mathcal{F}}\left(
F\right) \equiv \left\{ F^{\prime }\in \mathcal{F}:\pi _{\mathcal{F}%
}^{1}F^{\prime }=F\right\} $ for the collection of $\mathcal{F}$-children of 
$F$.

\begin{proposition}[The Intertwining Proposition]
\label{strongly adapted'}Suppose $1<p<\infty $ and $\sigma ,\omega $ are
locally finite positive Borel measures on $\mathbb{R}$. Then if $\mathcal{F}$
is the collection of stopping times for $f\in L^{p}\left( \sigma \right) $
as above, we have 
\begin{equation*}
\left\vert \sum_{F\in \mathcal{F}}\ \sum_{I:\ I\supsetneqq F}\ \left\langle
H_{\sigma }\left( \mathbf{1}_{I_{F}}\bigtriangleup _{I}^{\sigma }f\right) ,%
\mathsf{P}_{\mathcal{C}_{\mathcal{F}}\left( F\right) }^{\omega
}g\right\rangle _{\omega }\right\vert \lesssim \left( \mathcal{E}_{p}^{\ell
^{2},\limfunc{ext}}\left( \sigma ,\omega \right) +\mathfrak{T}_{H,p}^{\ell
^{2},\func{loc}}\left( \sigma ,\omega \right) \right) \ \left\Vert
f\right\Vert _{L^{p}\left( \sigma \right) }\left\Vert g\right\Vert
_{L^{p^{\prime }}\left( \omega \right) }.
\end{equation*}
\end{proposition}

\begin{proof}
We write the left hand side of the display above as%
\begin{equation*}
\sum_{F\in \mathcal{F}}\ \sum_{I:\ I\supsetneqq F}\ \left\langle H_{\sigma
}\left( \mathbf{1}_{I_{F}}\bigtriangleup _{I}^{\sigma }f\right)
,g_{F}\right\rangle _{\omega }=\sum_{F\in \mathcal{F}}\ \left\langle
H_{\sigma }\left( \sum_{I:\ I\supsetneqq F}\mathbf{1}_{I_{F}}\bigtriangleup
_{I}^{\sigma }f\right) ,g_{F}\right\rangle _{\omega }\equiv \sum_{F\in 
\mathcal{F}}\ \left\langle H_{\sigma }f_{F},g_{F}\right\rangle _{\omega }\ ,
\end{equation*}%
where%
\begin{equation*}
g_{F}=\mathsf{P}_{\mathcal{C}_{\mathcal{F}}\left( F\right) }^{\omega
}g=\sum_{J\in \mathcal{C}_{\mathcal{F}}\left( F\right) }\bigtriangleup
_{J}^{\omega }g\text{ and }f_{F}\equiv \sum_{I:\ I\supsetneqq F}\mathbf{1}%
_{I_{F}}\bigtriangleup _{I}^{\sigma }f\ .
\end{equation*}%
Note that $g_{F}$ is supported in $F$, and that $f_{F}$ is constant on $F$.
We note that the intervals $I$ occurring in this sum are linearly and
consecutively ordered by inclusion, along with the intervals $F^{\prime }\in 
\mathcal{F}$ that contain $F$. More precisely, we can write%
\begin{equation*}
F\equiv F_{0}\subsetneqq F_{1}\subsetneqq F_{2}\subsetneqq ...\subsetneqq
F_{n}\subsetneqq F_{n+1}\subsetneqq ...F_{N}
\end{equation*}%
where $F_{m}=\pi _{\mathcal{F}}^{m}F$ for all $m\geq 1$. We can also write%
\begin{equation*}
F=F_{0}\subsetneqq I_{1}\subsetneqq I_{2}\subsetneqq ...\subsetneqq
I_{k}\subsetneqq I_{k+1}\subsetneqq ...\subsetneqq I_{K}=F_{N}
\end{equation*}%
where $I_{k}=\pi _{\mathcal{D}}^{k}F$ for all $k\geq 1$, and by convention
we set $I_{0}=F$. There is a (unique) subsequence $\left\{ k_{m}\right\}
_{m=1}^{N}$ such that%
\begin{equation*}
F_{m}=I_{k_{m}},\ \ \ \ \ 1\leq m\leq N.
\end{equation*}

Recall that%
\begin{equation*}
f_{F}\left( x\right) =\sum_{k=1}^{\infty }\mathbf{1}_{\left( I_{k}\right)
_{F}}\left( x\right) \bigtriangleup _{I_{k}}^{\sigma }f\left( x\right)
=\sum_{k=1}^{\infty }\mathbf{1}_{I_{k}\setminus I_{k-1}}\left( x\right)
\sum_{\ell =k+1}^{\infty }\bigtriangleup _{I_{\ell }}^{\sigma }f\left(
x\right) .
\end{equation*}%
Assume now that $k_{m}\leq k<k_{m+1}$. Using a telescoping sum, we compute
that for 
\begin{equation*}
x\in I_{k+1}\setminus I_{k}\subset F_{m+1}\setminus F_{m},
\end{equation*}%
we have 
\begin{equation*}
\left\vert \sum_{\ell =k+2}^{\infty }\bigtriangleup _{I_{\ell }}^{\sigma
}f\left( x\right) \right\vert =\left\vert \mathbb{E}_{\theta
I_{k+2}}^{\sigma }f-\mathbb{E}_{I_{K}}^{\sigma }f\right\vert \lesssim 
\mathbb{E}_{F_{m+1}}^{\sigma }\left\vert f\right\vert \ .
\end{equation*}%
Note that $f_{F}$ is constant on $F$ and that%
\begin{eqnarray*}
\left\vert f_{F}\right\vert &\leq &\sum_{m=0}^{N}\left( \mathbb{E}%
_{F_{m+1}}^{\sigma }\left\vert f\right\vert \right) \ \mathbf{1}%
_{F_{m+1}\setminus F_{m}}=\left( \mathbb{E}_{F}^{\sigma }\left\vert
f\right\vert \right) \ \mathbf{1}_{F}+\sum_{m=0}^{N}\left( \mathbb{E}_{\pi _{%
\mathcal{F}}^{m+1}F}^{\sigma }\left\vert f\right\vert \right) \ \mathbf{1}%
_{\pi _{\mathcal{F}}^{m+1}F\setminus \pi _{\mathcal{F}}^{m}F} \\
&=&\left( \mathbb{E}_{F}^{\sigma }\left\vert f\right\vert \right) \ \mathbf{1%
}_{F}+\sum_{F^{\prime }\in \mathcal{F}:\ F\subset F^{\prime }}\left( \mathbb{%
E}_{\pi _{\mathcal{F}}F^{\prime }}^{\sigma }\left\vert f\right\vert \right)
\ \mathbf{1}_{\pi _{\mathcal{F}}F^{\prime }\setminus F^{\prime }} \\
&\leq &\alpha _{\mathcal{F}}\left( F\right) \ \mathbf{1}_{F}+\sum_{F^{\prime
}\in \mathcal{F}:\ F\subset F^{\prime }}\alpha _{\mathcal{F}}\left( \pi _{%
\mathcal{F}}F^{\prime }\right) \ \mathbf{1}_{\pi _{\mathcal{F}}F^{\prime
}\setminus F^{\prime }} \\
&\leq &\alpha _{\mathcal{F}}\left( F\right) \ \mathbf{1}_{F}+\sum_{F^{\prime
}\in \mathcal{F}:\ F\subset F^{\prime }}\alpha _{\mathcal{F}}\left( \pi _{%
\mathcal{F}}F^{\prime }\right) \ \mathbf{1}_{\pi _{\mathcal{F}}F^{\prime }}\ 
\mathbf{1}_{F^{c}} \\
&=&\alpha _{\mathcal{F}}\left( F\right) \ \mathbf{1}_{F}+\Phi \ \mathbf{1}%
_{F^{c}}\ ,\ \ \ \ \ \text{\ for all }F\in \mathcal{F},
\end{eqnarray*}%
where%
\begin{equation*}
\Phi \equiv \sum_{F^{\prime \prime }\in \mathcal{F}}\ \alpha _{\mathcal{F}%
}\left( F^{\prime \prime }\right) \ \mathbf{1}_{F^{\prime \prime }}\ .
\end{equation*}

Now we write%
\begin{equation*}
\sum_{F\in \mathcal{F}}\ \left\langle H_{\sigma }f_{F},g_{F}\right\rangle
_{\omega }=\sum_{F\in \mathcal{F}}\ \left\langle H_{\sigma }\left( \mathbf{1}%
_{F}f_{F}\right) ,g_{F}\right\rangle _{\omega }+\sum_{F\in \mathcal{F}}\
\left\langle H_{\sigma }\left( \mathbf{1}_{F^{c}}f_{F}\right)
,g_{F}\right\rangle _{\omega }\equiv I+II.
\end{equation*}%
Then quadratic interval testing, the square function inequalities in Theorem %
\ref{square thm}, and quasi-orthogonality together with the fact that $f_{F}$
is a constant on $F$ bounded by $\alpha _{\mathcal{F}}\left( F\right) $, give%
\begin{eqnarray}
&&\left\vert I\right\vert =\left\vert \int_{\mathbb{R}}\sum_{F\in \mathcal{F}%
}\mathbf{1}_{F}\left( x\right) H_{\sigma }\left( \mathbf{1}_{F}f_{F}\right)
\left( x\right) \ g_{F}\left( x\right) \ d\omega \left( x\right) \right\vert
\label{I} \\
&\leq &\int_{\mathbb{R}}\sum_{F\in \mathcal{F}}\alpha _{\mathcal{F}}\left(
F\right) \left\vert \mathbf{1}_{F}\left( x\right) H_{\sigma }\left( \mathbf{1%
}_{F}\right) \left( x\right) \ g_{F}\left( x\right) \right\vert \ d\omega
\left( x\right)  \notag \\
&\leq &\int_{\mathbb{R}}\sqrt{\sum_{F\in \mathcal{F}}\left\vert \alpha _{%
\mathcal{F}}\left( F\right) \mathbf{1}_{F}\left( x\right) H_{\sigma }\left( 
\mathbf{1}_{F}\right) \left( x\right) \right\vert ^{2}}\ \sqrt{\sum_{F\in 
\mathcal{F}}\left\vert g_{F}\left( x\right) \right\vert ^{2}}\ d\omega
\left( x\right)  \notag \\
&\leq &\ \left\Vert \left\vert \left\{ \alpha _{\mathcal{F}}\left( F\right) 
\mathbf{1}_{F}H_{\sigma }\left( \mathbf{1}_{F}\right) \right\} _{F\in 
\mathcal{F}}\right\vert _{\ell ^{2}}\right\Vert _{L^{p}\left( \omega \right)
}\ \left\Vert \left\vert \left\{ g_{F}\right\} _{F\in \mathcal{F}%
}\right\vert _{\ell ^{2}}\right\Vert _{L^{p^{\prime }}\left( \omega \right) }
\notag \\
&\lesssim &\mathfrak{T}_{H,p}^{\ell ^{2},\func{loc}}\left( \sigma ,\omega
\right) \left\Vert \left\vert \left\{ \alpha _{\mathcal{F}}\left( F\right) 
\mathbf{1}_{F}\right\} _{F\in \mathcal{F}}\right\vert _{\ell
^{2}}\right\Vert _{L^{p}\left( \sigma \right) }\left\Vert g\right\Vert
_{L^{p^{\prime }}\left( \omega \right) }\lesssim C_{\mathcal{F}}\left(
\sigma \right) \mathfrak{T}_{H,p}^{\ell ^{2},\func{loc}}\left( \sigma
,\omega \right) \left\Vert f\right\Vert _{L^{p}\left( \sigma \right)
}\left\Vert g\right\Vert _{L^{p^{\prime }}\left( \omega \right) }.  \notag
\end{eqnarray}

Now $\mathbf{1}_{F^{c}}f_{F}$ is supported outside $F$, and each $J$ in the
Haar support of $g_{F}=\mathsf{P}_{\mathcal{C}_{\mathcal{F}}\left( F\right)
}^{\omega }g$ is either in $\mathcal{N}_{r}\left( F\right) =\left\{ J\in 
\mathcal{D}\left[ F\right] :\ell \left( J\right) \geq 2^{-r}\ell \left(
F\right) \right\} $, in which case the desired bound for term $I$ is
straightforward, or $J$ is $\left( r,\varepsilon \right) $-deeply embedded
in $F$, i.e. $J\subset _{r,\varepsilon }F$, and so $J\subset _{r,\varepsilon
}W$ for some $W\in \mathcal{M}_{\left( r,\varepsilon \right) -\limfunc{deep}%
}\left( F\right) $. Since the corona $\mathcal{C}_{\mathcal{F}}\left(
F\right) $ is connected, it follows that $W\in \mathcal{C}_{\mathcal{F}%
}\left( F\right) $ if $\mathsf{P}_{\mathcal{C}_{\mathcal{F}}\left( F\right)
\cap \mathcal{D}\left[ W\right] }^{\omega }g$ is nonvanishing. Thus with $%
\mathcal{C}_{\mathcal{F}}^{\flat }\left( F\right) \equiv \mathcal{C}_{%
\mathcal{F}}\left( F\right) \setminus \mathcal{N}_{r}\left( F\right) $ we
can apply \ref{ener rev} to obtain%
\begin{eqnarray*}
\left\vert II\right\vert &=&\left\vert \int_{\mathbb{R}}\sum_{F\in \mathcal{F%
}}\ H_{\sigma }\left( \mathbf{1}_{F^{c}}f_{F}\right) \left( x\right) \ 
\mathsf{P}_{\mathcal{C}_{\mathcal{F}}^{\flat }\left( F\right) }^{\omega
}g\left( x\right) \ d\omega \left( x\right) \right\vert \\
&=&\left\vert \int_{\mathbb{R}}\sum_{F\in \mathcal{F}}\sum_{W\in \mathcal{M}%
_{\left( r,\varepsilon \right) -\limfunc{deep}}\left( F\right) \cap \mathcal{%
C}_{\mathcal{F}}\left( F\right) }\ \mathsf{P}_{\mathcal{C}_{\mathcal{F}%
}^{\flat }\left( F\right) \cap \mathcal{D}\left[ W\right] }^{\omega
}H_{\sigma }\left( \mathbf{1}_{F^{c}}f_{F}\right) \left( x\right) \ \mathsf{P%
}_{\mathcal{C}_{\mathcal{F}}^{\flat }\left( F\right) \cap \mathcal{D}\left[ W%
\right] }^{\omega }g\left( x\right) \ d\omega \left( x\right) \right\vert \\
&\lesssim &\int_{\mathbb{R}}\sqrt{\sum_{F\in \mathcal{F}}\sum_{W\in \mathcal{%
M}_{\left( r,\varepsilon \right) -\limfunc{deep}}\left( F\right) \cap 
\mathcal{C}_{\mathcal{F}}\left( F\right) }\ \left\vert \mathsf{P}_{\mathcal{C%
}_{\mathcal{F}}^{\flat }\left( F\right) \cap \mathcal{D}\left[ W\right]
}^{\omega }H_{\sigma }\left( \mathbf{1}_{F^{c}}f_{F}\right) \left( x\right)
\right\vert ^{2}}\  \\
&&\ \ \ \ \ \ \ \ \ \ \ \ \ \ \ \ \ \ \ \ \ \ \ \ \ \times \sqrt{\sum_{F\in 
\mathcal{F}}\sum_{W\in \mathcal{M}_{\left( r,\varepsilon \right) -\limfunc{%
deep}}\left( F\right) \cap \mathcal{C}_{\mathcal{F}}\left( F\right) }\
\left\vert \mathsf{P}_{\mathcal{C}_{\mathcal{F}}^{\flat }\left( F\right)
\cap \mathcal{D}\left[ W\right] }^{\omega }g\left( x\right) \right\vert ^{2}}%
\ d\omega \left( x\right) ,
\end{eqnarray*}%
which is at most%
\begin{eqnarray*}
&&\left\{ \int_{\mathbb{R}}\left( \sum_{F\in \mathcal{F}}\sum_{W\in \mathcal{%
M}_{\left( r,\varepsilon \right) -\limfunc{deep}}\left( F\right) \cap 
\mathcal{C}_{\mathcal{F}}\left( F\right) }\ \left\vert \mathsf{P}_{\mathcal{C%
}_{\mathcal{F}}^{\flat }\left( F\right) \cap \mathcal{D}\left[ W\right]
}^{\omega }H_{\sigma }\left( \mathbf{1}_{F^{c}}f_{F}\right) \left( x\right)
\right\vert ^{2}\right) ^{\frac{p}{2}}d\omega \left( x\right) \right\} ^{%
\frac{1}{p}} \\
&&\times \left\{ \int_{\mathbb{R}}\left( \sum_{F\in \mathcal{F}}\sum_{W\in 
\mathcal{M}_{\left( r,\varepsilon \right) -\limfunc{deep}}\left( F\right)
\cap \mathcal{C}_{\mathcal{F}}\left( F\right) }\ \left\vert \mathsf{P}_{%
\mathcal{C}_{\mathcal{F}}^{\flat }\left( F\right) \cap \mathcal{D}\left[ W%
\right] }^{\omega }g\left( x\right) \right\vert ^{2}\right) ^{\frac{%
p^{\prime }}{2}}d\omega \left( x\right) \right\} ^{\frac{1}{p^{\prime }}}.
\end{eqnarray*}%
The second factor is at most $C_{p}\left\Vert g\right\Vert _{L^{p^{\prime
}}\left( \omega \right) }$ by Theorem \ref{square thm}.

Then we use the Energy Lemma on the first factor to obtain that its $p^{th}$
power is at most,%
\begin{eqnarray*}
&&\int_{\mathbb{R}}\left( \sum_{F\in \mathcal{F}}\sum_{W\in \mathcal{M}%
_{\left( r,\varepsilon \right) -\limfunc{deep}}\left( F\right) \cap \mathcal{%
C}_{\mathcal{F}}\left( F\right) }\ \left( \frac{\mathrm{P}\left( W,\mathbf{1}%
_{F^{c}}f_{F}\sigma \right) }{\ell \left( W\right) }\right) ^{2}\left\vert 
\mathsf{P}_{\mathcal{C}_{\mathcal{F}}^{\flat }\left( F\right) \cap \mathcal{D%
}\left[ W\right] }^{\omega }\right\vert Z\left( x\right) ^{2}\right) ^{\frac{%
p}{2}}d\omega \left( x\right) \\
&\lesssim &\int_{\mathbb{R}}\left( \sum_{F\in \mathcal{F}}\sum_{W\in 
\mathcal{M}_{\left( r,\varepsilon \right) -\limfunc{deep}}\left( F\right)
\cap \mathcal{C}_{\mathcal{F}}\left( F\right) }\ \left( \frac{\mathrm{P}%
\left( W,\mathbf{1}_{F^{c}}\Phi _{F}f\sigma \right) }{\ell \left( W\right) }%
\right) ^{2}\left\vert \mathsf{P}_{\mathcal{C}_{\mathcal{F}}^{\flat }\left(
F\right) \cap \mathcal{D}\left[ W\right] }^{\omega }\right\vert Z\left(
x\right) ^{2}\right) ^{\frac{p}{2}}d\omega \left( x\right) \\
&\leq &\mathcal{E}_{p}^{\ell ^{2},\limfunc{ext}}\left( \sigma ,\omega
\right) ^{p}\lVert f\rVert _{L^{p}\left( \sigma \right) }^{p},
\end{eqnarray*}%
where the last line follows from the definition of the extended energy
characteristic.

This completes the proof of the Intertwining Proposition \ref{strongly
adapted'}.
\end{proof}

Thus we have the following controls of the $\func{far}$ form, 
\begin{eqnarray*}
\left\vert \mathsf{B}_{\func{far}}\left( f,g\right) \right\vert &\lesssim
&\left( \mathcal{E}_{p}^{\ell ^{2},\limfunc{ext}}\left( \sigma ,\omega
\right) +\mathfrak{T}_{H,p}^{\ell ^{2},\limfunc{loc}}\left( \sigma ,\omega
\right) +\mathcal{A}_{p}\left( \sigma ,\omega \right) +\mathcal{WBP}%
_{H,p}^{\ell ^{2}}\left( \sigma ,\omega \right) \right) \ \left\Vert
f\right\Vert _{L^{p}\left( \sigma \right) }\left\Vert g\right\Vert
_{L^{p^{\prime }}\left( \omega \right) }\ , \\
\left\vert \mathsf{B}_{\func{far}}\left( f,g\right) \right\vert &\lesssim
&\left( \mathcal{E}_{p}^{\ell ^{2},\limfunc{ext}}\left( \sigma ,\omega
\right) +\mathfrak{T}_{H,p}^{\ell ^{2},\limfunc{glob}}\left( \sigma ,\omega
\right) \right) \ \left\Vert f\right\Vert _{L^{p}\left( \sigma \right)
}\left\Vert g\right\Vert _{L^{p^{\prime }}\left( \omega \right) }\ .
\end{eqnarray*}%
In Section \ref{ex en section} below, we will show that the extended energy $%
\mathcal{E}_{p}^{\ell ^{2},\limfunc{ext}}\left( \sigma ,\omega \right) $
characteristic is controlled by the norm characteristic $\mathfrak{N}%
_{H,p}\left( \sigma ,\omega \right) $, and hence is a necessary condition.

\section{Reduction of the diagonal form by the NTV reach}

We first apply the\ clever `NTV reach' of \cite{NTV4}, which splits the
diagonal form%
\begin{equation*}
\mathsf{B}_{\func{diag}}\left( f,g\right) =\sum_{\left( I,J\right) \in 
\mathcal{P}_{\func{diag}}}\left\langle H_{\sigma }\left( \mathbf{1}%
_{I_{J}}\bigtriangleup _{I}^{\sigma }f\right) ,\bigtriangleup _{J}^{\omega
}g\right\rangle _{\omega }=\sum_{F\in \mathcal{F}}\sum_{\substack{ \left(
I,J\right) \in \mathcal{C}_{\mathcal{F}}\left( F\right) \times \mathcal{C}_{%
\mathcal{F}}\left( F\right)  \\ J\subset _{\tau }I}}\left( E_{I_{J}}^{\sigma
}\bigtriangleup _{I}^{\sigma }f\right) \left\langle H_{\sigma }\mathbf{1}%
_{I_{J}},\bigtriangleup _{J}^{\omega }g\right\rangle _{\omega },
\end{equation*}%
into a paraproduct and stopping form,%
\begin{eqnarray*}
\mathsf{B}_{\func{diag}}\left( f,g\right) &=&\sum_{F\in \mathcal{F}}\sum 
_{\substack{ \left( I,J\right) \in \mathcal{C}_{\mathcal{F}}\left( F\right)
\times \mathcal{C}_{\mathcal{F}}\left( F\right)  \\ J\subset _{\tau }I}}%
\left( E_{I_{J}}^{\sigma }\bigtriangleup _{I}^{\sigma }f\right) \left\langle
H_{\sigma }\mathbf{1}_{F},\bigtriangleup _{J}^{\omega }g\right\rangle
_{\omega } \\
&&+\sum_{F\in \mathcal{F}}\sum_{\substack{ \left( I,J\right) \in \mathcal{C}%
_{\mathcal{F}}\left( F\right) \times \mathcal{C}_{\mathcal{F}}\left(
F\right)  \\ J\subset _{\tau }I}}\left( E_{I_{J}}^{\sigma }\bigtriangleup
_{I}^{\sigma }f\right) \left\langle H_{\sigma }\mathbf{1}_{F\setminus
I_{J}},\bigtriangleup _{J}^{\omega }g\right\rangle _{\omega } \\
&\equiv &\mathsf{B}_{\limfunc{para}}\left( f,g\right) +\mathsf{B}_{\limfunc{%
stop}}\left( f,g\right) .
\end{eqnarray*}

\subsection{Paraproduct form}

Here we bound the paraproduct form,%
\begin{equation*}
\mathsf{B}_{\limfunc{para}}\left( f,g\right) =\sum_{F\in \mathcal{F}}\mathsf{%
B}_{\limfunc{para}}^{F}\left( f,g\right) =\sum_{F\in \mathcal{F}}\sum_{J\in 
\mathcal{C}_{\mathcal{F}}\left( F\right) }\left\langle \left( E_{J^{\ast
}}^{\sigma }f\right) H_{\sigma }\mathbf{1}_{F},\bigtriangleup _{J}^{\omega
}g\right\rangle _{\omega },
\end{equation*}%
for $1<p<\infty $, where $J^{\ast }=I_{J}$ where $I$ is the smallest
interval in the Haar support of $f$ for which $J$ is $\tau $-deeply embedded
in $I$. Define $\widetilde{g}=\sum_{J\in \mathcal{D}}\frac{E_{J^{\ast
}}^{\sigma }f}{E_{F}^{\sigma }f}\bigtriangleup _{J}^{\omega }g$ and note
that $\left\vert E_{J^{\ast }}^{\sigma }f\right\vert \lesssim \left\vert
E_{F}^{\sigma }f\right\vert $ since $J^{\ast }=I_{J}$ is $\func{good}$
because $I$ is in the Haar support of $f$. Then we obtain%
\begin{eqnarray*}
&&\left\vert \mathsf{B}_{\limfunc{para}}\left( f,g\right) \right\vert
=\left\vert \sum_{F\in \mathcal{F}}\mathsf{B}_{\limfunc{para}}^{F}\left(
f,g\right) \right\vert =\left\vert \sum_{F\in \mathcal{F}}\sum_{J\in 
\mathcal{C}_{\mathcal{F}}\left( F\right) }\left\langle \left( E_{J^{\ast
}}^{\sigma }f\right) H_{\sigma }\mathbf{1}_{F},\bigtriangleup _{J}^{\omega
}g\right\rangle _{\omega }\right\vert \\
&=&\left\vert \sum_{F\in \mathcal{F}}E_{F}^{\sigma }\left\vert f\right\vert
\sum_{J\in \mathcal{C}_{\mathcal{F}}\left( F\right) }\left\langle H_{\sigma }%
\mathbf{1}_{F},\frac{E_{J^{\ast }}^{\sigma }f}{E_{F}^{\sigma }\left\vert
f\right\vert }\bigtriangleup _{J}^{\omega }g\right\rangle _{\omega
}\right\vert =\left\vert \sum_{F\in \mathcal{F}}E_{F}^{\sigma }\left\vert
f\right\vert \left\langle H_{\sigma }\mathbf{1}_{F},\sum_{J\in \mathcal{C}_{%
\mathcal{F}}\left( F\right) }\bigtriangleup _{J}^{\omega }\widetilde{g}%
\right\rangle _{\omega }\right\vert \\
&=&\left\vert \sum_{F\in \mathcal{F}}E_{F}^{\sigma }\left\vert f\right\vert
\left\langle \mathbf{1}_{F}H_{\sigma }\mathbf{1}_{F},\mathsf{P}_{\mathcal{C}%
_{\mathcal{F}}\left( F\right) }^{\omega }\widetilde{g}\right\rangle _{\omega
}\right\vert =\left\vert \int_{\mathbb{R}}\sum_{F\in \mathcal{F}}\ \mathbf{1}%
_{F}H_{\sigma }\left( \mathbf{1}_{F}E_{F}^{\sigma }\left\vert f\right\vert
\right) \left( x\right) \ \mathsf{P}_{\mathcal{C}_{\mathcal{F}}\left(
F\right) }^{\omega }\widetilde{g}\left( x\right) \ d\omega \left( x\right)
\right\vert \\
&\leq &\int_{\mathbb{R}}\left( \sum_{F\in \mathcal{F}}\left\vert \mathbf{1}%
_{F}H_{\sigma }\left( \mathbf{1}_{F}E_{F}^{\sigma }\left\vert f\right\vert
\right) \left( x\right) \right\vert ^{2}\right) ^{\frac{1}{2}}\ \left(
\sum_{F\in \mathcal{F}}\left\vert \mathsf{P}_{\mathcal{C}_{\mathcal{F}%
}\left( F\right) }^{\omega }\widetilde{g}\left( x\right) \right\vert
^{2}\right) ^{\frac{1}{2}}\ d\omega \left( x\right) ,
\end{eqnarray*}%
and we can write%
\begin{equation*}
\left\vert \mathsf{B}_{\limfunc{para}}\left( f,g\right) \right\vert \leq
\left( \int_{\mathbb{R}}\left\vert \left\{ \alpha _{\mathcal{F}}\left(
F\right) \mathbf{1}_{F}H_{\sigma }\mathbf{1}_{F}\left( x\right) \right\}
_{F\in \mathcal{F}}\right\vert _{\ell ^{2}\left( \mathcal{F}\right)
}^{p}d\omega \left( x\right) \right) ^{\frac{1}{p}}\left( \int_{\mathbb{R}%
}\left( \sum_{F\in \mathcal{F}}\left\vert \mathsf{P}_{\mathcal{C}_{\mathcal{F%
}}\left( F\right) }^{\omega }\widetilde{g}\left( x\right) \right\vert
^{2}\right) ^{\frac{p^{\prime }}{2}}d\omega \left( x\right) \right) ^{\frac{1%
}{p^{\prime }}}.
\end{equation*}

We claim the following inequalities for all $1<p<\infty $,%
\begin{eqnarray}
&&\int_{\mathbb{R}}\left\vert \left\{ \alpha _{\mathcal{F}}\left( F\right) 
\mathbf{1}_{F}H_{\sigma }\mathbf{1}_{F}\left( x\right) \right\} _{F\in 
\mathcal{F}}\right\vert _{\ell ^{2}\left( \mathcal{F}\right) }^{p}d\omega
\left( x\right) \lesssim \mathfrak{T}_{H,p}^{\ell ^{2},\func{loc}}\left(
\sigma ,\omega \right) ^{p}\int_{\mathbb{R}}\left\vert \left\{ \alpha _{%
\mathcal{F}}\left( F\right) \mathbf{1}_{F}\left( x\right) \right\} _{F\in 
\mathcal{F}}\right\vert _{\ell ^{2}\left( \mathcal{F}\right) }^{p}d\sigma
\left( x\right)  \label{WTP} \\
&&\ \ \ \ \ \ \ \ \ \ \ \ \ \ \ \ \ \ \ \ \ \ \ \ \ \ \ \ \ \ \ =\mathfrak{T}%
_{H,p}^{\ell ^{2},\func{loc}}\left( \sigma ,\omega \right) ^{p}\int_{\mathbb{%
R}}\left( \sum_{F\in \mathcal{F}}\alpha _{\mathcal{F}}\left( F\right) ^{2}%
\mathbf{1}_{F}\left( x\right) \right) ^{\frac{p}{2}}d\sigma \left( x\right) ,
\notag \\
&&\int_{\mathbb{R}}\left\vert \left\{ \alpha _{\mathcal{F}}\left( F\right) 
\mathbf{1}_{F}\left( x\right) \right\} _{F\in \mathcal{F}}\right\vert _{\ell
^{2}\left( \mathcal{F}\right) }^{p}d\sigma \left( x\right) \lesssim
\sum_{F\in \mathcal{F}}\alpha _{\mathcal{F}}\left( F\right) ^{p}\left\vert
F\right\vert _{\sigma },  \notag \\
&&\text{and in particular }\int_{\mathbb{R}}\left\vert \left\{ \alpha _{%
\mathcal{F}}\left( F\right) \mathbf{1}_{F}H_{\sigma }\mathbf{1}_{F}\left(
x\right) \right\} _{F\in \mathcal{F}}\right\vert _{\ell ^{2}\left( \mathcal{F%
}\right) }^{p}d\omega \left( x\right) \lesssim \sum_{F\in \mathcal{F}}\alpha
_{\mathcal{F}}\left( F\right) ^{p}\left\vert F\right\vert _{\sigma }.  \notag
\end{eqnarray}%
Since by quasiorthogonality we have $\sum_{F\in \mathcal{F}}\alpha _{%
\mathcal{F}}\left( F\right) ^{p}\left\vert F\right\vert _{\sigma }\lesssim
\int_{\mathbb{R}}\left\vert f\right\vert ^{p}d\sigma $, the inequality (\ref%
{WTP}) will then yield%
\begin{equation*}
\int_{\mathbb{R}}\left\vert \left\{ \alpha _{\mathcal{F}}\left( F\right) 
\mathbf{1}_{F}H_{\sigma }\mathbf{1}_{F}\left( x\right) \right\} _{F\in 
\mathcal{F}}\right\vert _{\ell ^{2}\left( \mathcal{F}\right) }^{p}d\omega
\left( x\right) \lesssim \mathfrak{T}_{H,p}^{\ell ^{2},\func{loc}}\left(
\sigma ,\omega \right) ^{p}\int_{\mathbb{R}}\left\vert f\right\vert
^{p}d\sigma .
\end{equation*}%
If we now combine this inequality with the standard square function estimate,%
\begin{equation}
\int_{\mathbb{R}}\left( \sum_{F\in \mathcal{F}}\left\vert \mathsf{P}_{%
\mathcal{C}_{\mathcal{F}}\left( F\right) }^{\omega }\widetilde{g}\left(
x\right) \right\vert ^{2}\right) ^{\frac{p^{\prime }}{2}}d\omega \left(
x\right) \lesssim \int_{\mathbb{R}}\left\vert \widetilde{g}\right\vert
^{p^{\prime }}d\omega ,  \label{SSF}
\end{equation}%
and the inequality $\int_{\mathbb{R}}\left\vert \widetilde{g}\right\vert
^{p^{\prime }}d\omega \lesssim \int_{\mathbb{R}}\left\vert g\right\vert
^{p^{\prime }}d\omega $, which follows from $\left\vert E_{J^{\ast
}}^{\sigma }f\right\vert \lesssim \left\vert E_{F}^{\sigma }f\right\vert $,
we obtain%
\begin{equation*}
\left\vert \sum_{F\in \mathcal{F}}\mathsf{B}_{\limfunc{para}}^{F}\left(
f,g\right) \right\vert \lesssim \mathfrak{T}_{H,p}^{\ell ^{2},\func{loc}%
}\left( \sigma ,\omega \right) \left\Vert f\right\Vert _{L^{p}\left( \sigma
\right) }\left\Vert g\right\Vert _{L^{p^{\prime }}\left( \omega \right) },
\end{equation*}%
which is the desired estimate for the paraproduct form.

Now we turn to proving (\ref{WTP}). The local quadratic testing condition
gives the first inequality in (\ref{WTP}). Indeed, with $\mathcal{F}=\left\{
I_{i}\right\} _{i=1}^{\infty }$ and $a_{i}=\alpha _{\mathcal{F}}\left(
I_{i}\right) $, we have using that the projection $\mathsf{P}_{\mathcal{C}_{%
\mathcal{F}}\left( F\right) }^{\omega }$ is supported in $F$, 
\begin{eqnarray*}
&&\int_{\mathbb{R}}\left\vert \left\{ \alpha _{\mathcal{F}}\left( F\right) 
\mathbf{1}_{F}H_{\sigma }\mathbf{1}_{F}\left( x\right) \right\} _{F\in 
\mathcal{F}}\right\vert _{\ell ^{2}\left( \mathcal{F}\right) }^{p}d\omega
\left( x\right) \lesssim \int_{\mathbb{R}}\left\vert \left\{ \alpha _{%
\mathcal{F}}\left( F\right) \mathbf{1}_{F}\left( x\right) H_{\sigma }\mathbf{%
1}_{F}\left( x\right) \right\} _{F\in \mathcal{F}}\right\vert _{\ell
^{2}\left( \mathcal{F}\right) }^{p}d\omega \left( x\right) \\
&=&\left\Vert \left( \sum_{i=1}^{\infty }\left( a_{i}\mathbf{1}%
_{I_{i}}H_{\sigma }\mathbf{1}_{I_{i}}\right) ^{2}\right) ^{\frac{1}{2}%
}\right\Vert _{L^{p}\left( \omega \right) }^{p}\leq \mathfrak{T}_{H,p}^{\ell
^{2},\func{loc}}\left( \sigma ,\omega \right) ^{p}\left\Vert \left(
\sum_{i=1}^{\infty }\left( a_{i}\mathbf{1}_{I_{i}}\right) ^{2}\right) ^{%
\frac{1}{2}}\right\Vert _{L^{p}\left( \sigma \right) }^{p} \\
&=&\mathfrak{T}_{H,p}^{\ell ^{2},\func{loc}}\left( \sigma ,\omega \right)
^{p}\int_{\mathbb{R}}\left\vert \alpha _{\mathcal{F}}\left( F\right) \mathbf{%
1}_{F}\left( x\right) \right\vert _{\ell ^{2}\left( \mathcal{F}\right)
}^{p}d\sigma \left( x\right) .
\end{eqnarray*}%
The second inequality in (\ref{WTP}) is (\ref{The claim}) in Theorem \ref%
{using Carleson} with $\kappa =0$, and this completes the proof of (\ref{WTP}%
) and hence the control of the paraproduct form for $1<p<\infty $.

\section{Stopping form}

To control the stopping form%
\begin{equation*}
\mathsf{B}_{\limfunc{stop}}\left( f,g\right) =\sum_{F\in \mathcal{F}}\sum 
_{\substack{ \left( I,J\right) \in \mathcal{C}_{\mathcal{F}}\left( F\right)
\times \mathcal{C}_{\mathcal{F}}\left( F\right)  \\ J\subset _{\tau }I}}%
\left( E_{I_{J}}^{\sigma }\bigtriangleup _{I}^{\sigma }f\right) \left\langle
H_{\sigma }\mathbf{1}_{F\setminus I_{J}},\bigtriangleup _{J}^{\omega
}g\right\rangle _{\omega }
\end{equation*}%
we assume that the Haar supports of $f$ and $g$ are contained in $\digamma
\cap \mathcal{D}_{\func{good}}^{\limfunc{child}}$ for some large but finite
and connected subset $\digamma $ of the grid $\mathcal{D}$, and let $%
\mathcal{F}$ denote the Calder\'{o}n-Zygmund stopping times for $f$ with top 
$T$.

\begin{remark}
Note that for each $F\in \mathcal{F}$, the sums over $I$ and $J$ above
depend on $F$ through the corona $\mathcal{C}_{\mathcal{F}}\left( F\right) $%
, and the argument $\mathbf{1}_{F\setminus I_{J}}$ of the Hilbert transform
also depends on $F$.
\end{remark}

The estimate we prove here is%
\begin{equation}
\left\vert \mathsf{B}_{\limfunc{stop}}\left( f,g\right) \right\vert \lesssim 
\mathfrak{T}_{H,p}^{\func{loc}}\left( \sigma ,\omega \right) \left\Vert
f\right\Vert _{L^{p}\left( \sigma \right) }\left\Vert g\right\Vert
_{L^{p^{\prime }}\left( \omega \right) },\ \ \ \ \ 1<p<4,  \label{stop est}
\end{equation}%
for $f\in L^{p}\left( \sigma \right) \cap L^{2}\left( \sigma \right) $ and $%
g\in L^{p^{\prime }}\left( \omega \right) \cap L^{2}\left( \omega \right) $
with Haar supports in $\digamma $, and where the stopping times $\mathcal{F}$
that arise in the definition of $\mathsf{B}_{\limfunc{stop}}\left(
f,g\right) $ are the Calderon-Zygmund stopping times for $f$. The
restriction $\frac{4}{3}<p$ arises from treating the dual stopping form.
Control of the stopping form will require the most formidable estimates yet,
despite that only the simplest of characteristics is used, namely the scalar
testing characteristic $\mathfrak{T}_{H,p}^{\func{loc}}\left( \sigma ,\omega
\right) $. Recall that the inequality $\mathfrak{T}_{H,p}^{\func{loc}}\left(
\sigma ,\omega \right) \leq \mathfrak{T}_{H,p}^{\ell ^{2},\func{loc}}\left(
\sigma ,\omega \right) $ is trivial.

The key technical estimate needed for (\ref{stop est}) is the Quadratic $%
L^{p}$-Stopping Child Lemma below which controls off-diagonal terms, and
which has its roots in the `straddling' lemmas of M. Lacey in \cite[Lemmas
3.19 and 3.16]{Lac}. To prove this lemma, we will need a Corona Martingale
Comparison Principle that replaces a martingale difference sequence with
another having differences comparable in norm. Then we use a square function
generalization of Lacey's upside down corona construction in the dual tree
decomposition to finish the proof with a somewhat lengthy argument, in which
the restriction $\frac{4}{3}<p<4$ arises.

\subsection{Heuristics}

In order to explain the origin of the Quadratic $L^{p}$-Stopping Child Lemma
and Lemma \ref{final level} at the end of the paper, we point out a
significant obstacle arising from the lack of orthogonality when $p\neq 2$,
and which leads to the restriction $\frac{4}{3}<p<4$. The following short
discussion is intended to be heuristic and without precise notation and
definitions. The reader can safely skip this subsubsection and proceed
directly to the next subsubsection on the dual tree decomposition.

We now \textbf{fix} functions $f\in L^{p}\left( \sigma \right) \cap
L^{2}\left( \sigma \right) $ and $g\in L^{p^{\prime }}\left( \omega \right)
\cap L^{2}\left( \omega \right) $, along with the Calder\'{o}n-Zygmund
stopping times $\mathcal{F}$ for $f$, and of course the $p$-energy stopping
times for the measure pair $\left( \sigma ,\omega \right) $. The stopping
form 
\begin{equation*}
\mathsf{B}_{\limfunc{stop}}\left( f,g\right) =\sum_{F\in \mathcal{F}}\mathsf{%
B}_{\limfunc{stop}}\left( \mathsf{P}_{\mathcal{C}_{\mathcal{F}}\left(
F\right) }^{\sigma }f,\mathsf{P}_{\mathcal{C}_{\mathcal{F}}\left( F\right)
}^{\omega }g\right) =\sum_{F\in \mathcal{F}}\mathsf{B}_{\limfunc{stop}}^{%
\mathcal{F},F}\left( f,g\right)
\end{equation*}%
is already a `quadratic' form in the sense that it is a one parameter sum
over $F$, rather than the two parameter sum over $F$ and $G$ that appears
for example in the below form, 
\begin{equation*}
\mathsf{B}_{\func{below}}\left( f,g\right) =\mathsf{B}_{\func{below}}\left(
\sum_{F\in \mathcal{F}}\mathsf{P}_{\mathcal{C}_{\mathcal{F}}\left( F\right)
}^{\sigma }f,\sum_{G\in \mathcal{F}}\mathsf{P}_{\mathcal{C}_{\mathcal{F}%
}\left( G\right) }^{\omega }g\right) =\sum_{F,G\in \mathcal{F}}\mathsf{B}_{%
\func{below}}\left( \mathsf{P}_{\mathcal{C}_{\mathcal{F}}\left( F\right)
}^{\sigma }f,\mathsf{P}_{\mathcal{C}_{\mathcal{F}}\left( G\right) }^{\omega
}g\right) .
\end{equation*}%
The basic idea for controlling the stopping form in \cite[Lemmas 3.19 and
3.16]{Lac} when $p=2$, is to construct additional bottom-up stopping times $%
\mathcal{A}\left[ F\right] $ within each corona $\mathcal{C}_{\mathcal{F}%
}\left( F\right) $ that control energy associated with the Haar support of $%
g $, and then using certain `straddling' lemmas, to reduce control of the
resulting bilinear forms 
\begin{equation*}
\sum_{A,B\in \mathcal{A}\left[ F\right] }\mathsf{B}_{\limfunc{stop}}\left( 
\mathsf{P}_{\mathcal{C}_{\mathcal{A}\left[ F\right] }\left( A\right)
}^{\sigma }f,\mathsf{P}_{\mathcal{C}_{\mathcal{A}\left[ F\right] }\left(
B\right) }^{\omega }g\right)
\end{equation*}%
within each corona $\mathcal{C}_{\mathcal{F}}\left( F\right) $, to their
`quadratic' counterparts 
\begin{equation*}
\sum_{A\in \mathcal{A}\left[ F\right] }\mathsf{B}_{\limfunc{stop}\func{diag}%
}\left( \mathsf{P}_{\mathcal{C}_{\mathcal{A}\left[ F\right] }\left( A\right)
}^{\sigma }f,\mathsf{P}_{\mathcal{C}_{\mathcal{A}\left[ F\right] }\left(
A\right) }^{\sigma }g\right) =\sum_{A\in \mathcal{A}\left[ F\right] }\mathsf{%
B}_{\limfunc{stop}\func{diag}}^{\mathcal{A}\left[ F\right] ,A}\left(
f,g\right) ,
\end{equation*}%
where the norm of the forms $\mathsf{B}_{\limfunc{stop}\func{diag}}^{%
\mathcal{A}\left[ F\right] ,A}$ are small compared to the norm of $\mathsf{B}%
_{\limfunc{stop}}$. At this point one uses the Quasi-Orthogonality Argument
in \cite[page 6]{Lac} to control the entire sum of the iterated `quadratic'
forms, namely 
\begin{eqnarray}
&&\left\vert \sum_{F\in \mathcal{F}}\sum_{A\in \mathcal{A}\left[ F\right] }%
\mathsf{B}_{\limfunc{stop}\func{diag}}^{\mathcal{A}\left[ F\right] ,A}\left(
f,g\right) \right\vert \leq \sum_{F\in \mathcal{F}}\sum_{A\in \mathcal{A}%
\left[ F\right] }\left\Vert \mathsf{B}_{\limfunc{stop}\func{diag}}^{\mathcal{%
A}\left[ F\right] ,A}\right\Vert \left\Vert \mathsf{P}_{\mathcal{C}_{%
\mathcal{A}\left[ F\right] }\left( A\right) }^{\sigma }f\right\Vert
_{L^{2}\left( \sigma \right) }\left\Vert \mathsf{P}_{\mathcal{C}_{\mathcal{A}%
\left[ F\right] }\left( A\right) }^{\omega }g\right\Vert _{L^{2}\left(
\omega \right) }  \label{OL} \\
&\leq &\left( \sup_{\left( F,A\right) \in \mathcal{F}\times \mathcal{A}\left[
F\right] }\left\Vert \mathsf{B}_{\limfunc{stop}\func{diag}}^{\mathcal{A}%
\left[ F\right] ,A}\right\Vert \right) \left( \sum_{F\in \mathcal{F}%
}\sum_{A\in \mathcal{A}\left[ F\right] }\left\Vert \mathsf{P}_{\mathcal{C}_{%
\mathcal{A}\left[ F\right] }\left( A\right) }^{\sigma }f\right\Vert
_{L^{2}\left( \sigma \right) }^{2}\right) ^{\frac{1}{2}}\left( \sum_{F\in 
\mathcal{F}}\sum_{A\in \mathcal{A}\left[ F\right] }\left\Vert \mathsf{P}_{%
\mathcal{C}_{\mathcal{A}\left[ F\right] }\left( A\right) }^{\omega
}g\right\Vert _{L^{2}\left( \omega \right) }^{2}\right) ^{\frac{1}{2}} 
\notag \\
&\leq &\left( \sup_{\left( F,A\right) \in \mathcal{F}\times \mathcal{A}\left[
F\right] }\left\Vert \mathsf{B}_{\limfunc{stop}\func{diag}}^{\mathcal{A}%
\left[ F\right] ,A}\right\Vert \right) \left\Vert f\right\Vert _{L^{2}\left(
\sigma \right) }\left\Vert g\right\Vert _{L^{2}\left( \omega \right) }\leq
\varepsilon \left\Vert \mathsf{B}_{\limfunc{stop}}\right\Vert \left\Vert
f\right\Vert _{L^{2}\left( \sigma \right) }\left\Vert g\right\Vert
_{L^{2}\left( \omega \right) }.  \notag
\end{eqnarray}%
Then one can finish by recursion as in \cite{Lac}, or by absorption as in 
\cite{Saw7}, which formally (ignoring the nature of the smallness factor and
just inserting a small $\varepsilon >0$) becomes%
\begin{eqnarray*}
\left\Vert \mathsf{B}_{\limfunc{stop}}\right\Vert &\equiv &\frac{\left\vert 
\mathsf{B}_{\limfunc{stop}}\left( f,g\right) \right\vert }{\left\Vert
f\right\Vert _{L^{2}\left( \sigma \right) }\left\Vert g\right\Vert
_{L^{2}\left( \omega \right) }} \\
&\leq &\frac{C\left\Vert f\right\Vert _{L^{2}\left( \sigma \right)
}\left\Vert g\right\Vert _{L^{2}\left( \omega \right) }+\left\vert
\sum_{F\in \mathcal{F}}\sum_{A\in \mathcal{A}\left[ F\right] }\mathsf{B}_{%
\limfunc{stop}\func{diag}}^{\mathcal{A}\left[ F\right] ,A}\left( f,g\right)
\right\vert }{\left\Vert f\right\Vert _{L^{2}\left( \sigma \right)
}\left\Vert g\right\Vert _{L^{2}\left( \omega \right) }} \\
&\leq &\frac{C\left\Vert f\right\Vert _{L^{2}\left( \sigma \right)
}\left\Vert g\right\Vert _{L^{2}\left( \omega \right) }+\varepsilon
\left\Vert \mathsf{B}_{\limfunc{stop}}\right\Vert \left\Vert f\right\Vert
_{L^{2}\left( \sigma \right) }\left\Vert g\right\Vert _{L^{2}\left( \omega
\right) }}{\left\Vert f\right\Vert _{L^{2}\left( \sigma \right) }\left\Vert
g\right\Vert _{L^{2}\left( \omega \right) }}=C+\varepsilon \left\Vert 
\mathsf{B}_{\limfunc{stop}}\right\Vert ,
\end{eqnarray*}%
which yields $\left\Vert \mathsf{B}_{\limfunc{stop}}\right\Vert \leq \frac{C%
}{1-\varepsilon }$.

Unfortunately, the inequality (\ref{OL}) fails to generalize for all $%
1<p<\infty $, and we must arrange matters so as to avoid its use when $p\neq
2$. This will be accomplished by applying the above heuristic to more
general forms 
\begin{equation*}
\mathsf{B}_{\limfunc{stop}}^{\mathcal{Q}}\left( f,g\right) =\sum_{F\in 
\mathcal{F}}\sum_{Q\in \mathcal{Q}\left[ F\right] }\mathsf{B}_{\limfunc{stop}%
}\left( \mathsf{P}_{\mathcal{C}_{\mathcal{Q}\left[ F\right] }\left( Q\right)
}^{\sigma }f,\mathsf{P}_{\mathcal{C}_{\mathcal{Q}\left[ F\right] }\left(
Q\right) }^{\omega }g\right) =\sum_{F\in \mathcal{F}}\sum_{Q\in \mathcal{Q}%
\left[ F\right] }\mathsf{B}_{\limfunc{stop}}^{\mathcal{Q}\left[ F\right]
,Q}\left( f,g\right) ,
\end{equation*}%
in which a larger collection of stopping times $\mathcal{Q\supset F}$ is
used in order to introduce an additional layer of stopping times $\mathcal{Q}%
\left[ F\right] =\mathcal{Q}\cap \mathcal{C}_{\mathcal{F}}\left( F\right) $
within each corona $\mathcal{C}_{\mathcal{F}}\left( F\right) $. Then we can
interpret the iterated quadratic form%
\begin{equation*}
\sum_{F\in \mathcal{F}}\sum_{Q\in \mathcal{Q}\left[ F\right] }\sum_{A\in 
\mathcal{A}\left[ Q\right] }\mathsf{B}_{\limfunc{stop}\func{diag}}^{\mathcal{%
A}\left[ Q\right] ,A}\left( f,g\right) =\sum_{F\in \mathcal{F}}\sum_{Q\in 
\mathcal{Q}\left[ F\right] }\sum_{A\in \mathcal{A}\left[ Q\right] }\mathsf{B}%
_{\limfunc{stop}\func{diag}}\left( \mathsf{P}_{\mathcal{C}_{\mathcal{A}\left[
Q\right] }\left( A\right) }^{\sigma }f,\mathsf{P}_{\mathcal{C}_{\mathcal{A}%
\left[ Q\right] }\left( A\right) }^{\omega }g\right)
\end{equation*}%
as simply the form%
\begin{eqnarray}
&&  \label{Q'} \\
&&\mathsf{B}_{\limfunc{stop}}^{\mathcal{Q}\circ \mathcal{A}}\left(
f,g\right) =\sum_{F\in \mathcal{F}}\sum_{Q^{\prime }\in \left( \mathcal{Q}%
\circ \mathcal{A}\right) \left[ F\right] }\mathsf{B}_{\limfunc{stop}%
}^{\left( \mathcal{Q}\circ \mathcal{A}\right) \left[ F\right] ,Q^{\prime
}}\left( f,g\right) =\sum_{F\in \mathcal{F}}\sum_{Q^{\prime }\in \left( 
\mathcal{Q}\circ \mathcal{A}\right) \left[ F\right] }\mathsf{B}_{\limfunc{%
stop}}^{\left( \mathcal{Q}\circ \mathcal{A}\right) \left[ F\right]
,Q^{\prime }}\left( \mathsf{P}_{\mathcal{C}_{\left( \mathcal{Q}\circ 
\mathcal{A}\right) \left[ F\right] }\left( Q^{\prime }\right) }^{\sigma }f,%
\mathsf{P}_{\mathcal{C}_{\left( \mathcal{Q}\circ \mathcal{A}\right) \left[ F%
\right] }\left( Q^{\prime }\right) }^{\omega }g\right) ,  \notag \\
&&\ \ \ \ \ \ \ \ \ \ \ \ \ \ \ \ \ \ \ \ \ \ \ \ \ \text{where }\left( 
\mathcal{Q}\circ \mathcal{A}\right) \left[ F\right] =\bigcup_{Q\in \mathcal{Q%
}\left[ F\right] }\mathcal{A}\left[ Q\right] .  \notag
\end{eqnarray}

This will be shown below to circumvent use of the Quasi-Orthogonality
Argument in the restricted range $1<p<4$ (which is then further restricted
to $\frac{4}{3}<p$ when treating the dual stopping form). See also the
section on Concluding Remarks at the end of the paper for more discussion on
this point. We will be using special stopping collections $\mathcal{Q}$ and $%
\mathcal{A}$ constructed using the upside down corona of Lacey \cite{Lac},
but adapted to $p\neq 2$, and we now turn to this construction.

\subsection{Dual tree decomposition}

To control the stopping form $\mathsf{B}_{\limfunc{stop}}\left( f,g\right) $%
, we need to introduce further corona decompositions within each corona $%
\mathcal{C}_{\mathcal{F}}\left( F\right) $ to which we can apply the $L^{p}$%
-Stopping Child\ Lemma. These coronas will be associated to stopping
intervals $\mathcal{A}=\mathcal{A}\left[ F\right] \subset \mathcal{C}_{%
\mathcal{F}}\left( F\right) $, whose construction, following \cite{Saw7},
uses a dual tree decomposition originating with M. Lacey in \cite{Lac}.
However, our stopping criteria will be different when $p\neq 2$, and the
arguments more involved.

\begin{definition}
Let $\mathcal{T}$ be a tree with root $o$.

\begin{enumerate}
\item Let $P\left( \alpha \right) \equiv \left\{ \beta \in \mathcal{T}:\beta
\succeq \alpha \right\} $ and $S\left( \alpha \right) \equiv \left\{ \beta
\in \mathcal{T}:\beta \preceq \alpha \right\} $ denote the predessor and
successor sets of $\alpha \in \mathcal{T}$.

\item A \emph{geodesic} $\mathfrak{g}$ is a maximal linearly ordered subset
of $\mathcal{T}$. A finite geodesic $\mathfrak{g}$ is an interval $\mathfrak{%
g}=\left[ \alpha ,\beta \right] =P\left( \beta \right) \setminus S\left(
\alpha \right) $, and an infinite geodesic is an interval $\mathfrak{g}=%
\mathfrak{g}\setminus P\left( \alpha \right) $ for some $\alpha \in 
\mathfrak{g}$. Intervals $\left( \alpha ,\beta \right) $, $\left( \alpha
,\beta \right] $ and $\left[ \alpha ,\beta \right] $ are defined similarly.

\item A \emph{stopping time}\footnote{%
This different definition of stopping time used here, is that used in the
theory of trees, but should cause no confusion with the other definition we
use elsewhere, that a stopping time is any subset of $\mathcal{T}$.} $T$ for
a tree $\mathcal{T}$ is a subset $T\subset \mathcal{T}$ such that 
\begin{equation*}
S\left( \beta \right) \cap S\left( \beta ^{\prime }\right) =\emptyset \text{
for all }\beta ,\beta ^{\prime }\in T\text{ with }\beta \neq \beta ^{\prime
}.
\end{equation*}

\item A sequence $\left\{ T_{n}\right\} _{n=0}^{N}$ of stopping times $T_{n}$
is \emph{decreasing} if, for every $\beta \in T_{n+1}$ with $0\leq n<N$,
there is $\beta ^{\prime }\in T_{n}$ such that $S\left( \beta \right)
\subset S\left( \beta ^{\prime }\right) $. We think of such a sequence as
getting further from the root as $n$ increases.

\item For $T$ a stopping time in $\mathcal{T}$ and $\alpha \in \mathcal{T}$,
we define 
\begin{equation*}
\left[ T,\alpha \right) \equiv \bigcup_{\beta \in T}\left[ \beta ,\alpha
\right) ,
\end{equation*}%
where the interval $\left[ \beta ,\alpha \right) =\emptyset $ unless $\beta
\prec \alpha $. In the case $\left[ T,\alpha \right) =\emptyset $, we write $%
\alpha \preceq T$, and in the case $\left[ T,\alpha \right) \not=\emptyset $%
, we write $\alpha \succ T$. The set $\left[ T,\alpha \right) $ can be
thought of as the set of points in the tree $\mathcal{T}$ that `lie between' 
$T$ and $\alpha $ but are strictly less than $\alpha $. We also define $%
\left( T,\alpha \right) $, $\left( T,\alpha \right] $ and $\left[ T,\alpha %
\right] $ in similar fashion.

\item For any $\alpha \in \mathcal{T}$, we define the set of its \emph{%
children} $\mathfrak{C}_{T}\left( \alpha \right) $ to consist of the \emph{%
maximal} elements $\beta \in \mathcal{T}$ such that $\beta \prec \alpha $.
\end{enumerate}
\end{definition}

\bigskip

In the finite tree pictured below, downward arrows point to small tree
elements, and we have for example,%
\begin{eqnarray*}
&&\gamma \prec \beta \prec \alpha ,\ \ \ P\left( \gamma \right) =\left\{
\gamma ,\beta ,\alpha ,o\right\} =\left[ \gamma ,o\right] ,\ \ \ \left(
\gamma ^{\prime \prime \prime },o\right] =\left\{ \beta ^{\prime \prime
},\alpha ,o\right\} , \\
&&\ \ \ \ \ \ \ \ \ \ \ \ \ \ \ \text{and }S\left( \alpha \right) =\left\{
\alpha ,\beta ,\beta ^{\prime },\beta ^{\prime \prime },\gamma ,\gamma
^{\prime },\gamma ^{\prime \prime },\gamma ^{\prime \prime \prime }\right\} .
\end{eqnarray*}%
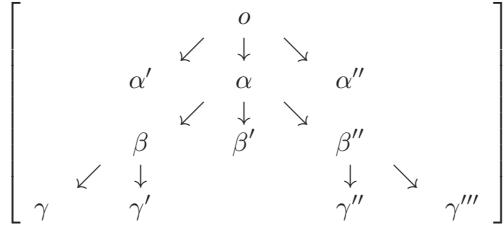
\begin{equation*}
\left[ 
\begin{array}{ccccccccc}
&  &  &  & o &  &  &  &  \\ 
&  &  & \swarrow & \downarrow & \searrow &  &  &  \\ 
&  & \alpha ^{\prime } &  & \alpha &  & \alpha ^{\prime \prime } &  &  \\ 
&  &  & \swarrow & \downarrow & \searrow &  &  &  \\ 
&  & \beta &  & \beta ^{\prime } &  & \beta ^{\prime \prime } &  &  \\ 
& \swarrow & \downarrow &  &  &  & \downarrow & \searrow &  \\ 
\gamma &  & \gamma ^{\prime } &  &  &  & \gamma ^{\prime \prime } &  & 
\gamma ^{\prime \prime \prime }%
\end{array}%
\right]
\end{equation*}

Lemma \ref{lem I*} below will create a set stopping times for any function $%
\nu :\mathcal{T}\rightarrow L^{p}\left( \ell ^{2}\left( \mathcal{T}\right)
;\omega \right) $ with finite support. It might be useful to point out the
application we have in mind. Namely, we will take $\mathcal{T}$ to be a
connected subset of the grid $\mathcal{D}$, where the root $o$ is a dyadic
interval $T\in \mathcal{T}$, and where $J\preccurlyeq I$ is defined to hold
if and only if $J\subset I$ (i.e. the symbols $\preccurlyeq $ and $\subset $
are consistent). For any subset $\Lambda \subset \mathcal{T}$, we will
consider the sequence-valued function $\nu _{\Lambda }:\mathcal{T}%
\rightarrow L^{p}\left( \ell ^{2}\left( \mathcal{T}\right) ;\omega \right) $
defined by 
\begin{equation}
\nu _{\Lambda }\left( I\right) \equiv \left\{ 
\begin{array}{ccc}
\left\{ \bigtriangleup _{J}^{\omega }Z\left( x\right) \mathbf{1}_{\left\{
I\right\} }\left( J\right) \right\} _{J\in \mathcal{T}} & \text{ if } & I\in
\Lambda \\ 
0 & \text{ if } & I\notin \Lambda%
\end{array}%
\right. \ ,  \label{def mu}
\end{equation}%
i.e. $\nu _{\Lambda }\left( I\right) $ is the sequence $\left\{ f_{J}\left(
x\right) \right\} _{J\in \mathcal{T}}$ where $f_{J}\left( x\right) =\left\{ 
\begin{array}{ccc}
\bigtriangleup _{J}^{\omega }Z\left( x\right) & \text{ if } & J=I\in \Lambda
\\ 
0 & \text{ if } & J\neq I%
\end{array}%
\right. $ and is the zero sequence otherwise. We define the dual integration
operator $I^{\ast }$ on $\nu $ by $I^{\ast }\nu \left( \alpha \right) \equiv
\sum_{\beta \in \mathcal{T}:\ \beta \prec \alpha }\nu \left( \beta \right) $%
. More generally, for any subset $\Omega $ of the tree $\mathcal{T}$, we
define%
\begin{equation*}
I^{\Omega }\nu \equiv \sum_{\beta \in \Omega }\nu \left( \beta \right) ,
\end{equation*}%
in which case $I^{\ast }\nu \left( \alpha \right) =I^{S\left( \alpha \right)
}\nu $. Note that in the application setting discussed above, we have%
\begin{equation*}
\left\Vert I^{\ast }\nu _{\Lambda }\left( I\right) \right\Vert _{L^{p}\left(
\ell ^{2};\omega \right) }^{p}=\left\Vert \left( \sum_{J\in \Lambda :\
J\subset I}\left\vert \bigtriangleup _{J}^{\omega }Z\right\vert ^{2}\right)
^{\frac{1}{2}}\right\Vert _{L^{p}\left( \omega \right) }^{p}=\int_{\mathbb{R}%
}\left( \sum_{J\in \Lambda :\ J\subset I}\left\vert \bigtriangleup
_{J}^{\omega }Z\left( x\right) \right\vert ^{2}\right) ^{\frac{p}{2}}d\omega
\left( x\right) .
\end{equation*}%
Here is the dual stopping time lemma that abstracts and extends that of M.
Lacey in \cite{Lac} to $p\neq 2$. We state this lemma for a tree with
bounded numbers of children, but we will only use the case of a dyadic tree,
which has at most two children.

\begin{lemma}
\label{lem I*}Let $\left( \mathcal{T},\preccurlyeq \right) $ be a tree with
root $o$ and $M\equiv \sup_{\alpha \in \mathcal{T}}\#\mathfrak{C}_{\mathcal{T%
}}\left( \alpha \right) <\infty $, and suppose $\nu :\mathcal{T}\rightarrow
L^{p}\left( \ell ^{2}\left( \mathcal{T}\right) ;\omega \right) $ is
nontrivial with finite support, and that $T_{0}$ is the stopping time
consisting of the minimal tree elements in the support of $\nu $. Fix $%
\Gamma >1$\footnote{%
This is not necessarily the same $\Gamma $ as used in the Calder\'{o}%
n-Zygmund stopping time construction (\ref{energy stop crit}).}. If there is
no element $\alpha \in \mathcal{T}$ with 
\begin{equation*}
\left\Vert I^{\ast }\nu \left( \alpha \right) \right\Vert _{L^{p}\left( \ell
^{2};\omega \right) }^{p}>\Gamma ^{p}\sum_{\beta \in \mathcal{T}:\ \beta
\prec \alpha }\left\Vert I^{\ast }\nu \left( \beta \right) \right\Vert
_{L^{p}\left( \ell ^{2};\omega \right) }^{p},
\end{equation*}%
we say the tree is $\Gamma $-irreducible. Otherwise, there is a unique
increasing sequence $\left\{ T_{n}\right\} _{n=0}^{N+1}$, with $%
T_{N+1}=\left\{ o\right\} $, of stopping times $T_{n}$ such that for all $%
n\in \mathbb{N}$ with $n\leq N$, 
\begin{eqnarray}
\left\Vert I^{\ast }\nu \left( \alpha \right) \right\Vert _{L^{p}\left( \ell
^{2};\omega \right) }^{p} &>&\Gamma ^{p}\sum_{\beta \in T_{n-1}:\ \beta
\prec \alpha }\left\Vert I^{\ast }\nu \left( \beta \right) \right\Vert
_{L^{p}\left( \ell ^{2};\omega \right) }^{p},\ \ \ \ \ \text{for all }\alpha
\in T_{n}\ ;  \label{dec corona} \\
\left\Vert I^{\ast }\nu \left( \gamma \right) \right\Vert _{L^{p}\left( \ell
^{2};\omega \right) }^{p} &\leq &\Gamma ^{p}\sum_{\beta \in T_{n-1}:\ \beta
\prec \gamma }\left\Vert I^{\ast }\nu \left( \beta \right) \right\Vert
_{L^{p}\left( \ell ^{2};\omega \right) }^{p},\text{ \ \ \ \ for all }\gamma
\in \left[ T_{n-1},\alpha \right) \text{ with }\alpha \in T_{n}\ ;  \notag \\
\left\Vert I^{\ast }\nu \left( o\right) \right\Vert _{L^{p}\left( \ell
^{2};\omega \right) }^{p} &\leq &\Gamma ^{p}\sum_{\beta \in T_{N}:\ \beta
\prec \gamma }\left\Vert I^{\ast }\nu \left( \beta \right) \right\Vert
_{L^{p}\left( \ell ^{2};\omega \right) }^{p}\ .  \notag
\end{eqnarray}%
Moreover, this unique sequence $\left\{ T_{n}\right\} _{n=0}^{N+1}$
satisfies the following inequalities,%
\begin{eqnarray}
&&  \label{dec corona small} \\
\frac{\left\Vert I^{\left( T_{n-1},\alpha \right) }\nu \right\Vert
_{L^{p}\left( \ell ^{2};\omega \right) }^{p}}{\sum_{\beta \in T_{n-1}:\
\beta \prec \alpha }\left\Vert I^{\ast }\nu \left( \beta \right) \right\Vert
_{L^{p}\left( \ell ^{2};\omega \right) }^{p}} &\leq &C_{p}\left( \Gamma
^{p}-1\right) ^{\natural },\ \ \ \ \ \text{for all }\alpha \in T_{n}\ ,\ \ \
\ \ 1\leq n\leq N+1,  \notag \\
\frac{\left\Vert I^{\left( T_{n-1},\gamma \right) }\nu \right\Vert
_{L^{p}\left( \ell ^{2};\omega \right) }^{p}}{\sum_{\beta \in T_{n-1}:\
\beta \prec \gamma }\left\Vert I^{\ast }\nu \left( \beta \right) \right\Vert
_{L^{p}\left( \ell ^{2};\omega \right) }^{p}} &\leq &C_{p}\left( \Gamma
^{p}-1\right) ^{\natural },\ \ \ \ \ \text{for all }\gamma \in \left(
S\left( \alpha \right) \setminus \left\{ \alpha \right\} \right) \setminus
\bigcup_{\beta \in T_{n-1}}S\left( \beta \right) ,  \notag
\end{eqnarray}%
where $\Theta ^{\natural }\equiv \max \left\{ \Theta ^{\frac{p}{2}},\Theta
\right\} $ for any $\Theta >0$. When $p\geq 2$, we may drop the constant $%
C_{p}$ and the $\natural $ in the exponent.
\end{lemma}

\begin{proof}
If $T_{n}$ is already defined, let $T_{n+1}$ consist of all minimal points $%
\alpha \in \mathcal{T}$ satisfying 
\begin{equation}
\left\Vert I^{\ast }\nu \left( \alpha \right) \right\Vert _{L^{p}\left( \ell
^{2};\omega \right) }^{p}>\Gamma ^{p}\sum_{\beta \in T_{n}:\ \beta \succ
\alpha }\left\Vert I^{\ast }\nu \left( \beta \right) \right\Vert
_{L^{p}\left( \ell ^{2};\omega \right) }^{p}\ ,  \label{stop crit}
\end{equation}%
provided at least one such point $\alpha $ exists. If not then set $N=n$ and
define $T_{N+1}\equiv \left\{ o\right\} $. It is easy to see that the
sequence $\left\{ T_{n}\right\} _{n=0}^{N+1}$ so constructed is an
increasing sequence of stopping times that satisfies (\ref{dec corona}), and
is unique with these properties.

Note that for $q\geq 1$, we have%
\begin{eqnarray}
\left( a+b\right) ^{q} &\geq &a^{q}+b^{q}\text{ for }a,b>0\text{ and }q\geq
1,  \label{q1} \\
\text{i.e. }b^{q} &\leq &\left( a+b\right) ^{q}-a^{q}\text{ for }a,b>0\text{
and }q\geq 1.  \notag
\end{eqnarray}%
Thus for $p\geq 2$, we have $q=\frac{p}{2}\geq 1$ and the first line in (\ref%
{dec corona small}) holds since using Lemma \ref{disjoint supp}, 
\begin{eqnarray*}
&&\left\Vert I^{\left( T_{n-1},\alpha \right) }\nu \right\Vert _{L^{p}\left(
\ell ^{2};\omega \right) }^{p}=\sum_{\gamma \in \mathfrak{C}_{\mathcal{T}%
}\left( \alpha \right) }\left\Vert I^{\left( T_{n-1},\gamma \right) }\nu
\right\Vert _{L^{p}\left( \ell ^{2};\omega \right) }^{p}=\sum_{\gamma \in 
\mathfrak{C}_{\mathcal{T}}\left( \alpha \right) }\left\Vert I^{\ast }\nu
\left( \gamma \right) -\sum_{\beta \in T_{n-1}:\ \beta \preccurlyeq \gamma
}I^{\ast }\nu \left( \beta \right) \right\Vert _{L^{p}\left( \ell
^{2};\omega \right) }^{p} \\
&\leq &\sum_{\gamma \in \mathfrak{C}_{\mathcal{T}}\left( \alpha \right)
}\left\Vert I^{\ast }\nu \left( \gamma \right) \right\Vert _{L^{p}\left(
\ell ^{2};\omega \right) }^{p}-\sum_{\gamma \in \mathfrak{C}_{\mathcal{T}%
}\left( \alpha \right) }\left\Vert \sum_{\beta \in T_{n-1}:\ \beta
\preccurlyeq \gamma }I^{\ast }\nu \left( \beta \right) \right\Vert
_{L^{p}\left( \ell ^{2};\omega \right) }^{p}\ \ \ \ \ \ \ \ \ \ \ \ \ \ \ \
\ \ \ \ \text{by (\ref{q1})}, \\
&\leq &\sum_{\gamma \in \mathfrak{C}_{\mathcal{T}}\left( \alpha \right)
}\left( \Gamma ^{p}\sum_{\beta \in T_{n-1}:\ \beta \prec \gamma }\left\Vert
I^{\ast }\nu \left( \beta \right) \right\Vert _{L^{p}\left( \ell ^{2};\omega
\right) }^{p}-\sum_{\beta \in T_{n-1}:\ \beta \preccurlyeq \gamma
}\left\Vert I^{\ast }\nu \left( \beta \right) \right\Vert _{L^{p}\left( \ell
^{2};\omega \right) }^{p}\right) \\
&\leq &\sum_{\gamma \in \mathfrak{C}_{\mathcal{T}}\left( \alpha \right)
}\left( \Gamma ^{p}-1\right) \sum_{\beta \in T_{n-1}:\ \beta \preccurlyeq
\gamma }\left\Vert I^{\ast }\nu \left( \beta \right) \right\Vert
_{L^{p}\left( \ell ^{2};\omega \right) }^{p}=\left( \Gamma ^{p}-1\right)
\sum_{\beta \in T_{n-1}:\ \beta \prec \alpha }\left\Vert I^{\ast }\nu \left(
\beta \right) \right\Vert _{L^{p}\left( \ell ^{2};\omega \right) }^{p}.
\end{eqnarray*}

In the case $1<p<2$, we must work harder since (\ref{q1}) fails when $q<1$.
In fact we now use Lemma \ref{q<1} from the section on preliminaries to show
that the first line in (\ref{dec corona small}) holds. Indeed, from Lemmas %
\ref{disjoint supp}\ and \ref{q<1}, and using that the maximal elements in $%
\left( T_{n-1},\alpha \right) $ are the children $\gamma \in \mathfrak{C}_{%
\mathcal{T}}\left( \alpha \right) $, we have%
\begin{eqnarray*}
&&\frac{\left\Vert I^{\left( T_{n-1},\alpha \right) }\nu \right\Vert
_{L^{p}\left( \ell ^{2};\omega \right) }^{p}}{\left\Vert \sum_{\beta \in
T_{n-1}:\ \beta \prec \alpha }I^{\ast }\nu \left( \beta \right) \right\Vert
_{L^{p}\left( \ell ^{2};\omega \right) }^{p}}\leq \sum_{\gamma \in \mathfrak{%
C}_{\mathcal{T}}\left( \alpha \right) }\frac{\left\Vert I^{\left(
T_{n-1},\gamma \right) }\nu \right\Vert _{L^{p}\left( \ell ^{2};\omega
\right) }^{p}}{\left\Vert \sum_{\beta \in T_{n-1}:\ \beta \preccurlyeq
\gamma }I^{\ast }\nu \left( \beta \right) \right\Vert _{L^{p}\left( \ell
^{2};\omega \right) }^{p}} \\
&=&\sum_{\gamma \in \mathfrak{C}_{\mathcal{T}}\left( \alpha \right) }\frac{%
\left\Vert I^{\ast }\nu \left( \gamma \right) -\sum_{\beta \in T_{n-1}:\
\beta \preccurlyeq \gamma }I^{\ast }\nu \left( \beta \right) \right\Vert
_{L^{p}\left( \ell ^{2};\omega \right) }^{p}}{\left\Vert \sum_{\beta \in
T_{n-1}:\ \beta \preccurlyeq \gamma }I^{\ast }\nu \left( \beta \right)
\right\Vert _{L^{p}\left( \ell ^{2};\omega \right) }^{p}} \\
&\leq &C_{p}\sum_{\gamma \in \mathfrak{C}_{\mathcal{T}}\left( \alpha \right)
}\left( \frac{\left\Vert I^{\ast }\nu \left( \gamma \right) \right\Vert
_{L^{p}\left( \ell ^{2};\omega \right) }^{p}}{\left\Vert \sum_{\beta \in
T_{n-1}:\ \beta \preccurlyeq \gamma }I^{\ast }\nu \left( \beta \right)
\right\Vert _{L^{p}\left( \ell ^{2};\omega \right) }^{p}}-1\right)
^{\natural },
\end{eqnarray*}%
by (\ref{suff}), which equals, again using Lemma \ref{disjoint supp}, 
\begin{equation*}
C_{p}\sum_{\gamma \in \mathfrak{C}_{\mathcal{T}}\left( \alpha \right)
}\left( \frac{\left\Vert I^{\ast }\nu \left( \gamma \right) \right\Vert
_{L^{p}\left( \ell ^{2};\omega \right) }^{p}}{\sum_{\beta \in T_{n-1}:\
\beta \preccurlyeq \gamma }\left\Vert I^{\ast }\nu \left( \beta \right)
\right\Vert _{L^{p}\left( \ell ^{2};\omega \right) }^{p}}-1\right)
^{\natural }\leq MC_{p}\left( \Gamma ^{p}-1\right) ^{\natural
}=C_{p,M}\left( \Gamma ^{p}-1\right) ^{\natural }.
\end{equation*}

The same arguments prove the second line in (\ref{dec corona small}), since $%
\gamma \in \left( S\left( \alpha \right) \setminus \left\{ \alpha \right\}
\right) \setminus \bigcup_{\beta \in T_{n-1}}S\left( \beta \right) $ was 
\emph{not} chosen by the stopping criterion in the first line of (\ref{dec
corona}), and hence%
\begin{equation*}
\left\Vert I^{\ast }\nu \left( \gamma \right) \right\Vert _{L^{p}\left( \ell
^{2};\omega \right) }^{p}\leq \Gamma ^{p}\sum_{\beta \in T_{n-1}:\ \beta
\prec \gamma }\left\Vert I^{\ast }\nu \left( \beta \right) \right\Vert
_{L^{p}\left( \ell ^{2};\omega \right) }^{p}.
\end{equation*}%
We can now proceed as above, and this completes the proof of Lemma \ref{lem
I*}.
\end{proof}

\subsection{Corona Martingale Comparison Principle}

Suppose $\mu $ is a locally finite positive Borel measure on the real line $%
\mathbb{R}$, that $\mathcal{L}\subset \mathcal{D}$ with top interval $T$,
and that $\left\{ \mathcal{C}_{\mathcal{L}}\left( L\right) \right\} _{L\in 
\mathcal{L}}$ is the associated collection of coronas. For each $k\in 
\mathbb{N}$ define%
\begin{equation*}
\mathsf{P}_{\mathcal{L},k}^{\mu }\equiv \sum_{L\in \mathfrak{C}_{\mathcal{L}%
}^{\left( k\right) }\left( T\right) }\mathsf{P}_{\mathcal{C}_{\mathcal{L}%
}\left( L\right) }^{\mu }\text{ where }\mathsf{P}_{\mathcal{C}_{\mathcal{L}%
}\left( F\right) }^{\mu }\equiv \sum_{I\in \mathcal{C}_{\mathcal{L}}\left(
F\right) }\bigtriangleup _{I}^{\mu }.
\end{equation*}%
Then as shown in \cite[see the section on square functions and vector-valued
inequalities]{SaWi}, the sequence $\left\{ \mathsf{P}_{\mathcal{L},k}^{\mu
}g\right\} _{k\in \mathbb{N}}$ is a martingale difference sequence of an $%
L^{p}$ bounded martingale for any $g\in L^{p}\left( \mu \right) $. We will
refer to such a martingale on the real line as an $L^{p}$\emph{-}$\mathcal{L}
$\emph{\ martingale}. We define $\mathfrak{C}_{\mathcal{L}}^{\left( \ell
\right) }\left( F\right) $ to be the set of $\ell $-grandchildren of $F$ in
the tree $\mathcal{L}$, and%
\begin{equation*}
\mathcal{C}_{\mathcal{L}}^{\left( \ell \right) }\left( F\right) \equiv
\bigcup_{F^{\prime }\in \mathfrak{C}_{\mathcal{L}}^{\left( \ell \right)
}\left( F\right) }\mathcal{C}_{\mathcal{L}}\left( F^{\prime }\right) \text{
and }\mathcal{C}_{\mathcal{L}}^{\left[ m\right] }\left( F\right) \equiv
\bigcup_{\ell =0}^{m}\mathcal{C}_{\mathcal{L}}^{\left( \ell \right) }\left(
F\right) .
\end{equation*}%
The comparison principle for corona martingales is a transplantation theorem
relying on the structure of corona martingales for its success. Variants of
this type of comparison principle for martingale differences arose in work
of J. Zinn almost four decades ago \cite{Zin}. See also Burkholder \cite{Bur}
for related inequalities.

\begin{proposition}[Corona Martingale Comparison Principle]
\label{CMCP}Let $1<p<\infty $, and $m,N\in \mathbb{N}$ with $m>N$.\ Let $\mu 
$ be a locally finite positive Borel measure on $\mathbb{R}$, let $\mathcal{L%
}\subset \mathcal{D}$ with top interval $T$, and suppose that $\left\{ 
\mathsf{P}_{\mathcal{L},k}g\right\} _{k\in \mathbb{N}}$ and $\left\{ \mathsf{%
P}_{\mathcal{L},k}b\right\} _{k\in \mathbb{N}}$ are martingale difference
sequences of $L^{p}$-$\mathcal{L}$ martingales with $\int_{T}gd\mu
=\int_{T}bd\mu =0$. Suppose moreover, that $\mathsf{P}_{\mathcal{C}_{%
\mathcal{L}}^{\left( k\right) }\left( T\right) }^{\mu }g=0$ for $0\leq k\leq
m-1$. Then we have 
\begin{equation}
\left\Vert \left\vert \left\{ \mathsf{P}_{\mathcal{L},k}^{\mu }g\right\}
_{k\in \mathbb{N}}\right\vert _{\ell ^{2}}\right\Vert _{L^{p}\left( \mu
\right) }\lesssim mNM_{\mathcal{L}}^{\left( m,N\right) }\left( g,b\right) \
\max_{0\leq s\leq N}\left\Vert \left\vert \left\{ \mathsf{P}_{\mathcal{L}%
,k+s}^{\mu }b\right\} _{k\in \mathbb{N}}\right\vert _{\ell ^{2}}\right\Vert
_{L^{p}\left( \mu \right) },  \label{part 2}
\end{equation}%
where%
\begin{equation}
M_{\mathcal{L}}^{\left( m,N\right) }\left( g,b\right) \equiv \sup_{L\in 
\mathcal{L}}\frac{\left\Vert \mathsf{P}_{\bigcup_{K\in \mathfrak{C}_{%
\mathcal{L}}^{\left( m\right) }\left( L\right) }\mathcal{D}\left( K\right)
}^{\mu }g\right\Vert _{L^{p}\left( \mu \right) }}{\left\Vert \mathsf{P}_{%
\mathcal{C}_{\mathcal{L}}^{\left[ N\right] }\left( L\right) }^{\mu
}b\right\Vert _{L^{p}\left( \mu \right) }}.  \label{M char}
\end{equation}
\end{proposition}

\begin{remark}
A crucial feature of the proof, peculiar to corona martingales, is that $%
\mathsf{P}_{\mathcal{C}_{\mathcal{L}}^{\left[ N\right] }\left( L\right)
}^{\mu }b$ is constant on the support of $\mathsf{P}_{\bigcup_{K\in 
\mathfrak{C}_{\mathcal{L}}^{\left( m\right) }\left( L\right) }\mathcal{D}%
\left( K\right) }^{\mu }g$ when $m>N$.
\end{remark}

We will frequently use the square function in Theorem \ref{square thm} in
the proof of Proposition \ref{CMCP}, which in particular implies that 
\begin{eqnarray*}
\left\Vert \mathsf{P}_{\mathcal{C}_{\mathcal{L}}^{\left( m\right) }\left(
L\right) }^{\mu }f\right\Vert _{L^{p}\left( \mu \right) } &\approx
&\left\Vert \left\vert \left\{ \mathsf{P}_{\mathcal{C}_{\mathcal{L}}^{\left(
m\right) }\left( L^{\prime }\right) }^{\mu }f\right\} _{L^{\prime }\in 
\mathfrak{C}_{\mathcal{L}}^{\left( m\right) }\left( L\right) }\right\vert
_{\ell ^{2}}\right\Vert _{L^{p}\left( \mu \right) }=\left\Vert \left(
\sum_{L^{\prime }\in \mathfrak{C}_{\mathcal{L}}^{\left( m\right) }\left(
L\right) }\left\vert \mathsf{P}_{\mathcal{C}_{\mathcal{L}}^{\left( m\right)
}\left( L^{\prime }\right) }^{\mu }f\right\vert ^{2}\right) ^{\frac{1}{2}%
}\right\Vert _{L^{p}\left( \mu \right) } \\
&\approx &\left\Vert \left\vert \left\{ \bigtriangleup _{I}^{\mu }f\right\}
_{I\in \mathcal{C}_{\mathcal{L}}^{\left( m\right) }\left( L\right)
}\right\vert _{\ell ^{2}}\right\Vert _{L^{p}\left( \mu \right) }=\left\Vert
\left( \sum_{I\in \mathcal{C}_{\mathcal{L}}^{\left( m\right) }\left(
L\right) }\left\vert \bigtriangleup _{I}^{\mu }f\right\vert ^{2}\right) ^{%
\frac{1}{2}}\right\Vert _{L^{p}\left( \mu \right) },
\end{eqnarray*}%
for any $f\in L^{p}\left( \mu \right) $ and $L\in \mathcal{L}$. Note that
whenever the supports of scalar functions $f_{n}\left( x\right) $ for $n\in 
\mathbb{N}$ are pairwise disjoint for $x\in \mathbb{R}$, then at most one $%
f_{n}\left( x\right) \neq 0$ for any fixed $x$, and so we have 
\begin{equation*}
\left\vert \left\{ f_{n}\left( x\right) \right\} _{n\in \mathbb{N}%
}\right\vert _{\ell ^{2}}=\left( \sum_{n\in \mathbb{N}}\left\vert
f_{n}\left( x\right) \right\vert ^{2}\right) ^{\frac{1}{2}}=\sum_{n\in 
\mathbb{N}}\left\vert f_{n}\left( x\right) \right\vert =\left\vert
\sum_{n\in \mathbb{N}}f_{n}\left( x\right) \right\vert ,\text{\ \ \ \ \ for
all }x\in \mathbb{R}.
\end{equation*}%
In particular this applies to the sequence $\left\{ \mathsf{P}_{\mathcal{C}_{%
\mathcal{L}}^{\left( m\right) }\left( L^{\prime }\right) }^{\mu }f\right\}
_{L^{\prime }\in \mathfrak{C}_{\mathcal{L}}^{\left( m\right) }\left(
L\right) }$ for each fixed $m$.

\begin{proof}
We suppose that $g,b\in L^{p}\left( \mu \right) \cap L^{2}\left( \mu \right) 
$ and define $L^{2}\left( \mu \right) $-projections,%
\begin{eqnarray*}
g\left( x\right) &\equiv &\sum_{k=m}^{\infty }g_{k}\left( x\right) \text{
and }b\left( x\right) \equiv \sum_{k=1}^{\infty }b_{k}\left( x\right) \ ,\ \
\ \ \text{where} \\
g_{k}\left( x\right) &=&\mathsf{P}_{\mathcal{L},k}^{\mu }g\left( x\right)
=\sum_{L\in \mathfrak{C}_{\mathcal{L}}^{\left( k\right) }\left( T\right) }%
\mathsf{P}_{\mathcal{C}_{\mathcal{L}}\left( L\right) }^{\mu }g\left(
x\right) =\sum_{L\in \mathfrak{C}_{\mathcal{L}}^{\left( k\right) }\left(
T\right) \ }\sum_{I\in \mathcal{C}_{\mathcal{L}}\left( L\right) \
}\bigtriangleup _{I}^{\mu }g\left( x\right) , \\
b_{k}\left( x\right) &=&\mathsf{P}_{\mathcal{L},k}^{\mu }b\left( x\right)
=\sum_{L\in \mathfrak{C}_{\mathcal{L}}^{\left( k\right) }\left( T\right) }%
\mathsf{P}_{\mathcal{C}_{\mathcal{L}}\left( L\right) }^{\mu }b\left(
x\right) =\sum_{L\in \mathfrak{C}_{\mathcal{L}}^{\left( k\right) }\left(
T\right) }\sum_{I\in \mathcal{C}_{\mathcal{L}}\left( L\right)
}\bigtriangleup _{I}^{\mu }b\left( x\right) ,
\end{eqnarray*}%
along with their corresponding sequences using capital letters,%
\begin{eqnarray*}
G_{k}\left( x\right) &=&\left\{ \mathsf{P}_{\mathcal{C}_{\mathcal{L}}\left(
L\right) }^{\mu }g\left( x\right) \right\} _{L\in \mathfrak{C}_{\mathcal{L}%
}^{\left( k\right) }\left( T\right) }=\left\{ \left\{ \bigtriangleup
_{I}^{\mu }g\left( x\right) \right\} _{I\in \mathcal{C}_{\mathcal{L}}\left(
L\right) }\right\} _{L\in \mathfrak{C}_{\mathcal{L}}^{\left( k\right)
}\left( T\right) }, \\
B_{k}\left( x\right) &=&\left\{ \mathsf{P}_{\mathcal{C}_{\mathcal{L}}\left(
L\right) }^{\mu }b\left( x\right) \right\} _{L\in \mathfrak{C}_{\mathcal{L}%
}^{\left( k\right) }\left( T\right) }=\left\{ \left\{ \bigtriangleup
_{I}^{\mu }b\left( x\right) \right\} _{I\in \mathcal{C}_{\mathcal{L}}\left(
L\right) }\right\} _{L\in \mathfrak{C}_{\mathcal{L}}^{\left( k\right)
}\left( T\right) }.
\end{eqnarray*}%
Given $m\in \mathbb{N}\cup \left\{ 0\right\} $, we also define more
sequences using capital letters,%
\begin{eqnarray}
&&  \label{more} \\
G_{m,L}\left( x\right) &\equiv &\left\{ \bigtriangleup _{I}^{\mu }g\left(
x\right) \right\} _{I\in \mathcal{C}_{\mathcal{L}}^{\left( m\right) }\left(
L\right) }\text{ and }B_{L}\left( x\right) \equiv \left\{ \bigtriangleup
_{I}^{\mu }b\left( x\right) \right\} _{I\in \mathcal{C}_{\mathcal{L}}\left(
L\right) },  \notag \\
G_{m,k}\left( x\right) &\equiv &\left\{ \mathsf{P}_{\mathcal{C}_{\mathcal{L}%
}^{\left( m\right) }\left( L\right) }^{\mu }g\left( x\right) \right\} _{L\in 
\mathfrak{C}_{\mathcal{L}}^{\left( k\right) }\left( T\right) }\text{ and }%
B_{k}\left( x\right) \equiv \left\{ \mathsf{P}_{\mathcal{C}_{\mathcal{L}%
}\left( L\right) }^{\mu }b\left( x\right) \right\} _{L\in \mathfrak{C}_{%
\mathcal{L}}^{\left( k\right) }\left( T\right) }\ ,  \notag \\
G_{m}^{\limfunc{doub}}\left( x\right) &\equiv &\left\{ \left\{ \mathsf{P}_{%
\mathcal{C}_{\mathcal{L}}^{\left( m\right) }\left( L\right) }^{\mu }g\left(
x\right) \right\} _{L\in \mathfrak{C}_{\mathcal{L}}^{\left( k\right) }\left(
T\right) }\right\} _{k\in \mathbb{N}}=\left\{ G_{m,k}\left( x\right)
\right\} _{k\in \mathbb{N}}\text{ and }B^{\limfunc{doub}}\left( x\right)
\equiv \left\{ \left\{ \mathsf{P}_{\mathcal{C}_{\mathcal{L}}\left( L\right)
}^{\mu }b\left( x\right) \right\} _{L\in \mathfrak{C}_{\mathcal{L}}^{\left(
k\right) }\left( T\right) }\right\} _{k\in \mathbb{N}}\ ,  \notag \\
G_{m,k}^{\limfunc{doub}}\left( x\right) &\equiv &\left\{ \left\{
\bigtriangleup _{I}^{\mu }g\left( x\right) \right\} _{I\in \mathcal{C}_{%
\mathcal{L}}^{\left( m\right) }\left( L\right) }\right\} _{L\in \mathfrak{C}%
_{\mathcal{L}}^{\left( k\right) }\left( T\right) }=\left\{ G_{m,L}\left(
x\right) \right\} _{L\in \mathfrak{C}_{\mathcal{L}}^{\left( k\right) }\left(
T\right) }\text{ and }B_{k}^{\limfunc{doub}}\left( x\right) \equiv \left\{
\left\{ \bigtriangleup _{I}^{\mu }b\left( x\right) \right\} _{I\in \mathcal{C%
}_{\mathcal{L}}\left( L\right) }\right\} _{L\in \mathfrak{C}_{\mathcal{L}%
}^{\left( k\right) }\left( T\right) }\ ,  \notag
\end{eqnarray}%
where the superscript $\limfunc{doub}$ designates a doubly indexed sequence.
We will often write simply%
\begin{eqnarray}
&&G_{m}\left( x\right) \text{ in place of }G_{m}^{\limfunc{doub}}\left(
x\right) ,  \label{convent} \\
&&B\left( x\right) \text{ in place of }B^{\limfunc{doub}}\left( x\right) . 
\notag
\end{eqnarray}%
Note that the projection in the numerator of (\ref{M char}) can be written
in several different ways,%
\begin{eqnarray*}
\mathsf{P}_{\bigcup_{K\in \mathfrak{C}_{\mathcal{L}}^{\left( m\right)
}\left( L\right) }\mathcal{D}\left( K\right) }^{\mu }g &=&\sum_{K\in 
\mathfrak{C}_{\mathcal{L}}^{\left( m\right) }\left( L\right) }\mathsf{P}_{%
\mathcal{D}\left( K\right) }^{\mu }g=\sum_{K\in \mathfrak{C}_{\mathcal{L}%
}^{\left( m\right) }\left( L\right) }\left( \sum_{M\in \mathcal{L}:\
M\subset K}\mathsf{P}_{\mathcal{C}_{\mathcal{L}}\left( M\right) }^{\mu
}g\right) \\
&=&\sum_{L^{\prime }\in \mathcal{L}:\ L^{\prime }\subset L}\mathsf{P}_{%
\mathcal{C}_{\mathcal{L}}^{\left( m\right) }\left( L^{\prime }\right) }^{\mu
}g=\sum_{L^{\prime }\in \mathcal{L}:\ L^{\prime }\preccurlyeq L}\mathsf{P}_{%
\mathcal{C}_{\mathcal{L}}^{\left( m\right) }\left( L^{\prime }\right) }^{\mu
}g,
\end{eqnarray*}%
where in the final expression, we are using the tree ordering on $\mathcal{D}%
\left[ T\right] $. From the square function equivalences we then have%
\begin{equation*}
\left\Vert \mathsf{P}_{\bigcup_{K\in \mathfrak{C}_{\mathcal{L}}^{\left(
m\right) }\left( L\right) }\mathcal{D}\left( K\right) }^{\mu }g\right\Vert
_{L^{p}\left( \mu \right) }\approx \left\Vert \left\vert \left\{ \mathsf{P}_{%
\mathcal{C}_{\mathcal{L}}^{\left( m\right) }\left( L^{\prime }\right) }^{\mu
}g\right\} _{L^{\prime }\preccurlyeq L}\right\vert _{\ell ^{2}}\right\Vert
_{L^{p}\left( \mu \right) }\approx \left\Vert \left\vert G_{m,L^{\prime
}}\left( x\right) _{L^{\prime }\preccurlyeq L}\right\vert _{\ell
^{2}}\right\Vert _{L^{p}\left( \mu \right) },
\end{equation*}%
where the sequence to which the norm $\left\vert \cdot \right\vert _{\ell
^{2}}$ applies is understood by context. For example, 
\begin{eqnarray*}
\left\vert \left\{ \mathsf{P}_{\mathcal{C}_{\mathcal{L}}^{\left( m\right)
}\left( L^{\prime }\right) }^{\mu }g\left( x\right) \right\} _{L^{\prime
}\preccurlyeq L}\right\vert _{\ell ^{2}}^{2} &=&\sum_{L^{\prime
}\preccurlyeq L}\left\vert \mathsf{P}_{\mathcal{C}_{\mathcal{L}}^{\left(
m\right) }\left( L^{\prime }\right) }^{\mu }g\left( x\right) \right\vert ^{2}%
\text{,} \\
\text{and }\left\vert G_{m,L^{\prime }}\left( x\right) _{L^{\prime
}\preccurlyeq L}\right\vert _{\ell ^{2}}^{2} &=&\left\vert \left\{ \left\{
\bigtriangleup _{I}^{\mu }g\left( x\right) \right\} _{I\in \mathcal{C}_{%
\mathcal{L}}^{\left( m\right) }\left( L^{\prime }\right) }\right\}
_{L^{\prime }\preccurlyeq L}\right\vert _{\ell ^{2}}^{2} \\
&=&\left\vert \left\{ \bigtriangleup _{I}^{\mu }g\left( x\right) \right\} 
_{\substack{ \left( L^{\prime },I\right) \in \mathcal{L}\times \mathcal{D} 
\\ L^{\prime }\preccurlyeq L,I\in \mathcal{C}_{\mathcal{L}}^{\left( m\right)
}\left( L^{\prime }\right) }}\right\vert _{\ell ^{2}}^{2}=\sum_{L^{\prime
}\preccurlyeq L}\sum_{I\in \mathcal{C}_{\mathcal{L}}^{\left( m\right)
}\left( L^{\prime }\right) }\left\vert \bigtriangleup _{I}^{\mu }g\left(
x\right) \right\vert ^{2},
\end{eqnarray*}%
where by context, the iterated sequence $G_{m,L^{\prime }}\left( x\right)
_{L^{\prime }\preccurlyeq L}=\left\{ \left\{ \bigtriangleup _{I}^{\mu
}g\left( x\right) \right\} _{I\in \mathcal{C}_{\mathcal{L}}^{\left( m\right)
}\left( L^{\prime }\right) }\right\} _{L^{\prime }\preccurlyeq L}$ is
understood to be the sequence $\left\{ \bigtriangleup _{I}^{\mu }g\left(
x\right) \right\} _{\substack{ \left( L^{\prime },I\right) \in \mathcal{L}%
\times \mathcal{D}  \\ L^{\prime }\preccurlyeq L,I\in \mathcal{C}_{\mathcal{L%
}}^{\left( m\right) }\left( L^{\prime }\right) }}$, with some ordering of
the countable set of such intervals $I$.

Define%
\begin{eqnarray*}
G_{m,L}^{\left[ \infty \right] }\left( x\right) &\equiv &\left\{
G_{m,K}\left( x\right) \right\} _{K\in \mathcal{L}:\ K\subset L}\text{ and }%
B_{L}^{\left[ N\right] }\left( x\right) \equiv \left\{ B_{K}\left( x\right)
\right\} _{K\subset \mathcal{C}_{\mathcal{L}}^{\left[ N\right] }\left(
L\right) }, \\
G_{m,k}^{\left[ \infty \right] }\left( x\right) &\equiv &\left\{ G_{m,L}^{%
\left[ \infty \right] }\left( x\right) \right\} _{L\in \mathfrak{C}_{%
\mathcal{L}}^{\left( k\right) }\left( T\right) }\ \text{ and }B_{k}^{\left[ N%
\right] }\left( x\right) \equiv \left\{ B_{L}^{\left[ N\right] }\left(
x\right) \right\} _{L\in \mathfrak{C}_{\mathcal{L}}^{\left( k\right) }\left(
T\right) }\ ,
\end{eqnarray*}%
where the last line can be interpreted as doubly indexed sequences. Note
that for each \emph{fixed} $k\in \mathbb{N}$, and $m\in \mathbb{N},$ and $%
0\leq \ell \leq N$, both collections of functions 
\begin{equation*}
\left\{ \mathsf{P}_{\mathcal{C}_{\mathcal{L}}^{\left( m\right) }\left(
L\right) }^{\mu }g\right\} _{L\in \mathfrak{C}_{\mathcal{L}}^{\left(
k\right) }\left( T\right) }\text{ and }\left\{ \mathsf{P}_{\mathcal{C}_{%
\mathcal{L}}^{\left( \ell \right) }\left( L\right) }^{\mu }b\right\} _{L\in 
\mathfrak{C}_{\mathcal{L}}^{\left( k\right) }\left( T\right) }
\end{equation*}%
have pairwise disjoint supports in both $\mathcal{D}$ and $\mathbb{R}$, i.e.
both (\ref{N supp}) and (\ref{R supp}) hold for each collection of
functions. Thus from Corollary \ref{disjoint supp'}, we have for each fixed $%
k\in \mathbb{N}$, 
\begin{eqnarray}
&&\left\Vert \left\vert G_{m,k}\right\vert _{\ell ^{2}}\right\Vert
_{L^{p}\left( \mu \right) }^{p}\lesssim \sum_{L\in \mathfrak{C}_{\mathcal{L}%
}^{\left( k\right) }\left( T\right) }\left\Vert \left\vert
G_{m,L}\right\vert _{\ell ^{2}}\right\Vert _{L^{p}\left( \mu \right)
}^{p}\leq M_{\mathcal{L}}^{\left( m,N\right) }\left( g,b\right)
^{p}\sum_{L\in \mathfrak{C}_{\mathcal{L}}^{\left( k\right) }\left( T\right)
}\left\Vert \left\vert B_{L}^{\left[ N\right] }\right\vert _{\ell
^{2}}\right\Vert _{L^{p}\left( \mu \right) }^{p}  \label{thus we have'} \\
&=&M_{\mathcal{L}}^{\left( m,N\right) }\left( g,b\right) ^{p}\sum_{L\in 
\mathfrak{C}_{\mathcal{L}}^{\left( k\right) }\left( T\right) }\left\Vert
\left\vert \sum_{\ell =0}^{N}B_{L}^{\left( \ell \right) }\right\vert _{\ell
^{2}}\right\Vert _{L^{p}\left( \mu \right) }^{p}\lesssim N^{p}M_{\mathcal{L}%
}^{\left( m,N\right) }\left( g,b\right) ^{p}\sum_{\ell =0}^{N}\sum_{L\in 
\mathfrak{C}_{\mathcal{L}}^{\left( k\right) }\left( T\right) }\left\Vert
\left\vert B_{L}^{\left( \ell \right) }\right\vert _{\ell ^{2}}\right\Vert
_{L^{p}\left( \mu \right) }^{p}  \notag \\
&\lesssim &N^{p+1}M_{\mathcal{L}}^{\left( m,N\right) }\left( g,b\right)
^{p}\max_{0\leq \ell \leq N}\left\Vert \left\vert B_{k}^{\left( \ell \right)
}\right\vert _{\ell ^{2}}\right\Vert _{L^{p}\left( \mu \right) }^{p}\leq
N^{p+1}M_{\mathcal{L}}^{\left( m,N\right) }\left( g,b\right) ^{p}\left\Vert
\left\vert B_{k}^{\left[ N\right] }\right\vert _{\ell ^{2}}\right\Vert
_{L^{p}\left( \mu \right) }^{p},  \notag
\end{eqnarray}%
where $B_{L}^{\left( \ell \right) }\equiv \left\{ B_{K}\left( x\right)
\right\} _{K\subset \mathcal{C}_{\mathcal{L}}^{\left( \ell \right) }\left(
L\right) }$.

Note that we cannot apply Corollary \ref{disjoint supp'} to the doubly
indexed sequences 
\begin{equation*}
G_{m}^{\limfunc{doub}}=\left\{ \left\{ \mathsf{P}_{\bigcup_{K\in \mathfrak{C}%
_{\mathcal{L}}^{\left( m\right) }\left( L\right) }\mathcal{D}\left( K\right)
}^{\mu }g\right\} _{L\in \mathfrak{C}_{\mathcal{L}}^{\left( k\right) }\left(
T\right) }\right\} _{k\in \mathbb{N}}\text{ and }B_{N}^{\limfunc{doub}%
}=\left\{ \left\{ \mathsf{P}_{\mathcal{C}_{\mathcal{L}}^{\left[ N\right]
}\left( L\right) }^{\mu }b\left( x\right) \right\} _{L\in \mathfrak{C}_{%
\mathcal{L}}^{\left( k\right) }\left( T\right) }\right\} _{k\in \mathbb{N}}
\end{equation*}%
since we lose the pairwise disjoint property in both $\mathcal{D}$ and $%
\mathbb{R}$. We must work harder to handle this general situation and the
remainder of the proof is devoted to this end. We begin by treating the
function $g$, and then $b$ will be treated using similar ideas at the end of
the argument.

Let $\lambda >1$. For each $L^{\prime }\in \mathfrak{C}_{\mathcal{L}}\left(
L\right) $ with $L\in \mathfrak{C}_{\mathcal{L}}^{\left( k\right) }\left(
T\right) $, define%
\begin{eqnarray}
\Omega _{m,N,L^{\prime }} &\equiv &\left\{ x\in L^{\prime }:\left\vert
G_{m,L^{\prime }}^{\left[ \infty \right] }\left( x\right) \right\vert _{\ell
^{2}}>\lambda M_{\mathcal{L}}^{\left( m,N\right) }\left( g,b\right)
\left\vert B_{L}^{\left[ N\right] }\left( x\right) \right\vert _{\ell
^{2}}\right\} ,  \label{def OmegamNL'} \\
\text{and }\Omega _{m,N,k} &\equiv &\bigcup_{L^{\prime }\in \mathfrak{C}_{%
\mathcal{L}}^{\left( k+1\right) }\left( T\right) }\Omega _{m,N,L^{\prime }}\
,  \notag
\end{eqnarray}%
and note that $\Omega _{m,N,k+1}\subset \Omega _{m,N,k}$ since $\left\vert
G_{m,L^{\prime \prime }}^{\left[ \infty \right] }\left( x\right) \right\vert
_{\ell ^{2}}\leq \left\vert G_{m,L^{\prime }}^{\left[ \infty \right] }\left(
x\right) \right\vert _{\ell ^{2}}\,$for $L^{\prime \prime }\in \mathfrak{C}_{%
\mathcal{L}}\left( L\right) $. Then on $L^{\prime }\setminus \Omega
_{m,N,L^{\prime }}$, we have the pointwise inequality $\left\vert
G_{m,L^{\prime }}^{\left[ \infty \right] }\left( x\right) \right\vert _{\ell
^{2}}\leq \lambda M_{\mathcal{L}}^{\left( m\right) }\left( g,b\right)
\left\vert B_{L}^{\left[ N\right] }\left( x\right) \right\vert _{\ell ^{2}}$%
, and if we write%
\begin{eqnarray}
G_{m} &\equiv &\left\{ \mathsf{P}_{\mathcal{C}_{\mathcal{L}}^{\left(
m\right) }\left( L\right) }^{\mu }g\right\} _{L\in \mathcal{L}}=\left\{ 
\mathsf{P}_{\mathcal{C}_{\mathcal{L}}^{\left( m-1\right) }\left( L^{\prime
}\right) }^{\mu }g\right\} _{\substack{ L\in \mathcal{L}  \\ L^{\prime }\in 
\mathfrak{C}_{\mathcal{L}}\left( L\right) }}  \label{good and bad} \\
&=&\left\{ \mathbf{1}_{\mathbb{R}\setminus \Omega _{m,N,L^{\prime }}}\mathsf{%
P}_{\mathcal{C}_{\mathcal{L}}^{\left( m-1\right) }\left( L^{\prime }\right)
}^{\mu }g\right\} _{\substack{ L\in \mathcal{L}  \\ L^{\prime }\in \mathfrak{%
C}_{\mathcal{L}}\left( L\right) }}+\left\{ \mathbf{1}_{\Omega
_{m,N,L^{\prime }}}\mathsf{P}_{\mathcal{C}_{\mathcal{L}}^{\left( m-1\right)
}\left( L^{\prime }\right) }^{\mu }g\right\} _{\substack{ L\in \mathcal{L} 
\\ L^{\prime }\in \mathfrak{C}_{\mathcal{L}}\left( L\right) }}  \notag \\
&\equiv &G_{m}^{\func{good}}+G_{m}^{\func{bad}},  \notag
\end{eqnarray}%
then it follows that%
\begin{eqnarray}
&&\left\Vert \left\vert G_{m}^{\func{good}}\right\vert _{\ell
^{2}}\right\Vert _{L^{p}\left( \mu \right) }^{p}=\left\Vert \left\vert
\left\{ \mathbf{1}_{\mathbb{R}\setminus \Omega _{m,N,L^{\prime }}}\mathsf{P}%
_{\mathcal{C}_{\mathcal{L}}^{\left( m-1\right) }\left( L^{\prime }\right)
}^{\mu }g\right\} _{\substack{ L\in \mathcal{L}  \\ L^{\prime }\in \mathfrak{%
C}_{\mathcal{L}}\left( L\right) }}\right\vert _{\ell ^{2}}\right\Vert
_{L^{p}\left( \mu \right) }^{p}  \label{follows that} \\
&\leq &\lambda ^{p}M_{\mathcal{L}}^{\left( m,N\right) }\left( g,b\right)
^{p}\left\Vert \left\vert \left\{ \mathsf{P}_{\mathcal{C}_{\mathcal{L}}^{%
\left[ N\right] }\left( L\right) }^{\mu }b\right\} _{\substack{ L\in 
\mathcal{L}  \\ L^{\prime }\in \mathfrak{C}_{\mathcal{L}}\left( L\right) }}%
\right\vert _{\ell ^{2}}\right\Vert _{L^{p}\left( \mu \right) }^{p}\lesssim
\lambda ^{p}M_{\mathcal{L}}^{\left( m,N\right) }\left( g,b\right)
^{p}\left\Vert \left\vert B^{\left[ N\right] }\right\vert _{\ell
^{2}}\right\Vert _{L^{p}\left( \mu \right) }^{p}.  \notag
\end{eqnarray}

To handle the term $G_{m}^{\func{bad}}$ we must work harder. Now since $%
\left\vert B_{L}^{\left[ N\right] }\right\vert _{\ell ^{2}}$ is constant on
each set $L^{\prime }\in \mathfrak{C}_{\mathcal{L}}^{\left( m\right) }\left(
L\right) $ when $m>N$, we note that for $m>N$, $L\in \mathcal{L}$ and$%
\,L^{\prime }\in \mathfrak{C}_{\mathcal{L}}\left( L\right) $, 
\begin{eqnarray}
&&\left\vert \Omega _{m,N,L^{\prime }}\right\vert _{\mu }=\left\vert
L^{\prime }\cap \Omega _{m,N,L^{\prime }}\right\vert _{\mu }\leq
\int_{L^{\prime }}\frac{\left\vert G_{m,L^{\prime }}^{\left[ \infty \right]
}\left( x\right) \right\vert _{\ell ^{2}}^{p}}{\lambda ^{p}M_{\mathcal{L}%
}^{\left( m\right) }\left( g,b\right) ^{p}\left\vert B_{L}^{\left[ N\right]
}\left( x\right) \right\vert _{\ell ^{2}}^{p}}d\mu \left( x\right) =\frac{%
\int_{L^{\prime }}\left\vert G_{m,L^{\prime }}^{\left[ \infty \right]
}\left( x\right) \right\vert _{\ell ^{2}}^{p}d\mu \left( x\right) }{\lambda
^{p}M_{\mathcal{L}}^{\left( m,N\right) }\left( g,b\right) ^{p}E_{L^{\prime
}}^{\mu }\left\vert B_{L}^{\left[ N\right] }\right\vert _{\ell ^{2}}^{p}}
\label{note'} \\
&=&\frac{1}{\lambda ^{p}M_{\mathcal{L}}^{\left( m,N\right) }\left(
g,b\right) ^{p}}\frac{\int_{L^{\prime }}\left\vert G_{m,L^{\prime }}^{\left[
\infty \right] }\left( x\right) \right\vert _{\ell ^{2}}^{p}d\mu \left(
x\right) }{\int_{L^{\prime }}\left\vert B_{L}^{\left[ N\right] }\left(
x\right) \right\vert _{\ell ^{2}}^{p}d\mu \left( x\right) }\left\vert
L^{\prime }\right\vert _{\mu }\leq C_{0}\frac{1}{\lambda ^{p}M_{\mathcal{L}%
}^{\left( m,N\right) }\left( g,b\right) ^{p}}M_{\mathcal{L}}^{\left(
m,N\right) }\left( g,b\right) ^{p}\left\vert L^{\prime }\right\vert _{\mu }=%
\frac{C_{0}}{\lambda ^{p}}\left\vert L^{\prime }\right\vert _{\mu }\ , 
\notag
\end{eqnarray}%
since 
\begin{eqnarray*}
\left\Vert \left\vert G_{m,L^{\prime }}^{\left[ \infty \right] }\left(
x\right) \right\vert _{\ell ^{2}}\right\Vert _{L^{p}\left( \mu \right) }^{p}
&=&\left\Vert \left\vert \sum_{K\subset L^{\prime }}G_{m,K}\left( x\right)
\right\vert _{\ell ^{2}}\right\Vert _{L^{p}\left( \mu \right) }^{p} \\
&=&\left\Vert \left\vert \sum_{K\subset L^{\prime }}\left\{ \bigtriangleup
_{I}^{\mu }g\left( x\right) \right\} _{I\in \mathcal{C}_{\mathcal{L}%
}^{\left( m\right) }\left( K\right) }\right\vert _{\ell ^{2}}\right\Vert
_{L^{p}\left( \mu \right) }^{p}\approx \left\Vert \mathsf{P}_{\bigcup_{G\in 
\mathfrak{C}_{\mathcal{L}}^{\left( m\right) }\left( L\right) }\mathcal{D}%
\left( G\right) }^{\mu }g\right\Vert _{L^{p}\left( \mu \right) }^{p},
\end{eqnarray*}%
by the square function estimates in Theorem \ref{square thm}. In particular,
for $\lambda >C_{0}^{\frac{1}{p}}$ we have 
\begin{equation}
\left\vert L^{\prime }\right\vert _{\mu }\geq \left\vert L^{\prime
}\setminus \Omega _{m,N,L^{\prime }}\right\vert _{\mu }=\left\vert L^{\prime
}\right\vert _{\mu }-\left\vert \Omega _{m,N,L^{\prime }}\right\vert _{\mu
}\geq \left\vert L^{\prime }\right\vert _{\mu }-\frac{C_{0}}{\lambda ^{p}}%
\left\vert L^{\prime }\right\vert _{\mu }=\left( 1-\frac{C_{0}}{\lambda ^{p}}%
\right) \left\vert L^{\prime }\right\vert _{\mu }\ ,  \label{note''}
\end{equation}%
independent of $L^{\prime }$ and $m$ and $N$. The facts that $\left\vert
G_{m,k}^{\left[ \infty \right] }\left( x\right) \right\vert _{\ell ^{2}}$ is
constant on $L^{\prime }$ for $L^{\prime }\in \mathfrak{C}_{\mathcal{L}%
}\left( L\right) $ and $L\in \mathfrak{C}_{\mathcal{L}}^{\left( k\right)
}\left( T\right) $, and\ that 
\begin{equation}
\left\vert L^{\prime }\setminus \Omega _{m,N,L^{\prime }}\right\vert _{\mu
}\approx \left\vert L^{\prime }\right\vert _{\mu },  \label{note'''}
\end{equation}%
will permit us to replace the sequence $\left\{ g_{k}\right\} _{k\geq 1}$ of
functions with $m$ sequences that are pairwise disjoint in $\mathbb{R}$ as
well as in $\mathcal{D}$, and which will in turn permit us to use the
definition of $M_{\mathcal{L}}^{\left( m,N\right) }\left( g,b\right) $ on
each of these $m$ sequences.

Indeed, for $1\leq \kappa \leq m+1$, define 
\begin{eqnarray}
&&  \label{kappa} \\
g^{\left[ \kappa \right] }\left( x\right) &\equiv &\sum_{k\in \kappa +\left(
m+1\right) \mathbb{N}}g_{k}\left( x\right) \text{ and }G^{\left[ \kappa %
\right] }\left( x\right) \equiv \left\{ G_{k}\left( x\right) \right\} _{k\in
\kappa +\left( m+1\right) \mathbb{N}}\ ,  \notag \\
\widehat{g^{\left[ \kappa \right] }}\left( x\right) &\equiv &\sum_{k\in
\kappa +\left( m+1\right) \mathbb{N}}\mathbf{1}_{\Omega _{m,N,k}}\left(
x\right) g_{k}\left( x\right) \text{ and }\widehat{G^{\left[ \kappa \right] }%
}\left( x\right) \equiv \left\{ \mathbf{1}_{\Omega _{m,N,k}}\left( x\right)
G_{k}\left( x\right) \right\} _{k\in \kappa +\left( m+1\right) \mathbb{N}}\ ,
\notag \\
\widetilde{g^{\left[ \kappa \right] }}\left( x\right) &\equiv &\sum_{k\in
\kappa +\left( m+1\right) \mathbb{N}}\mathbf{1}_{\Omega _{m,N,k}\setminus
\Omega _{m,N,k+m+1}}\left( x\right) g_{k}\left( x\right) \text{ and }%
\widetilde{G^{\left[ \kappa \right] }}\left( x\right) \equiv \left\{ \mathbf{%
1}_{\Omega _{m,N,k}\setminus \Omega _{m,N,k+m+1}}\left( x\right) G_{k}\left(
x\right) \right\} _{k\in \kappa +\left( m+1\right) \mathbb{N}},  \notag
\end{eqnarray}%
so that for each $\kappa $, the sequence of functions $\widetilde{g_{\left[
\kappa \right] }}\left( x\right) $ satisfies both (\ref{N supp}) and (\ref{R
supp}).

\textbf{Notation}: \emph{We also make corresponding definitions with the
caret decoration }$\widehat{}$\emph{\ and the tilde decoration }$\widetilde{}
$\emph{, analogous to those in (\ref{more}), but with a superscript }$^{%
\left[ \kappa \right] }$\emph{. Thus the caret denotes restriction to }$%
\Omega _{m,N,k}$\emph{, while tilde denotes restriction to the `annular' set 
}$\Omega _{m,N,k}\setminus \Omega _{m,N,k+m+1}$\emph{.}

Moreover, recalling our convention (\ref{convent}) that we often write $%
G_{m}^{\left[ \kappa \right] }\left( x\right) $ in place of $G_{m}^{\left[
\kappa \right] ,\limfunc{doub}}$, etc., we claim that%
\begin{equation}
\left\vert \widehat{G_{m}^{\left[ \kappa \right] }}\left( x\right)
\right\vert _{\ell ^{2}}\lesssim \left\vert G_{m}^{\func{good}}\left(
x_{0}\right) \right\vert _{\ell ^{2}}+M_{\mu }^{\func{dy}}\left( \left\vert 
\widetilde{G_{m}^{\left[ \kappa \right] }}\right\vert _{\ell ^{2}}\right)
\left( x\right) ,\ \ \ \ \ x\in T,  \label{G tilde}
\end{equation}%
follows from (\ref{note''}). To see this, pick a point $x_{0}\in T$. Since
the Haar support of $g$ is finite, the sets $\Omega _{m,N,k+m+1}$ are
eventually empty, and so there is $k$ such that $x_{0}\in \Omega
_{m,N,k}\setminus \Omega _{m,N,k+m+1}$. Suppose that $L\in \mathfrak{C}_{%
\mathcal{L}}^{\left( k\right) }\left( T\right) $ contains $x_{0}$. Now note
that $\sum_{\substack{ s\in \kappa +\left( m+1\right) \mathbb{N}  \\ s\leq
k-1 }}\left\vert G_{m,s}^{\left[ \kappa \right] }\right\vert _{\ell ^{2}}$
is constant on $L$, since if $L^{\ast }=\pi _{\mathcal{L}}L\in \mathfrak{C}_{%
\mathcal{L}}^{\left( k-1\right) }\left( T\right) $ is the child at level $%
k-1 $ that contains $L$, then for $x\in L$, 
\begin{equation*}
\left\vert G_{L^{\ast }}^{\left[ \kappa \right] }\left( x\right) \right\vert
_{\ell ^{2}}^{2}=\left\vert \left\{ \bigtriangleup _{I}^{\mu }g\left(
x\right) \right\} _{I\in \mathcal{C}_{\mathcal{L}}^{\left( m\right) }\left(
L^{\ast }\right) }\right\vert _{\ell ^{2}}^{2}=\sum_{I\in \left( L,L^{\ast }%
\right] }\left\vert \bigtriangleup _{I}^{\mu }g\left( x\right) \right\vert
^{2}\text{ is constant}.
\end{equation*}%
Thus (\ref{note''}) implies 
\begin{eqnarray*}
&&\sum_{\substack{ s\in \kappa +\left( m+1\right) \mathbb{N}  \\ s\leq k-1}}%
\left\vert G_{m,s}^{\left[ \kappa \right] }\left( x_{0}\right) \right\vert
_{\ell ^{2}}=\sum_{\substack{ s\in \kappa +\left( m+1\right) \mathbb{N}  \\ %
s\leq k-1}}\left\vert G_{m,s}^{\left[ \kappa \right] ,\func{good}}\left(
x_{0}\right) \right\vert _{\ell ^{2}}+\sum_{\substack{ s\in \kappa +\left(
m+1\right) \mathbb{N}  \\ s\leq k-1}}\left\vert \widehat{G_{m,s}^{\left[
\kappa \right] }}\left( x_{0}\right) \right\vert _{\ell ^{2}} \\
&=&\sum_{\substack{ s\in \kappa +\left( m+1\right) \mathbb{N}  \\ s\leq k-1}}%
\left\vert G_{m,s}^{\left[ \kappa \right] ,\func{good}}\left( x_{0}\right)
\right\vert _{\ell ^{2}}+\frac{1}{\left\vert L\setminus \Omega
_{m,N,L}\right\vert _{\mu }}\int_{L\setminus \Omega _{m,N,L}}\sum_{\substack{
s\in \kappa +\left( m+1\right) \mathbb{N}  \\ s\leq k-1}}\left\vert \widehat{%
G_{m,s}^{\left[ \kappa \right] }}\left( x\right) \right\vert _{\ell
^{2}}d\mu \left( x\right) \\
&\lesssim &\left\vert G_{m}^{\func{good}}\left( x_{0}\right) \right\vert
_{\ell ^{2}}+\frac{1}{\left\vert L\right\vert _{\mu }}\int_{L\setminus
\Omega _{m,N,L}}\sum_{\substack{ s\in \kappa +\left( m+1\right) \mathbb{N} 
\\ s\leq k-1}}\left\vert \widetilde{G_{m,s}^{\left[ \kappa \right] }}\left(
x\right) \right\vert _{\ell ^{2}}d\mu \left( x\right) \\
&\lesssim &\left\vert G_{m}^{\func{good}}\left( x_{0}\right) \right\vert
_{\ell ^{2}}+M_{\mu }^{\func{dy}}\left( \sum_{\substack{ s\in \kappa +\left(
m+1\right) \mathbb{N}  \\ s\leq k-1}}\left\vert \widetilde{G_{m,s}^{\left[
\kappa \right] }}\right\vert _{\ell ^{2}}\right) \left( x_{0}\right) ,
\end{eqnarray*}%
since $\widehat{G_{m,s}^{\left[ \kappa \right] }}\left( x\right) =\widetilde{%
G_{m,s}^{\left[ \kappa \right] }}\left( x\right) $ on $L\setminus \Omega
_{m,N,L}$ if $s\leq k-1$. We also have $\left\vert \widehat{G_{m,k}^{\left[
\kappa \right] }}\right\vert _{\ell ^{2}}\left( x\right) =\left\vert 
\widetilde{G_{m,k}^{\left[ \kappa \right] }}\right\vert _{\ell ^{2}}\left(
x\right) $ since $x\notin \Omega _{m,N,k+m+1}$, and thus we conclude%
\begin{eqnarray*}
\sum_{\substack{ s\in \kappa +\left( m+1\right) \mathbb{N}  \\ s\leq k}}%
\left\vert G_{m,k}^{\left[ \kappa \right] }\right\vert _{\ell ^{2}}\left(
x_{0}\right) &\leq &\left\vert G_{m}^{\func{good}}\left( x_{0}\right)
\right\vert _{\ell ^{2}}+\left( \sum_{\substack{ s\in \kappa +\left(
m+1\right) \mathbb{N}  \\ s\leq k-1}}\left\vert \widehat{G_{m,s}^{\left[
\kappa \right] }}\left( x\right) \right\vert _{\ell ^{2}}\right) +\left\vert 
\widehat{G_{k}}\right\vert _{\ell ^{2}}\left( x_{0}\right) \\
&\lesssim &\left\vert G_{m}^{\func{good}}\left( x_{0}\right) \right\vert
_{\ell ^{2}}+M_{\mu }^{\func{dy}}\left( \sum_{\substack{ s\in \kappa +\left(
m+1\right) \mathbb{N}  \\ s\leq k-1}}\left\vert \widetilde{G_{m,s}^{\left[
\kappa \right] }}\right\vert _{\ell ^{2}}\right) \left( x_{0}\right)
+\left\vert \widetilde{G_{k}}\right\vert _{\ell ^{2}}\left( x\right) \\
&\leq &\left\vert G_{m}^{\func{good}}\left( x_{0}\right) \right\vert _{\ell
^{2}}+M_{\mu }^{\func{dy}}\left( \sum_{k=1}^{\infty }\left\vert \widetilde{%
G_{m,k}^{\left[ \kappa \right] }}\right\vert _{\ell ^{2}}\right) \left(
x_{0}\right) ,
\end{eqnarray*}%
which is (\ref{G tilde}).

Altogether then, combining the boundedness of $M_{\mu }^{\func{dy}}$ on $%
L^{p}\left( \mu \right) $ with Lemma \ref{disjoint supp}, (\ref{follows that}%
) and (\ref{thus we have'}), we have\ for each $\kappa $ and $m$,%
\begin{eqnarray*}
&&\left\Vert g^{\left[ \kappa \right] }\right\Vert _{L^{p}\left( \mu \right)
}^{p}\approx \left\Vert \left\vert \left\{ G_{m,k}^{\left[ \kappa \right]
}\right\} _{k\in \kappa +m\mathbb{N}}\right\vert _{\ell ^{2}}\right\Vert
_{L^{p}\left( \mu \right) }^{p} \\
&\lesssim &\lambda ^{p}M_{\mathcal{L}}^{\left( m,N\right) }\left( g,b\right)
^{p}\left\Vert \left\vert B^{\left[ N\right] }\right\vert _{\ell
^{2}}\right\Vert _{L^{p}\left( \mu \right) }^{p}+\left\Vert \left(
\left\vert \left\{ M_{\mu }^{\func{dy}}\widetilde{G_{m,k}^{\left[ \kappa %
\right] }}\right\} _{k\in \kappa +m\mathbb{N}}\right\vert _{\ell
^{2}}\right) \right\Vert _{L^{p}\left( \mu \right) }^{p},
\end{eqnarray*}%
where, recalling that $b_{k}^{\left[ \kappa \right] ,\left[ N\right] }$ is
defined in the above \textbf{Notation},%
\begin{eqnarray*}
&&\left\Vert \left( \left\vert \left\{ M_{\mu }^{\func{dy}}\widetilde{%
G_{m,k}^{\left[ \kappa \right] }}\right\} _{k\in \kappa +m\mathbb{N}%
}\right\vert _{\ell ^{2}}\right) \right\Vert _{L^{p}\left( \mu \right) }^{p}
\\
&\lesssim &\left\Vert \left\vert \left\{ \widetilde{G_{m,k}^{\left[ \kappa %
\right] }}\right\} _{k\in \kappa +m\mathbb{N}}\right\vert _{\ell
^{2}}\right\Vert _{L^{p}\left( \mu \right) }^{p}\approx \sum_{k\in \kappa +m%
\mathbb{N}}\left\Vert \widetilde{g_{m,k}^{\left[ \kappa \right] }}%
\right\Vert _{L^{p}\left( \mu \right) }^{p}\leq \sum_{k\in \kappa +m\mathbb{N%
}}\left\Vert g_{m,k}^{\left[ \kappa \right] }\right\Vert _{L^{p}\left( \mu
\right) }^{p} \\
&\lesssim &M_{\mathcal{L}}^{\left( m,N\right) }\left( g,b\right)
^{p}\sum_{k\in \kappa +m\mathbb{N}}\left\Vert b_{k}^{\left[ \kappa \right] ,%
\left[ N\right] }\right\Vert _{L^{p}\left( \mu \right) }^{p}=M_{\mathcal{L}%
}^{\left( m,N\right) }\left( g,b\right) ^{p}\sum_{k\in \kappa +m\mathbb{N}%
}\int_{\mathbb{R}}\left\vert b_{k}^{\left[ \kappa \right] ,\left[ N\right]
}\left( x\right) \right\vert ^{p}d\mu \left( x\right) \\
&\lesssim &M_{\mathcal{L}}^{\left( m,N\right) }\left( g,b\right) ^{p}\int_{%
\mathbb{R}}\left\vert \sum_{k\in \kappa +m\mathbb{N}}b_{k}^{\left[ \kappa %
\right] ,\left[ N\right] }\left( x\right) \right\vert ^{p}d\mu \left(
x\right) ,\ \ \ \ \ \text{for }p>1\text{,}
\end{eqnarray*}%
so that altogether,%
\begin{equation}
\left\Vert g^{\left[ \kappa \right] }\right\Vert _{L^{p}\left( \mu \right)
}^{p}\lesssim \lambda ^{p}M_{\mathcal{L}}^{\left( m,N\right) }\left(
g,b\right) ^{p}\left\Vert \left\vert B^{\left[ N\right] }\right\vert _{\ell
^{2}}\right\Vert _{L^{p}\left( \mu \right) }^{p}+M_{\mathcal{L}}^{\left(
m,N\right) }\left( g,b\right) ^{p}\int_{\mathbb{R}}\left\vert \sum_{k\in
\kappa +m\mathbb{N}}b_{k}^{\left[ \kappa \right] ,\left[ N\right] }\left(
x\right) \right\vert ^{p}d\mu \left( x\right) .  \label{porism stop}
\end{equation}

At this point we write $B_{k}^{\left[ \kappa \right] ,\left[ N\right] }$ as
a sum of $N+1$ martingale difference sequences $B_{k}^{\left[ \kappa \right]
,\left( s\right) }\left( x\right) $, $0\leq s\leq N$, i.e.%
\begin{eqnarray}
&&B_{k}^{\left[ \kappa \right] ,\left[ N\right] }\left( x\right) =\left\{ 
\mathsf{P}_{\mathcal{C}_{\mathcal{L}}^{\left[ N\right] }\left( L\right)
}^{\mu }b\left( x\right) \right\} _{L\in \mathfrak{C}_{\mathcal{L}}^{\left(
k\right) }\left( T\right) }=\left\{ \sum_{s=0}^{N}\mathsf{P}_{\mathcal{C}_{%
\mathcal{L}}^{\left( s\right) }\left( L\right) }^{\mu }b\left( x\right)
\right\} _{L\in \mathfrak{C}_{\mathcal{L}}^{\left( k\right) }\left( T\right)
}=\sum_{s=0}^{N}B_{k}^{\left[ \kappa \right] ,\left( s\right) }\left(
x\right) ,  \label{Bks} \\
&&\ \ \ \ \ \ \ \ \ \ \text{where }B_{k}^{\left[ \kappa \right] ,\left(
s\right) }\left( x\right) \equiv \left\{ \mathsf{P}_{\mathcal{C}_{\mathcal{L}%
}^{\left( s\right) }\left( L\right) }^{\mu }b\left( x\right) \right\} _{L\in 
\mathfrak{C}_{\mathcal{L}}^{\left( k\right) }\left( T\right) }.  \notag
\end{eqnarray}%
Note that $B_{k}^{\left[ \kappa \right] ,\left( s\right) }$ is a martingale
difference sequence for the $L^{p}\left( \mu \right) $ function 
\begin{equation*}
\mathsf{P}b\left( x\right) \equiv \sum_{k=0}^{\infty }\sum_{L\in \mathfrak{C}%
_{\mathcal{L}}^{\left( k\right) }\left( T\right) }\mathsf{P}_{\mathcal{C}_{%
\mathcal{L}}^{\left[ N\right] }\left( L\right) }^{\mu }b\left( x\right) ,
\end{equation*}%
and hence by the square function estimates in Theorem \ref{square thm}, we
have for each $s$,%
\begin{equation*}
\int_{\mathbb{R}}\left\vert \sum_{k\in \kappa +m\mathbb{N}}b_{k}^{\left[
\kappa \right] ,\left( s\right) }\left( x\right) \right\vert ^{p}d\mu \left(
x\right) \approx \left\Vert \left\vert \left\{ \widehat{B_{m,k}^{\left[
\kappa \right] ,\left( s\right) }}\right\} _{k\in \kappa +m\mathbb{N}%
}\right\vert _{\ell ^{2}}\right\Vert _{L^{p}\left( \mu \right)
}^{p}=\left\Vert \left\vert \left\{ B_{k}^{\left( s\right) }\right\} _{k\in
\kappa +m\mathbb{N}}\right\vert _{\ell ^{2}}\right\Vert _{L^{p}\left( \mu
\right) }^{p},
\end{equation*}%
and hence that%
\begin{eqnarray*}
\left\Vert \widehat{g^{\left[ \kappa \right] }}\right\Vert _{L^{p}\left( \mu
\right) }^{p} &\lesssim &M_{\mathcal{L}}^{\left( m,N\right) }\left(
g,b\right) ^{p}\int_{\mathbb{R}}\left\vert \sum_{k\in \kappa +m\mathbb{N}%
}b_{k}^{\left[ \kappa \right] ,\left[ N\right] }\left( x\right) \right\vert
^{p}d\mu \left( x\right) \\
&=&M_{\mathcal{L}}^{\left( m,N\right) }\left( g,b\right) ^{p}\int_{\mathbb{R}%
}\left\vert \sum_{k\in \kappa +m\mathbb{N}}\sum_{s=0}^{N}B_{k}^{\left[
\kappa \right] ,\left( s\right) }\left( x\right) \right\vert ^{p}d\mu \left(
x\right) \\
&\lesssim &M_{\mathcal{L}}^{\left( m,N\right) }\left( g,b\right)
^{p}N^{p}\sup_{0\leq s\leq N}\int_{\mathbb{R}}\left\vert \sum_{k\in \kappa +m%
\mathbb{N}}B_{k}^{\left[ \kappa \right] ,\left( s\right) }\left( x\right)
\right\vert ^{p}d\mu \left( x\right) .
\end{eqnarray*}%
Altogether,%
\begin{eqnarray*}
\left\Vert \left\vert G_{m}^{\func{bad}}\right\vert _{\ell ^{2}}\right\Vert
_{L^{p}\left( \mu \right) }^{p} &\lesssim &\sum_{\kappa =1}^{m+1}\left\Vert 
\widehat{g^{\left[ \kappa \right] }}\right\Vert _{L^{p}\left( \mu \right)
}^{p}\lesssim M_{\mathcal{L}}^{\left( m,N\right) }\left( g,b\right)
^{p}N^{p}\sum_{\kappa =1}^{m+1}\sup_{0\leq s\leq N}\int_{\mathbb{R}%
}\left\vert \sum_{k\in \kappa +m\mathbb{N}}B_{k}^{\left[ \kappa \right]
,\left( s\right) }\left( x\right) \right\vert ^{p}d\mu \left( x\right) \\
&\lesssim &mN^{p}M_{\mathcal{L}}^{\left( m,N\right) }\left( g,b\right)
^{p}\sup_{0\leq s\leq N}\left\Vert \left\vert \left\{ B_{k}^{\left( s\right)
}\right\} _{k\in \mathbb{N}}\right\vert _{\ell ^{2}}\right\Vert
_{L^{p}\left( \mu \right) }^{p}\ .
\end{eqnarray*}%
Using (\ref{follows that}) and the previous line,%
\begin{eqnarray*}
\left\Vert \left\vert G\right\vert _{\ell ^{2}}\right\Vert _{L^{p}\left( \mu
\right) } &\leq &\left\Vert \left\vert G_{m}^{\func{good}}\right\vert _{\ell
^{2}}\right\Vert _{L^{p}\left( \mu \right) }+\left\Vert \left\vert G_{m}^{%
\func{bad}}\right\vert _{\ell ^{2}}\right\Vert _{L^{p}\left( \mu \right) } \\
&\lesssim &M_{\mathcal{L}}^{\left( m,N\right) }\left( g,b\right) \left\Vert
\left\vert B^{\left[ N\right] }\right\vert _{\ell ^{2}}\right\Vert
_{L^{p}\left( \mu \right) }+mNM_{\mathcal{L}}^{\left( m,N\right) }\left(
g,b\right) \sup_{0\leq s\leq N}\left\Vert \left\vert \left\{ B_{k}^{\left(
s\right) }\right\} _{k\in \mathbb{N}}\right\vert _{\ell ^{2}}\right\Vert
_{L^{p}\left( \mu \right) } \\
&\leq &mNM_{\mathcal{L}}^{\left( m,N\right) }\left( g,b\right) \sup_{0\leq
s\leq N}\left\Vert \left\vert \left\{ B_{k}^{\left( s\right) }\right\}
_{k\in \mathbb{N}}\right\vert _{\ell ^{2}}\right\Vert _{L^{p}\left( \mu
\right) }\ ,
\end{eqnarray*}%
which completes the proof of inequality (\ref{part 2}) in the conclusion of
the Corona Martingale Comparison Principle.
\end{proof}

\subsection{$L^{p}$-Stopping Child Lemma}

We begin by defining the iteration of general stopping times, which we
remind the reader are simply subsets of the dyadic grid $\mathcal{D}$.

\begin{definition}
\label{def comp}Suppose $\mathcal{Q\subset D}$, and for each $Q\in \mathcal{Q%
}$, let $\mathcal{A}\left[ Q\right] \subset \mathcal{C}_{\mathcal{Q}}\left(
Q\right) $ with $Q\in \mathcal{A}\left[ Q\right] $, which can be thought of
as a family $\left\{ \mathcal{A}\left[ Q\right] \right\} _{Q\in \mathcal{Q}}$
of stopping times indexed by $Q\in \mathcal{Q}$. Then we define the
composition $\mathcal{Q\circ A}$ to be $\mathcal{Q\circ A}\equiv
\bigcup_{Q\in \mathcal{Q}}\mathcal{A}\left[ Q\right] $, which can also be
written simply as $\mathcal{A}$ when the additional structure arising from $%
\mathcal{Q}$ is unimportant.
\end{definition}

Recall that the $p$-energy defined in (\ref{def p energy}) by, 
\begin{equation*}
\mathsf{E}_{p}\left( I,\omega \right) \equiv \left( \frac{1}{\left\vert
I\right\vert _{\omega }}\int_{I}\left\vert x-\frac{1}{\left\vert
I\right\vert _{\omega }}\int_{I}zd\omega \left( z\right) \right\vert
^{p}d\omega \left( x\right) \right) ^{\frac{1}{p}}\approx \left( \frac{1}{%
\left\vert I\right\vert _{\omega }}\int_{I}\left( \sum_{J\subset
I}\left\vert \bigtriangleup _{J}^{\omega }Z\left( x\right) \right\vert
^{2}\right) ^{\frac{p}{2}}d\omega \left( x\right) \right) ^{\frac{1}{p}}.
\end{equation*}%
For $\Lambda \subset \mathcal{D}\left[ I\right] $ let 
\begin{equation}
\mathsf{E}_{p}\left( \Lambda ;\omega \right) \equiv \sqrt[p]{\frac{1}{%
\left\vert I\right\vert _{\omega }}\int_{I}\left( \sum_{J\in \Lambda
}\left\vert \bigtriangleup _{J}^{\omega }Z\left( x\right) \right\vert
^{2}\right) ^{\frac{p}{2}}d\omega \left( x\right) }\approx \left\Vert \frac{1%
}{\left\vert I\right\vert _{\omega }}\sum_{J\in \Lambda }\bigtriangleup
_{J}^{\omega }Z\right\Vert _{L^{p}\left( \omega \right) }=\left\Vert \frac{1%
}{\left\vert I\right\vert _{\omega }}\mathsf{P}_{\Lambda }^{\omega
}Z\right\Vert _{L^{p}\left( \omega \right) },  \label{equiv Ep}
\end{equation}%
which generalizes the $p$-energy $\mathsf{E}_{p}\left( I,\omega \right) $
defined for an interval $I$ - indeed, one immediately checks that $\mathsf{E}%
_{p}\left( I,\omega \right) =\mathsf{E}_{p}\left( \mathcal{D}\left[ I\right]
;\omega \right) $.

We will state our stopping child lemma in the context of an iterated
stopping time $\mathcal{Q\circ A}$ where only $\mathcal{A}$ is assumed to
have the structure arising from Lemma \ref{lem I*}. In fact, we will only
apply this lemma later on to the special case $\mathcal{S}^{\left( n\right)
}=\mathcal{S}^{\left( n-1\right) }\circ \mathcal{A}_{n}$ with $\mathcal{Q}=%
\mathcal{S}^{\left( n-1\right) }$ and $\mathcal{A}=\mathcal{A}_{n}$, so we
will state our stopping child lemma only in this case, but observe that the
estimate for the off-diagonal terms here involves no structure from $%
\mathcal{Q}$, unlike the diagonal terms treated later on, which rely
crucially on the structure of $\mathcal{Q}=\mathcal{S}^{\left( n-1\right) }$%
, and which will be treated using the lemmas from the previous subsection.
Here are the details.

We now recursively define the sequence of stopping times $\mathfrak{Q}\equiv
\left\{ \mathcal{S}^{\left( n\right) }\right\} _{n=1}^{\infty }$ that we
will consider in the remainder of the proof. Set $\mathcal{S}^{\left(
0\right) }\equiv \mathcal{F}$ and $\mathcal{S}^{\left( 1\right) }\equiv 
\mathcal{S}^{\left( 0\right) }\mathcal{\circ A}_{1}=\mathcal{F\circ A}_{1}$
where $\mathcal{A}_{1}$ is constructed using Lemma \ref{lem I*} for a dyadic
tree with the parameter $\Gamma >1$ fixed, but close to $1$, and where $\nu
=\nu _{\Lambda _{g_{A_{0}}}^{\omega }}$ for $A_{0}\in \mathcal{S}^{\left(
0\right) }$, \ and $g_{A_{0}}=\mathsf{P}_{\mathcal{D}\left[ A_{0}\right] }g$%
. See Definition \ref{def comp} for $\mathcal{F\circ A}_{1}$. Then set $%
\mathcal{S}^{\left( 2\right) }\equiv \mathcal{S}^{\left( 1\right) }\mathcal{%
\circ A}_{2}$ where $\mathcal{A}_{2}$ is now constructed using Lemma \ref%
{lem I*} relative to the stopping times $\mathcal{S}^{\left( 1\right) }$
instead of $\mathcal{S}^{\left( 0\right) }=\mathcal{F}$, and with $\nu =\nu
_{\Lambda _{g_{A_{1}}}^{\omega }}$ for $A_{1}\in \mathcal{A}_{1}\left[ F%
\right] $, and $g_{A_{1}}=\mathsf{P}_{\mathcal{D}\left[ A_{1}\right] }g$.
Continue by defining recursively,%
\begin{equation*}
\mathcal{S}^{\left( n+1\right) }\equiv \mathcal{S}^{\left( n\right) }%
\mathcal{\circ A}_{n+1},\ \ \ \ \ \text{for all }n\geq 1.
\end{equation*}%
Note that all of the stopping times $\mathcal{A}_{k}$ for $k\geq 1$ are
constructed with the same fixed parameter $\Gamma >1$ in Lemma \ref{lem I*},
but with smaller and smaller collections $\Lambda _{g_{A_{n}}}^{\omega }$ of 
$\mathcal{D}$ as $n$ increases.

We define the separated stopping form%
\begin{equation*}
\mathsf{B}_{\limfunc{stop}\func{sep}}^{\mathcal{A}}\left( f,g\right)
=\sum_{F\in \mathcal{F}}\sum_{Q\in \mathcal{Q}\left[ F\right] }\sum_{A\in 
\mathcal{A}\left[ Q\right] }\mathsf{B}_{\limfunc{stop}\func{sep}}^{\mathcal{A%
}\left[ Q\right] ,A}\left( f,g\right) ,
\end{equation*}%
where%
\begin{equation*}
\mathsf{B}_{\limfunc{stop}\func{sep}}^{\mathcal{A}\left[ Q\right] ,A}\left(
f,g\right) =\sum_{S\in \mathfrak{C}_{\mathcal{A}}\left( A\right) }\sum_{J\in 
\mathcal{C}_{\mathcal{Q}}\left( Q\right) \cap \mathcal{D}\left[ S\right]
}\left\langle \bigtriangleup _{J}^{\omega }H_{\sigma }\varphi
_{J}^{F,S},\bigtriangleup _{J}^{\omega }g\right\rangle _{\omega }
\end{equation*}%
is the local separated form - called `separated' because there is a child $S$
separating the intervals $J$ from\ the intervals $I$ arising in the sum for $%
\varphi _{J}^{F,S}$.

Finally for any \emph{sequence} $\Lambda _{\mathcal{Q}}=\left\{ \Lambda
_{Q}\right\} _{Q\in \mathcal{Q}}$ of subsets $\Lambda _{Q}\subset \mathcal{D}%
\left[ Q\right] $ for $Q\in \mathcal{Q}$, and $\delta >0$, we define%
\begin{eqnarray}
&&  \label{def fAQ} \\
&&\left\vert f\right\vert _{\mathcal{F},\mathcal{Q},\mathcal{A}}^{\Lambda _{%
\mathcal{Q}}}\left( x\right) \equiv \sqrt{\sum_{F\in \mathcal{F}}\sum_{Q\in 
\mathcal{Q}\left[ F\right] }\sum_{A\in \mathcal{A}\left[ Q\right] }2^{-%
\limfunc{dist}\left( A,Q\right) \delta }\sum_{S\in \mathfrak{C}_{\mathcal{A}%
}\left( A\right) }\sum_{K\in \mathcal{W}_{\func{good},\tau }\left( S\right)
}\alpha _{A}\left( S\right) ^{2}\left( \frac{\mathrm{P}\left( K,\mathbf{1}%
_{F\setminus S}\sigma \right) }{\ell \left( K\right) }\right) ^{2}\left\vert 
\mathsf{P}_{S;K}^{\omega ,\Lambda _{Q}}\right\vert Z\left( x\right) ^{2}}, 
\notag \\
&&\text{\ \ \ \ \ \ \ \ \ \ \ \ \ \ \ \ \ \ \ \ where }\mathsf{P}%
_{S;K}^{\omega ,\Lambda _{Q}}h\equiv \sum_{\substack{ J\in \Lambda _{Q}:\
J\subset _{\tau }\Lambda _{f}^{\sigma }\left[ S\right]  \\ J\subset K}}%
\bigtriangleup _{J}^{\omega }h\text{ and }\left\vert \mathsf{P}%
_{S;K}^{\omega ,\Lambda _{Q}}\right\vert h\equiv \sqrt{\sum_{\substack{ J\in
\Lambda _{Q}:\ J\subset _{\tau }\Lambda _{f}^{\sigma }\left[ S\right]  \\ %
J\subset K}}\left\vert \bigtriangleup _{J}^{\omega }h\right\vert ^{2}}. 
\notag
\end{eqnarray}%
Certain special cases of this rather complicated expression can be thought
of as substituting for the role of Lacey's size condition, but constrained
to live in the world of the measure $\omega $.

In our application of the $L^{p}$-Stopping Child Lemma, the main hypothesis (%
\ref{geo dec bound'}) below will follow from iterating the negation of the
first line in (\ref{dec corona}) of the dual tree decomposition.

\begin{lemma}[Quadratic $L^{p}$-Stopping Child Lemma]
\label{straddle3}Let $1<p<\infty $, and $f\in L^{p}\left( \sigma \right)
\cap L^{2}\left( \sigma \right) $, $g\in L^{p^{\prime }}\left( \omega
\right) \cap L^{2}\left( \omega \right) $ have their Haar supports in $%
\mathcal{D}_{\func{good}}^{\limfunc{child}}$, and let $\mathcal{F}$ be a
collection of $\func{good}$ stopping times satisfying a $\sigma $-Carleson
condition. Let $\mathcal{Q}=\mathcal{S}^{\left( n-1\right) }$ so that $%
\mathcal{Q}\left[ F\right] \subset \mathcal{C}_{\mathcal{F}}\left( F\right) $
is a set of $\func{good}$ stopping times with top $F$, and note that $%
\mathcal{Q}=\bigcup_{F\in \mathcal{F}}\mathcal{Q}\left[ F\right] $. Set $%
\mathcal{A}=\mathcal{S}^{\left( n\right) }$ so that for each $Q\in \mathcal{Q%
}$, the collection $\mathcal{A}\left[ Q\right] \subset \mathcal{C}_{\mathcal{%
Q}}\left( Q\right) $ is a set of $\func{good}$ stopping times with top
interval $Q$. For $A\in \mathcal{A}$ and $S\in \mathfrak{C}_{\mathcal{A}%
}\left( A\right) $, set 
\begin{equation}
\alpha _{A}\left( S\right) \equiv \sup_{I\in \left( \Lambda _{f}^{\sigma } 
\left[ S\right] ,A\right] \cap \mathcal{D}_{\func{good}}}\left\vert
E_{I}^{\sigma }f\right\vert ,  \label{def alpha AS}
\end{equation}%
where $\Lambda _{f}^{\sigma }\left[ S\right] $ is the smallest interval in
the Haar support $\Lambda _{f}^{\sigma }$ of $f$ that contains $S$. Finally,
we suppose there is $N\in \mathbb{N}$ and $\delta >0$ such that for all 
\begin{equation*}
\left( F,Q,A,S,K\right) \in \mathcal{F}\times \mathcal{Q}\left[ F\right]
\times \mathcal{A}\left[ Q\right] \times \mathfrak{C}_{\mathcal{A}}\left(
A\right) \times \mathcal{W}_{\func{good}}^{\limfunc{trip}}\left( S\right) 
\text{ and }m\geq 1,
\end{equation*}%
we have the equivalence,%
\begin{equation}
\mathsf{E}_{p}\left( \mathcal{C}_{\mathcal{Q}}\left( Q\right) \cap \left\{ 
\mathcal{D}_{\func{good}}\left[ K\right] \setminus \mathcal{D}_{\func{good}}%
\left[ \cup \mathfrak{C}_{\mathcal{A}}^{\left( N+1\right) }\left( S\right)
\cap K\right] \right\} \cap \Lambda _{g}^{\omega };\omega \right) \approx 
\mathsf{E}_{p}\left( \mathcal{C}_{\mathcal{Q}}\left( Q\right) \cap \mathcal{D%
}_{\func{good}}\left[ K\right] \cap \Lambda _{g}^{\omega };\omega \right) ,
\label{N opt}
\end{equation}%
and also the geometric decay bound (which is trivial when $m=1$),%
\begin{eqnarray}
&&\mathsf{E}_{p}\left( \mathcal{C}_{\mathcal{Q}}\left( Q\right) \cap \left\{ 
\mathcal{D}_{\func{good}}\left[ \bigcup_{B\in \mathfrak{C}_{\mathcal{A}%
}^{\left( m-1\right) }\left( S\right) }B\cap K\right] \right\} \cap \Lambda
_{g}^{\omega };\omega \right)  \label{geo dec bound'} \\
&\leq &C2^{-m\delta }\mathsf{E}_{p}\left( \mathcal{C}_{\mathcal{Q}}\left(
Q\right) \cap \left\{ \mathcal{D}_{\func{good}}\left[ K\right] \setminus 
\mathcal{D}_{\func{good}}\left[ \cup \mathfrak{C}_{\mathcal{A}}^{\left(
N+1\right) }\left( S\right) \cap K\right] \right\} \cap \Lambda _{g}^{\omega
};\omega \right) ,  \notag
\end{eqnarray}%
where $\mathfrak{C}_{\mathcal{A}}^{\left( k\right) }\left( S\right) $ is
defined in Notation \ref{corona notation} (and note that $\delta $ in(\ref%
{geo dec bound'}) is different than the $\delta $ appearing in (\ref{geom
decay})).\newline
Then we have the following nonlinear bound for all $\mathcal{A}=\mathcal{S}%
^{\left( n\right) }\supset \mathcal{Q}=\mathcal{S}^{\left( n-1\right)
}\supset \mathcal{F}$,%
\begin{equation}
\left\vert \mathsf{B}_{\limfunc{stop}\func{sep}}^{\mathcal{A}}\left(
f,g\right) \right\vert \lesssim \left( 1+2^{N\delta }\right) \frac{N^{2}}{%
\delta ^{2}}\left\Vert \left\vert f\right\vert _{\mathcal{F},\mathcal{Q},%
\mathcal{A}}^{\left\{ \mathcal{C}_{\mathcal{A}}^{\left[ N\right] }\left(
A\right) \cap \Lambda _{g_{Q}}^{\omega }\right\} _{A\in \mathcal{A}%
}}\right\Vert _{L^{p}\left( \omega \right) }\left\Vert g\right\Vert
_{L^{p^{\prime }}\left( \omega \right) }\ ,\ \ \ \ \ 1<p<\infty .
\label{stop est'}
\end{equation}
\end{lemma}

Note that in the definition of $\left\vert f\right\vert _{\mathcal{F},%
\mathcal{Q},\mathcal{A}}^{\left\{ \mathcal{C}_{\mathcal{A}}^{\left[ N\right]
}\left( A\right) \cap \Lambda _{g_{Q}}^{\omega }\right\} _{A\in \mathcal{A}%
}} $ in (\ref{def fAQ}), the restrictions $J\in \mathcal{C}_{\mathcal{A}}^{%
\left[ N\right] }\left( S\right) $ and $J\subset K\in \mathcal{W}_{\func{good%
},\tau }\left( S\right) $ on the intervals $J$ arising in the absolute
projection $\left\vert \mathsf{P}_{S;K}^{\omega ,\mathcal{C}_{\mathcal{A}}^{%
\left[ N\right] }\left( A\right) \cap \Lambda _{g_{Q}}^{\omega }}\right\vert 
$, imply that $K$ satisfies%
\begin{equation}
K\in \mathcal{W}_{\func{good},\tau }^{\mathcal{A},\left[ N\right] }\left(
S\right) \equiv \mathcal{C}_{\mathcal{A}}^{\left[ N\right] }\left( S\right)
\cap \mathcal{W}_{\func{good},\tau }\left( S\right) .  \label{K sat}
\end{equation}

\begin{proof}
Recall that%
\begin{eqnarray*}
&&\mathsf{B}_{\limfunc{stop}}^{\mathcal{A}}\left( f,g\right) =\sum_{F\in 
\mathcal{F}}\sum_{A\in \mathcal{A}\left[ F\right] }\mathsf{B}_{\limfunc{stop}%
}^{\mathcal{F},F}\left( \mathsf{P}_{\mathcal{C}_{\mathcal{A}\left[ Q\right]
}\left( A\right) }^{\sigma }f,\mathsf{P}_{\mathcal{C}_{\mathcal{A}\left[ Q%
\right] }\left( A\right) }^{\omega }g\right) , \\
&&\mathsf{B}_{\limfunc{stop}\func{sep}}^{\mathcal{A}}\left( f,g\right)
=\sum_{F\in \mathcal{F}}\sum_{\substack{ A,B\in \mathcal{A}\left[ F\right] 
\\ B\subsetneqq A}}\mathsf{B}_{\limfunc{stop}}^{\mathcal{F},F}\left( \mathsf{%
P}_{\mathcal{C}_{\mathcal{A}\left[ Q\right] }\left( A\right) }^{\sigma }f,%
\mathsf{P}_{\mathcal{C}_{\mathcal{A}\left[ Q\right] }\left( B\right)
}^{\omega }g\right) \\
&&\ \ \ \ \ \ \ \ \ \ \ \ \ \ \ \ \ \ \ =\sum_{F\in \mathcal{F}}\sum_{A\in 
\mathcal{A}\left[ F\right] }\sum_{m=1}^{\infty }\sum_{B\in \mathcal{C}_{%
\mathcal{A}\left[ Q\right] }^{\left( m\right) }\left( A\right) }\mathsf{B}_{%
\limfunc{stop}}^{\mathcal{F},F}\left( \mathsf{P}_{\mathcal{C}_{\mathcal{A}%
\left[ Q\right] }\left( A\right) }^{\sigma }f,\mathsf{P}_{\mathcal{C}_{%
\mathcal{A}\left[ Q\right] }\left( B\right) }^{\omega }g\right) \\
&&\ \ \ \ \ \ \ \ \ \ \ \ \ \ \ \ \ \ \ =\sum_{F\in \mathcal{F}}\sum_{A\in 
\mathcal{A}\left[ F\right] }\sum_{m=1}^{\infty }\mathsf{B}_{\limfunc{stop}}^{%
\mathcal{F},F}\left( \mathsf{P}_{\mathcal{C}_{\mathcal{A}\left[ Q\right]
}\left( A\right) }^{\sigma }f,\mathsf{P}_{\mathcal{C}_{\mathcal{A}\left[ Q%
\right] }^{\left( m\right) }\left( A\right) }^{\omega }g\right) ,
\end{eqnarray*}%
where for each $m\geq 1$, $F\in \mathcal{F}$ and $A\in \mathcal{A}\left[ F%
\right] $, we can write, 
\begin{eqnarray*}
&&\mathsf{B}_{\limfunc{stop}}^{\mathcal{F},F}\left( \mathsf{P}_{\mathcal{C}_{%
\mathcal{A}\left[ Q\right] }\left( A\right) }^{\sigma }f,\mathsf{P}_{%
\mathcal{C}_{\mathcal{A}\left[ Q\right] }^{\left( m\right) }\left( A\right)
}^{\omega }g\right) =\sum_{S\in \mathfrak{C}_{\mathcal{A}}\left( A\right) }%
\mathsf{B}_{\limfunc{stop}}^{\mathcal{F},F}\left( \mathsf{P}_{\mathcal{C}_{%
\mathcal{A}\left[ Q\right] }\left( A\right) }^{\sigma }f,\mathsf{P}_{%
\mathcal{C}_{\mathcal{A}\left[ Q\right] }^{\left( m-1\right) }\left(
S\right) }^{\omega }g\right) \\
&=&\sum_{S\in \mathfrak{C}_{\mathcal{A}}\left( A\right) }\sum_{\substack{ %
\left( I,J\right) \in \left( S,A\right] \times \left\{ \mathcal{C}_{\mathcal{%
F}}\left( F\right) \cap \mathcal{D}\left[ S\right] \right\}  \\ J\subset
_{\tau }I}}\left( E_{I_{J}}^{\sigma }\bigtriangleup _{I}^{\sigma }f\right)
\left\langle H_{\sigma }\left( \mathbf{1}_{F\setminus S}\right)
,\bigtriangleup _{J}^{\omega }\mathsf{P}_{\mathcal{C}_{\mathcal{A}\left[ Q%
\right] }^{\left( m-1\right) }\left( S\right) }^{\omega }g\right\rangle
_{\omega } \\
&=&\sum_{S\in \mathfrak{C}_{\mathcal{A}}\left( A\right) }\sum_{J\in \mathcal{%
C}_{\mathcal{A}\left[ Q\right] }^{\left( m-1\right) }\left( S\right)
}\left\langle \bigtriangleup _{J}^{\omega }H_{\sigma }\varphi _{J}^{\mathcal{%
A},F,S},\bigtriangleup _{J}^{\omega }\mathsf{P}_{\mathcal{C}_{\mathcal{A}%
\left[ Q\right] }^{\left( m-1\right) }\left( S\right) }^{\omega
}g\right\rangle _{\omega }\ ,
\end{eqnarray*}%
where%
\begin{equation*}
\varphi _{J}^{\mathcal{A},F,S}\equiv \sum_{I\in \left( S,A\right] :\
J\subset _{\tau }\Lambda _{f}^{\sigma }\left[ S\right] }\left(
E_{I_{J}}^{\sigma }\bigtriangleup _{I}^{\sigma }f\right) \mathbf{1}%
_{F\setminus S},\ \ \ \ \ A=\pi _{\mathcal{A}}S\ ,
\end{equation*}%
and where $\Lambda _{f}^{\sigma }\left[ S\right] $ is the smallest interval
in the Haar support $\Lambda _{f}^{\sigma }$ of $f$ that strictly contains $%
S $. We rename,%
\begin{equation*}
\mathsf{B}_{\limfunc{stop}\func{sep}}^{\mathcal{A}\left[ F\right]
,A,m}\left( f,g\right) \equiv \mathsf{B}_{\limfunc{stop}}^{\mathcal{F}%
,F}\left( \mathsf{P}_{\mathcal{C}_{\mathcal{A}\left[ Q\right] }\left(
A\right) }^{\sigma }f,\mathsf{P}_{\mathcal{C}_{\mathcal{A}\left[ Q\right]
}^{\left( m\right) }\left( A\right) }^{\omega }g\right) ,
\end{equation*}%
and write,%
\begin{equation*}
\mathsf{B}_{\limfunc{stop}\func{sep}}^{\mathcal{A}\left[ Q\right] ,A}\left(
f,g\right) =\sum_{S\in \mathfrak{C}_{\mathcal{A}}\left( A\right) }\sum_{J\in 
\mathcal{C}_{\mathcal{Q}}\left( Q\right) \cap \mathcal{D}\left[ S\right]
}\left\langle \bigtriangleup _{J}^{\omega }H_{\sigma }\varphi _{J}^{\mathcal{%
A},F,S},\bigtriangleup _{J}^{\omega }g\right\rangle _{\omega
}=\sum_{m=1}^{\infty }\mathsf{B}_{\limfunc{stop}}^{\mathcal{A}\left[ Q\right]
,A,m}\left( f,g\right) \ ,
\end{equation*}%
where%
\begin{equation*}
\mathsf{B}_{\limfunc{stop}\func{sep}}^{\mathcal{A}\left[ Q\right]
,A,m}\left( f,g\right) \equiv \sum_{S\in \mathfrak{C}_{\mathcal{A}}\left(
A\right) }\sum_{J\in \mathcal{C}_{\mathcal{A}}^{\left( m-1\right) }\left(
S\right) }\left\langle \bigtriangleup _{J}^{\omega }H_{\sigma }\varphi _{J}^{%
\mathcal{A},F,S},\bigtriangleup _{J}^{\omega }g\right\rangle _{\omega }\ .
\end{equation*}%
Recalling our convention regarding iterated sums, we define%
\begin{eqnarray*}
\mathsf{B}_{\limfunc{stop}\func{sep}}^{\mathcal{Q}\circ \mathcal{A}}\left(
f,g\right) &\equiv &\sum_{m=1}^{\infty }\mathsf{B}_{\limfunc{stop}\func{sep}%
}^{\mathcal{Q}\circ \mathcal{A},m}\left( f,g\right) , \\
\mathsf{B}_{\limfunc{stop}\func{sep}}^{\mathcal{Q}\circ \mathcal{A},m}\left(
f,g\right) &\equiv &\sum_{F\in \mathcal{F}}\sum_{Q\in \mathcal{Q}}\sum_{A\in 
\mathcal{A}}\mathsf{B}_{\limfunc{stop}\func{sep}}^{\mathcal{A}\left[ Q\right]
,A,m}\left( f,g\right) \\
&=&\sum_{F\in \mathcal{F}}\sum_{Q\in \mathcal{Q}}\sum_{A\in \mathcal{A}%
}\sum_{S\in \mathfrak{C}_{\mathcal{A}}\left( A\right) }\sum_{J\in \mathcal{C}%
_{\mathcal{A}}^{\left( m-1\right) }\left( S\right) }\left\langle
\bigtriangleup _{J}^{\omega }H_{\sigma }\varphi _{J}^{\mathcal{A}%
,F,S},\bigtriangleup _{J}^{\omega }g\right\rangle _{\omega }.
\end{eqnarray*}%
With $g_{A,m}\equiv \mathsf{P}_{\mathcal{C}_{\mathcal{A}}^{\left( m\right)
}\left( A\right) }g$, where $\mathcal{C}_{\mathcal{A}}^{\left( m\right)
}\left( A\right) $ is as in Notation \ref{corona notation}, we obtain for
each $m\in \mathbb{N}$, 
\begin{eqnarray*}
&&\left\vert \mathsf{B}_{\limfunc{stop}\func{sep}}^{\mathcal{Q}\circ 
\mathcal{A},m}\left( f,g\right) \right\vert =\left\vert \sum_{F\in \mathcal{F%
}}\sum_{Q\in \mathcal{Q}}\sum_{A\in \mathcal{A}}\mathsf{B}_{\limfunc{stop}%
\func{sep}}^{\mathcal{A}\left[ Q\right] ,A,m}\left( f,g\right) \right\vert \\
&=&\left\vert \sum_{F\in \mathcal{F}}\sum_{Q\in \mathcal{Q}}\sum_{A\in 
\mathcal{A}}\sum_{S\in \mathfrak{C}_{\mathcal{A}}\left( A\right) }\sum_{J\in 
\mathcal{C}_{\mathcal{A}}^{\left( m-1\right) }\left( S\right) \cap \Lambda
_{g_{A,m}}^{\omega }:\ J\subset _{\tau }\Lambda _{f}^{\sigma }\left[ S\right]
}\int_{\mathbb{R}}\bigtriangleup _{J}^{\omega }H_{\sigma }\varphi _{J}^{%
\mathcal{A},F,S}\left( x\right) \bigtriangleup _{J}^{\omega }g\left(
x\right) d\omega \left( x\right) \right\vert \\
&\leq &\int_{\mathbb{R}}\left\vert \sum_{F\in \mathcal{F}}\sum_{Q\in 
\mathcal{Q}}\sum_{A\in \mathcal{A}}\sum_{S\in \mathfrak{C}_{\mathcal{A}%
}\left( A\right) }\sum_{J\in \mathcal{C}_{\mathcal{A}}^{\left( m-1\right)
}\left( S\right) \cap \Lambda _{g_{A,m}}^{\omega }:\ J\subset _{\tau
}\Lambda _{f}^{\sigma }\left[ S\right] }\bigtriangleup _{J}^{\omega
}H_{\sigma }\varphi _{J}^{\mathcal{A},F,S}\left( x\right) \bigtriangleup
_{J}^{\omega }g\left( x\right) \right\vert d\omega \left( x\right) \\
&\leq &\int_{\mathbb{R}}\left( \sum_{F\in \mathcal{F}}\sum_{Q\in \mathcal{Q}%
}\sum_{A\in \mathcal{A}}\sum_{S\in \mathfrak{C}_{\mathcal{A}}\left( A\right)
}\sum_{J\in \mathcal{C}_{\mathcal{A}}^{\left( m-1\right) }\left( S\right)
\cap \Lambda _{g_{A,m}}^{\omega }:\ J\subset _{\tau }\Lambda _{f}^{\sigma }%
\left[ S\right] }\left\vert \bigtriangleup _{J}^{\omega }H_{\sigma }\varphi
_{J}^{\mathcal{A},F,S}\left( x\right) \right\vert ^{2}\right) ^{\frac{1}{2}}
\\
&&\ \ \ \ \ \ \ \ \ \ \ \ \ \ \ \ \ \ \ \ \times \left( \sum_{F\in \mathcal{F%
}}\sum_{Q\in \mathcal{Q}}\sum_{A\in \mathcal{A}}\sum_{S\in \mathfrak{C}_{%
\mathcal{A}}\left( A\right) }\sum_{J\in \mathcal{C}_{\mathcal{A}}^{\left(
m-1\right) }\left( S\right) \cap \Lambda _{g_{A,m}}^{\omega }:\ J\subset
_{\tau }\Lambda _{f}^{\sigma }\left[ S\right] }\left\vert \bigtriangleup
_{J}^{\omega }g\left( x\right) \right\vert ^{2}\right) ^{\frac{1}{2}}d\omega
\left( x\right) ,
\end{eqnarray*}%
which is at most%
\begin{eqnarray*}
&&\left\Vert \left( \sum_{F\in \mathcal{F}}\sum_{Q\in \mathcal{Q}}\sum_{A\in 
\mathcal{A}}\sum_{S\in \mathfrak{C}_{\mathcal{A}}\left( A\right) }\sum_{J\in 
\mathcal{C}_{\mathcal{A}}^{\left( m-1\right) }\left( S\right) \cap \mathcal{C%
}_{\mathcal{Q}}\left( Q\right) \cap \Lambda _{g_{A,m}}^{\omega }:\ J\subset
_{\tau }\Lambda _{f}^{\sigma }\left[ S\right] }\left\vert \bigtriangleup
_{J}^{\omega }H_{\sigma }\varphi _{J}^{\mathcal{A},F,S}\left( x\right)
\right\vert ^{2}\right) ^{\frac{1}{2}}\right\Vert _{L^{p}\left( \omega
\right) } \\
&&\ \ \ \ \ \ \ \ \ \ \ \ \ \ \ \ \ \ \ \ \times \left\Vert \left(
\sum_{F\in \mathcal{F}}\sum_{Q\in \mathcal{Q}}\sum_{A\in \mathcal{A}%
}\sum_{S\in \mathfrak{C}_{\mathcal{A}}\left( A\right) }\sum_{J\in \mathcal{C}%
_{\mathcal{A}}^{\left( m-1\right) }\left( S\right) \cap \mathcal{C}_{%
\mathcal{Q}}\left( Q\right) \cap \Lambda _{g_{A,m}}^{\omega }:\ J\subset
_{\tau }\Lambda _{f}^{\sigma }\left[ S\right] }\left\vert \bigtriangleup
_{J}^{\omega }g\left( x\right) \right\vert ^{2}\right) ^{\frac{1}{2}%
}\right\Vert _{L^{p^{\prime }}\left( \omega \right) },
\end{eqnarray*}%
where the square function inequality in Theorem \ref{square thm} shows that
the second norm is bounded by $C\left\Vert g\right\Vert _{L^{p^{\prime
}}\left( \omega \right) }$, and because of this we assume without loss of
generality, 
\begin{equation}
\left\Vert g\right\Vert _{L^{p^{\prime }}\left( \omega \right) }=1.
\label{g=1}
\end{equation}%
By the telescoping property of martingale differences, together with the
bound $\alpha _{A}\left( S\right) $ in (\ref{def alpha AS}) on the averages
of $\mathsf{P}_{\mathcal{C}_{\mathcal{A}}\left( A\right) }^{\sigma }f$ in
the tower $\left( \Lambda _{f}^{\sigma }\left[ S\right] ,A\right] $, we have%
\begin{equation}
\left\vert \varphi _{J}^{\mathcal{A},F,S}\left( x\right) \right\vert
=\left\vert \sum_{I\in \left( \Lambda _{f}^{\sigma }\left[ S\right] ,A\right]
:\ J\subset _{\tau }I}\left( E_{I_{J}}^{\sigma }\bigtriangleup _{I}^{\sigma
}f\right) \left( x\right) \mathbf{1}_{A\setminus I_{J}}\left( x\right)
\right\vert \lesssim M_{\sigma }^{\func{dy}}\mathsf{P}_{\mathcal{C}_{%
\mathcal{A}}\left( A\right) }^{\sigma }\left( x\right) \mathbf{1}%
_{A\setminus S}\left( x\right) \equiv \alpha _{A}\left( S\right) \left(
x\right) \ .  \label{bfi 3}
\end{equation}

Next we use the Monotonicity Lemma and the fact that%
\begin{equation*}
\Lambda _{g}^{\omega }\cap \mathcal{D}\left[ S\right] \subset \left(
\bigcup_{K\in \mathcal{W}_{\func{good},\tau }\left( S\right) }K\right) \cup 
\mathcal{N}_{\tau }\left( S\right)
\end{equation*}%
where $\mathcal{W}_{\func{good},\tau }\left( S\right) $ is the collection of
maximal $\func{good}$ intervals $K$ in $I$ with $K\subset _{\tau }S$, and $%
\mathcal{N}_{\tau }\left( S\right) \equiv \left\{ J\subset S:\ell \left(
J\right) \geq 2^{-\tau }\ell \left( S\right) \right\} $ is the set of `$\tau 
$-nearby' dyadic intervals in $S$. Then remembering that $\left\Vert
g\right\Vert _{L^{p^{\prime }}\left( \omega \right) }=1$, we have the
following estimate for the sum in $m\in \mathbb{N}$, 
\begin{eqnarray*}
&&\left\vert \sum_{m=1}^{\infty }\sum_{F\in \mathcal{F}}\sum_{Q\in \mathcal{Q%
}}\sum_{A\in \mathcal{A}}\mathsf{B}_{\limfunc{stop}\func{sep}}^{\mathcal{A}%
\left[ Q\right] ,A,m}\left( f,g\right) \right\vert \\
&\leq &\left\Vert \left( \sum_{m=1}^{\infty }\sum_{F\in \mathcal{F}%
}\sum_{Q\in \mathcal{Q}}\sum_{A\in \mathcal{A}}\sum_{S\in \mathfrak{C}_{%
\mathcal{A}}\left( A\right) }\sum_{J\in \mathcal{C}_{\mathcal{A}}^{\left(
m-1\right) }\left( S\right) \cap \Lambda _{g_{A}}^{\omega }:\ J\subset
_{\tau }\Lambda _{f}^{\sigma }\left[ S\right] }\left\vert \bigtriangleup
_{J}^{\omega }H_{\sigma }\varphi _{J}^{\mathcal{A},F,S}\left( x\right)
\right\vert ^{2}\right) ^{\frac{1}{2}}\right\Vert _{L^{p}\left( \omega
\right) } \\
&\leq &\left\Vert \left( \sum_{m=1}^{\infty }\sum_{F\in \mathcal{F}%
}\sum_{Q\in \mathcal{Q}}\sum_{A\in \mathcal{A}}\sum_{S\in \mathfrak{C}_{%
\mathcal{A}}\left( A\right) }\alpha _{A}\left( S\right) ^{2}\sum_{K\in 
\mathcal{W}_{\func{good},\tau }\left( S\right) }\sum_{J\in \mathcal{C}_{%
\mathcal{A}}^{\left( m-1\right) }\left( S\right) \cap \Lambda
_{g_{A}}^{\omega }:\ J\subset _{\tau }\Lambda _{f}^{\sigma }\left[ S\right]
\cap K}\left( \frac{\mathrm{P}\left( J,\mathbf{1}_{F\setminus S}\sigma
\right) }{\ell \left( J\right) }\right) ^{2}\left\vert \bigtriangleup
_{J}^{\omega }Z\left( x\right) \right\vert ^{2}\right) ^{\frac{1}{2}%
}\right\Vert _{L^{p}\left( \omega \right) } \\
&&+\left\Vert \left( \sum_{m=1}^{\infty }\sum_{F\in \mathcal{F}}\sum_{Q\in 
\mathcal{Q}}\sum_{A\in \mathcal{A}}\sum_{S\in \mathfrak{C}_{\mathcal{A}%
}\left( A\right) }\alpha _{A}\left( S\right) ^{2}\sum_{J\in \mathcal{C}_{%
\mathcal{A}}^{\left( m-1\right) }\left( S\right) \cap \mathcal{N}_{\tau
}\left( S\right) \cap \Lambda _{g}^{\omega }:\ J\subset _{\tau }\Lambda
_{f}^{\sigma }\left[ S\right] }\left( \frac{\mathrm{P}\left( J,\mathbf{1}%
_{F\setminus S}\sigma \right) }{\ell \left( J\right) }\right) ^{2}\left\vert
\bigtriangleup _{J}^{\omega }Z\left( x\right) \right\vert ^{2}\right) ^{%
\frac{1}{2}}\right\Vert _{L^{p}\left( \omega \right) } \\
&\leq &\left( \sum_{m=1}^{\infty }\left\Vert \sqrt{\left\vert \mathsf{B}%
^{m}\right\vert _{\func{straddle}}^{\mathcal{Q}\circ \mathcal{A},\limfunc{%
trip}}\left( f\right) }\right\Vert _{L^{p}\left( \omega \right) }\right)
+\left\Vert \sqrt{\left\vert \mathsf{B}\right\vert _{\func{straddle}}^{%
\mathcal{Q}\circ \mathcal{A},\func{near}}\left( f\right) }\right\Vert
_{L^{p}\left( \omega \right) },
\end{eqnarray*}%
where 
\begin{eqnarray*}
\left\vert \mathsf{B}^{m}\right\vert _{\func{straddle}}^{\mathcal{Q}\circ 
\mathcal{A},\limfunc{trip}}\left( f\right) &\equiv &\sum_{F\in \mathcal{F}%
}\sum_{Q\in \mathcal{Q}}\sum_{A\in \mathcal{A}}\sum_{S\in \mathfrak{C}_{%
\mathcal{A}}\left( A\right) }\alpha _{A}\left( S\right) ^{2}\sum_{K\in 
\mathcal{W}_{\func{good},\tau }\left( S\right) }\sum_{J\in \mathcal{C}_{%
\mathcal{A}}^{\left( m-1\right) }\left( S\right) \cap \Lambda
_{g_{A}}^{\omega }:\ J\subset _{\tau }\Lambda _{f}^{\sigma }\left[ S\right]
\cap K}\left( \frac{\mathrm{P}\left( J,\mathbf{1}_{F\setminus S}\sigma
\right) }{\ell \left( J\right) }\right) ^{2}\left\vert \bigtriangleup
_{J}^{\omega }Z\left( x\right) \right\vert ^{2}, \\
\left\vert \mathsf{B}\right\vert _{\func{straddle}}^{\mathcal{Q}\circ 
\mathcal{A},\func{near}}\left( f\right) &\equiv &\sum_{F\in \mathcal{F}%
}\sum_{Q\in \mathcal{Q}}\sum_{A\in \mathcal{A}}\sum_{S\in \mathfrak{C}_{%
\mathcal{A}}\left( A\right) }\alpha _{A}\left( S\right) ^{2}\sum_{J\in 
\mathcal{N}_{\tau }\left( S\right) \cap \Lambda _{g}^{\omega }:\ J\subset
_{\tau }\Lambda _{f}^{\sigma }\left[ S\right] }\left( \frac{\mathrm{P}\left(
J,\mathbf{1}_{F\setminus S}\sigma \right) }{\ell \left( J\right) }\right)
^{2}\left\vert \bigtriangleup _{J}^{\omega }Z\left( x\right) \right\vert
^{2}.
\end{eqnarray*}

Now%
\begin{eqnarray*}
\frac{\mathrm{P}\left( J,\mathbf{1}_{F\setminus S}\sigma \right) }{\ell
\left( J\right) } &=&\int_{F\setminus S}\frac{1}{\left[ \ell \left( J\right)
+\limfunc{dist}\left( y,c_{J}\right) \right] ^{2}}d\sigma \left( y\right) \\
&\approx &\int_{F\setminus S}\frac{1}{\left[ \ell \left( K\right) +\limfunc{%
dist}\left( y,c_{K}\right) \right] ^{2}}d\sigma \left( y\right) =\frac{%
\mathrm{P}\left( K,\mathbf{1}_{F\setminus S}\sigma \right) }{\ell \left(
K\right) },
\end{eqnarray*}%
for $K\in \mathcal{W}_{\func{good},\tau }\left( S\right) \cup \mathcal{N}%
_{\tau }\left( S\right) $, as one easily verifies using that 
\begin{equation*}
\ell \left( J\right) +\limfunc{dist}\left( y,c_{J}\right) \approx \ell
\left( K\right) +\limfunc{dist}\left( y,c_{K}\right)
\end{equation*}%
in both cases. Thus with absolute projections, as defined in (\ref{abs proj}%
), 
\begin{equation}
\left\vert \mathsf{P}_{S;K}^{\omega ,\Lambda _{g}^{\omega },m}\right\vert
h\equiv \sqrt{\sum_{\substack{ J\in \mathcal{C}_{\mathcal{A}\left[ Q\right]
}^{\left( m-1\right) }\left( S\right) \cap \Lambda _{g}^{\omega }:\ J\subset
_{\tau }\Lambda _{f}^{\sigma }\left[ S\right]  \\ J\subset K}}\left\vert
\bigtriangleup _{J}^{\omega }h\right\vert ^{2}}\ \ \ \text{and \ \ }%
\left\vert \mathsf{P}_{S;K}^{\omega ,N}\right\vert h\equiv \sqrt{\sum 
_{\substack{ J\in \mathcal{C}_{\mathcal{A}}^{\left[ N\right] }\left(
S\right) :\ J\subset _{\tau }\Lambda _{f}^{\sigma }\left[ S\right]  \\ %
J\subset K}}\left\vert \bigtriangleup _{J}^{\omega }h\right\vert ^{2}},
\label{def PSK}
\end{equation}%
for $\left( F,Q,A,S,K\right) \in \mathcal{F}\times \mathcal{Q}\times 
\mathcal{A}\times \mathfrak{C}_{\mathcal{A}}\left( A\right) \times \mathcal{W%
}_{\func{good},\tau }\left( S\right) $, we can `lift' the intervals $J$ to
their corresponding $\func{good}$ Whitney interval $K$, to obtain that, 
\begin{eqnarray*}
&&\left\vert \mathsf{B}^{m}\right\vert _{\func{straddle}}^{\mathcal{Q}\circ 
\mathcal{A},\limfunc{trip}}\left( f\right) \lesssim \sum_{F\in \mathcal{F}%
}\sum_{Q\in \mathcal{Q}}\sum_{A\in \mathcal{A}}\sum_{S\in \mathfrak{C}_{%
\mathcal{A}}\left( A\right) }\alpha _{A}\left( S\right) ^{2}\sum_{K\in 
\mathcal{W}_{\func{good},\tau }\left( S\right) }\left( \frac{\mathrm{P}%
\left( K,\mathbf{1}_{F\setminus S}\sigma \right) }{\ell \left( K\right) }%
\right) ^{2} \\
&&\ \ \ \ \ \ \ \ \ \ \ \ \ \ \ \ \ \ \ \ \ \ \ \ \ \times \sum_{J\in 
\mathcal{C}_{\mathcal{A}}^{\left( m-1\right) }\left( S\right) \cap \Lambda
_{g}^{\omega }:\ J\subset _{\tau }\Lambda _{f}^{\sigma }\left[ S\right] \cap
K}\left\vert \bigtriangleup _{J}^{\omega }Z\left( x\right) \right\vert ^{2}
\\
&=&\sum_{F\in \mathcal{F}}\sum_{Q\in \mathcal{Q}}\sum_{A\in \mathcal{A}%
}\sum_{S\in \mathfrak{C}_{\mathcal{A}}\left( A\right) }\alpha _{A}\left(
S\right) ^{2}\sum_{K\in \mathcal{W}_{\func{good},\tau }\left( S\right)
}\left( \frac{\mathrm{P}\left( K,\mathbf{1}_{F\setminus S}\sigma \right) }{%
\ell \left( K\right) }\right) ^{2}\left\vert \mathsf{P}_{S;K}^{\omega
,\Lambda _{g_{Q}}^{\omega },m}\right\vert Z\left( x\right) ^{2},
\end{eqnarray*}%
and hence%
\begin{equation*}
\left\Vert \sqrt{\left\vert \mathsf{B}^{m}\right\vert _{\func{straddle}}^{%
\mathcal{Q}\circ \mathcal{A},\limfunc{trip}}\left( f\right) }\right\Vert
_{L^{p}\left( \omega \right) }\lesssim \left\Vert \sqrt{\sum_{F\in \mathcal{F%
}}\sum_{Q\in \mathcal{Q}}\sum_{A\in \mathcal{A}\left[ Q\right] }\sum_{S\in 
\mathfrak{C}_{\mathcal{A}}\left( A\right) }\alpha _{A}\left( S\right)
^{2}\sum_{K\in \mathcal{W}_{\func{good},\tau }\left( S\right) }\left( \frac{%
\mathrm{P}\left( K,\mathbf{1}_{F\setminus S}\sigma \right) }{\ell \left(
K\right) }\right) ^{2}\left\vert \mathsf{P}_{S;K}^{\omega ,\Lambda
_{g_{Q}}^{\omega },m}\right\vert Z\left( x\right) ^{2}}\right\Vert
_{L^{p}\left( \omega \right) }.
\end{equation*}

Using the corollary to the disjoint support Lemma \ref{disjoint supp},
together with the geometric decay bound (\ref{geo dec bound'}), we will now
use inequality (\ref{part 2}) in the Corona Martingale Comparison Principle
in Proposition \ref{CMCP} to prove that for every $m\geq 1$,%
\begin{eqnarray}
&&\ \ \ \ \ \ \ \ \ \ \ \ \ \ \ \left\Vert \sqrt{\sum_{F\in \mathcal{F}%
}\sum_{Q\in \mathcal{Q}}\sum_{A\in \mathcal{A}\left[ Q\right] }\sum_{S\in 
\mathfrak{C}_{\mathcal{A}}\left( A\right) }\sum_{K\in \mathcal{W}_{\func{good%
},\tau }\left( S\right) }\alpha _{A}\left( S\right) ^{2}\left( \frac{\mathrm{%
P}\left( K,\mathbf{1}_{F\setminus S}\sigma \right) }{\ell \left( K\right) }%
\right) ^{2}\left\vert \mathsf{P}_{S;K}^{\omega ,\Lambda _{g_{Q}}^{\omega
},m}\right\vert Z\left( x\right) ^{2}}\right\Vert _{L^{p}\left( \omega
\right) }  \label{will now} \\
&\lesssim &\left( 1+2^{N\delta }\right) 2^{-m\delta }mN^{2}\left\Vert \sqrt{%
\sum_{F\in \mathcal{F}}\sum_{Q\in \mathcal{Q}}\sum_{A\in \mathcal{A}\left[ Q%
\right] }\sum_{S\in \mathfrak{C}_{\mathcal{A}}\left( A\right) }\sum_{K\in 
\mathcal{W}_{\func{good},\tau }^{\mathcal{A},\left[ N\right] }\left(
S\right) }\alpha _{A}\left( S\right) ^{2}\left( \frac{\mathrm{P}\left( K,%
\mathbf{1}_{F\setminus S}\sigma \right) }{\ell \left( K\right) }\right)
^{2}\left\vert \mathsf{P}_{S;K}^{\omega ,N}\right\vert Z\left( x\right) ^{2}}%
\right\Vert _{L^{p}\left( \omega \right) }.  \notag
\end{eqnarray}%
where $N$ is as in (\ref{def N}).

Indeed, for $k\geq 1$, $F\in \mathcal{F}$ and $Q\in \mathcal{Q}\left[ F%
\right] $, we begin by letting%
\begin{equation*}
\Lambda _{k}^{F,Q}\equiv \left\{ \left( A,S,K\right) \in \mathfrak{C}_{%
\mathcal{A}\left[ Q\right] }^{\left( k\right) }\left( Q\right) \times 
\mathfrak{C}_{\mathcal{A}}\left( A\right) \times \mathcal{W}_{\func{good}%
,\tau }\left( S\right) \right\} ,
\end{equation*}%
denote the collection of triples $\left( A,S,K\right) $ where $A\in 
\mathfrak{C}_{\mathcal{A}\left[ Q\right] }^{\left( k\right) }\left( Q\right) 
$ lies $k$ levels below the top $Q$ in the tree $\mathcal{A}\left[ Q\right] $%
. Note that the intervals $K$ above are pairwise disjoint in $Q$ for each
fixed $k$. Indeed, they are pairwise disjoint in the intervals $S$, which
are in turn pairwise disjoint in\ the intervals $A$, which are in turn
pairwise disjoint in the interval $Q$ for each fixed $m$.

With $F\in \mathcal{F}$ and $Q\in \mathcal{Q}\left[ F\right] $ fixed, let%
\begin{eqnarray*}
g\left( x\right) &\equiv &\sum_{k}g_{k}\left( x\right) \text{ and }b\left(
x\right) \equiv \sum_{k}b_{k}\left( x\right) \ ,\ \ \ \ \text{where} \\
g_{k}\left( x\right) &=&\sum_{\left( A,S,K\right) \in \Lambda
_{k}^{F,Q}}g_{K,k}\left( x\right) \equiv \sum_{\left( A,S,K\right) \in
\Lambda _{k}^{F,Q}}\left( \frac{\mathrm{P}\left( K,\mathbf{1}_{F\setminus
S}\sigma \right) }{\ell \left( K\right) }\right) \alpha _{A}\left( S\right) 
\mathsf{P}_{S;K}^{\omega ,\Lambda _{g_{Q}}^{\omega },m}Z\left( x\right) , \\
b_{k}\left( x\right) &=&\sum_{\left( A,S,K\right) \in \Lambda
_{k}^{F,Q}}b_{K,k}\left( x\right) \equiv \sum_{\left( A,S,K\right) \in
\Lambda _{k}^{F,Q}}\left( \frac{\mathrm{P}\left( K,\mathbf{1}_{F\setminus
S}\sigma \right) }{\ell \left( K\right) }\right) \alpha _{A}\left( S\right) 
\mathsf{P}_{S;K}^{\omega ,N}Z\left( x\right) , \\
G_{k}\left( x\right) &\equiv &\left\{ g_{K,k}\left( x\right) \right\}
_{\left( A,S,K\right) \in \Lambda _{k}^{F,Q}}\text{ and }B_{k}\left(
x\right) \equiv \left\{ b_{K,k}\left( x\right) \right\} _{\left(
A,S,K\right) \in \Lambda _{k}^{F,Q}}.
\end{eqnarray*}%
Then for fixed $F$ and $Q$, we have from (\ref{equiv Ep}) that the ratio in (%
\ref{M char}) is%
\begin{eqnarray*}
&&M_{\mathcal{L}}^{\left( m,N\right) }\left( g,b\right) \approx
\sup_{k}\sup_{\left( A,S,K\right) \in \Lambda _{k}^{F,Q}}\frac{\left\Vert 
\mathsf{P}_{\mathcal{C}_{\mathcal{A}}^{\left( m\right) }\left( A\right)
}^{\omega }g\left( x\right) \right\Vert _{L^{p}\left( \omega \right) }}{%
\left\Vert \mathsf{P}_{\mathcal{C}_{\mathcal{A}}^{\left[ N\right] }\left(
A\right) }^{\omega }b\left( x\right) \right\Vert _{L^{p}\left( \omega
\right) }} \\
&=&\sup_{k}\sup_{\left( A,S,K\right) \in \Lambda _{k}^{F,Q}}\frac{\left\Vert
\left( \frac{\mathrm{P}\left( K,\mathbf{1}_{F\setminus S}\sigma \right) }{%
\ell \left( K\right) }\right) \alpha _{A}\left( S\right) \mathsf{P}%
_{S;K}^{\omega ,\Lambda _{g}^{\omega },m}Z\left( x\right) \right\Vert
_{L^{p}\left( \omega \right) }}{\left\Vert \left( \frac{\mathrm{P}\left( K,%
\mathbf{1}_{F\setminus S}\sigma \right) }{\ell \left( K\right) }\right)
\alpha _{A}\left( S\right) \mathsf{P}_{S;K}^{\omega ,N}Z\left( x\right)
\right\Vert _{L^{p}\left( \omega \right) }}=\sup_{k}\sup_{\left(
A,S,K\right) \in \Lambda _{k}^{F,Q}}\frac{\left\Vert \mathsf{P}%
_{S;K}^{\omega ,\Lambda _{g}^{\omega },m}Z\left( x\right) \right\Vert
_{L^{p}\left( \omega \right) }}{\left\Vert \mathsf{P}_{S;K}^{\omega
,N}Z\left( x\right) \right\Vert _{L^{p}\left( \omega \right) }} \\
&=&\sup_{k}\sup_{\left( A,S,K\right) \in \Lambda _{k}^{F,Q}}\frac{\mathsf{E}%
_{p}\left( \mathcal{C}_{\mathcal{Q}}\left( Q\right) \cap \mathcal{D}_{\func{%
good}}\left[ \bigcup_{B\in \mathfrak{C}_{\mathcal{A}}^{\left( m-1\right)
}\left( S\right) }B\cap K\right] \cap \Lambda _{g}^{\omega };\omega \right) 
}{\mathsf{E}_{p}\left( \mathcal{C}_{\mathcal{Q}}\left( Q\right) \cap \left\{ 
\mathcal{D}_{\func{good}}\left[ K\right] \setminus \mathcal{D}_{\func{good}}%
\left[ \mathfrak{C}_{\mathcal{A}\left[ Q\right] }^{\left( N+1\right) }\left(
S\right) \cap K\right] \right\} \cap \Lambda _{g}^{\omega };\omega \right) }.
\end{eqnarray*}%
Now the bound, 
\begin{equation}
\frac{\mathsf{E}_{p}\left( \mathcal{C}_{\mathcal{Q}}\left( Q\right) \cap 
\mathcal{D}_{\func{good}}\left[ \bigcup_{B\in \mathfrak{C}_{\mathcal{A}%
}^{\left( m-1\right) }\left( S\right) }B\cap K\right] \cap \Lambda
_{g}^{\omega };\omega \right) }{\mathsf{E}_{p}\left( \mathcal{C}_{\mathcal{Q}%
}\left( Q\right) \cap \left\{ \mathcal{D}_{\func{good}}\left[ K\right]
\setminus \mathcal{D}_{\func{good}}\left[ \mathfrak{C}_{\mathcal{A}\left[ Q%
\right] }^{\left( N+1\right) }\left( S\right) \cap K\right] \right\} \cap
\Lambda _{g}^{\omega };\omega \right) }\lesssim m2^{-m\delta },
\label{the bound}
\end{equation}%
follows directly from (\ref{geo dec bound'}) and then applying (\ref{N opt}%
). Now we write%
\begin{eqnarray*}
&&\sum_{m=1}^{\infty }\left\Vert \sqrt{\left\vert \mathsf{B}^{m}\right\vert
_{\func{straddle}}^{\mathcal{Q}\circ \mathcal{A},\limfunc{trip}}\left(
f\right) }\right\Vert _{L^{p}\left( \omega \right) } \\
&\lesssim &\left\{ \sum_{m=1}^{N}+\sum_{m=N+1}^{\infty }\right\} \left\Vert 
\sqrt{\sum_{F\in \mathcal{F}}\sum_{Q\in \mathcal{Q}\left[ F\right]
}\sum_{A\in \mathcal{A}\left[ Q\right] }\sum_{S\in \mathfrak{C}_{\mathcal{A}%
}\left( A\right) }\sum_{K\in \mathcal{W}_{\func{good},\tau }\left( S\right)
}\alpha _{A}\left( S\right) ^{2}\left( \frac{\mathrm{P}\left( K,\mathbf{1}%
_{F\setminus S}\sigma \right) }{\ell \left( K\right) }\right) ^{2}\left\vert 
\mathsf{P}_{S;K}^{\omega ,\Lambda _{g_{A,m}}^{\omega }}\right\vert Z\left(
x\right) ^{2}}\right\Vert _{L^{p}\left( \omega \right) } \\
&\equiv &\Sigma _{1}^{N}+\Sigma _{N+1}^{\infty },
\end{eqnarray*}%
and apply inequality (\ref{part 2}) in the Martingale Comparison Principle
in Proposition \ref{CMCP}, to the functions $g$ and $b$ above to obtain%
\begin{eqnarray*}
&&\Sigma _{N+1}^{\infty }\lesssim \sum_{m=N+1}^{\infty }\sqrt{mN2^{-m\delta }%
} \\
&&\times \left\Vert \sqrt{\sum_{F\in \mathcal{F}}\sum_{Q\in \mathcal{Q}\left[
F\right] }\sum_{A\in \mathcal{A}\left[ Q\right] }2^{-\limfunc{dist}\left(
A,Q\right) \delta }\sum_{S\in \mathfrak{C}_{\mathcal{A}}\left( A\right)
}\sum_{K\in \mathcal{W}_{\func{good},\tau }\left( S\right) }\alpha
_{A}\left( S\right) ^{2}\left( \frac{\mathrm{P}\left( K,\mathbf{1}%
_{F\setminus S}\sigma \right) }{\ell \left( K\right) }\right) ^{2}\left\vert 
\mathsf{P}_{S;K}^{\omega ,\Lambda _{g_{Q}^{\left[ N\right] }}^{\omega
}}\right\vert Z\left( x\right) ^{2}}\right\Vert _{L^{p}\left( \omega \right)
}
\end{eqnarray*}%
where 
\begin{equation}
g_{Q}\equiv \mathsf{P}_{\mathcal{C}_{\mathcal{Q}}\left( Q\right) }g\text{
and }g_{Q}^{\left[ N\right] }\equiv \mathsf{P}_{\mathcal{C}_{\mathcal{Q}}^{%
\left[ N\right] }\left( Q\right) }g,  \label{def gAs}
\end{equation}%
and where we have used (\ref{the bound}), which gives the decay $%
mN2^{-m\delta }2^{-\limfunc{dist}\left( A,Q\right) \delta }$ when passing
from $\left\vert \mathsf{P}_{S;K}^{\omega ,\Lambda _{g_{A,m}}^{\omega
}}\right\vert Z\left( x\right) $ to $\left\vert \mathsf{P}_{S;K}^{\omega
,N}\right\vert Z\left( x\right) $ to $\left\vert \mathsf{P}_{S;K}^{\omega
,\Lambda _{g_{Q}^{\left[ N\right] }}^{\omega }}\right\vert Z\left( x\right) $%
.

For the finite sum, we cannot\ \emph{directly} use the Martingale Comparison
Principle in Proposition \ref{CMCP} since $m\leq N$. On the other hand, if
we set%
\begin{equation*}
\Omega _{\gamma }\equiv \left\{ x\in K:\left\vert \mathsf{P}_{S;K}^{\omega
,\Lambda _{g_{A,m}}^{\omega }}\right\vert Z\left( x\right) ^{2}>\gamma 2^{-%
\limfunc{dist}\left( A,Q\right) \delta }\left\vert \mathsf{P}_{S;K}^{\omega
,\Lambda _{g_{Q}^{\left[ N\right] }}^{\omega }}\right\vert Z\left( x\right)
^{2}\right\} ,
\end{equation*}%
then we obtain from the negation of the pointwise inequality above, 
\begin{eqnarray*}
&&\sum_{m=1}^{N}\left\Vert \sqrt{\sum_{F\in \mathcal{F}}\sum_{Q\in \mathcal{Q%
}\left[ F\right] }\sum_{A\in \mathcal{A}\left[ Q\right] }\sum_{S\in 
\mathfrak{C}_{\mathcal{A}}\left( A\right) }\sum_{K\in \mathcal{W}_{\func{good%
},\tau }\left( S\right) }\alpha _{A}\left( S\right) ^{2}\left( \frac{\mathrm{%
P}\left( K,\mathbf{1}_{F\setminus S}\sigma \right) }{\ell \left( K\right) }%
\right) ^{2}\left\vert \mathsf{P}_{S;K}^{\omega ,\Lambda _{g_{A,m}}^{\omega
}}\right\vert Z\left( x\right) ^{2}\mathbf{1}_{K\setminus \Omega _{\gamma
}}\left( x\right) }\right\Vert _{L^{p}\left( \omega \right) } \\
&\lesssim &\sqrt{\gamma }N\left\Vert \sqrt{\sum_{F\in \mathcal{F}}\sum_{Q\in 
\mathcal{Q}\left[ F\right] }\sum_{A\in \mathcal{A}\left[ Q\right] }2^{-%
\limfunc{dist}\left( A,Q\right) \delta }\sum_{S\in \mathfrak{C}_{\mathcal{A}%
}\left( A\right) }\sum_{K\in \mathcal{W}_{\func{good},\tau }\left( S\right)
}\alpha _{A}\left( S\right) ^{2}\left( \frac{\mathrm{P}\left( K,\mathbf{1}%
_{F\setminus S}\sigma \right) }{\ell \left( K\right) }\right) ^{2}\left\vert 
\mathsf{P}_{S;K}^{\omega ,\Lambda _{g_{Q}^{\left[ N\right] }}^{\omega
}}\right\vert Z\left( x\right) ^{2}}\right\Vert _{L^{p}\left( \omega \right)
}.
\end{eqnarray*}%
Now note that $\left\vert \mathsf{P}_{S;K}^{\omega ,\Lambda _{g_{Q}^{\left[ N%
\right] }}^{\omega }}\right\vert Z\left( x\right) ^{2}$ is constant on $K$
if $m+\limfunc{dist}\left( A,Q\right) >N$, and so in this case we have 
\begin{eqnarray*}
\left\vert \Omega _{\gamma }\right\vert _{\omega }\gamma 2^{-\limfunc{dist}%
\left( A,Q\right) \delta } &\leq &\int_{\Omega _{\gamma }}\left( \frac{%
\left\vert \mathsf{P}_{S;K}^{\omega ,\Lambda _{g_{A,m}}^{\omega
}}\right\vert Z\left( x\right) ^{2}}{\left\vert \mathsf{P}_{S;K}^{\omega
,\Lambda _{g_{Q}^{\left[ N\right] }}^{\omega }}\right\vert Z\left( x\right)
^{2}}\right) ^{\frac{p}{2}}d\omega \left( x\right) \\
&\leq &\frac{\int_{K}\left( \left\vert \mathsf{P}_{S;K}^{\omega ,\Lambda
_{g_{A,m}}^{\omega }}\right\vert Z\left( x\right) ^{2}\right) ^{\frac{p}{2}%
}d\omega \left( x\right) }{\int_{K}\left( \left\vert \mathsf{P}%
_{S;K}^{\omega ,\Lambda _{g_{Q}^{\left[ N\right] }}^{\omega }}\right\vert
Z\left( x\right) ^{2}\right) ^{\frac{p}{2}}d\omega \left( x\right) }%
\left\vert K\right\vert _{\omega }\lesssim mN2^{-\limfunc{dist}\left(
A,Q\right) \delta }\left\vert K\right\vert _{\omega },
\end{eqnarray*}%
which implies 
\begin{equation*}
\left\vert \Omega _{\gamma }\right\vert _{\omega }\leq \frac{1}{\gamma }%
N^{2}\left\vert K\right\vert _{\omega }<\frac{1}{2}\left\vert K\right\vert
_{\omega }\text{ for }\gamma >2N^{2},
\end{equation*}%
and hence that $\frac{1}{2}\leq \frac{\left\vert K\setminus \Omega _{\gamma
}\right\vert _{\omega }}{\left\vert K\right\vert _{\omega }}\leq 1$. Thus in
the case that $m+\limfunc{dist}\left( A,Q\right) >N$, we can apply a slight
variant of the Martingale Comparison Principle in Proposition \ref{CMCP} with%
\begin{eqnarray*}
g_{k}\left( x\right) &=&\sum_{\left( A,S,K\right) \in \Lambda
_{k}^{F,Q}}g_{K,k}\left( x\right) \equiv \sum_{\left( A,S,K\right) \in
\Lambda _{k}^{F,Q}}\left( \frac{\mathrm{P}\left( K,\mathbf{1}_{F\setminus
S}\sigma \right) }{\ell \left( K\right) }\right) \alpha _{A}\left( S\right) 
\mathsf{P}_{S;K}^{\omega ,\Lambda _{g}^{\omega },m}Z\left( x\right) , \\
b_{k}\left( x\right) &=&\sum_{\left( A,S,K\right) \in \Lambda
_{k}^{F,Q}}b_{K,k}\left( x\right) \equiv \sum_{\left( A,S,K\right) \in
\Lambda _{k}^{F,Q}}\left( \frac{\mathrm{P}\left( K,\mathbf{1}_{F\setminus
S}\sigma \right) }{\ell \left( K\right) }\right) \alpha _{A}\left( S\right) 
\mathsf{P}_{S;K}^{\omega ,\Lambda _{g_{Q}^{\left[ N\right] }}^{\omega
}}Z\left( x\right) ,
\end{eqnarray*}%
where $g_{k}$ is unchanged, but $b_{k}$ now has projection onto the larger
set of frequencies $\Lambda _{g_{Q}^{\left[ N\right] }}^{\omega }$.\ The
only difference in the proof of this variant of Proposition \ref{CMCP}, is
the use\ of the inequality $\frac{1}{2}\leq \frac{\left\vert K\setminus
\Omega _{\gamma }\right\vert _{\omega }}{\left\vert K\right\vert _{\omega }}%
\leq 1$ in place of the inequality (\ref{note'''}). The factor in (\ref{M
char}) also changes to 
\begin{equation*}
M_{\mathcal{L}}^{\left( m,N\right) }\left( g,b\right) \equiv \sup_{A\in 
\mathcal{A}\left[ Q\right] }\frac{\left\Vert \mathsf{P}_{\bigcup_{A\in 
\mathfrak{C}_{\mathcal{A}\left[ Q\right] }^{\left( m\right) }\left( Q\right)
}\mathcal{D}\left( K\right) }^{\omega }g\right\Vert _{L^{p}\left( \omega
\right) }}{\left\Vert \mathsf{P}_{\mathcal{C}_{\mathcal{Q}}^{\left[ N\right]
}\left( Q\right) }^{\omega }b\right\Vert _{L^{p}\left( \omega \right) }},
\end{equation*}%
which has the bound $mN2^{-\limfunc{dist}\left( A,Q\right) \delta }$. Now we
writing the sum in shorthand form,%
\begin{equation*}
\sum_{\left( F,Q,A,S,K\right) \in \mathcal{F}\times \mathcal{Q}\left[ F%
\right] \times \mathcal{A}\left[ Q\right] \times \mathfrak{C}_{\mathcal{A}%
}\left( A\right) \times \mathcal{W}_{\func{good},\tau }\left( S\right)
}=\sum_{\left( F,Q,A,S,K\right) \in \Omega ^{5}}\ ,
\end{equation*}%
we obtain,%
\begin{eqnarray*}
&&\sum_{m+\limfunc{dist}\left( A,Q\right) >N}\left\Vert \sqrt{\sum_{\left(
F,Q,A,S,K\right) \in \Omega ^{5}}\alpha _{A}\left( S\right) ^{2}\left( \frac{%
\mathrm{P}\left( K,\mathbf{1}_{F\setminus S}\sigma \right) }{\ell \left(
K\right) }\right) ^{2}\left\vert \mathsf{P}_{S;K}^{\omega ,\Lambda
_{g_{A,m}}^{\omega }}\right\vert Z\left( x\right) ^{2}\mathbf{1}_{\Omega
_{\gamma }}\left( x\right) }\right\Vert _{L^{p}\left( \omega \right) } \\
&\lesssim &N^{2}\left\Vert \sqrt{\sum_{\left( F,Q,A,S,K\right) \in \Omega
^{5}}2^{-\limfunc{dist}\left( A,Q\right) \delta }\alpha _{A}\left( S\right)
^{2}\left( \frac{\mathrm{P}\left( K,\mathbf{1}_{F\setminus S}\sigma \right) 
}{\ell \left( K\right) }\right) ^{2}\left\vert \mathsf{P}_{S;K}^{\omega
,\Lambda _{g_{Q}^{\left[ N\right] }}^{\omega }}\right\vert Z\left( x\right)
^{2}}\right\Vert _{L^{p}\left( \omega \right) }.
\end{eqnarray*}

Finally in the case that $m+\limfunc{dist}\left( A,Q\right) <N$, we have 
\begin{equation*}
2^{-\limfunc{dist}\left( A,Q\right) \delta }\geq 2^{-\left( N-m\right)
\delta }\geq 2^{-N\delta }>0,
\end{equation*}%
and so we can simply use the pointwise inequality $\left\vert \mathsf{P}%
_{S;K}^{\omega ,\Lambda _{g_{A,m}}^{\omega }}\right\vert Z\left( x\right)
\leq \left\vert \mathsf{P}_{S;K}^{\omega ,Q}\right\vert Z\left( x\right) $
to obtain 
\begin{eqnarray*}
&&\sum_{m+\limfunc{dist}\left( A,Q\right) \leq N}\left\Vert \sqrt{%
\sum_{\left( F,Q,A,S,K\right) \in \Omega ^{5}}\alpha _{A}\left( S\right)
^{2}\left( \frac{\mathrm{P}\left( K,\mathbf{1}_{F\setminus S}\sigma \right) 
}{\ell \left( K\right) }\right) ^{2}\left\vert \mathsf{P}_{S;K}^{\omega
,\Lambda _{g_{A,m}}^{\omega }}\right\vert Z\left( x\right) ^{2}\mathbf{1}%
_{\Omega _{\gamma }}\left( x\right) }\right\Vert _{L^{p}\left( \omega
\right) } \\
&\leq &2^{N\delta }N\max_{\substack{ 0\leq s\leq N  \\ 1\leq m\leq N}}%
\left\Vert \sqrt{\sum_{\left( F,Q,A,S,K\right) \in \Omega ^{5}}2^{-\limfunc{%
dist}\left( A,Q\right) \delta }\alpha _{A}\left( S\right) ^{2}\left( \frac{%
\mathrm{P}\left( K,\mathbf{1}_{F\setminus S}\sigma \right) }{\ell \left(
K\right) }\right) ^{2}\left\vert \mathsf{P}_{S;K}^{\omega ,\Lambda
_{g_{A,\left( s\right) }}^{\omega }}\right\vert Z\left( x\right) ^{2}}%
\right\Vert _{L^{p}\left( \omega \right) }.
\end{eqnarray*}%
These two estimates combine to bound $\Sigma _{1}^{N}$ by $\left(
1+2^{N\delta }\right) N^{2}$ times%
\begin{equation*}
\max_{\substack{ 0\leq s\leq N  \\ 1\leq m\leq N}}\left\Vert \sqrt{%
\sum_{\left( F,Q,A,S,K\right) \in \Omega ^{5}}2^{-\limfunc{dist}\left(
A,Q\right) \delta }\alpha _{A}\left( S\right) ^{2}\left( \frac{\mathrm{P}%
\left( K,\mathbf{1}_{F\setminus S}\sigma \right) }{\ell \left( K\right) }%
\right) ^{2}\left\vert \mathsf{P}_{S;K}^{\omega ,\Lambda _{g_{A,\left(
s\right) }}^{\omega }}\right\vert Z\left( x\right) ^{2}}\right\Vert
_{L^{p}\left( \omega \right) }.
\end{equation*}

Collecting all of our estimates, and using $\Lambda _{g_{A,\left( s\right)
}}^{\omega }\subset \Lambda _{g_{Q}}^{\omega }$, we have shown that%
\begin{equation*}
\left\Vert \sqrt{\left\vert \mathsf{B}^{m}\right\vert _{\func{straddle}}^{%
\mathcal{Q\circ A},\limfunc{trip}}\left( f\right) }\right\Vert _{L^{p}\left(
\omega \right) }\lesssim \left\{ 
\begin{array}{ccc}
2^{-m\delta }mN^{2}\left\Vert \left\vert f\right\vert _{\mathcal{F},\mathcal{%
Q},\mathcal{A}}^{\left\{ \mathcal{C}_{\mathcal{A}}^{\left[ N\right] }\left(
A\right) \cap \Lambda _{g_{Q}}^{\omega }\right\} _{A\in \mathcal{A}%
}}\right\Vert _{L^{p}\left( \omega \right) } & \text{ if } & m>N \\ 
\left( 1+2^{N\delta }\right) N^{2}\left\Vert \left\vert f\right\vert _{%
\mathcal{F},\mathcal{Q},\mathcal{A}}^{\left\{ \mathcal{C}_{\mathcal{A}}^{%
\left[ N\right] }\left( A\right) \cap \Lambda _{g_{Q}}^{\omega }\right\}
_{A\in \mathcal{A}}}\right\Vert _{L^{p}\left( \omega \right) } & \text{ if }
& m\leq N%
\end{array}%
\right. ,
\end{equation*}%
where $\left\vert f\right\vert _{\mathcal{F},\mathcal{Q},\mathcal{A}%
}^{\left\{ \mathcal{C}_{\mathcal{A}}^{\left[ N\right] }\left( A\right) \cap
\Lambda _{g_{Q}}^{\omega }\right\} _{A\in \mathcal{A}}}\left( x\right) $ is
defined in (\ref{def fAQ}) above, and $g_{Q}$ is defined in (\ref{def gAs}).
Now summing in $m$ yields%
\begin{equation*}
\sum_{m=1}^{\infty }\left\Vert \sqrt{\left\vert \mathsf{B}^{m}\right\vert _{%
\func{straddle}}^{\mathcal{Q\circ A},\limfunc{trip}}\left( f\right) }%
\right\Vert _{L^{p}\left( \omega \right) }\lesssim \left( 1+2^{N\delta
}\right) \frac{N^{2}}{\delta ^{2}}\max_{0\leq s\leq N}\left\Vert \left\vert
f\right\vert _{\mathcal{F},\mathcal{Q},\mathcal{A}}^{\left\{ \mathcal{C}_{%
\mathcal{A}}^{\left( s\right) }\left( A\right) \cap \Lambda _{g_{Q}}^{\omega
}\right\} _{A\in \mathcal{A}}}\right\Vert _{L^{p}\left( \omega \right) }.
\end{equation*}

The estimate%
\begin{equation*}
\left\Vert \sqrt{\left\vert \mathsf{B}\right\vert _{\func{straddle}}^{%
\mathcal{Q\circ A},\func{near}}\left( f\right) }\right\Vert _{L^{p}\left(
\omega \right) }\lesssim \left( 1+2^{N\delta }\right) \frac{N^{2}}{\delta
^{2}}\max_{0\leq s\leq N}\left\Vert \left\vert f\right\vert _{\mathcal{F},%
\mathcal{Q},\mathcal{A}}^{\left\{ \mathcal{C}_{\mathcal{A}}^{\left( s\right)
}\left( A\right) \cap \Lambda _{g_{Q}}^{\omega }\right\} _{A\in \mathcal{A}%
}}\right\Vert _{L^{p}\left( \omega \right) }
\end{equation*}%
is similar but easier, since there at most $2^{\tau }$ intervals that are $%
\tau $-nearby any given interval. This completes the proof of the Quadratic $%
L^{p}$-Stopping Child Lemma \ref{straddle3} in view of our assumption (\ref%
{g=1}) that $\left\Vert g\right\Vert _{L^{p^{\prime }}\left( \omega \right)
}=1$.
\end{proof}

\subsection{Completion of the proof}

First, we obtain geometric decay in grandchildren from the second line in (%
\ref{dec corona}) and Lemma \ref{disjoint supp},%
\begin{eqnarray}
\sum_{A^{\prime }\in \mathfrak{C}_{\mathcal{A}\left[ Q\right] }^{\left(
m\right) }\left[ A\right] }\left\Vert \left\vert \mathsf{P}_{\mathcal{D}%
\left[ A^{\prime }\right] \cap \mathcal{C}_{\mathcal{Q}}\left( Q\right) \cap
\Lambda _{g}^{\omega }}^{\omega }\right\vert Z\right\Vert _{L^{p}\left(
\omega \right) }^{p} &=&\left\Vert \left\vert \sum_{A^{\prime }\in \mathfrak{%
C}_{\mathcal{A}\left[ Q\right] }^{\left( m\right) }\left[ A\right] }\mathsf{P%
}_{\mathcal{D}\left[ A^{\prime }\right] \cap \mathcal{C}_{\mathcal{Q}}\left(
Q\right) \cap \Lambda _{g}^{\omega }}^{\omega }\right\vert Z\right\Vert
_{L^{p}\left( \omega \right) }^{p}  \label{pre geo dec} \\
&\leq &\frac{1}{\Gamma ^{pm}}\left\Vert \left\vert \mathsf{P}_{\mathcal{D}%
\left[ A\right] \cap \mathcal{C}_{\mathcal{Q}}\left( Q\right) \cap \Lambda
_{g}^{\omega }}^{\omega }\right\vert Z\right\Vert _{L^{p}\left( \omega
\right) }^{p}\ ,\ \ \ \ \ m\geq 1.  \notag
\end{eqnarray}%
This last estimate can be improved to%
\begin{equation}
\sum_{A^{\prime }\in \mathfrak{C}_{\mathcal{A}\left[ Q\right] }^{\left(
m\right) }\left( A\right) }\left\Vert \left\vert \mathsf{P}_{\mathcal{D}%
\left[ A^{\prime }\right] \cap \mathcal{C}_{\mathcal{Q}}\left( Q\right) \cap
\Lambda _{g}^{\omega }}^{\omega }\right\vert Z\right\Vert _{L^{p}\left(
\omega \right) }^{p}\leq \frac{2}{\Gamma ^{pm}}\left\Vert \left\vert \mathsf{%
P}_{\mathcal{C}_{\mathcal{A}\left[ Q\right] }^{\left[ N\right] }\left(
A\right) \cap \Lambda _{g}^{\omega }}^{\omega }\right\vert Z\right\Vert
_{L^{p}\left( \omega \right) }^{p}\ ,\ \ \ \ \ m\geq 1,  \label{geo dec}
\end{equation}%
provided $N$ is chosen, depending only on $\Gamma $, so that%
\begin{equation}
\frac{1}{\Gamma ^{pN}}<\frac{1}{2},\ \ \ \ \ \text{e.g. }N=\left\lceil \frac{%
\ln 2}{\ln \Gamma ^{p}}\right\rceil ,  \label{def N}
\end{equation}%
where we note for future reference that%
\begin{equation*}
N=\left\lceil \frac{\ln 2}{\ln \left( 1+\theta \right) }\right\rceil \approx 
\frac{1}{\theta },\ \ \ \ \ \text{for }\Gamma ^{p}=1+\theta \text{ and }%
0<\theta \ll 1.
\end{equation*}%
Indeed, with this choice of $N$, we have by (\ref{second disj}) and (\ref%
{pre geo dec}) with $m=N$,%
\begin{eqnarray*}
\left\Vert \left\vert \mathsf{P}_{\mathcal{D}\left[ A\right] \cap \mathcal{C}%
_{\mathcal{Q}}\left( Q\right) \cap \Lambda _{g}^{\omega }}^{\omega
}\right\vert Z\right\Vert _{L^{p}\left( \omega \right) } &\leq &\left\Vert
\left\vert \mathsf{P}_{\mathcal{C}_{\mathcal{A}\left[ Q\right] }^{\left[ N%
\right] }\left( A\right) \cap \Lambda _{g}^{\omega }}^{\omega }\right\vert
Z\right\Vert _{L^{p}\left( \omega \right) }+\left\Vert \left\vert
\sum_{m=N+1}^{\infty }\mathsf{P}_{\mathcal{C}_{\mathcal{A}\left[ Q\right]
}^{\left( m\right) }\left( A\right) \cap \Lambda _{g}^{\omega }}^{\omega
}\right\vert Z\right\Vert _{L^{p}\left( \omega \right) } \\
&=&\left\Vert \left\vert \mathsf{P}_{\mathcal{C}_{\mathcal{A}\left[ Q\right]
}^{\left[ N\right] }\left( A\right) \cap \Lambda _{g}^{\omega }}^{\omega
}\right\vert Z\right\Vert _{L^{p}\left( \omega \right) }+\left(
\sum_{m=N+1}^{\infty }\left\Vert \left\vert \mathsf{P}_{\mathcal{C}_{%
\mathcal{A}\left[ Q\right] }^{\left( m\right) }\left( A\right) \cap \Lambda
_{g}^{\omega }}^{\omega }\right\vert Z\right\Vert _{L^{p}\left( \omega
\right) }^{p}\right) ^{\frac{1}{p}} \\
&\leq &\left\Vert \left\vert \mathsf{P}_{\mathcal{C}_{\mathcal{A}\left[ Q%
\right] }^{\left[ N\right] }\left( A\right) \cap \Lambda _{g}^{\omega
}}^{\omega }\right\vert Z\right\Vert _{L^{p}\left( \omega \right) }+\left(
\sum_{A^{\prime }\in \mathfrak{C}_{\mathcal{A}\left[ Q\right] }^{\left(
m\right) }\left( A\right) }\left\Vert \left\vert \mathsf{P}_{\mathcal{D}%
\left[ A^{\prime }\right] \cap \mathcal{C}_{\mathcal{Q}}\left( Q\right) \cap
\Lambda _{g}^{\omega }}^{\omega }\right\vert Z\right\Vert _{L^{p}\left(
\omega \right) }^{p}\right) ^{\frac{1}{p}} \\
&\leq &\left\Vert \left\vert \mathsf{P}_{\mathcal{C}_{\mathcal{A}\left[ Q%
\right] }^{\left[ N\right] }\left( A\right) \cap \Lambda _{g}^{\omega
}}^{\omega }\right\vert Z\right\Vert _{L^{p}\left( \omega \right) }+\frac{1}{%
2^{\frac{1}{p}}}\left\Vert \left\vert \mathsf{P}_{\mathcal{D}\left[ A\right]
\cap \mathcal{C}_{\mathcal{Q}}\left( Q\right) \cap \Lambda _{g}^{\omega
}}^{\omega }\right\vert Z\right\Vert _{L^{p}\left( \omega \right) }\ .
\end{eqnarray*}%
Thus we obtain%
\begin{equation}
1-\frac{1}{2^{\frac{1}{p}}}\leq \frac{\left\Vert \left\vert \mathsf{P}_{%
\mathcal{C}_{\mathcal{A}\left[ Q\right] }^{\left[ N\right] }\left( A\right)
\cap \Lambda _{g}^{\omega }}^{\omega }\right\vert Z\right\Vert _{L^{p}\left(
\omega \right) }}{\left\Vert \left\vert \mathsf{P}_{\mathcal{D}\left[ A%
\right] \cap \mathcal{C}_{\mathcal{Q}}\left( Q\right) \cap \Lambda
_{g}^{\omega }}^{\omega }\right\vert Z\right\Vert _{L^{p}\left( \omega
\right) }}\leq 1.  \label{N opt'}
\end{equation}%
In particular,%
\begin{equation*}
\left\Vert \left\vert \mathsf{P}_{\mathcal{D}\left[ A\right] \cap \mathcal{C}%
_{\mathcal{Q}}\left( Q\right) \cap \Lambda _{g}^{\omega }}^{\omega
}\right\vert Z\right\Vert _{L^{p}\left( \omega \right) }\leq \frac{2^{\frac{1%
}{p}}}{2^{\frac{1}{p}}-1}\left\Vert \left\vert \mathsf{P}_{\mathcal{C}_{%
\mathcal{A}\left[ Q\right] }^{\left[ N\right] }\left( A\right) \cap \Lambda
_{g}^{\omega }}^{\omega }\right\vert Z\right\Vert _{L^{p}\left( \omega
\right) },
\end{equation*}%
implies (\ref{geo dec}) and, together with (\ref{dec corona small}) implies,%
\begin{equation}
\left\Vert \left\vert \mathsf{P}_{\left[ \mathcal{D}^{\limfunc{no}\func{top}}%
\left[ A\right] \setminus \bigcup_{S\in \mathfrak{C}_{\mathcal{A}\left[ Q%
\right] }\left( A\right) }\mathcal{D}\left[ S\right] \right] \cap \Lambda
_{g}^{\omega }}\right\vert Z\right\Vert _{L^{p}\left( \omega \right)
}^{p}\leq \theta ^{\natural }\left\Vert \left\vert \mathsf{P}_{\mathcal{D}%
\left[ A\right] \cap \mathcal{C}_{\mathcal{Q}}\left( Q\right) \cap \Lambda
_{g}^{\omega }}\right\vert Z\right\Vert _{L^{p}\left( \omega \right)
}^{p}\leq \frac{2}{\left( 2^{\frac{1}{p}}-1\right) ^{p}}\left\Vert
\left\vert \mathsf{P}_{\mathcal{C}_{\mathcal{A}\left[ Q\right] }^{\left[ N%
\right] }\left( A\right) \cap \Lambda _{g}^{\omega }}^{\omega }\right\vert
Z\right\Vert _{L^{p}\left( \omega \right) }^{p}\ ,  \label{tight}
\end{equation}%
for $A\in \mathcal{A}\left[ Q\right] $ and $Q\in \mathcal{Q}\left[ F\right] $%
, where $\mathcal{D}^{\limfunc{no}\func{top}}\left( A\right) \equiv \mathcal{%
D}\left( A\right) \setminus \left\{ A\right\} $.

The stopping form $\mathsf{B}_{\limfunc{stop}}\left( f,g\right) =\mathsf{B}_{%
\limfunc{stop}}^{\mathcal{F}}\left( f,g\right) $ depends on the construction
of the stopping times $\mathcal{F}$, and is given by%
\begin{eqnarray*}
&&\mathsf{B}_{\limfunc{stop}}^{\mathcal{F}}\left( f,g\right) =\sum_{F\in 
\mathcal{F}}\sum_{\substack{ \left( I,J\right) \in \mathcal{C}_{\mathcal{F}%
}\left( F\right) \times \mathcal{C}_{\mathcal{F}}\left( F\right)  \\ %
J\subset _{\tau }I}}\left( E_{I_{J}}^{\sigma }\bigtriangleup _{I}^{\sigma
}f\right) \left\langle H_{\sigma }\mathbf{1}_{F\setminus
I_{J}},\bigtriangleup _{J}^{\omega }g\right\rangle _{\omega }=\sum_{F\in 
\mathcal{F}}\mathsf{B}_{\limfunc{stop}}^{\mathcal{F},F}\left( \mathsf{P}_{%
\mathcal{C}_{\mathcal{F}}\left( F\right) }^{\sigma }f,\mathsf{P}_{\mathcal{C}%
_{\mathcal{F}}\left( F\right) }^{\omega }g\right) , \\
&&\text{where }\mathsf{B}_{\limfunc{stop}}^{\mathcal{F},F}\left( h,k\right)
\equiv \sum_{\substack{ \left( I,J\right) \in \mathcal{C}_{\mathcal{F}%
}\left( F\right) \times \mathcal{C}_{\mathcal{F}}\left( F\right)  \\ %
J\subset _{\tau }I}}\left( E_{I_{J}}^{\sigma }\bigtriangleup _{I}^{\sigma
}h\right) \left\langle H_{\sigma }\mathbf{1}_{F\setminus
I_{J}},\bigtriangleup _{J}^{\omega }k\right\rangle _{\omega }\ .
\end{eqnarray*}%
We will keep $f$, $g$ and $\mathcal{F}$ \emph{fixed} throughout our
treatment of the stopping form, although we will often consider projections $%
\mathsf{P}f$ and $\mathsf{P}g$ of $f$ and $g$, and then abuse notation by
writing simply $f$ or $g$ instead of $\mathsf{P}f$ or $\mathsf{P}g$. More
generally, for any collection of stopping times $\mathcal{A}\supset \mathcal{%
F}$, we have%
\begin{eqnarray*}
\mathsf{B}_{\limfunc{stop}}^{\mathcal{F}}\left( f,g\right) &=&\sum_{F\in 
\mathcal{F}}\sum_{\substack{ \left( I,J\right) \in \mathcal{C}_{\mathcal{A}%
\left[ Q\right] }\left( A\right) \times \mathcal{C}_{\mathcal{A}\left[ Q%
\right] }\left( B\right)  \\ J\subset _{\tau }I}}\left( E_{I_{J}}^{\sigma
}\bigtriangleup _{I}^{\sigma }f\right) \left\langle H_{\sigma }\mathbf{1}%
_{F\setminus I_{J}},\bigtriangleup _{J}^{\omega }g\right\rangle _{\omega } \\
&=&\sum_{F\in \mathcal{F}}\left\{ \sum_{\substack{ A,B\in \mathcal{A}\left[ F%
\right]  \\ B\subsetneqq A}}+\sum_{\substack{ A,B\in \mathcal{A}\left[ F%
\right]  \\ A\subsetneqq B}}\right\} \sum_{\substack{ \left( I,J\right) \in 
\mathcal{C}_{\mathcal{A}\left[ Q\right] }\left( A\right) \times \mathcal{C}_{%
\mathcal{A}\left[ Q\right] }\left( B\right)  \\ J\subset _{\tau }I}}\left(
E_{I_{J}}^{\sigma }\bigtriangleup _{I}^{\sigma }f\right) \left\langle
H_{\sigma }\mathbf{1}_{F\setminus I_{J}},\bigtriangleup _{J}^{\omega
}g\right\rangle _{\omega } \\
&&+\sum_{F\in \mathcal{F}}\sum_{\substack{ J\in \mathcal{C}_{\mathcal{A}%
\left[ Q\right] }\left( A\right)  \\ J\subset _{\tau }A}}\left(
E_{A_{J}}^{\sigma }\bigtriangleup _{A}^{\sigma }f\right) \left\langle
H_{\sigma }\mathbf{1}_{F\setminus A_{J}},\bigtriangleup _{J}^{\omega
}g\right\rangle _{\omega } \\
&&+\sum_{F\in \mathcal{F}}\sum_{\substack{ \left( I,J\right) \in \mathcal{C}%
_{\mathcal{A}\left[ Q\right] }^{\limfunc{no}\func{top}}\left( A\right)
\times \mathcal{C}_{\mathcal{A}\left[ Q\right] }\left( A\right)  \\ J\subset
_{\tau }I}}\left( E_{I_{J}}^{\sigma }\bigtriangleup _{I}^{\sigma }f\right)
\left\langle H_{\sigma }\mathbf{1}_{F\setminus I_{J}},\bigtriangleup
_{J}^{\omega }g\right\rangle _{\omega } \\
&=&\mathsf{B}_{\limfunc{stop}\func{sep}}^{\mathcal{A}}\left( f,g\right) +%
\mathsf{B}_{\limfunc{stop}\func{sep}}^{\mathcal{A},\ast }\left( f,g\right) +%
\mathsf{B}_{\limfunc{stop}\func{top}\func{only}}^{\mathcal{A}}\left(
f,g\right) +\mathsf{B}_{\limfunc{stop}\limfunc{no}\func{top}}^{\mathcal{A}%
}\left( f,g\right) ,
\end{eqnarray*}%
\newline
where%
\begin{eqnarray*}
\mathsf{B}_{\limfunc{stop}\func{sep}}^{\mathcal{A}}\left( f,g\right) &\equiv
&\sum_{F\in \mathcal{F}}\sum_{\substack{ A,B\in \mathcal{A}\left[ F\right] 
\\ B\subsetneqq A}}\sum_{\substack{ \left( I,J\right) \in \mathcal{C}_{%
\mathcal{A}\left[ Q\right] }\left( A\right) \times \mathcal{C}_{\mathcal{A}%
\left[ Q\right] }\left( B\right)  \\ J\subset _{\tau }I}}\left(
E_{I_{J}}^{\sigma }\bigtriangleup _{I}^{\sigma }f\right) \left\langle
H_{\sigma }\mathbf{1}_{F\setminus I_{J}},\bigtriangleup _{J}^{\omega
}g\right\rangle _{\omega }\ , \\
\mathsf{B}_{\limfunc{stop}\func{sep}}^{\mathcal{A},\ast }\left( f,g\right)
&\equiv &\sum_{F\in \mathcal{F}}\sum_{\substack{ A,B\in \mathcal{A}\left[ F%
\right]  \\ A\subsetneqq B}}\sum_{\substack{ \left( I,J\right) \in \mathcal{C%
}_{\mathcal{A}\left[ Q\right] }\left( A\right) \times \mathcal{C}_{\mathcal{A%
}\left[ Q\right] }\left( B\right)  \\ J\subset _{\tau }I}}\left(
E_{I_{J}}^{\sigma }\bigtriangleup _{I}^{\sigma }f\right) \left\langle
H_{\sigma }\mathbf{1}_{F\setminus I_{J}},\bigtriangleup _{J}^{\omega
}g\right\rangle _{\omega }\ , \\
\mathsf{B}_{\limfunc{stop}\func{top}\func{only}}^{\mathcal{A}}\left(
f,g\right) &\equiv &\sum_{F\in \mathcal{F}}\sum_{\substack{ J\in \mathcal{C}%
_{\mathcal{A}\left[ Q\right] }\left( A\right)  \\ J\subset _{\tau }A}}\left(
E_{A_{J}}^{\sigma }\bigtriangleup _{A}^{\sigma }f\right) \left\langle
H_{\sigma }\mathbf{1}_{F\setminus A_{J}},\bigtriangleup _{J}^{\omega
}g\right\rangle _{\omega }\ , \\
\mathsf{B}_{\limfunc{stop}\limfunc{no}\func{top}}^{\mathcal{A}}\left(
f,g\right) &\equiv &\sum_{F\in \mathcal{F}}\sum_{\substack{ \left(
I,J\right) \in \mathcal{C}_{\mathcal{A}\left[ Q\right] }^{\limfunc{no}\func{%
top}}\left( A\right) \times \mathcal{C}_{\mathcal{A}\left[ Q\right] }\left(
A\right)  \\ J\subset _{\tau }I}}\left( E_{I_{J}}^{\sigma }\bigtriangleup
_{I}^{\sigma }f\right) \left\langle H_{\sigma }\mathbf{1}_{F\setminus
I_{J}},\bigtriangleup _{J}^{\omega }g\right\rangle _{\omega }\ ,
\end{eqnarray*}%
and $\mathcal{C}_{\mathcal{A}\left[ Q\right] }^{\limfunc{no}\func{top}%
}\left( A\right) \equiv \mathcal{C}_{\mathcal{A}\left[ Q\right] }\left(
A\right) \setminus \left\{ A\right\} $.

We now consider a stopping time $\mathcal{Q}\in \mathfrak{Q}$, which can by
definition be written as $\mathcal{S}^{\left( n\right) }\circ \mathcal{A}%
_{n+1}$ for some $n\geq 0$. Then from above we have 
\begin{equation*}
\mathsf{B}_{\limfunc{stop}}^{\mathcal{S}^{\left( n\right) }}\left(
f,g\right) =\mathsf{B}_{\limfunc{stop}\func{sep}}^{\mathcal{S}^{\left(
n+1\right) }}\left( f,g\right) +\mathsf{B}_{\limfunc{stop}\func{sep}}^{%
\mathcal{S}^{\left( n+1\right) },\ast }\left( f,g\right) +\mathsf{B}_{%
\limfunc{stop}\func{top}\func{only}}^{\mathcal{S}^{\left( n+1\right)
}}\left( f,g\right) +\mathsf{B}_{\limfunc{stop}\limfunc{no}\func{top}}^{%
\mathcal{S}^{\left( n+1\right) }}\left( f,g\right) ,
\end{equation*}%
where we can control the separated stopping form $\mathsf{B}_{\limfunc{stop}%
\func{sep}}^{\mathcal{S}^{\left( n+1\right) }}\left( f,g\right) $ using the
quadratic $L^{p}$-Stopping Child Lemma \ref{straddle3}. Indeed, we can use
the geometric gain in the index $m$ that measures the distance between the
coronas in the separated form, i.e.%
\begin{eqnarray*}
&&\mathsf{E}_{p}\left( \mathcal{C}_{\mathcal{Q}}\left( Q\right) \cap 
\mathcal{D}_{\func{good}}\left[ \bigcup_{B\in \mathfrak{C}_{\mathcal{A}%
}^{\left( m-1\right) }\left( S\right) }B\cap K\right] \cap \Lambda
_{g}^{\omega };\omega \right) \\
&\lesssim &\Gamma ^{-pm}\mathsf{E}_{p}\left( \mathcal{C}_{\mathcal{Q}}\left(
Q\right) \cap \left\{ \mathcal{D}_{\func{good}}\left[ \mathfrak{C}_{\mathcal{%
A}}\left( S\right) \cap K\right] \setminus \mathcal{D}_{\func{good}}\left[ 
\mathfrak{C}_{\mathcal{A}}^{\left( N+1\right) }\left( S\right) \cap K\right]
\right\} \cap \Lambda _{g}^{\omega };\omega \right) , \\
&&\text{for }\left( F,Q,A,S,K\right) \in \mathcal{F}\times \mathcal{Q}\left[
F\right] \times \mathcal{A}\left[ Q\right] \times \mathfrak{C}_{\mathcal{A}%
}\left( A\right) \times \mathcal{W}_{\func{good}}^{\limfunc{trip}}\left(
S\right) \text{ and }m\geq 1,
\end{eqnarray*}%
together with the equivalence (\ref{N opt'}), to obtain,%
\begin{eqnarray}
\left\vert \mathsf{B}_{\limfunc{stop}\func{sep}}^{\mathcal{S}^{\left(
n+1\right) }}\left( f,g\right) \right\vert &\lesssim &\left( 1+2^{N\delta
}\right) \frac{N^{2}}{\delta ^{2}}\ \left\Vert \left\vert f\right\vert _{%
\mathcal{F},\mathcal{S}^{\left( n\right) },\mathcal{S}^{\left( n+1\right)
}}^{\left\{ \mathcal{C}_{\mathcal{S}^{\left( n+1\right) }}^{\left[ N\right]
}\left( A\right) \cap \Lambda _{g_{Q}}^{\omega }\right\} _{A\in \mathcal{S}%
^{\left( n+1\right) }}}\right\Vert _{L^{p}\left( \omega \right) }\left\Vert
g\right\Vert _{L^{p^{\prime }}\left( \omega \right) }  \label{far below stop}
\\
&\lesssim &\left( 1+2^{N\delta }\right) \frac{N^{3}}{\delta ^{2}}\
\sup_{0\leq s\leq N}\left\Vert \left\vert f\right\vert _{\mathcal{F},%
\mathcal{S}^{\left( n\right) },\mathcal{S}^{\left( n+1\right) }}^{\left\{ 
\mathcal{C}_{\mathcal{S}^{\left( n+1\right) }}^{\left( s\right) }\left(
A\right) \cap \Lambda _{g_{Q}}^{\omega }\right\} _{A\in \mathcal{S}^{\left(
n+1\right) }}}\right\Vert _{L^{p}\left( \omega \right) }\left\Vert
g\right\Vert _{L^{p^{\prime }}\left( \omega \right) }.  \notag
\end{eqnarray}%
Note that the hypothesis (\ref{geo dec bound'}) of the $L^{p}$-Stopping
Child Lemma \ref{straddle3} holds with $2^{-\delta }=\frac{1}{\Gamma }$ by
iterating the negation of the first line in (\ref{dec corona}) of Lemma \ref%
{lem I*}, and as mentioned above, the hypothesis (\ref{N opt}) holds by (\ref%
{N opt'}). Note also that 
\begin{equation*}
\delta =\frac{\ln \Gamma ^{p}}{\ln 2}=\frac{\ln \left( 1+\theta \right) }{%
\ln 2}\approx \theta \approx \frac{1}{N}.
\end{equation*}

The dual separated stopping form $\mathsf{B}_{\limfunc{stop}\func{sep}}^{%
\mathcal{S}^{\left( n+1\right) },\ast }\left( f,g\right) $ is handled
similarly, and the top only stopping form $\mathsf{B}_{\limfunc{stop}\func{%
top}\func{only}}^{\mathcal{S}^{\left( n+1\right) }}\left( f,g\right) $ is
handled as an easy corollary of the proof of the $L^{p}$-Stopping Child
Lemma, since there are only pairs $\left( A,J\right) $ arising in this form,
and they are effectively separated by $A$ itself. The only difference in the
course of the proof is that $\alpha _{\mathcal{A}\left[ Q\right] }\left(
S\right) =\sup_{I\in \left( \Lambda _{f}^{\sigma }\left[ S\right] ,A\right]
\cap \mathcal{D}_{\func{good}}}\left\vert E_{I}^{\sigma }f\right\vert $ is
replaced by $\sup_{I\in \left\{ A_{J},A\right\} }\left\vert E_{I}^{\sigma
}f\right\vert $.

Thus we have using $N\approx \frac{1}{\delta }\approx \frac{1}{\theta }$,
that $\left( 1+2^{N\delta }\right) \frac{N^{3}}{\delta ^{2}}\approx \frac{1}{%
\theta ^{5}}$ and so,%
\begin{eqnarray*}
\left\vert \mathsf{B}_{\limfunc{stop}}^{\mathcal{S}^{\left( n\right)
}}\left( f,g\right) \right\vert &\leq &C\frac{1}{\theta ^{5}}\ \sup_{0\leq
s\leq N}\left\Vert \left\vert f\right\vert _{\mathcal{F},\mathcal{S}^{\left(
n-1\right) },\mathcal{S}^{\left( n\right) }}^{\left\{ \mathcal{C}_{\mathcal{S%
}^{\left( n\right) }}^{\left( s\right) }\left( A\right) \cap \Lambda
_{g_{Q}}^{\omega }\right\} _{A\in \mathcal{S}^{\left( n\right)
}}}\right\Vert _{L^{p}\left( \omega \right) }\left\Vert g\right\Vert
_{L^{p^{\prime }}\left( \omega \right) }+\left\vert \mathsf{B}_{\limfunc{stop%
}\limfunc{no}\func{top}}^{\mathcal{S}^{\left( n+1\right) }}\left( f,g\right)
\right\vert \\
&=&C\frac{1}{\theta ^{5}}\ \sup_{0\leq s\leq N}\left\Vert \left\vert
f\right\vert _{\mathcal{F},\mathcal{S}^{\left( n-1\right) },\mathcal{S}%
^{\left( n\right) }}^{\left\{ \mathcal{C}_{\mathcal{S}^{\left( n\right)
}}^{\left( s\right) }\left( A\right) \cap \Lambda _{g_{Q}}^{\omega }\right\}
_{A\in \mathcal{S}^{\left( n\right) }}}\right\Vert _{L^{p}\left( \omega
\right) }\left\Vert g\right\Vert _{L^{p^{\prime }}\left( \omega \right)
}+\left\vert \mathsf{B}_{\limfunc{stop}}^{\mathcal{S}^{\left( n+1\right)
}}\left( \mathsf{P}^{\limfunc{no}\func{top}}f,\mathsf{P}^{\limfunc{no}\func{%
top}}g\right) \right\vert ,
\end{eqnarray*}%
and iteration yields%
\begin{eqnarray}
&&\left\vert \mathsf{B}_{\limfunc{stop}}^{\mathcal{F}}\left( f,g\right)
\right\vert =\left\vert \mathsf{B}_{\limfunc{stop}}^{\mathcal{S}^{\left(
0\right) }}\left( f,g\right) \right\vert \leq C\frac{1}{\theta ^{5}}%
\sup_{0\leq s\leq N}\left\Vert \left\vert f\right\vert _{\mathcal{F},%
\mathcal{S}^{\left( 0\right) },\mathcal{S}^{\left( 1\right) }}^{\left\{ 
\mathcal{C}_{\mathcal{S}^{\left( 1\right) }}^{\left( s\right) }\left(
A\right) \cap \Lambda _{g_{Q}}^{\omega }\right\} _{A\in \mathcal{S}^{\left(
1\right) }}}\right\Vert _{L^{p}\left( \omega \right) }\left\Vert
g\right\Vert _{L^{p^{\prime }}\left( \omega \right) }+\left\vert \mathsf{B}_{%
\limfunc{stop}}^{\mathcal{S}^{\left( 1\right) }}\left( f,g\right) \right\vert
\label{stopping point} \\
&\leq &C\frac{1}{\theta ^{5}}\left\{ \sum_{k=0}^{n}\sup_{0\leq s\leq
N}\left\Vert \left\vert f\right\vert _{\mathcal{F},\mathcal{S}^{\left(
k-1\right) },\mathcal{S}^{\left( k\right) }}^{\left\{ \mathcal{C}_{\mathcal{S%
}^{\left( k\right) }}^{\left( s\right) }\left( A\right) \cap \Lambda
_{g_{Q}}^{\omega }\right\} _{A\in \mathcal{S}^{\left( k\right)
}}}\right\Vert _{L^{p}\left( \omega \right) }\right\} \left\Vert
g\right\Vert _{L^{p^{\prime }}\left( \omega \right) }+\left\vert \mathsf{B}_{%
\limfunc{stop}}^{\mathcal{S}^{\left( n\right) }}\left( f,g\right)
\right\vert ,  \notag
\end{eqnarray}%
for $n\in \mathbb{N}$, where we have suppressed the projections $\mathsf{P}^{%
\limfunc{no}\func{top}}$ that accummulate as we iterate. Once the lemma in
the next subsubsection is proved, we are done since $\mathsf{B}_{\limfunc{%
stop}}^{\mathcal{S}^{\left( n\right) }}\left( f,g\right) $ vanishes for $n$
sufficiently large because of the finite Haar support assumptions on $f$ and 
$g$, and then using $\min \left\{ \frac{1}{p^{\prime }},\frac{4-p}{2p}%
\right\} >0$ for $1<p<4$, we obtain 
\begin{eqnarray*}
\left\vert \mathsf{B}_{\limfunc{stop}}^{\mathcal{F}}\left( f,g\right)
\right\vert &\lesssim &\frac{1}{\theta ^{5}}\sum_{n=0}^{\infty }\sup_{0\leq
s\leq N}\left\Vert \left\vert f\right\vert _{\mathcal{F},\mathcal{S}^{\left(
n-1\right) },\mathcal{S}^{\left( n\right) }}^{\left\{ \mathcal{C}_{\mathcal{S%
}^{\left( n\right) }}^{\left( s\right) }\left( A\right) \cap \Lambda
_{g_{A}}^{\omega }\right\} _{A\in \mathcal{S}^{\left( n\right)
}}}\right\Vert _{L^{p}\left( \omega \right) }\left\Vert g\right\Vert
_{L^{p^{\prime }}\left( \omega \right) } \\
&\lesssim &\frac{1}{\theta ^{5}}\left( \sum_{n=0}^{\infty }\left(
C_{p}\theta \right) ^{n\min \left\{ \frac{1}{p^{\prime }},\frac{4-p}{2p}%
\right\} }\right) \mathfrak{T}_{H,p}^{\func{loc}}\left( \sigma ,\omega
\right) \left\Vert f\right\Vert _{L^{p}\left( \sigma \right) }\left\Vert
g\right\Vert _{L^{p^{\prime }}\left( \omega \right) } \\
&\leq &C_{p,\theta }\mathfrak{T}_{H,p}^{\func{loc}}\left( \sigma ,\omega
\right) \left\Vert f\right\Vert _{L^{p}\left( \sigma \right) }\left\Vert
g\right\Vert _{L^{p^{\prime }}\left( \omega \right) },
\end{eqnarray*}%
where $C_{p,\theta }<\infty $ provided $C_{p}\theta <1$. Thus we see that
the stopping form is controlled by the scalar testing characteristic $%
\mathfrak{T}_{H,p}^{\func{loc}}\left( \sigma ,\omega \right) $, which is of
course at most $\mathfrak{T}_{H,p}^{\ell ^{2},\func{loc}}\left( \sigma
,\omega \right) $.

\subsubsection{The decay lemma}

Here is the final lemma of the paper.

\begin{lemma}
\label{final level}Let $n\in \mathbb{N}$, $0\leq s\leq N$, and $1<p<4$.\
There is $C_{p}>0$ such that for all $0<\theta <1$, there is a positive
constant $B_{\theta }$ such that,%
\begin{equation}
\left\Vert \left\vert f\right\vert _{\mathcal{F},\mathcal{S}^{\left(
n-1\right) },\mathcal{S}^{\left( n\right) }}^{\left\{ \mathcal{C}_{\mathcal{S%
}^{\left( n\right) }}^{\left( s\right) }\left( Q\right) \cap \Lambda
_{g_{Q}}^{\omega }\right\} _{Q\in \mathcal{S}^{\left( n-1\right)
}}}\right\Vert _{L^{p}\left( \omega \right) }\lesssim \left\{ 
\begin{array}{ccc}
B_{\theta }\left( C_{p}\theta \right) ^{\frac{n}{p^{\prime }}}\mathfrak{T}%
_{H,p}^{\func{loc}}\left( \sigma ,\omega \right) \left\Vert f\right\Vert
_{L^{p}\left( \sigma \right) } & \text{ if } & 1<p\leq 2 \\ 
B_{\theta }\left( C_{p}\theta \right) ^{n\frac{4-p}{2p}}\mathfrak{T}_{H,p}^{%
\func{loc}}\left( \sigma ,\omega \right) \left\Vert f\right\Vert
_{L^{p}\left( \sigma \right) } & \text{ if } & 2\leq p<4%
\end{array}%
\right. .  \label{final level '}
\end{equation}
\end{lemma}

\begin{proof}
Fix $0\leq s\leq N$ throughout the following arguments. We begin with the
function%
\begin{equation*}
h\left( x\right) \equiv \sum_{Q\in \mathcal{S}^{\left( n-1\right)
}}\sum_{A\in \mathcal{A}^{\left( s\right) }\left[ Q\right] }\sum_{S\in 
\mathfrak{C}_{\mathcal{A}}\left( A\right) }\sum_{K\in \mathcal{W}_{\func{good%
},\tau }\left( S\right) }2^{-\frac{1}{2}\limfunc{dist}\left( A,Q\right)
\delta }\alpha _{A}\left( S\right) \frac{\mathrm{P}\left( K,\mathbf{1}%
_{F\setminus S}\sigma \right) }{\ell \left( K\right) }\mathsf{P}%
_{S;K}^{\omega ,\mathcal{C}_{\mathcal{A}}^{\left( s\right) }\left( A\right)
\cap \Lambda _{g_{Q}}^{\omega }}Z\left( x\right) ,
\end{equation*}%
where 
\begin{equation*}
g_{Q}\equiv \mathsf{P}_{\mathcal{C}_{\mathcal{S}^{\left( n-1\right) }}\left(
Q\right) }^{\omega }g\text{ for }Q\in \mathcal{S}^{\left( n-1\right) },
\end{equation*}%
and the associated sequence associated with $h$ and $\mathcal{A}$,%
\begin{equation*}
\left\{ 2^{-\frac{1}{2}\limfunc{dist}\left( A,Q\right) \delta }\alpha
_{A}\left( S\right) \frac{\mathrm{P}\left( K,\mathbf{1}_{F\setminus S}\sigma
\right) }{\ell \left( K\right) }\mathsf{P}_{S;K}^{\omega ,\mathcal{C}_{%
\mathcal{A}}^{\left( s\right) }\left( A\right) \cap \Lambda _{g_{Q}}^{\omega
}}Z\left( x\right) \right\} _{Q,A,S,K\in \mathcal{S}^{\left( n-1\right)
}\times \mathcal{A}\left[ Q\right] \times \mathfrak{C}_{\mathcal{A}}\left(
A\right) \times \mathcal{W}_{\func{good},\tau }\left( S\right) }.
\end{equation*}%
Note that $\left\Vert h\right\Vert _{L^{p}\left( \omega \right) }=\left\Vert
\left\vert f\right\vert _{\mathcal{F},\mathcal{S}^{\left( n-1\right) },%
\mathcal{S}^{\left( n\right) }}^{\left\{ \mathcal{C}_{\mathcal{S}^{\left(
n\right) }}^{\left( s\right) }\left( Q\right) \cap \Lambda _{g_{Q}}^{\omega
}\right\} _{Q\in \mathcal{S}^{\left( n-1\right) }}}\right\Vert _{L^{p}\left(
\omega \right) }$. In analogy with (\ref{K sat}), we note\ from the
definition in (\ref{def fAQ}) that the projection 
\begin{equation*}
\mathsf{P}_{S;K}^{\omega ,\mathcal{C}_{\mathcal{A}}^{\left( s\right) }\left(
A\right) \cap \Lambda _{g_{Q}}^{\omega }}=\sum_{\substack{ J\in \mathcal{C}_{%
\mathcal{A}}^{\left( s\right) }\left( A\right) \cap \Lambda _{g_{Q}}^{\omega
}:\ J\subset _{\tau }\Lambda _{f}^{\sigma }\left[ S\right]  \\ J\subset K}}%
\bigtriangleup _{J}^{\omega }
\end{equation*}%
vanishes unless $K\in \mathcal{C}_{\mathcal{A}}^{\left( s\right) }\left(
A\right) \cap \mathcal{W}_{\func{good},\tau }\left( S\right) $.

Now recall the definition of an iterated martingale difference sequence $%
\left\{ h_{k}\right\} _{k=1}^{\infty }$ from subsubsection \ref{iter subsub
sec}, where for an iterated stopping time $\mathcal{Q}\circ \mathcal{A}$ ,
we defined in (\ref{iter defs}) and (\ref{def iter dist}), the martingale
differences $h_{k}$, the maximal depths $D_{k}$, and the iterated corona
distance $\limfunc{dist}{}_{\mathcal{Q}\circ \mathcal{A}}\left( A,T\right) =%
\limfunc{xdist}\ _{\mathcal{A}}\left( A,T\right) $ from the root $T$ to $%
A\in \mathcal{A}$. In Conclusion \ref{concl it} we referred to this
construction there as the \emph{regularization} of the `standard' definition
of the $\mathcal{A}$-corona martingale difference sequence given in
Definition \ref{def mart}.

Now we apply this regularization to the multiply iterated stopping times $%
\mathcal{S}^{\left( n\right) }=\mathcal{S}^{\left( 0\right) }\circ \mathcal{S%
}^{\left( 1\right) }\circ ...\circ \mathcal{S}^{\left( n\right) }$. We then
have the following \emph{regularizing} property. If $Q_{n}\in \mathcal{S}%
^{\left( n\right) }$ has associated tower%
\begin{equation}
\mathcal{C}_{\mathcal{S}^{\left( n\right) }}\left( Q_{n}\right) \subset 
\mathcal{C}_{\mathcal{S}^{\left( n-1\right) }}\left( Q_{n-1}\right) \subset 
\mathcal{C}_{\mathcal{S}^{\left( n-2\right) }}\left( Q_{n-2}\right) \subset
...\subset \mathcal{C}_{\mathcal{S}^{\left( 1\right) }}\left( Q_{1}\right)
\subset \mathcal{C}_{\mathcal{S}^{\left( 0\right) }}\left( Q_{0}\right) ,
\label{tower'}
\end{equation}%
then 
\begin{equation*}
\limfunc{xdist}\ _{\mathcal{S}^{\left( n\right) }}\left( Q_{n},T\right)
=D_{1}+D_{2}+...+D_{n-1}+\limfunc{dist}{}_{\mathcal{S}^{\left( n\right)
}}\left( Q_{n},Q_{n-1}\right) ,
\end{equation*}%
where the $D_{k}$ are defined using the single iteration $\mathcal{I}%
^{\left( n-1\right) }\circ \mathcal{S}^{\left( n\right) }$ where $\mathcal{I}%
^{\left( n-1\right) }\equiv \mathcal{S}^{\left( 0\right) }\circ \mathcal{S}%
^{\left( 1\right) }\circ ...\circ \mathcal{S}^{\left( n-1\right) }$. This
regularizing property makes it easier to track levels in the stopping
collection $\mathcal{S}^{\left( n\right) }$ in terms of levels within each
corona $\mathcal{C}_{\mathcal{S}^{\left( k\right) }}\left( Q_{k}\right) $, $%
0\leq k\leq n$, when it comes time for estimates later on. So, keeping in
mind that we write $\mathcal{A}=\mathcal{S}^{\left( n\right) }$ and $A=Q_{n}$
interchangeably, we will now use the \emph{iterated} martingale difference
sequence associated with the stopping times $\mathcal{S}^{\left( n\right) }$%
, which we write as,%
\begin{equation*}
\left\{ h_{k}\left( x\right) \right\} _{k\in \mathbb{N}}=\left\{
\sum_{\left( S,K\right) \in \mathfrak{C}_{\mathcal{S}^{\left( n\right)
}}\left( Q_{n}\right) \times \mathcal{W}_{\func{good},\tau }\left( S\right)
}\alpha _{Q_{n}}\left( S\right) \frac{\mathrm{P}\left( K,\mathbf{1}%
_{F\setminus S}\sigma \right) }{\ell \left( K\right) }\mathsf{P}%
_{S;K}^{\omega ,\mathcal{C}_{\mathcal{S}^{\left( n\right) }}^{\left(
s\right) }\left( Q_{n}\right) \cap \Lambda _{g_{Q_{n-1}}}^{\omega }}Z\left(
x\right) \right\} _{\substack{ Q_{n}\in \mathcal{S}^{\left( n\right) }  \\ 
\limfunc{itdist}{}_{\mathcal{S}^{\left( n\right) }}\left( Q_{n},T\right) =k}}
\end{equation*}%
where $Q_{n-1}$ is determined in terms of $Q_{n}$ by (\ref{tower'}). This
iterated martingale difference sequence has the martingale property because
the projections $\mathsf{P}_{S;K}^{\omega ,\mathcal{C}_{\mathcal{S}^{\left(
n\right) }}^{\left( s\right) }\left( Q_{n}\right) \cap \Lambda
_{g_{Q_{n-1}}}^{\omega }}$ have pairwise disjoint Haar supports. In
addition, recall the pair $\left( d_{1}\left( t\right) ,d_{2}\left( t\right)
\right) $ that in the present case of an $n$-fold iteration, becomes an $n$%
-tuple $\left( d_{1}\left( t\right) ,d_{2}\left( t\right) ,...,d_{n}\left(
t\right) \right) $ associated \ to $A\in \mathcal{A}$ with $t=\limfunc{xdist}%
\ _{\mathcal{S}^{\left( n\right) }}\left( A,T\right) $ and $d_{k}\left(
t\right) =\limfunc{dist}_{\mathcal{S}^{\left( k\right) }}\left(
Q_{k},Q_{k-1}\right) $.

Recall that for each $Q\in \mathcal{S}^{\left( n-1\right) }$ in the sum
above, and each $A\in \mathcal{A}\left[ Q_{n-1}\right] $, we relabel $A$ as $%
Q_{n}$ and $\mathcal{A}$ as $\mathcal{A}_{n}$ so that there is a tower of
coronas (\ref{tower'}), which we repeat here,%
\begin{equation}
\mathcal{C}_{\mathcal{S}^{\left( n-1\right) }}\left( Q_{n-1}\right) \subset 
\mathcal{C}_{\mathcal{S}^{\left( n-2\right) }}\left( Q_{n-2}\right) \subset
...\subset \mathcal{C}_{\mathcal{S}^{\left( 1\right) }}\left( Q_{1}\right)
\subset \mathcal{C}_{\mathcal{S}^{\left( 0\right) }}\left( Q_{0}\right) ,
\label{tower}
\end{equation}%
with $A=Q_{n}\subset Q=Q_{n-1}\subset Q_{n-2}\subset ...\subset Q_{1}\subset
Q_{0}$. Now define 
\begin{equation*}
\Omega \left[ Q_{n-1}\right] \equiv \left\{ \left( S,K\right) \in \mathfrak{C%
}_{\mathcal{S}^{\left( n-1\right) }}\left( Q_{n-1}\right) \times \mathcal{W}%
_{\func{good},\tau }\left( S\right) \right\} .
\end{equation*}%
Then for $K\in \mathcal{W}_{\func{good},\tau }\left( S\right) $ with $%
Q_{n}\in \mathcal{A}_{n}\left[ Q_{n-1}\right] $ and $S\in \mathfrak{C}_{%
\mathcal{A}_{n}}\left( Q_{n}\right) $, i.e. $\left( S,K\right) \in \Omega %
\left[ Q_{n-1}\right] $, the projection $\mathsf{P}_{S;K}^{\omega ,\mathcal{C%
}_{\mathcal{S}^{\left( n\right) }}\left( Q_{n}\right) }Z\left( x\right) $
vanishes unless $K\in \mathcal{C}_{\mathcal{S}^{\left( n\right) }}\left(
Q\right) =\mathcal{C}_{\mathcal{S}^{\left( n\right) }}\left( Q_{n}\right) $,
in which case there is $C_{p}>0$ such that for $2\leq j\leq n$,%
\begin{equation*}
\left\Vert \mathsf{P}_{S;K}^{\omega ,\mathcal{C}_{\mathcal{S}^{\left(
n-1\right) }}\left( Q_{j-1}\right) \cap \Lambda _{g_{_{Q_{n-2}}}}^{\omega
}}Z\left( x\right) \right\Vert _{L^{p}\left( \omega \right) }^{p}\leq
C_{p}\theta ^{\natural }2^{-\limfunc{dist}\left( Q_{j-1},Q_{j-2}\right)
\delta }\left\Vert \mathsf{P}_{S;K}^{\omega ,\mathcal{C}_{\mathcal{S}%
^{\left( n-2\right) }}\left( Q_{n-2}\right) \cap \Lambda
_{g_{_{Q_{j-3}}}}^{\omega }}Z\left( x\right) \right\Vert _{L^{p}\left(
\omega \right) }^{p},
\end{equation*}%
since $K\in \mathcal{C}_{\mathcal{S}^{\left( j-1\right) }}\left(
Q_{j-1}\right) $.

Define%
\begin{equation}
\Psi _{Q_{n}}\equiv 2^{-\sum_{k=1}^{n}\limfunc{dist}\left(
Q_{k},Q_{k-1}\right) \delta },\ \ \ \ \ \text{for }Q=Q_{n}\in \mathcal{Q}%
_{n},  \label{def Psi Q}
\end{equation}%
where $\left\{ Q_{k}\right\} _{k=1}^{n}$ is the tower of intervals $Q_{k}\in 
\mathfrak{C}_{S^{\left( k-1\right) }}\left( Q_{k-1}\right) $ for $1\leq
k\leq n$. Note that%
\begin{equation}
\Psi _{A}=2^{-d\left( t\right) \delta },\ \ \ \ \ \text{where }d\left(
t\right) \equiv \sum_{k=1}^{n}d_{k}\left( t\right) \text{ and }t=\limfunc{%
xdist}\ _{\mathcal{S}^{\left( n\right) }}\left( A,T\right) .
\label{connection}
\end{equation}%
For convenience we set%
\begin{equation*}
B_{n}\left( S,K\right) \equiv \Psi _{Q_{n}}\alpha _{Q_{n}}\left( S\right)
\left( \frac{\mathrm{P}\left( K,\mathbf{1}_{F\setminus S}\sigma \right) }{%
\ell \left( K\right) }\right) .
\end{equation*}%
It follows that for each $Q_{n}\in \mathcal{A}$ and $S\in \mathfrak{C}_{%
\mathcal{A}}$,%
\begin{equation}
\left\Vert \alpha _{Q_{n}}\left( S\right) \left( \frac{\mathrm{P}\left( K,%
\mathbf{1}_{F\setminus S}\sigma \right) }{\ell \left( K\right) }\right) 
\mathsf{P}_{S;K}^{\omega ,\mathcal{C}_{\mathcal{S}^{\left( n\right) }}\left(
Q_{n}\right) }Z\left( x\right) \right\Vert _{L^{p}\left( \omega \right)
}^{p}\lesssim \left( C_{p}\theta ^{\natural }\right) ^{n}\left\Vert
B_{n}\left( S,K\right) \mathsf{P}_{S;K}^{\omega ,\mathcal{C}_{\mathcal{S}%
^{\left( 0\right) }}\left( Q_{0}\right) }Z\left( x\right) \right\Vert
_{L^{p}\left( \omega \right) }^{p},  \label{big gain}
\end{equation}%
where $\left\{ Q_{k}\right\} _{k=1}^{n}$ is the tower associated with $Q_{n}$%
. As a consequence, for each pair $Q_{n}\in \mathcal{A}$ and $S\in \mathfrak{%
C}_{\mathcal{A}}\left( Q_{n}\right) $, the factor $2^{-\left( \sum_{k=1}^{n}%
\limfunc{dist}\left( Q_{k},Q_{k-1}\right) \right) \delta }$ is at most one
and becomes smaller as the distances $\limfunc{dist}\left(
Q_{k},Q_{k-1}\right) $ grow. Since summing over all $Q_{n}\in \mathcal{S}%
^{\left( n\right) }$ can be reindexed as summing over all towers in (\ref%
{tower}), we will consequently write $\sum_{Q_{n}\in \mathcal{S}^{\left(
n\right) }}=\sum_{\left\{ Q_{k}\right\} _{k=0}^{n}}$ interchangeably
depending on context. We now estimate the norm $\left\Vert h\right\Vert
_{L^{p}\left( \omega \right) }^{p}$ separately in the cases $1<p\leq 2$ and $%
2\leq p<4$, beginning with $2\leq p<4$.

\textbf{The case }$2\leq p<4$. In this case $\theta ^{\natural }=\theta $
and this will be reflected in what follows. We now claim that for $2\leq p<4$%
,%
\begin{eqnarray}
&&\left\Vert \left\vert \left\{ \alpha _{A}\left( S\right) \left( \frac{%
\mathrm{P}\left( K,\mathbf{1}_{F\setminus S}\sigma \right) }{\ell \left(
K\right) }\right) \mathsf{P}_{S;K}^{\omega ,\Lambda _{g_{A}}}Z\left(
x\right) \ \mathbf{1}_{K}\left( x\right) \right\} _{\left( F,A,S,K\right)
\in \Omega ^{4}}\right\vert _{\ell ^{2}}\right\Vert _{L^{p}\left( \omega
\right) }^{p}  \label{like} \\
&\lesssim &\left( C_{p}\theta \right) ^{n\left( 2-\frac{p}{2}\right) }%
\mathfrak{X}_{\mathcal{F};p}\left( \sigma ,\omega \right) ^{p}\left\Vert
f\right\Vert _{L^{p}\left( \sigma \right) }^{p}\ ,  \notag
\end{eqnarray}%
or more succinctly,%
\begin{equation*}
\int_{\mathbb{R}}\left( \sum_{K\in \Omega ^{4}}c_{F,A,S,K}^{2}\left\vert 
\mathsf{P}_{S;K}^{\omega ,\Lambda _{g_{A}}}Z\left( x\right) \right\vert
^{2}\ \mathbf{1}_{K}\left( x\right) \right) ^{\frac{p}{2}}d\omega \left(
x\right) \lesssim \left( C_{p}\theta \right) ^{n\left( 2-\frac{p}{2}\right) }%
\mathfrak{X}_{\mathcal{F};p}\left( \sigma ,\omega \right) ^{p}\left\Vert
f\right\Vert _{L^{p}\left( \sigma \right) }^{p},
\end{equation*}%
where 
\begin{equation*}
c_{F,A,S,K}\equiv \alpha _{A}\left( S\right) \left( \frac{\mathrm{P}\left( K,%
\mathbf{1}_{F\setminus S}\sigma \right) }{\ell \left( K\right) }\right) ,
\end{equation*}%
and where the tower $\left\{ Q_{k}\right\} _{k=0}^{n}$ is determined from $%
Q_{k}=A$ and $Q_{0}=F\in \mathcal{F}$.

Using the iterated corona decomposition, we have for $\eta >0$ to be chosen
later,%
\begin{eqnarray*}
&&\int_{\mathbb{R}}\left( \sum_{\left( F,A,S,K\right) \in \Omega
^{4}}c_{F,A,S,K}^{2}\left\vert \mathsf{P}_{K}^{\Lambda _{g_{A}}}Z\left(
x\right) \right\vert ^{2}\ \mathbf{1}_{K}\left( x\right) \right) ^{\frac{p}{2%
}}d\omega \left( x\right) \\
&=&\int_{\mathbb{R}}\left( \sum_{t=1}^{\infty }\sum_{\left( F,A,S,K\right)
\in \Omega _{t}^{4}}c_{F,A,S,K}^{2}\left\vert \mathsf{P}_{K}^{\Lambda
_{g_{A}}}Z\left( x\right) \right\vert ^{2}\ \mathbf{1}_{K}\left( x\right)
\right) ^{\frac{p}{2}}d\omega \left( x\right) \\
&=&\int_{\mathbb{R}}\left( \sum_{t=1}^{\infty }2^{-d\left( t\right) \eta
}\sum_{\left( F,A,S,K\right) \in \Omega _{t}^{4}}2^{d\left( t\right) \eta
}c_{F,A,S,K}^{2}\left\vert \mathsf{P}_{K}^{\Lambda _{g_{A}}}Z\left( x\right)
\right\vert ^{2}\ \mathbf{1}_{K}\left( x\right) \right) ^{\frac{p}{2}%
}d\omega \left( x\right) ,
\end{eqnarray*}%
where $d\left( t\right) $ is defined in (\ref{connection}), and we write $%
\left( F,A,S,K\right) \in \Omega _{t}^{4}$ to mean that $A$ is $t$ levels
below $T$ in the iterated stopping time construction, i.e. $t=\limfunc{xdist%
}{}_{\mathcal{A}}\left( A,T\right) $. Typically, $d\left( t\right) $ is much
smaller that $t$, and this is what gives rise to large negative powers of $%
\theta $ below. By H\"{o}lder's inequality with exponent $\frac{p}{2}$, this
is at most%
\begin{eqnarray}
&&\int_{\mathbb{R}}\left( \left[ \sum_{t=1}^{\infty }2^{-d\left( t\right)
\eta \frac{p}{p-2}}\right] ^{1-\frac{2}{p}}\left( \sum_{t=1}^{\infty }\left[
\sum_{\left( F,A,S,K\right) \in \Omega _{t}^{4}}2^{d\left( t\right) \eta
}c_{F,A,S,K}^{2}\left\vert \mathsf{P}_{K}^{\Lambda _{g_{A}}}Z\left( x\right)
\right\vert ^{2}\ \mathbf{1}_{K}\left( x\right) \right] ^{\frac{p}{2}%
}\right) ^{\frac{2}{p}}\right) ^{\frac{p}{2}}d\omega \left( x\right)
\label{Holder} \\
&=&\left[ \sum_{t=1}^{\infty }2^{-d\left( t\right) \eta \frac{p}{p-2}}\right]
^{\frac{p}{2}-1}\int_{\mathbb{R}}\sum_{t=1}^{\infty }\left[ \sum_{\left(
F,A,S,K\right) \in \Omega _{t}^{4}}2^{d\left( t\right) \eta
}c_{F,A,S,K}^{2}\left\vert \mathsf{P}_{K}^{\Lambda _{g_{A}}}Z\left( x\right)
\right\vert ^{2}\ \mathbf{1}_{K}\left( x\right) \right] ^{\frac{p}{2}%
}d\omega \left( x\right)  \notag \\
&=&\left[ \sum_{t=1}^{\infty }2^{-d\left( t\right) \eta \frac{p}{p-2}}\right]
^{\frac{p}{2}-1}\int_{\mathbb{R}}\sum_{t=1}^{\infty }\sum_{\left(
F,A,S,K\right) \in \Omega _{t}^{4}}2^{d\left( t\right) \eta \frac{p}{2}%
}c_{F,A,S,K}^{p}\left\vert \mathsf{P}_{K}^{\Lambda _{g_{A}}}Z\left( x\right)
\right\vert ^{p}\ \mathbf{1}_{K}\left( x\right) d\omega \left( x\right) . 
\notag
\end{eqnarray}%
Now we compute%
\begin{eqnarray*}
&&\int_{\mathbb{R}}\sum_{t=1}^{\infty }\sum_{\left( F,A,S,K\right) \in
\Omega _{t}^{4}}2^{d\left( t\right) \eta \frac{p}{2}}c_{F,A,S,K}^{p}\left%
\vert \mathsf{P}_{K}^{\Lambda _{g_{A}}}Z\left( x\right) \right\vert ^{p}\ 
\mathbf{1}_{K}\left( x\right) d\omega \left( x\right) \\
&=&\sum_{t=1}^{\infty }\sum_{\left( F,A,S,K\right) \in \Omega
_{t}^{4}}2^{d\left( t\right) \eta \frac{p}{2}}c_{F,A,S,K}^{p}\int_{\mathbb{R}%
}\left\vert \mathsf{P}_{K}^{\Lambda _{g_{A}}}Z\left( x\right) \right\vert
^{p}d\omega \left( x\right) \\
&\leq &\sum_{t=1}^{\infty }\sum_{\left( F,A,S,K\right) \in \Omega
_{t}^{4}}2^{d\left( t\right) \eta \frac{p}{2}}c_{F,A,S,K}^{p}\left(
C_{p}\theta \right) ^{n}2^{-d\left( t\right) \delta }\int_{\mathbb{R}%
}\left\vert \mathsf{P}_{K}^{\Lambda _{g_{F}}}Z\left( x\right) \right\vert
^{p}d\omega \left( x\right) \\
&=&\left( C_{p}\theta \right) ^{n}\sum_{t=1}^{\infty }2^{d\left( t\right)
\left( \eta \frac{p}{2}-\delta \right) }\sum_{\left( F,A,S,K\right) \in
\Omega _{t}^{4}}c_{F,A,S,K}^{p}\int_{\mathbb{R}}\left\vert \mathsf{P}%
_{K}^{\Lambda _{g_{F}}}Z\left( x\right) \right\vert ^{p}d\omega \left(
x\right) ,
\end{eqnarray*}%
which gives,%
\begin{eqnarray*}
&&\int_{\mathbb{R}}\left( \sum_{\left( F,A,S,K\right) \in \Omega
^{4}}c_{F,A,S,K}^{2}\left\vert \mathsf{P}_{K}^{\Lambda _{g_{A}}}Z\left(
x\right) \right\vert ^{2}\ \mathbf{1}_{K}\left( x\right) \right) ^{\frac{p}{2%
}}d\omega \left( x\right) \\
&\leq &\left[ \sum_{t=1}^{\infty }2^{-d\left( t\right) \eta \frac{p}{p-2}}%
\right] ^{\frac{p}{2}-1}\int_{\mathbb{R}}\sum_{t=1}^{\infty }\sum_{\left(
F,A,S,K\right) \in \Omega _{t}^{4}}2^{d\left( t\right) \eta \frac{p}{2}%
}c_{F,A,S,K}^{p}\left\vert \mathsf{P}_{K}^{\Lambda _{g_{A}}}Z\left( x\right)
\right\vert ^{p}\ \mathbf{1}_{K}\left( x\right) d\omega \left( x\right) \\
&\leq &\left( C_{p}\theta \right) ^{n}\left[ \sum_{t=1}^{\infty }2^{-d\left(
t\right) \eta \frac{p}{p-2}}\right] ^{\frac{p}{2}-1}\sum_{t=1}^{\infty
}2^{d\left( t\right) \left( \eta \frac{p}{2}-\delta \right) }\sum_{\left(
F,A,S,K\right) \in \Omega _{t}^{4}}c_{F,A,S,K}^{p}\int_{\mathbb{R}%
}\left\vert \mathsf{P}_{K}^{\Lambda _{g_{F}}}Z\left( x\right) \right\vert
^{p}d\omega \left( x\right) ,
\end{eqnarray*}%
where%
\begin{equation*}
\left[ \sum_{t=1}^{\infty }2^{-d\left( t\right) \eta \frac{p}{p-2}}\right] ^{%
\frac{p}{2}-1}\approx \left( \frac{C_{p}}{\eta }\right) ^{n\left( \frac{p}{2}%
-1\right) },
\end{equation*}%
and%
\begin{eqnarray*}
&&\sum_{t=1}^{\infty }2^{d\left( t\right) \left( \eta \frac{p}{2}-\delta
\right) }\sum_{\left( F,A,S,K\right) \in \Omega ^{4}:K\sim
t}c_{F,A,S,K}^{p}\int_{\mathbb{R}}\left\vert \mathsf{P}_{K}^{\Lambda
_{g_{F}}}Z\left( x\right) \right\vert ^{p}d\omega \left( x\right) \\
&\leq &\sum_{t=1}^{\infty }\sum_{\left( F,A,S,K\right) \in \Omega ^{4}:K\sim
t}c_{F,A,S,K}^{p}\int_{\mathbb{R}}\left\vert \mathsf{P}_{K}^{\Lambda
_{g_{F}}}Z\left( x\right) \right\vert ^{p}d\omega \left( x\right) ,
\end{eqnarray*}%
provided $\eta <\frac{2\delta }{p}$ (note that we are only using $2^{d\left(
t\right) \left( \eta \frac{p}{2}-\delta \right) }\leq 1$ here).

Indeed, we can bound the sum of the decay factors $2^{-d\left( t\right)
\delta }=\Psi _{Q^{\prime }}=2^{-\sum_{k=1}^{n-1}\limfunc{dist}\left(
Q_{k}^{\prime },Q_{k-1}^{\prime }\right) \delta }$ by setting $j_{k}=%
\limfunc{dist}\left( Q_{k}^{\prime },Q_{k-1}^{\prime }\right) $ and
computing, 
\begin{equation}
\sum_{t=1}^{\infty }2^{-d\left( t\right) \beta }=\sum_{t=1}^{\infty
}2^{-\sum_{k=1}^{n}d_{k}\left( t\right) \beta }\leq
\prod_{k=1}^{n}\sum_{j_{k}=0}^{\infty }2^{-j_{k}\beta }=\prod_{k=1}^{n}\frac{%
1}{1-2^{-\beta }}\sim \left( \frac{1}{\beta }\right) ^{n}.
\label{towers above}
\end{equation}%
So with for example $\eta =\frac{\delta }{p}<\frac{2\delta }{p}$, we get
altogether that,

\begin{eqnarray*}
&&\int_{\mathbb{R}}\left( \sum_{\left( F,A,S,K\right) \in \Omega
^{4}}c_{F,A,S,K}^{2}\left\vert \mathsf{P}_{K}^{\Lambda _{g_{A}}}Z\left(
x\right) \right\vert ^{2}\ \mathbf{1}_{K}\left( x\right) \right) ^{\frac{p}{2%
}}d\omega \left( x\right) \\
&\lesssim &\left( C_{p}\theta \right) ^{n}\left( \frac{C_{p}}{\delta }%
\right) ^{n\left( \frac{p}{2}-1\right) }\sum_{t=1}^{\infty }\sum_{\left(
F,A,S,K\right) \in \Omega _{t}^{4}}c_{F,A,S,K}^{p}\int_{\mathbb{R}%
}\left\vert \mathsf{P}_{K}^{\Lambda _{g_{F}}}Z\left( x\right) \right\vert
^{p}d\omega \left( x\right) \\
&\approx &\left( C_{p}\theta \right) ^{n}\left( \frac{C_{p}}{\theta }\right)
^{n\left( \frac{p}{2}-1\right) }\sum_{\left( F,A,S,K\right) \in \Omega
^{4}}c_{F,A,S,K}^{p}\int_{\mathbb{R}}\left\vert \mathsf{P}_{K}^{\Lambda
_{g_{F}}}Z\left( x\right) \right\vert ^{p}d\omega \left( x\right) \\
&=&\left( C_{p}\theta \right) ^{n\left( 2-\frac{p}{2}\right) }\sum_{\left(
F,A,S,K\right) \in \Omega ^{4}}c_{F,A,S,K}^{p}\int_{\mathbb{R}}\left\vert 
\mathsf{P}_{K}^{\Lambda _{g_{F}}}Z\left( x\right) \right\vert ^{p}d\omega
\left( x\right) \ .
\end{eqnarray*}%
Thus we are left to bound the term,%
\begin{eqnarray*}
&&\left( C_{p}\theta \right) ^{n\left( 2-\frac{p}{2}\right) }\sum_{\left(
F,A,S,K\right) \in \Omega ^{4}}c_{F,A,S,K}^{p}\int_{\mathbb{R}}\left\vert 
\mathsf{P}_{K}^{\Lambda _{g_{F}}}Z\left( x\right) \right\vert ^{p}d\omega
\left( x\right) \\
&=&\left( C_{p}\theta \right) ^{n\left( 2-\frac{p}{2}\right) }\sum_{\left(
F,A,S,K\right) \in \Omega ^{4}}\alpha _{A}\left( S\right) ^{p}\left( \frac{%
\mathrm{P}\left( K,\mathbf{1}_{F\setminus S}\sigma \right) }{\ell \left(
K\right) }\right) ^{p}\int_{\mathbb{R}}\left\vert \mathsf{P}_{K}^{\Lambda
_{g_{F}}}Z\left( x\right) \right\vert ^{p}d\omega \left( x\right) \\
&=&\left( C_{p}\theta \right) ^{n\left( 2-\frac{p}{2}\right) }\int_{\mathbb{R%
}}\sum_{F\in \mathcal{F}}\sum_{Q_{n}\in \mathcal{S}^{\left( n\right) }\left[
F\right] }\sum_{\left( S,K\right) \in \Omega \left[ Q_{n}\right] }\alpha
_{Q_{n}}\left( S\right) ^{p}\left( \frac{\mathrm{P}\left( K,\mathbf{1}%
_{F\setminus S}\sigma \right) }{\ell \left( K\right) }\right) ^{p}\left\vert 
\mathsf{P}_{S;K}^{\omega ,\mathcal{C}_{\mathcal{F}}\left( F\right) }Z\left(
x\right) \right\vert ^{p}d\omega \left( x\right) .
\end{eqnarray*}

Using 
\begin{eqnarray}
&&\int_{\mathbb{R}}\left\vert \mathsf{P}_{S;K}^{\omega ,\mathcal{C}_{%
\mathcal{F}}\left( F\right) }Z\left( x\right) \right\vert ^{p}d\omega \left(
x\right) \approx \int_{\mathbb{R}}\left( \sum_{J\ \func{appears}\ \limfunc{in%
}\ \mathsf{P}_{S;K}^{\omega ,\mathcal{C}_{\mathcal{F}}\left( F\right)
}}\left\vert \bigtriangleup _{J}^{\omega }Z\left( x\right) \right\vert
^{2}\right) ^{\frac{p}{2}}d\omega \left( x\right)  \label{from 1} \\
&\leq &\left\vert K\right\vert _{\omega }\left( \frac{1}{\left\vert
K\right\vert _{\omega }}\int_{K}\left( \sum_{J^{\prime }\subset K}\left\vert
\bigtriangleup _{J^{\prime }}^{\omega }Z\left( x\right) \right\vert
^{2}\right) ^{\frac{p}{2}}d\omega \left( x\right) \right) =\left\vert
K\right\vert _{\omega }\mathsf{E}_{p}\left( K,\omega \right) ^{p}  \notag
\end{eqnarray}%
together with (\ref{PE char}), 
\begin{equation}
\sup_{K\in \mathcal{C}_{\mathcal{F}}\left( F\right) \cap \mathcal{D}_{\func{%
good}}^{\limfunc{child}}}\left( \frac{\mathrm{P}\left( K,\mathbf{1}%
_{F\setminus K}\sigma \right) }{\ell \left( K\right) }\right) ^{p}\mathsf{E}%
_{p}\left( K,\omega \right) ^{p}\frac{\left\vert K\right\vert _{\omega }}{%
\left\vert K\right\vert _{\sigma }}\leq \mathfrak{X}_{F;p}\left( \sigma
,\omega \right) ^{p},  \label{from 2}
\end{equation}%
and the boundedness of $M_{\sigma }^{\func{dy}}$,\ and the square function
estimate in Theorem \ref{square thm}, we can\ finally bound the above
integral by%
\begin{eqnarray*}
&&\left( C_{p}\theta \right) ^{n\left( 2-\frac{p}{2}\right) }\mathfrak{X}_{%
\mathcal{F};p}\left( \sigma ,\omega \right) ^{p}\sum_{F\in \mathcal{F}%
}\sum_{Q_{n}\in \mathcal{S}^{\left( n\right) }\left[ F\right] }\sum_{\left(
S,K\right) \in \Omega \left[ Q_{n}\right] }\alpha _{Q_{n}}\left( S\right)
^{p}\left\vert K\right\vert _{\sigma } \\
&\leq &\left( C_{p}\theta \right) ^{n\left( 2-\frac{p}{2}\right) }\mathfrak{X%
}_{\mathcal{F};p}\left( \sigma ,\omega \right) ^{p}\int_{\mathbb{R}%
}\sum_{F\in \mathcal{F}}\sum_{Q_{n}\in \mathcal{S}^{\left( n\right) }\left[ F%
\right] }\sum_{S\in \mathfrak{C}_{\mathcal{S}^{\left( n\right) }}\left[ Q_{n}%
\right] }\left\vert M_{\sigma }\left( \mathsf{P}_{\mathcal{C}_{\mathcal{S}%
^{\left( n\right) }}\left( Q_{n}\right) }^{\sigma }f\right) \left( x\right)
\right\vert ^{p}d\sigma \left( x\right) \\
&\lesssim &\left( C_{p}\theta \right) ^{n\left( 2-\frac{p}{2}\right) }%
\mathfrak{X}_{\mathcal{F};p}\left( \sigma ,\omega \right) ^{p}\int_{\mathbb{R%
}}\sum_{\left( F,Q_{n}\right) \in \mathcal{F}\times \mathcal{S}^{\left(
n\right) }\left[ F\right] }\left\vert \mathsf{P}_{\mathcal{C}_{\mathcal{S}%
^{\left( n\right) }}\left( Q_{n}\right) }^{\sigma }f\left( x\right)
\right\vert ^{p}d\sigma \left( x\right) \\
&\lesssim &\left( C_{p}\theta \right) ^{n\left( 2-\frac{p}{2}\right) }%
\mathfrak{X}_{\mathcal{F};p}\left( \sigma ,\omega \right) ^{p}\int_{\mathbb{R%
}}\left( \sum_{\left( F,Q_{n}\right) \in \mathcal{F}\times \mathcal{S}%
^{\left( n\right) }\left[ F\right] }\left\vert \mathsf{P}_{\mathcal{C}_{%
\mathcal{S}^{\left( n\right) }}\left( Q_{n}\right) }^{\sigma }f\left(
x\right) \right\vert ^{2}\right) ^{\frac{p}{2}}d\sigma \left( x\right) \ \ \
\ \ \text{(since }p>2\text{)} \\
&\lesssim &\left( C_{p}\theta \right) ^{n\left( 2-\frac{p}{2}\right) }%
\mathfrak{X}_{\mathcal{F};p}\left( \sigma ,\omega \right) ^{p}\left\Vert
f\right\Vert _{L^{p}\left( \sigma \right) }^{p}.
\end{eqnarray*}

\textbf{The case }$1<p\leq 2$. In this case $\theta ^{\natural }=\theta ^{%
\frac{p}{2}}$. We have, recalling that $\Omega ^{4}\equiv \mathcal{F}\times 
\mathcal{S}^{\left( n\right) }\left[ F\right] \times \mathfrak{C}_{\mathcal{S%
}^{\left( n\right) }}\left( Q_{n}\right) \times \mathcal{W}_{\func{good}%
,\tau }\left( S\right) $,

\begin{eqnarray*}
&&\left\Vert h\right\Vert _{L^{p}\left( \omega \right) }^{p}=\left\Vert
\left\{ \Psi _{Q_{n}}\alpha _{Q_{n}}\left( S\right) \left( \frac{\mathrm{P}%
\left( K,\mathbf{1}_{F\setminus S}\sigma \right) }{\ell \left( K\right) }%
\right) \mathsf{P}_{S;K}^{\omega ,\mathcal{C}_{\mathcal{S}^{\left( n\right)
}}\left( Q_{n}\right) }Z\left( x\right) \right\} _{\left( F,Q_{n},S,K\right)
\in \Omega ^{4}}\right\Vert _{L^{p}\left( \ell ^{2};\omega \right) }^{p} \\
&=&\int_{\mathbb{R}}\left( \sum_{\left( F,Q_{n},S,K\right) \in \Omega
^{4}}\left\vert \Psi _{Q_{n}}\alpha _{Q_{n}}\left( S\right) \left( \frac{%
\mathrm{P}\left( K,\mathbf{1}_{F\setminus S}\sigma \right) }{\ell \left(
K\right) }\right) \mathsf{P}_{S;K}^{\omega ,\mathcal{C}_{\mathcal{S}^{\left(
n\right) }}\left( Q_{n}\right) }Z\left( x\right) \right\vert ^{2}\right) ^{%
\frac{p}{2}}d\omega \left( x\right) ,
\end{eqnarray*}%
which using $1<p\leq 2$, is at most 
\begin{eqnarray*}
&&\int_{\mathbb{R}}\sum_{\left( F,Q_{n},S,K\right) \in \Omega
^{4}}\left\vert \Psi _{Q_{n}}\alpha _{Q_{n}}\left( S\right) \left( \frac{%
\mathrm{P}\left( K,\mathbf{1}_{F\setminus S}\sigma \right) }{\ell \left(
K\right) }\right) \mathsf{P}_{S;K}^{\omega ,\mathcal{C}_{\mathcal{S}^{\left(
n\right) }}\left( Q_{n}\right) }Z\left( x\right) \right\vert ^{p}d\omega
\left( x\right) \\
&\leq &\left( C_{p}\theta ^{\frac{p}{2}}\right) ^{n}\int_{\mathbb{R}%
}\sum_{\left( F,Q_{n},S,K\right) \in \Omega ^{4}}\Psi _{Q_{n}}^{p}\alpha
_{Q_{n}}\left( S\right) ^{p}\left( \frac{\mathrm{P}\left( K,\mathbf{1}%
_{F\setminus S}\sigma \right) }{\ell \left( K\right) }\right) ^{p}\left\vert 
\mathsf{P}_{S;K}^{\omega ,\mathcal{C}_{\mathcal{F}}\left( F\right) }Z\left(
x\right) \right\vert ^{p}d\omega \left( x\right) ,
\end{eqnarray*}%
where the sequence $\left\{ Q_{k}\right\} _{k=1}^{n-1}$ is the tower
associated with $Q_{n-1}$. Now using (\ref{from 1}) and (\ref{from 2}), we
bound the above integral by%
\begin{eqnarray*}
&&\left( C_{p}\theta ^{\frac{p}{2}}\right) ^{n}\mathfrak{X}_{\mathcal{F}%
;p}\left( \sigma ,\omega \right) ^{p}\sum_{\left( F,Q_{n},S,K\right) \in
\Omega ^{4}}\Psi _{Q_{n}}^{p}\alpha _{Q_{n}}\left( S\right) ^{p}\left\vert
K\right\vert _{\sigma } \\
&\leq &\left( C_{p}\theta ^{\frac{p}{2}}\right) ^{n}\mathfrak{X}_{\mathcal{F}%
;p}\left( \sigma ,\omega \right) ^{p}\int_{\mathbb{R}}\sum_{\left(
F,Q_{n},S,K\right) \in \Omega ^{4}}\Psi _{Q_{n}}^{p}\left\vert M_{\sigma
}\left( \mathsf{P}_{\mathcal{C}_{\mathcal{S}^{\left( n\right) }}\left(
Q_{n}\right) }^{\sigma }f\right) \left( x\right) \right\vert ^{p}d\sigma
\left( x\right) \\
&\lesssim &\left( C_{p}\theta ^{\frac{p}{2}}\right) ^{n}\mathfrak{X}_{%
\mathcal{F};p}\left( \sigma ,\omega \right) ^{p}\int_{\mathbb{R}%
}\sum_{\left( F,Q_{n},S,K\right) \in \Omega ^{4}}\Psi _{Q_{n}}^{p}\left\vert 
\mathsf{P}_{\mathcal{C}_{\mathcal{S}^{\left( n\right) }}\left( Q_{n}\right)
}^{\sigma }f\left( x\right) \right\vert ^{p}d\sigma \left( x\right) ,
\end{eqnarray*}%
where the last line follows from the boundedness of $M_{\sigma }^{\func{dy}}$
on $L^{p}\left( \sigma \right) $.

Again, we will exploit the decay of the factors $\Psi _{Q_{n}}$, as well as
the orthogonality hiding in the corona projections $\mathsf{P}_{\mathcal{C}_{%
\mathcal{S}^{\left( n\right) }}\left( Q_{n}\right) }^{\sigma }f$. We
successively apply the following inequalities for $1\leq k\leq n$,%
\begin{eqnarray*}
&&\int_{\mathbb{R}}\sum_{\left( Q_{k-1},Q_{k}\right) \in \mathcal{S}^{\left(
k-1\right) }\times \mathcal{S}^{\left( k\right) }\left[ Q_{k-1}\right] }2^{-p%
\limfunc{dist}\left( \left( Q_{k},Q_{k-1}\right) \right) \delta }\left\vert 
\mathsf{P}_{\mathcal{C}_{\mathcal{S}^{\left( k\right) }}\left( Q_{k}\right)
}^{\sigma }f\left( x\right) \right\vert ^{p}d\sigma \left( x\right) \\
&\leq &\int_{\mathbb{R}}\left( \sum_{\left( Q_{k-1},Q_{k}\right) \in 
\mathcal{S}^{\left( k-1\right) }\times \mathcal{S}^{\left( k\right) }\left[
Q_{k-1}\right] }\left( 2^{-p\limfunc{dist}\left( \left( Q_{k},Q_{k-1}\right)
\right) \delta }\right) ^{\frac{2}{2-p}}\right) ^{\frac{2-p}{2}} \\
&&\ \ \ \ \ \ \ \ \ \ \ \ \ \ \ \ \ \ \ \ \times \left( \sum_{\left(
Q_{k-1},Q_{k}\right) \in \mathcal{S}^{\left( k-1\right) }\times \mathcal{S}%
^{\left( k\right) }\left[ Q_{k-1}\right] }\left\vert \mathsf{P}_{\mathcal{C}%
_{\mathcal{S}^{\left( k\right) }}\left( Q_{k}\right) }^{\sigma }f\left(
x\right) \right\vert ^{2}\right) ^{\frac{p}{2}}d\sigma \left( x\right) \\
&\leq &N^{\frac{2-p}{2}}\int_{\mathbb{R}}\sum_{\left( Q_{k-1}\right) \in 
\mathcal{S}^{\left( k-1\right) }}\left\vert \mathsf{P}_{\mathcal{C}_{%
\mathcal{S}^{\left( k-1\right) }}\left( Q_{k-1}\right) }^{\sigma }f\left(
x\right) \right\vert ^{p}d\sigma \left( x\right) ,
\end{eqnarray*}%
to obtain%
\begin{eqnarray*}
&&\int_{\mathbb{R}}\sum_{\left( F,Q_{n},S,K\right) \in \Omega ^{4}}\Psi
_{Q_{n}}^{p}\left\vert \mathsf{P}_{\mathcal{C}_{\mathcal{S}^{\left( n\right)
}}\left( Q_{n}\right) }^{\sigma }f\left( x\right) \right\vert ^{p}d\sigma
\left( x\right) \\
&\lesssim &N^{\frac{2-p}{2}n}\int_{\mathbb{R}}\sum_{Q_{0}\in \mathcal{S}%
^{\left( 0\right) }}\left\vert \mathsf{P}_{\mathcal{C}_{\mathcal{S}^{\left(
0\right) }}\left( Q_{0}\right) }^{\sigma }f\left( x\right) \right\vert
^{p}d\sigma \left( x\right) =N^{\frac{2-p}{2}n}\sum_{F\in \mathcal{F}}\int_{%
\mathbb{R}}\left\vert \mathsf{P}_{\mathcal{C}_{\mathcal{F}}\left( F\right)
}^{\sigma }f\left( x\right) \right\vert ^{p}d\sigma \left( x\right) \\
&\lesssim &N^{\frac{2-p}{2}n}\sum_{F\in \mathcal{F}}\left\{ \left( \frac{1}{%
\left\vert F\right\vert _{\sigma }}\int_{F}\left\vert f\right\vert d\sigma
\right) ^{p}\left\vert F\right\vert _{\sigma }+\sum_{F^{\prime }\in 
\mathfrak{C}_{\mathcal{F}}\left( F\right) }\left( \frac{1}{\left\vert
F^{\prime }\right\vert _{\sigma }}\int_{F^{\prime }}\left\vert f\right\vert
d\sigma \right) ^{p}\left\vert F^{\prime }\right\vert _{\sigma }\right\}
\lesssim N^{\frac{2-p}{2}n}\left\Vert f\right\Vert _{L^{p}\left( \sigma
\right) }^{p},
\end{eqnarray*}%
where the last line follows from (\ref{en con}), the pointwise inequality, 
\begin{equation*}
\left\vert \mathsf{P}_{\mathcal{C}_{\mathcal{F}}\left( F\right) }^{\sigma
}f\left( x\right) \right\vert \lesssim \sup_{I\in \mathcal{C}_{\mathcal{F}%
}\left( F\right) }\left( E_{I}^{\sigma }\left\vert f\right\vert \right) 
\mathbf{1}_{F}\left( x\right) +\sum_{F^{\prime }\in \mathfrak{C}_{\mathcal{F}%
}\left( F\right) }\left( E_{F^{\prime }}^{\sigma }\left\vert f\right\vert
\right) \mathbf{1}_{F^{\prime }}\left( x\right) \lesssim \left(
E_{F}^{\sigma }\left\vert f\right\vert \right) \mathbf{1}_{F}\left( x\right)
+\sum_{F^{\prime }\in \mathfrak{C}_{\mathcal{F}}\left( F\right) }\left(
E_{F^{\prime }}^{\sigma }\left\vert f\right\vert \right) \mathbf{1}%
_{F^{\prime }}\left( x\right) ,
\end{equation*}%
the quasiorthogonality inequality (3) in the subsection on Carleson measures
above, and the $\sigma $-Carleson property of $\mathcal{F}$. This completes
the proof of the case $1<p\leq 2$ with the estimate%
\begin{eqnarray*}
&&\left\Vert \left\vert f\right\vert _{\mathcal{F},\mathcal{S}^{\left(
n-1\right) },\mathcal{S}^{\left( n\right) }}^{\left\{ \mathcal{C}_{\mathcal{S%
}^{\left( n\right) }}^{\left( s\right) }\left( A\right) \cap \Lambda
_{g_{A}}^{\omega }\right\} _{A\in \mathcal{S}^{\left( n\right)
}}}\right\Vert _{L^{p}\left( \omega \right) }\lesssim \left( C_{p}\theta ^{%
\frac{p}{2}}\right) ^{\frac{n}{p}}N^{n\frac{2-p}{2p}}\mathfrak{X}_{\mathcal{F%
};p}\left( \sigma ,\omega \right) \left\Vert f\right\Vert _{L^{p}\left(
\sigma \right) } \\
&\approx &C_{p}^{\frac{n}{p}}\theta ^{\frac{n}{2}}\left( \frac{C}{\theta }%
\right) ^{\frac{n}{p}-\frac{n}{2}}\mathfrak{X}_{\mathcal{F};p}\left( \sigma
,\omega \right) \left\Vert f\right\Vert _{L^{p}\left( \sigma \right)
}=\left( \left( C_{p}C^{1-\frac{p}{2}}\right) ^{\frac{p^{\prime }}{p}}\theta
\right) ^{\frac{n}{p^{\prime }}}\mathfrak{X}_{\mathcal{F};p}\left( \sigma
,\omega \right) \left\Vert f\right\Vert _{L^{p}\left( \sigma \right) }.
\end{eqnarray*}%
Now Lemma \ref{energy cond} yields, 
\begin{equation*}
\mathfrak{X}_{\mathcal{F};p}\left( \sigma ,\omega \right) \lesssim \mathcal{E%
}_{p}\left( \sigma ,\omega \right) \lesssim \mathfrak{T}_{H,p}^{\func{loc}%
}\left( \sigma ,\omega \right) ,
\end{equation*}%
which completes the proof of Lemma \ref{final level}.
\end{proof}

This finishes our control of the stopping form $\mathsf{B}_{\limfunc{stop}%
}\left( f,g\right) $\ in (\ref{stop est}) for $1<p<4$. The dual stopping
form requires $1<p^{\prime }<4$.

Now we turn to the final ingredient in the proof Theorems \ref{main glob}
and \ref{main}, which is the control of the extended energy characteristic
by the norm characteristic.

\section{Control of extended energy \label{ex en section}}

Fix $p\in \left( \frac{4}{3},2\right) \cup \left( 2,4\right) $ and consider
the two `dual' intertwining forms arising in the Intertwining Proposition, 
\begin{eqnarray*}
\mathsf{B}_{\limfunc{lower}}^{\limfunc{inter}}\left( f,g\right) &\equiv
&\sum_{F\in \mathcal{F}}\ \sum_{I:\ I\supsetneqq F}\ \left\langle H_{\sigma
}\left( \mathbf{1}_{I_{F}}\bigtriangleup _{I}^{\sigma }f\right) ,\mathsf{P}_{%
\mathcal{C}_{\mathcal{F}}\left( F\right) }^{\omega }g\right\rangle _{\omega
}\ , \\
\mathsf{B}_{\limfunc{upper}}^{\limfunc{inter}}\left( f,g\right) &\equiv
&\sum_{G\in \mathcal{G}}\ \sum_{J:\ J\supsetneqq G}\ \left\langle H_{\omega
}\left( \mathbf{1}_{J_{G}}\bigtriangleup _{J}^{\omega }g\right) ,\mathsf{P}_{%
\mathcal{C}_{\mathcal{G}}\left( G\right) }^{\sigma }f\right\rangle _{\sigma
}\ ,
\end{eqnarray*}%
With 
\begin{eqnarray*}
g_{F} &=&\mathsf{P}_{\mathcal{C}_{\mathcal{F}}\left( F\right) }^{\omega
}g=\sum_{J\in \mathcal{C}_{\mathcal{F}}\left( F\right) }\bigtriangleup
_{J}^{\omega }g\text{ }, \\
f_{F} &\equiv &\sum_{I:\ I\supsetneqq F}\mathbf{1}_{I_{F}}\bigtriangleup
_{I}^{\sigma }f=\sum_{I:\ I\supsetneqq F}\mathbf{1}_{I_{F}}\left\langle
f,h_{I}\right\rangle h_{I}=\sum_{I:\ I\supsetneqq F}\mathbb{E}_{\theta
I_{F}}f\ ,
\end{eqnarray*}%
the proof of the Intertwining Proposition shows that $\mathsf{B}_{\limfunc{%
lower}}^{\limfunc{inter}}\left( f,g\right) $ satisfies%
\begin{eqnarray}
&&\mathsf{B}_{\limfunc{lower}}^{\limfunc{inter}}\left( f,g\right)
=\sum_{F\in \mathcal{F}}\ \sum_{I:\ I\supsetneqq F}\ \left\langle H_{\sigma
}\left( \mathbf{1}_{I_{F}}\bigtriangleup _{I}^{\sigma }f\right)
,g_{F}\right\rangle _{\omega }  \label{proof shows} \\
&=&\sum_{F\in \mathcal{F}}\ \left\langle H_{\sigma }\left( \sum_{I:\
I\supsetneqq F}\mathbf{1}_{I_{F}}\bigtriangleup _{I}^{\sigma }f\right) ,%
\mathsf{P}_{\mathcal{C}_{\mathcal{F}}\left( F\right) }^{\omega
}g\right\rangle _{\omega }\equiv \sum_{F\in \mathcal{F}}\ \left\langle
H_{\sigma }f_{F},\mathsf{P}_{\mathcal{C}_{\mathcal{F}}\left( F\right)
}^{\omega }g\right\rangle _{\omega }  \notag \\
&=&\sum_{F\in \mathcal{F}}\ \left\langle H_{\sigma }\left( \mathbf{1}%
_{F}f_{F}\right) ,\mathsf{P}_{\mathcal{C}_{\mathcal{F}}\left( F\right)
}^{\omega }g\right\rangle _{\omega }+\sum_{F\in \mathcal{F}}\ \left\langle
H_{\sigma }\left( \mathbf{1}_{F^{c}}f_{F}\right) ,\mathsf{P}_{\mathcal{C}_{%
\mathcal{F}}\left( F\right) }^{\omega }g\right\rangle _{\omega }\equiv 
\mathsf{I}\left( f,g\right) +\mathsf{B}_{\limfunc{lower}}^{\limfunc{hole}%
}\left( f,g\right) \ ,  \notag
\end{eqnarray}%
where the form $\mathsf{I}\left( f,g\right) $ is controlled in (\ref{I})
above by local quadratric testing,%
\begin{equation*}
\left\vert \mathsf{I}\left( f,g\right) \right\vert =\left\vert \sum_{F\in 
\mathcal{F}}\ \left\langle H_{\sigma }\left( \mathbf{1}_{F}f_{F}\right) ,%
\mathsf{P}_{\mathcal{C}_{\mathcal{F}}\left( F\right) }^{\omega
}g\right\rangle _{\omega }\right\vert \lesssim \mathfrak{T}_{H,p}^{\ell ^{2},%
\limfunc{loc}}\left( \sigma ,\omega \right) \left\Vert f\right\Vert
_{L^{p}\left( \sigma \right) }\left\Vert g\right\Vert _{L^{p^{\prime
}}\left( \omega \right) }.
\end{equation*}

Let $\mathfrak{U}_{H,p,\limfunc{lower}}^{\ell ^{2},\limfunc{hole}}\left(
\sigma ,\omega \right) $ denote the best constant in the bilinear inequality
for $\mathsf{B}_{\limfunc{lower}}^{\limfunc{hole}}\left( f,g\right) $, i.e. 
\begin{equation}
\left\vert \mathsf{B}_{\limfunc{lower}}^{\limfunc{hole}}\left( f,g\right)
\right\vert =\left\vert \sum_{F\in \mathcal{F}}\ \left\langle H_{\sigma
}\left( \mathbf{1}_{F^{c}}f_{F}\right) ,\mathsf{P}_{\mathcal{C}_{\mathcal{F}%
}\left( F\right) }^{\omega }g\right\rangle _{\omega }\right\vert \lesssim 
\mathfrak{U}_{H,p,\limfunc{lower}}^{\ell ^{2},\limfunc{hole}}\left( \sigma
,\omega \right) \left\Vert f\right\Vert _{L^{p}\left( \sigma \right)
}\left\Vert g\right\Vert _{L^{p^{\prime }}\left( \omega \right) },
\label{hole ineq}
\end{equation}%
and similarly define the `dual' characteristic $\mathfrak{U}_{H,p^{\prime },%
\limfunc{upper}}^{\ell ^{2},\limfunc{hole}}\left( \sigma ,\omega \right) $
as the bound for the bilinear form $\mathsf{B}_{\limfunc{upper}}^{\limfunc{%
hole}}\left( f,g\right) $.

\begin{lemma}
\label{hole ext}For $1<p<\infty $ we have%
\begin{equation*}
\mathfrak{U}_{H,p,\limfunc{lower}}^{\ell ^{2},\limfunc{hole}}\left( \sigma
,\omega \right) \approx \mathcal{E}_{p}^{\ell ^{2},\limfunc{ext}}\left(
\sigma ,\omega \right) \text{ and }\mathfrak{U}_{H,p^{\prime },\limfunc{upper%
}}^{\ell ^{2},\limfunc{hole}}\left( \sigma ,\omega \right) \approx \mathcal{E%
}_{p^{\prime }}^{\ell ^{2},\limfunc{ext},\ast }\left( \omega ,\sigma \right)
.
\end{equation*}
\end{lemma}

\begin{proof}
Upon considering positive and negative parts of a function, we may assume
without loss of generality that both $f\geq 0$ and $g\geq 0$ when bounding\
the Hilbert transform bilinear form $\left\langle H_{\sigma
}f,g\right\rangle _{\omega }$ at the outset of the proof.

First note by duality that the bilinear inequality%
\begin{equation*}
\left\vert \mathsf{B}_{\limfunc{lower}}^{\limfunc{hole}}\left( f,g\right)
\right\vert =\left\vert \sum_{F\in \mathcal{F}}\left\langle \mathsf{P}_{%
\mathcal{C}_{\mathcal{F}}\left( F\right) }^{\omega }H_{\sigma }\left( 
\mathbf{1}_{F^{c}}f_{F}\right) ,g\right\rangle _{\omega }\right\vert
\lesssim \mathfrak{U}_{H,p,\limfunc{lower}}^{\ell ^{2},\limfunc{hole}}\left(
\sigma ,\omega \right) \left\Vert f\right\Vert _{L^{p}\left( \sigma \right)
}\left\Vert g\right\Vert _{L^{p^{\prime }}\left( \omega \right) }\ ,
\end{equation*}%
is equivalent to the norm inequality 
\begin{equation*}
\int_{\mathbb{R}}\left\vert \sum_{F\in \mathcal{F}}\mathsf{P}_{\mathcal{C}_{%
\mathcal{F}}\left( F\right) }^{\omega }H_{\sigma }\left( \mathbf{1}%
_{F^{c}}f_{F}\right) \right\vert ^{p}d\omega \lesssim \left( \mathfrak{U}%
_{H,p,\limfunc{lower}}^{\ell ^{2},\limfunc{hole}}\right) ^{p}\left( \sigma
,\omega \right) \left\Vert f\right\Vert _{L^{p}\left( \sigma \right) }^{p}\ .
\end{equation*}%
Now we exploit the fact that $\left\langle Z,h_{J}^{\omega }\right\rangle
_{\omega }\geq 0$, our assumption that $f\geq 0$, and the fact that the
support of the projection $\mathsf{P}_{\mathcal{C}_{\mathcal{F}}\left(
F\right) }^{\omega }=\sum_{I\in \mathcal{C}_{\mathcal{F}}\left( F\right)
}\bigtriangleup _{I}^{\omega }$ is \emph{separated} from the support of the
function $\mathbf{1}_{F^{c}}f_{F}$. By the Monotonicity Lemma,%
\begin{equation}
\left\langle H_{\sigma }\left( \mathbf{1}_{F^{c}}f_{F}\right) ,h_{J}^{\omega
}\right\rangle _{\omega }\approx \frac{\mathrm{P}\left( J,\mathbf{1}%
_{F^{c}}f_{F}\sigma \right) }{\ell \left( J\right) }\left\langle
Z,h_{J}^{\omega }\right\rangle _{\omega }\ ,  \label{ML}
\end{equation}%
since 
\begin{eqnarray*}
f_{F} &=&\sum_{m=0}^{N}\left( E_{F_{m+1}}^{\sigma }f\right) \ \mathbf{1}%
_{F_{m+1}\setminus F_{m}}=\left( E_{F}^{\sigma }f\right) \ \mathbf{1}%
_{F}+\sum_{m=0}^{N}\left( E_{\pi _{\mathcal{F}}^{m+1}F}^{\sigma }f\right) \ 
\mathbf{1}_{\pi _{\mathcal{F}}^{m+1}F\setminus \pi _{\mathcal{F}}^{m}F} \\
&=&\left( E_{F}^{\sigma }f\right) \ \mathbf{1}_{F}+\sum_{F^{\prime }\in 
\mathcal{F}:\ F\subset F^{\prime }}\left( E_{\pi _{\mathcal{F}}F^{\prime
}}^{\sigma }f\right) \ \mathbf{1}_{\pi _{\mathcal{F}}F^{\prime }\setminus
F^{\prime }}\ \geq 0.
\end{eqnarray*}

We conclude that $\mathfrak{U}_{H,p,\limfunc{lower}}^{\ell ^{2},\limfunc{hole%
}}\left( \sigma ,\omega \right) $ is comparable to the best constant, which
we continue to denote by $\mathfrak{U}_{H,p,\limfunc{lower}}^{\ell ^{2},%
\limfunc{hole}}\left( \sigma ,\omega \right) $, in the norm inequality,%
\begin{equation*}
\int_{\mathbb{R}}\left\vert \sum_{J\in \mathcal{C}_{\mathcal{F}}\left(
F\right) }\frac{\mathrm{P}\left( J,\sum_{F^{\prime }\in \mathcal{F}:\
F\subset F^{\prime }}\left( E_{\pi _{\mathcal{F}}F^{\prime }}^{\sigma
}f\right) \ \mathbf{1}_{\pi _{\mathcal{F}}F^{\prime }\setminus F^{\prime
}}\sigma \right) }{\ell \left( J\right) }\bigtriangleup _{J}^{\omega
}Z\left( x\right) \right\vert ^{p}d\omega \lesssim \left( \mathfrak{U}_{H,p,%
\limfunc{lower}}^{\ell ^{2},\limfunc{hole}}\left( \sigma ,\omega \right)
\right) ^{p}\left\Vert f\right\Vert _{L^{p}\left( \sigma \right) }^{p},\ \ \
\ \ \text{for }f\geq 0.
\end{equation*}%
By the square function theorem, $\mathfrak{U}_{H,p,\limfunc{lower}}^{\ell
^{2},\limfunc{hole}}$ is thus comparable to the best constant in the
inequality,%
\begin{eqnarray*}
&&\int_{\mathbb{R}}\left( \sum_{F\in \mathcal{F}}\sum_{J\in \mathcal{C}_{%
\mathcal{F}}\left( F\right) }\left( \frac{\mathrm{P}\left( J,\sum_{F^{\prime
}\in \mathcal{F}:\ F\subset F^{\prime }}\left( E_{\pi _{\mathcal{F}%
}F^{\prime }}^{\sigma }f\right) \ \mathbf{1}_{\pi _{\mathcal{F}}F^{\prime
}\setminus F^{\prime }}\sigma \right) }{\ell \left( J\right) }\right)
^{2}\left\vert \bigtriangleup _{J}^{\omega }Z\left( x\right) \right\vert
^{2}\right) ^{\frac{p}{2}}d\omega \\
&&\ \ \ \ \ \ \ \ \ \ \ \ \ \ \ \lesssim \left( \mathfrak{U}_{H,p,\limfunc{%
lower}}^{\ell ^{2},\limfunc{hole}}\left( \sigma ,\omega \right) \right)
^{p}\left\Vert f\right\Vert _{L^{p}\left( \sigma \right) }^{p},\ \ \ \ \ 
\text{for }f\geq 0,
\end{eqnarray*}%
equivalently,%
\begin{eqnarray*}
&&\int_{\mathbb{R}}\left( \sum_{F\in \mathcal{F}}\sum_{W\in \mathcal{M}_{%
\mathbf{r}-\limfunc{deep}}\left( F\right) }\left( \frac{\mathrm{P}\left(
W,\sum_{F^{\prime }\in \mathcal{F}:\ F\subset F^{\prime }}\left( E_{\pi _{%
\mathcal{F}}F^{\prime }}^{\sigma }f\right) \ \mathbf{1}_{\pi _{\mathcal{F}%
}F^{\prime }\setminus F^{\prime }}\sigma \right) }{\ell \left( W\right) }%
\right) ^{2}\left\vert \mathsf{P}_{\mathcal{C}_{\mathcal{F}}\left( F\right)
\cap \mathcal{D}\left[ W\right] }^{\omega }\right\vert Z\left( x\right)
^{2}\right) ^{\frac{p}{2}}d\omega \\
&&\ \ \ \ \ \ \ \ \ \ \ \ \ \ \ \ \ \ \ \ \lesssim \left( \mathfrak{U}_{H,p,%
\limfunc{lower}}^{\ell ^{2},\limfunc{hole}}\left( \sigma ,\omega \right)
\right) ^{p}\left\Vert f\right\Vert _{L^{p}\left( \sigma \right) }^{p},\ \ \
\ \ \text{for }f\geq 0,
\end{eqnarray*}%
where $\mathcal{M}_{\mathbf{r}-\limfunc{deep}}\left( F\right) $ consists of
the maximal $\mathbf{r}$-deeply embedded subcubes of $F$, and 
\begin{equation*}
\left\vert \mathsf{P}_{\mathcal{C}_{\mathcal{F}}\left( F\right) \cap 
\mathcal{D}\left[ W\right] }^{\omega }\right\vert Z\left( x\right) \equiv
\left( \sum_{J\in \mathcal{C}_{\mathcal{F}}\left( F\right) \cap \mathcal{D}%
\left[ W\right] }\left\vert \bigtriangleup _{J}^{\omega }Z\left( x\right)
\right\vert ^{2}\right) ^{\frac{1}{2}}.
\end{equation*}%
We can rewrite the sum $\sum_{F\in \mathcal{F}}\sum_{W\in \mathcal{M}_{%
\mathbf{r}-\limfunc{deep}}\left( F\right) }$ as $\sum_{W\in \mathbb{W}}$ due
to the fact that $W$ determines $F$ uniquely. If we write $\mathsf{P}%
_{W}^{\omega }=\mathsf{P}_{\mathcal{C}_{\mathcal{F}}\left( F\right) \cap 
\mathcal{D}\left[ W\right] }^{\omega }$ for $W\in \mathbb{W}$, then we have%
\begin{equation*}
\int_{\mathbb{R}}\left( \sum_{W\in \mathbb{W}}\left( \frac{\mathrm{P}\left(
W,\mathbf{1}_{F^{c}}f_{F}\sigma \right) }{\ell \left( W\right) }\right)
^{2}\left\vert \mathsf{P}_{\mathcal{C}_{\mathcal{F}}\left( F\right) \cap 
\mathcal{D}\left[ W\right] }^{\omega }\right\vert ^{2}Z\right) ^{\frac{p}{2}%
}d\omega \lesssim \left( \mathfrak{U}_{H,p,\limfunc{lower}}^{\ell ^{2},%
\limfunc{hole}}\left( \sigma ,\omega \right) \right) ^{p}\left\Vert
f\right\Vert _{L^{p}\left( \sigma \right) }^{p},\ \ \ \ \ \text{for }f\geq 0.
\end{equation*}

For $F\in \mathcal{F}$, where $\mathcal{F}$ is the collection of stopping
times associated with a function $f$, define $\Phi _{F}f$ by%
\begin{equation*}
\Phi _{F}f\equiv \mathbb{E}_{F}^{\sigma }f+f_{F}=\mathbb{E}_{F}^{\sigma
}f+\sum_{I:\ I\supsetneqq F}\mathbf{1}_{I_{F}}\bigtriangleup _{I}^{\sigma }f.
\end{equation*}%
Then the `plugged' norm%
\begin{equation*}
\int \left( \sum_{W\in \mathbb{W}}\left( \frac{\mathrm{P}\left( W,\mathbf{1}%
_{F}\Phi _{F}f\sigma \right) }{\ell \left( W\right) }\right) ^{2}\left\vert 
\mathsf{P}_{\mathcal{C}_{\mathcal{F}}\left( F\right) \cap \mathcal{D}\left[ W%
\right] }^{\omega }\right\vert ^{2}Z\right) ^{\frac{p}{2}}d\omega ,
\end{equation*}%
can be controlled using the fact that the collection $\mathbb{W}$ is $\sigma 
$-Carleson. Thus altogether we obtain that $\mathfrak{U}_{H,p,\limfunc{lower}%
}^{\ell ^{2},\limfunc{hole}}\left( \sigma ,\omega \right) $ is comparable to
the best constant $\mathcal{E}_{p}^{\ell ^{2},\limfunc{ext}}\left( \sigma
,\omega \right) $ in the extended energy inequality,%
\begin{equation*}
\int_{\mathbb{R}}\left( \sum_{W\in \mathbb{W}}\left( \frac{\mathrm{P}\left(
W,\Phi _{F}f\sigma \right) }{\ell \left( W\right) }\right) ^{2}\left\vert 
\mathsf{P}_{\mathcal{C}_{\mathcal{F}}\left( F\right) \cap \mathcal{D}\left[ W%
\right] }^{\omega }\right\vert ^{2}Z\right) ^{\frac{p}{2}}d\omega \lesssim
\left( \mathcal{E}_{p}^{\ell ^{2},\limfunc{ext}}\left( \sigma ,\omega
\right) \right) ^{p}\left\Vert f\right\Vert _{L^{p}\left( \sigma \right)
}^{p},\ \ \ \ \ \text{for }f\geq 0.
\end{equation*}%
This completes the proof of the lemma.
\end{proof}

Lemma \ref{hole ext}, together with the discussion at the beginning of this
section, and the earlier estimates on the other forms in this paper,
finishes the proofs of Theorems \ref{main glob}, \ref{main}, \ref{half} and %
\ref{necc func ener}.

\section{Concluding remarks on the stopping form}

The methods we used above for bounding the stopping form, actually yield the
following weaker form of (\ref{like}), with a smaller $\ell ^{q}$ norm
inside, 
\begin{eqnarray}
&&\left\Vert \left\vert \left\{ \alpha _{A}\left( S\right) \left( \frac{%
\mathrm{P}\left( K,\mathbf{1}_{F\setminus S}\sigma \right) }{\ell \left(
K\right) }\right) \mathsf{P}_{S;K}^{\omega ,\Lambda _{g_{A}}}Z\left(
x\right) \ \mathbf{1}_{K}\left( x\right) \right\} _{\left( F,A,S,K\right)
\in \Omega ^{4}}\right\vert _{\ell ^{q}}\right\Vert _{L^{p}\left( \omega
\right) }^{p}  \label{weaker} \\
&\lesssim &\left( C_{p}\theta \right) ^{n\left( 2-\frac{p}{q}\right) }%
\mathfrak{X}_{\mathcal{F};p}\left( \sigma ,\omega \right) ^{p}\left\Vert
f\right\Vert _{L^{p}\left( \sigma \right) }^{p}\ ,\ \ \ \ \ \text{for all }%
2\leq q\leq p<2q.  \notag
\end{eqnarray}%
This follows easily upon applying H\"{o}lder's inequality with exponent $%
\frac{p}{q}$ in (\ref{Holder}), and helps shed light on why our method fails
to prove the $L^{p}$ conjectures for all $1<p<\infty $. Indeed, experience
in this paper and in \cite{SaWi}, suggests that we must manipulate the left
hand side of (\ref{like}) or (\ref{weaker}) in the world of the measure $%
\omega $, to obtain some gain in $\theta $ before applying a characteristic,
such as $\mathfrak{X}_{\mathcal{F};p}\left( \sigma ,\omega \right) $, to
transfer analysis to the world of the measure $\sigma $. But this prevents
us from exploiting the `orthogonality' in the function $f$ that is hidden in
the coefficents $\alpha _{A}\left( S\right) $, while obtaining a gain in $%
\theta $, something that is avoided when $p=2$ using the Orthogonality
Argument of Lacey in \cite{Lac}, see the subsection on heuristics above.

As a consequence, we are forced to work in the world of the measure $\omega $
with `one hand tied behind our back', and perform sums in the factors $\Psi
_{A}$ which produce large negative powers of $\frac{1}{\theta }$, which
ultimately accounts for the restriction to $\frac{4}{3}<p<4$.

\end{document}